\documentclass[12pt]{amsart}

\usepackage{
	amsmath,
 	amsfonts,
  	amssymb,
 	amsthm,
        datetime,
	}

\usepackage{tikz-cd}
\usepackage{tikz}

\tikzset{Rightarrow/.style={double equal sign distance,>={Implies},->},
triple/.style={-,preaction={draw,Rightarrow}}}

\usepackage{rotating}

\usetikzlibrary{decorations.pathmorphing, decorations.pathreplacing,decorations.markings}
 \usetikzlibrary{backgrounds,}
 \usetikzlibrary{shapes}

\usepackage{color}

\usepackage{stmaryrd}
\usepackage[cal=boondoxo]{mathalfa}
\usepackage{mathtools}
\usepackage{bbm}

\usepackage{hyperref}
\usepackage[capitalise]{cleveref}

\usepackage{a4wide}

\usepackage{mathabx}

\usepackage{enumerate}
\usepackage{comment}

\theoremstyle{plain}
\newtheorem{thm}{Theorem}[section]
\newtheorem{cor}[thm]{Corollary}
\newtheorem{lem}[thm]{Lemma}
\newtheorem{prop}[thm]{Proposition}

\theoremstyle{definition}
\newtheorem{rem}[thm]{Remark}

\newtheorem{ex}[thm]{Example}
\newtheorem{exe}[thm]{Example}

\theoremstyle{definition}
\newtheorem{defn}[thm]{Definition}

\def\makeautorefname#1#2{\expandafter\def\csname#1autorefname\endcsname{#2}}
\makeautorefname{lem}{Lemma}%
\makeautorefname{prop}{Proposition}%
\makeautorefname{rem}{Remark}%
\makeautorefname{section}{Section}%

\newcommand{\cC}{\mathcal{C}}
\newcommand{\cD}{\mathcal{D}}

\newcommand{\cP}{\mathcal{P}}

\newcommand{\bC}{\mathbb{C}}

\newcommand{\bN}{\mathbb{N}}

\newcommand{\bQ}{\mathbb{Q}}

\newcommand{\bZ}{\mathbb{Z}}

\newcommand{\bg}{{\boldsymbol{g}}}

\newcommand{\bj}{{\boldsymbol{j}}}

\newcommand{\pp}[1]{(\!( #1 )\!)}
\newcommand{\opalg}[1]{{#1}^{\opp}}
\newcommand{\und}[1]{{\underline{#1}}}

\newcommand{\xrightarrowdbl}[2][]{%
  \xrightarrow[#1]{#2}\mathrel{\mkern-14mu}\rightarrow
}

\newcommand{\sssum}[1]{{\sum\limits_{\substack{#1}}}}
\newcommand{\ssbigoplus}[1]{{\bigoplus\limits_{\substack{#1}}}}

\newcommand{\bV}{\raisebox{0.03cm}{\mbox{\footnotesize$\textstyle{\bigwedge}$}}}

\DeclareMathOperator{\id}{Id}
\DeclareMathOperator{\Hom}{Hom}

\DeclareMathOperator{\Ind}{Ind}
\DeclareMathOperator{\Res}{Res}

\DeclareMathOperator{\HOM}{HOM}

\DeclareMathOperator{\RHOM}{RHOM}

\DeclareMathOperator{\cRHom}{\mathcal{RH}\mathit{om}}

\DeclareMathOperator{\Lderiv}{L}

\newcommand{\Lotimes}{\otimes^{\Lderiv}}

\DeclareMathOperator{\opp}{{op}}

\DeclareMathOperator{\Image}{im}
\DeclareMathOperator{\cok}{cok}

\DeclareMathOperator{\cone}{Cone}

\newcommand{\bKO}{\boldsymbol{K}_0}

\newcommand{\param}{\delta}

\newcommand{\slt}{\mathfrak{sl_2}}
\newcommand{\g}{{\mathfrak{g}}}
\newcommand{\p}{{\mathfrak{p}}}
\newcommand{\bo}{{\mathfrak{b}}}

\newcommand\E{{\sf{E}}}
\newcommand\F{{\sf{F}}}
\newcommand\K{{\sf{K}}}

\newcommand{\antimapslt}{\bar \tau}

\newcommand{\muT}{\dgT^{\und \mu}}
\newcommand{\muP}{\dgP^{\und \mu}}
\newcommand{\muS}{\dgS^{\und \mu}}
\newcommand{\muDiag}{\Gamma^{\und \mu}}

\newcommand{\stdFunct}{\mathbb{S}}

\newcommand{\dgT}{\mathrm{ \mathbf T}}
\newcommand{\dgP}{\mathrm{ \mathbf P}}
\newcommand{\dgS}{\mathrm{ \mathbf S}}

\newcommand{\BQ}{\mathrm{ \mathbf Q}}
\newcommand{\BS}{\mathrm{ \mathbf S}}
\newcommand{\BV}{\mathrm{ \mathbf V}}

\newcommand{\br}{\mathrm{ \mathbf p}} 

\DeclareMathOperator{\amod}{\mathrm{-}mod}

\DeclareMathOperator{\Hqe}{Hqe}

\DeclareMathOperator{\nh}{NH}

\DeclareMathOperator{\Pol}{Pol} 

\newcommand{\fl}{\rightarrow}

\newcommand{\curl}[1]{\mathcal{#1}}

\definecolor{myblue}{rgb}{0,.5,1}
\definecolor{mypurple}{rgb}{.75,0,.75}
\definecolor{mygreen}{rgb}{.3,.75,.1}

\definecolor{vcolor}{rgb}{0,.5,1}
\definecolor{pcolor}{rgb}{.75,0,.75}

\newcommand{\tikzdiagh}[2][]{\tikz[#1,anchor= center,very thick,baseline={([yshift=1ex+#2]current bounding box.center)}]}
\newcommand{\tikzdiagl}[1][]{\tikzdiagh[#1]{0}}
\newcommand{\tikzdiag}[1][]{\tikzdiagh[#1]{-1.5ex}}
\tikzstyle{tikzdot}=[fill, circle, inner sep=2pt]
\tikzstyle{tikzstar}=[star,star points=6,star point ratio=3,fill, inner sep=1pt]

\newcommand{\tikzbrace}[5][]{\draw[#1,decoration={brace,mirror,raise=-10pt},decorate]  (#2-.1,#4 -.35) -- node {#5} (#3+.1,#4-.35)}
\newcommand{\tikzbraceop}[5][]{\draw[#1,decoration={brace,raise=-10pt},decorate]  (#2-.1,#4 +.35) -- node {#5} (#3+.1,#4+.35)}

\tikzstyle{stdhl}=[red,double=red!30,double distance=1pt]
\tikzstyle{vstdhl}=[vcolor,double=vcolor!30,double distance=1pt]
\tikzstyle{pstdhl}=[mypurple,double=mypurple!30,double distance=1pt]
\tikzstyle{nail}=[draw,color=black,fill=white!30,circle,inner sep=2pt]

\newcommand{\criticaladash}{\tikzdiag[yscale=0.75,xscale=0.7]{
\draw (0,0.5) to (1,0.5) to (1,1) to (0,1) to (0,0.5);
\draw[dashed] (0.2,1) to (0.2,1.4);
\draw[dashed] (0.8,1) to (0.8,1.4);
\draw[dashed] (0.2,0.5) to (0.2,-1.1);
\draw[dashed] (0.8,0.5) to (0.8,-1.1);
\node at (0.5,0.75) {$D$};
\node at (1.6,0.15) {$\cdots$};
\draw (2,-0.1) to (2,0.4) to (3,0.4) to (3,-0.1) to (2,-0.1);
\node at (2.5, 0.15) {$D'$};
\draw[dashed] (2.2,-0.1) to (2.2,-1.1);
\draw[dashed] (2.8,-0.1) to (2.8,-1.1);
\draw[dashed] (2.2,0.4) to (2.2,1.4);
\draw[dashed] (2.8,0.4) to (2.8,1.4);
\node at (3.6,0.15) {$\cdots$};
\draw (4,-0.2) to (4,-0.7) to (5,-0.7) to (5,-0.2) to (4,-0.2);
\node at (4.5,-0.45) {$D''$};
\draw[dashed] (4.2,-0.7) to (4.2,-1.1);
\draw[dashed] (4.8,-0.7) to (4.8,-1.1);
\draw[dashed] (4.2,-0.2) to (4.2,1.4);
\draw[dashed] (4.8,-0.2) to (4.8,1.4);
}}

\newcommand{\criticalbdash}{\tikzdiag[yscale=0.75,xscale=0.7]{
\draw (2,0.5) to (3,0.5) to (3,1) to (2,1) to (2,0.5);
\draw[dashed] (2.2,1) to (2.2,1.4);
\draw[dashed] (2.8,1) to (2.8,1.4);
\draw[dashed] (2.2,0.5) to (2.2,-1.1);
\draw[dashed] (2.8,0.5) to (2.8,-1.1);
\node at (2.5,0.75) {$D'$};
\node at (1.6,0.15) {$\cdots$};
\draw (0,-0.1) to (0,0.4) to (1,0.4) to (1,-0.1) to (0,-0.1);
\node at (0.5, 0.15) {$D$};
\draw[dashed] (0.2,-0.1) to (0.2,-1.1);
\draw[dashed] (0.8,-0.1) to (0.8,-1.1);
\draw[dashed] (0.2,0.4) to (0.2,1.4);
\draw[dashed] (0.8,0.4) to (0.8,1.4);
\node at (3.6,0.15) {$\cdots$};
\draw (4,-0.2) to (4,-0.7) to (5,-0.7) to (5,-0.2) to (4,-0.2);
\node at (4.5,-0.45) {$D''$};
\draw[dashed] (4.2,-0.7) to (4.2,-1.1);
\draw[dashed] (4.8,-0.7) to (4.8,-1.1);
\draw[dashed] (4.2,-0.2) to (4.2,1.4);
\draw[dashed] (4.8,-0.2) to (4.8,1.4);
}}

\newcommand{\criticalcdash}{\tikzdiag[yscale=0.75,xscale=0.7]{
\draw (4,0.5) to (5,0.5) to (5,1) to (4,1) to (4,0.5);
\draw[dashed] (4.2,1) to (4.2,1.4);
\draw[dashed] (4.8,1) to (4.8,1.4);
\draw[dashed] (4.2,0.5) to (4.2,-1.1);
\draw[dashed] (4.8,0.5) to (4.8,-1.1);
\node at (4.5,0.75) {$D''$};
\node at (1.6,0.15) {$\cdots$};
\draw (0,-0.1) to (0,0.4) to (1,0.4) to (1,-0.1) to (0,-0.1);
\node at (0.5, 0.15) {$D$};
\draw[dashed] (0.2,-0.1) to (0.2,-1.1);
\draw[dashed] (0.8,-0.1) to (0.8,-1.1);
\draw[dashed] (0.2,0.4) to (0.2,1.4);
\draw[dashed] (0.8,0.4) to (0.8,1.4);
\node at (3.6,0.15) {$\cdots$};
\draw (2,-0.2) to (2,-0.7) to (3,-0.7) to (3,-0.2) to (2,-0.2);
\node at (2.5,-0.45) {$D'$};
\draw[dashed] (2.2,-0.7) to (2.2,-1.1);
\draw[dashed] (2.8,-0.7) to (2.8,-1.1);
\draw[dashed] (2.2,-0.2) to (2.2,1.4);
\draw[dashed] (2.8,-0.2) to (2.8,1.4);
}}

\newcommand{\criticalenddash}{\tikzdiag[yscale=0.75,xscale=0.7]{
\draw (4,0.5) to (5,0.5) to (5,1) to (4,1) to (4,0.5);
\draw[dashed] (4.2,1) to (4.2,1.4);
\draw[dashed] (4.8,1) to (4.8,1.4);
\draw[dashed] (4.2,0.5) to (4.2,-1.1);
\draw[dashed] (4.8,0.5) to (4.8,-1.1);
\node at (4.5,0.75) {$D''$};
\node at (1.6,0.15) {$\cdots$};
\draw (2,-0.1) to (2,0.4) to (3,0.4) to (3,-0.1) to (2,-0.1);
\node at (2.5, 0.15) {$D'$};
\draw[dashed] (2.2,-0.1) to (2.2,-1.1);
\draw[dashed] (2.8,-0.1) to (2.8,-1.1);
\draw[dashed] (2.2,0.4) to (2.2,1.4);
\draw[dashed] (2.8,0.4) to (2.8,1.4);
\node at (3.6,0.15) {$\cdots$};
\draw (0,-0.2) to (0,-0.7) to (1,-0.7) to (1,-0.2) to (0,-0.2);
\node at (0.5,-0.45) {$D$};
\draw[dashed] (0.2,-0.7) to (0.2,-1.1);
\draw[dashed] (0.8,-0.7) to (0.8,-1.1);
\draw[dashed] (0.2,-0.2) to (0.2,1.4);
\draw[dashed] (0.8,-0.2) to (0.8,1.4);
}}

\newcommand{\bcriticaldash}{\tikzdiag[yscale=0.75,xscale=0.7]{
\draw (0,0.5) to (1,0.5) to (1,1) to (0,1) to (0,0.5);
\draw[dashed] (0.2,1) to (0.2,1.4);
\draw[dashed] (0.8,1) to (0.8,1.4);
\draw[dashed] (0.2,0.5) to (0.2,-1.1);
\draw[dashed] (0.8,0.5) to (0.8,-1.1);
\node at (0.5,0.75) {$D$};
\node at (1.6,0.15) {$\cdots$};
\draw (4,-0.1) to (4,0.4) to (5,0.4) to (5,-0.1) to (4,-0.1);
\node at (4.5, 0.15) {$D''$};
\draw[dashed] (4.2,-0.1) to (4.2,-1.1);
\draw[dashed] (4.8,-0.1) to (4.8,-1.1);
\draw[dashed] (4.2,0.4) to (4.2,1.4);
\draw[dashed] (4.8,0.4) to (4.8,1.4);
\node at (3.6,0.15) {$\cdots$};
\draw (2,-0.2) to (2,-0.7) to (3,-0.7) to (3,-0.2) to (2,-0.2);
\node at (2.5,-0.45) {$D'$};
\draw[dashed] (2.2,-0.7) to (2.2,-1.1);
\draw[dashed] (2.8,-0.7) to (2.8,-1.1);
\draw[dashed] (2.2,-0.2) to (2.2,1.4);
\draw[dashed] (2.8,-0.2) to (2.8,1.4);
}}

\newcommand{\ccriticaldash}{\tikzdiag[yscale=0.75,xscale=0.7]{
\draw (2,0.5) to (3,0.5) to (3,1) to (2,1) to (2,0.5);
\draw[dashed] (2.2,1) to (2.2,1.4);
\draw[dashed] (2.8,1) to (2.8,1.4);
\draw[dashed] (2.2,0.5) to (2.2,-1.1);
\draw[dashed] (2.8,0.5) to (2.8,-1.1);
\node at (2.5,0.75) {$D'$};
\node at (1.6,0.15) {$\cdots$};
\draw (4,-0.1) to (4,0.4) to (5,0.4) to (5,-0.1) to (4,-0.1);
\node at (4.5, 0.15) {$D''$};
\draw[dashed] (4.2,-0.1) to (4.2,-1.1);
\draw[dashed] (4.8,-0.1) to (4.8,-1.1);
\draw[dashed] (4.2,0.4) to (4.2,1.4);
\draw[dashed] (4.8,0.4) to (4.8,1.4);
\node at (3.6,0.15) {$\cdots$};
\draw (0,-0.2) to (0,-0.7) to (1,-0.7) to (1,-0.2) to (0,-0.2);
\node at (0.5,-0.45) {$D$};
\draw[dashed] (0.2,-0.7) to (0.2,-1.1);
\draw[dashed] (0.8,-0.7) to (0.8,-1.1);
\draw[dashed] (0.2,-0.2) to (0.2,1.4);
\draw[dashed] (0.8,-0.2) to (0.8,1.4);
}}

\newcommand{\regcba}[2]{\mathord{
\begin{tikzpicture}[baseline=0,scale=1.211]
\draw[thick,color=black] (-1,1) -- (0,1) -- (0,0) -- (-1,0) -- (-1,1);
\draw[thick,color=black] (-0.5,-0.3) -- (-0.5,-0.6) -- (0.4,-0.6) -- (0.4,-0.3) -- (-0.5,-0.3);
\draw[thick,color=black] (0.1,-0.3) -- (0.1,1.2);
\draw[thick,color=black] (0.3,-0.3) -- (0.3,1.2);
\draw[thick,color=black] (-0.2,-0.3) -- (-0.2,0);
\draw[thick,color=black] (-0.4,0) -- (-0.4,-0.3);
\draw[thick,color=black] (-0.8,1) -- (-0.8,1.2);
\draw[thick,color=black] (-0.2,1) -- (-0.2,1.2);
\draw[thick,color=black] (-0.8,0) -- (-0.8,-0.8);
\draw[thick,color=black] (-0.6,0) -- (-0.6,-0.8);
\draw[thick,color=black] (-0.4,-0.6) -- (-0.4,-0.8);
\draw[thick,color=black] (0.3,-0.6) -- (0.3,-0.8);
\node at (-0.5,0.5) {$\scriptstyle{#1}$};
\node at (0,-0.45) {$\scriptstyle{#2}$};
\node at (-0.5,1.1) {$\ldots$};
\node at (0,-0.7) {$\ldots$};
\end{tikzpicture}}}

\newcommand{\regcbb}[3]{\mathord{
\begin{tikzpicture}[baseline=0,scale=1.211]
\draw[thick,color=black] (-1,1) -- (0,1) -- (0,0.7) -- (-1,0.7) -- (-1,1);
\draw[thick,color=black] (-0.5,0.35) -- (-0.5,0.05) -- (0,0.05) -- (0,0.35) -- (-0.5,0.35);
\draw[thick,color=black] (-0.5,-0.3) -- (-0.5,-0.6) -- (0.4,-0.6) -- (0.4,-0.3) -- (-0.5,-0.3);

\draw[thick,color=black] (-0.35,0.7) -- (-0.35,0.35);
\draw[thick,color=black] (-0.15,0.7) -- (-0.15,0.35);
\draw[thick,color=black] (-0.35,-0.3) -- (-0.35,0.05);
\draw[thick,color=black] (-0.15,-0.3) -- (-0.15,0.05);
\draw[thick,color=black] (-0.8,0.7) -- (-0.8,-0.8);
\draw[thick,color=black] (-0.6,0.7) -- (-0.6,-0.8);

\draw[thick,color=black] (-0.4,-0.6) -- (-0.4,-0.8);
\draw[thick,color=black] (0.3,-0.6) -- (0.3,-0.8);
\draw[thick,color=black] (0.1,-0.3) -- (0.1,1.2);
\draw[thick,color=black] (0.3,-0.3) -- (0.3,1.2);
\draw[thick,color=black] (-0.8,1) -- (-0.8,1.2);
\draw[thick,color=black] (-0.2,1) -- (-0.2,1.2);
\node at (-0.5,0.85) {$\scriptstyle{#1}$};
\node at (-0.25,0.2) {$\scriptstyle{#2}$};
\node at (0,-0.45) {$\scriptstyle{#3}$};
\node at (-0.5,1.1) {$\ldots$};
\node at (0,-0.7) {$\ldots$};
\end{tikzpicture}}}

\newcommand{\regcbc}[2]{\mathord{
\begin{tikzpicture}[baseline=0,scale=1.211]
\draw[thick,color=black] (-1,1) -- (0,1) -- (0,0.7) -- (-1,0.7) -- (-1,1);
\draw[thick,color=black] (-0.5,0.4) -- (0.4,0.4) -- (0.4,-0.6) -- (-0.5,-0.6) -- (-0.5,0.4);
\draw[thick,color=black] (0.1,0.4) -- (0.1,1.2);
\draw[thick,color=black] (0.3,0.4) -- (0.3,1.2);
\draw[thick,color=black] (-0.3,0.7) -- (-0.3,0.4);
\draw[thick,color=black] (-0.1,0.7) -- (-0.1,0.4);
\draw[thick,color=black] (-0.8,0.7) -- (-0.8,-0.8);
\draw[thick,color=black] (-0.6,0.7) -- (-0.6,-0.8);
\draw[thick,color=black] (-0.4,-0.6) -- (-0.4,-0.8);
\draw[thick,color=black] (0.3,-0.6) -- (0.3,-0.8);
\draw[thick,color=black] (-0.8,1) -- (-0.8,1.2);
\draw[thick,color=black] (-0.2,1) -- (-0.2,1.2);
\node at (-0.5,0.85) {$\scriptstyle{#1}$};
\node at (0,-0.1) {$\scriptstyle{#2}$};
\node at (-0.5,1.1) {$\ldots$};
\node at (0,-0.7) {$\ldots$};
\end{tikzpicture}}}

\newcommand{\inccba}[1]{\mathord{
\begin{tikzpicture}[baseline=0,scale=1.211]
\draw[thick,color=black] (-0.5,1) -- (-0.5,-0.6) -- (0.9,-0.6) -- (0.9,1) -- (-0.5,1);
\draw[thick,color=black] (-0.3,-0.6) -- (-0.3,-0.8);
\draw[thick,color=black] (0.7,-0.6) -- (0.7,-0.8);
\draw[thick,color=black] (-0.3,1) -- (-0.3,1.2);
\draw[thick,color=black] (0.7,1) -- (0.7,1.2);
\node at (0.2,0.2) {$\scriptstyle{#1}$};
\node at (0.2,1.1) {$\ldots$};
\node at (0.2,-0.7) {$\ldots$};
\end{tikzpicture}}}

\newcommand{\inccbb}[2]{\mathord{
\begin{tikzpicture}[baseline=0,scale=1.211]
\draw[thick,color=black] (-0.5,1) -- (-0.5,-0.6) -- (0.9,-0.6) -- (0.9,1) -- (-0.5,1);
\draw[thick,color=black] (-0.3,0.8) -- (0.7,0.8) -- (0.7,-0.4) -- (-0.3,-0.4) -- (-0.3,0.8);
\draw[thick,color=black] (-0.1,0.5) -- (0.5,0.5) -- (0.5,-0.1) -- (-0.1,-0.1) -- (-0.1,0.5);
\draw[thick,color=black] (-0.3,-0.6) -- (-0.3,-0.8);
\draw[thick,color=black] (0.7,-0.6) -- (0.7,-0.8);
\draw[thick,color=black] (-0.3,1) -- (-0.3,1.2);
\draw[thick,color=black] (0.7,1) -- (0.7,1.2);
\draw[thick,color=black] (0.3,0.5) -- (0.3,0.8);
\draw[thick,color=black] (0.1,0.5) -- (0.1,0.8);
\draw[thick,color=black] (0.3,-0.1) -- (0.3,-0.4);
\draw[thick,color=black] (0.1,-0.1) -- (0.1,-0.4);
\node at (0.2,0.2) {$\scriptstyle{#1}$};
\node at (0.8,-0.5) {$\scriptstyle{#2}$};
\node at (0.2,1.1) {$\ldots$};
\node at (0.2,-0.7) {$\ldots$};
\end{tikzpicture}}}

\newcommand{\leftindcba}[3]{\mathord{
\begin{tikzpicture}[baseline=0,scale=1.211]
\draw[thick,color=black] (1.5,1) -- (0.6,1) -- (0.6,0.6) -- (1,0.6) -- (1,0)-- (1.5,0) -- (1.5,1);
\draw[thick,color=black] (0.9,0.4) -- (0.1,0.4) -- (0.1,0) -- (0.9,0) -- (0.9,0.4);
\draw[thick,color=black] (1.5,-0.2) -- (0.6,-0.2) -- (0.6,-0.6) -- (1.5,-0.6) -- (1.5,-0.2);
\draw[thick,color=black] (0.4,0.4) -- (0.4,1.2);
\draw[thick,color=black] (0.2,0.4) -- (0.2,1.2);
\draw[thick,color=black] (0.4,0) -- (0.4,-0.8);
\draw[thick,color=black] (0.2,0) -- (0.2,-0.8);
\draw[thick,color=black] (1.3,1) -- (1.3,1.2);
\draw[thick,color=black] (0.8,1) -- (0.8,1.2);
\draw[thick,color=black] (1.3,-0.6) -- (1.3,-0.8);
\draw[thick,color=black] (0.8,-0.6) -- (0.8,-0.8);
\draw[thick,color=black] (1.3,-0.2) -- (1.3,0);
\draw[thick,color=black] (1.1,-0.2) -- (1.1,0);
\draw[thick,color=black] (0.7,0.4) -- (0.7,0.6);
\draw[thick,color=black] (0.85,0.4) -- (0.85,0.6);
\draw[thick,color=black] (0.7,0) -- (0.7,-0.2);
\draw[thick,color=black] (0.85,0) -- (0.85,-0.2);
\node at (1.15,0.8) {$\scriptstyle{#1}$};
\node at (1,-0.4) {$\scriptstyle{#3}$};
\node at (0.5,0.2) {$\scriptstyle{#2}$};
\node at (1.05,1.1) {$\ldots$};
\node at (1.05,-0.7) {$\ldots$};
\end{tikzpicture}}}

\newcommand{\leftindcbb}[4]{\mathord{
\begin{tikzpicture}[baseline=0,scale=1.211]
\draw[thick,color=black] (1.5,1) -- (0.6,1) -- (0.6,0.6) -- (1.5,0.6) -- (1.5,1);
\draw[thick,color=black] (1.5,0.4) -- (1,0.4) -- (1,0) -- (1.5,0) -- (1.5,0.4);
\draw[thick,color=black] (0.9,0.4) -- (0.1,0.4) -- (0.1,0) -- (0.9,0) -- (0.9,0.4);
\draw[thick,color=black] (1.5,-0.2) -- (0.6,-0.2) -- (0.6,-0.6) -- (1.5,-0.6) -- (1.5,-0.2
);
\draw[thick,color=black] (0.4,0.4) -- (0.4,1.2);
\draw[thick,color=black] (0.2,0.4) -- (0.2,1.2);
\draw[thick,color=black] (0.4,0) -- (0.4,-0.8);
\draw[thick,color=black] (0.2,0) -- (0.2,-0.8);
\draw[thick,color=black] (1.3,1) -- (1.3,1.2);
\draw[thick,color=black] (0.8,1) -- (0.8,1.2);
\draw[thick,color=black] (1.3,-0.6) -- (1.3,-0.8);
\draw[thick,color=black] (0.8,-0.6) -- (0.8,-0.8);
\draw[thick,color=black] (1.3,0.4) -- (1.3,0.6);
\draw[thick,color=black] (1.1,0.4) -- (1.1,0.6);
\draw[thick,color=black] (1.3,-0.2) -- (1.3,0);
\draw[thick,color=black] (1.1,-0.2) -- (1.1,0);
\draw[thick,color=black] (0.7,0.4) -- (0.7,0.6);
\draw[thick,color=black] (0.85,0.4) -- (0.85,0.6);
\draw[thick,color=black] (0.7,0) -- (0.7,-0.2);
\draw[thick,color=black] (0.85,0) -- (0.85,-0.2);
\node at (1,0.8) {$\scriptstyle{#1}$};
\node at (1,-0.4) {$\scriptstyle{#4}$};
\node at (0.5,0.2) {$\scriptstyle{#3}$};
\node at (1.25,0.2) {$\scriptstyle{#2}$};
\node at (1.05,1.1) {$\ldots$};
\node at (1.05,-0.7) {$\ldots$};
\end{tikzpicture}}}

\newcommand{\leftindcbc}[3]{\mathord{
\begin{tikzpicture}[baseline=0,scale=1.211]
\draw[thick,color=black] (1.5,1) -- (0.6,1) -- (0.6,0.6) -- (1.5,0.6) -- (1.5,1);
\draw[thick,color=black] (0.9,0.4) -- (0.1,0.4) -- (0.1,0) -- (0.9,0) -- (0.9,0.4);
\draw[thick,color=black] (1.5,0.4) -- (1,0.4) -- (1,-0.2) -- (0.6,-0.2) -- (0.6,-0.6) -- (1.5,-0.6) -- (1.5,0.4);
\draw[thick,color=black] (0.4,0.4) -- (0.4,1.2);
\draw[thick,color=black] (0.2,0.4) -- (0.2,1.2);
\draw[thick,color=black] (0.4,0) -- (0.4,-0.8);
\draw[thick,color=black] (0.2,0) -- (0.2,-0.8);
\draw[thick,color=black] (1.3,1) -- (1.3,1.2);
\draw[thick,color=black] (0.8,1) -- (0.8,1.2);
\draw[thick,color=black] (1.3,-0.6) -- (1.3,-0.8);
\draw[thick,color=black] (0.8,-0.6) -- (0.8,-0.8);
\draw[thick,color=black] (1.3,0.4) -- (1.3,0.6);
\draw[thick,color=black] (1.1,0.4) -- (1.1,0.6);
\draw[thick,color=black] (0.7,0.4) -- (0.7,0.6);
\draw[thick,color=black] (0.85,0.4) -- (0.85,0.6);
\draw[thick,color=black] (0.7,0) -- (0.7,-0.2);
\draw[thick,color=black] (0.85,0) -- (0.85,-0.2);
\node at (1,0.8) {$\scriptstyle{#1}$};
\node at (0.5,0.2) {$\scriptstyle{#2}$};
\node at (1.15,-0.4) {$\scriptstyle{#3}$};
\node at (1.05,1.1) {$\ldots$};
\node at (1.05,-0.7) {$\ldots$};
\end{tikzpicture}}}

\title[Categorification of infinite-dimensional $\mathfrak{sl}_2$-modules I]
{Categorification of infinite-dimensional $\mathfrak{sl}_2$-modules and braid group 2-actions I : tensor products}

\author{Benjamin Dupont}
\address{Institut Camille Jordan\\
43 rue du 11 Novembre 1918\\ 
69100 Villeurbanne \\ 
France}
\email{bdupont@math.univ-lyon1.fr}

\author{Gr\'egoire Naisse}
\address{Max-Planck Institute for Mathematics\\
 Vivatsgasse 7 \\ 
53111 Bonn\\ 
Germany}
\email{gregoire.naisse@gmail.com}
\begin{document}

\begin{abstract}
This is the first part of a series of two papers aiming to construct a categorification of the braiding on tensor products of Verma modules, and in particular of the Lawrence--Krammer--Bigelow representations.  \\
In this part, we categorify all tensor products of Verma modules and integrable modules for quantum $\mathfrak{sl_2}$. 
The categorification is given by derived categories of dg versions of KLRW algebras which generalize both the tensor product algebras of Webster, and the dg-algebras used by Lacabanne, the second author and Vaz.  
We compute a basis for these dgKLRW algebras by using rewriting methods modulo braid-like isotopy, which we develop in an Appendix.  
\end{abstract}


\maketitle



\section{Introduction}\label{sec:intro}

Categorification was motivated since its beginning by low-dimensional topology and physics. 
For instance, one of the goals of the program of categorifying quantum groups was to give a representation theoretic explanation for the existence of link homology theories. 
Indeed  Khovanov~\cite{khovanovJones} and Khovanov--Rozansky~\cite{KR1}
constructed categorifications of the Reshetikhin--Turaev~\cite{RTinv} polynomial link invariants associated to (the fundamental representations of) quantum $\mathfrak{sl}_n$. However their constructions rely on the categorification of certain combinatorial descriptions of the link invariants, and not on the representation theoretic ones. 

\smallskip

The above-mentioned program has been very fruitful since its start with the seminal work of  Bernstein--Frenkel--Khovanov~\cite{BFK} and Frenkel--Khovanov--Stroppel~\cite{FKS} who gave a categorification of the tensor products of quantum $\mathfrak{sl}_2$ fundamental representations using category~$\mathcal{O}$.
Categorification of Lusztig integral versions of the quantum groups was developed by Khovanov--Lauda~\cite{KL1,KL2,KL3} and independently Rouquier~\cite{rouquier}, extending on the grounding work of Chuang--Rouquier~\cite{chuangrouquier} and Lauda~\cite{laudasl2}. 
At the heart of these constructions are the \emph{KLR algebras}. These are $\bZ$-graded algebras which control the higher structure between compositions of categorical analog of the Chevalley generators. 
Categorification of the integrable modules for all quantum Kac--Moody algebras was conjectured in \cite{KL1} and proved in \cite{kashiwara} and independently in \cite{webster}, using certain finite dimensional quotients of KLR algebras called \emph{cyclotomic quotients}. 
More precisely, to each Kac--Moody algebra $\bg$ is associated a KLR algebra $R_\g$, and to each integral dominant $\bg$-weight $\Lambda$ is associated a quotient $R^\Lambda_\g$. The category of graded modules over $R^\Lambda_\g$ categorifies the integrable $U_q(\bg)$-module $V(\Lambda)$ of highest weight $\Lambda$.
Categorifications of all tensor products of integrable modules were constructed by Webster in \cite{webster}, using KLR-like diagrammatic algebras  that we refer to as \emph{KLRW algebras}, generalizing $R^\Lambda_\g$. He also defined a categorical braid group action on his construction, giving a higher version of the action of the $R$-matrix, as well as higher versions of evaluation and coevalution maps. These allowed the construction of homology link invariants for any $\g$ which coincides with Khovanov--Rozansky for quantum~$\mathfrak{sl}_n$ \cite{MW}. 
Alternatively, these link homologies can also be obtained from higher representation theory of quantum groups through a categorical instance of skew Howe duality~\cite{khovanovIsSkewHowe}.

\smallskip

While the theory of categorification of integrable modules is already well-studied and understood, with deep connections to geometry (e.g. \cite{coherentSheavesCatAction,catgeom}), 
to category $\mathcal{O}$ (e.g. \cite{FKS,stroppel}) and to low-dimensional topology (e.g. \cite{webster, khovanovIsSkewHowe}),
the categorification of infinite dimensional (in the sense non-integrable) representations is still quite new and not so well understood. The second author and Vaz constructed categorifications of universal Verma modules for $\slt$ in \cite{NV1,NV2}, and extended it to any generic parabolic Verma module for any quantum Kac--Moody algebra in \cite{NV3}. They also showed in \cite{NVhomfly} that their construction is related to Khovanov--Rozansky triply-graded link homology \cite{KR2}. Moreover, in a collaboration \cite{LNV} with Lacabanne, they gave a categorification of the tensor product of a Verma module with multiple integrable modules for quantum $\slt$. They also showed that their construction yields a categorification of the blob algebra of Martin--Saleur \cite{blob}, which allow the construction of invariants of tangles in the annulus. 

One of the main ingredients in the categorification of Verma modules in the above-mentioned papers is the notion of a \emph{dg-enhancement}. The idea is to replace the cyclotomic quotient of the KLR algebra by a resolution of the quotiented ideal. It turns out that all cyclotomic quotients can then be encoded by a universal dg-algebra that we refer to as \emph{dgKLR algebra}, with the same underlying graded algebra but equipping it with different differentials $d_\Lambda$ (there is one for each choice of integral highest weight $\Lambda$). 
The dg-algebra with differential $d_\Lambda$ is then quasi-isomorphic to the cyclotomic quotient $R^\Lambda_\g$. 
Setting the differential to zero instead yields a categorification of a Verma module.

\subsection{Content of the paper}

This is the first part of a series of two papers aiming to construct and study more general tensor products of Verma and integrable modules. In this first part, we propose a categorification of any such tensor product for quantum $\slt$ using \emph{dgKLRW algebras}, generalizing the construction in \cite{webster} and in \cite{LNV}. In a second part in preparation \cite{secondPart}, we construct a categorical braid group action lifting the action of the $R$-matrix. By considering the categorical analog of Jackson--Kerler \cite{jacksonkerler}, this yields categorifications of the Burau and of the two parameters Lawrence--Krammer--Bigelow representations by restricting to certain categorified weight spaces.

\subsubsection{The dgKLRW algebras}

KLR algebras are usually defined by generators and relations, and pictured in the form of braid-like diagrams with strands colored by simple roots and decorated by dots. Since we will consider only the $\slt$ case here, all strands will be implicitly colored by the unique simple root of $\slt$, and drawn as a solid black line. 
For a string of dominant integral weights $\und \mu = (\mu_1, \dots, \mu_r)$, one defines the KLRW algebra $T^{\und \mu}$ by considering KLR-like diagrams, but containing $r$ additional red strands labeled from left to right by $\mu_1, \dots, \mu_r$, and that are not allowed to intersect each other. These red strands respect the following local relations with the black strands, depending on their label $\mu_i$:
\begin{equation}\label{eq:R2intro}
\begin{split}
\tikzdiagh[yscale=1.5]{0}{
	\draw (1,0) ..controls (1,.25) and (0,.25) .. (0,.5)..controls (0,.75) and (1,.75) .. (1,1)  ;
	\draw[stdhl] (0,0) node[below]{\small $\mu_i$} ..controls (0,.25) and (1,.25) .. (1,.5) ..controls (1,.75) and (0,.75) .. (0,1)  ;
} 
\ = \ 
\tikzdiagh[yscale=1.5]{0}{
	\draw[stdhl] (0,0) node[below]{\small $\mu_i$} -- (0,1)  ;
	\draw (1,0) -- (1,1)  node[midway,tikzdot]{}  node[midway,xshift=1.75ex,yshift=.75ex]{\small $\mu_i$} ;
} 
\qquad \qquad
\tikzdiagh[yscale=1.5]{0}{
	\draw (0,0) ..controls (0,.25) and (1,.25) .. (1,.5) ..controls (1,.75) and (0,.75) .. (0,1)  ;
	\draw[stdhl] (1,0) node[below]{\small $\mu_i$} ..controls (1,.25) and (0,.25) .. (0,.5)..controls (0,.75) and (1,.75) .. (1,1)  ;
} 
\ = \ 
\tikzdiagh[yscale=1.5]{0}{
	\draw (0,0) -- (0,1)  node[midway,tikzdot]{}   node[midway,xshift=1.75ex,yshift=.75ex]{\small $\mu_i$} ;
	\draw[stdhl] (1,0) node[below]{\small $\mu_i$} -- (1,1)  ;
} 
\\
\tikzdiagh[scale=1.5]{0}{
	\draw  (0,0) .. controls (0,0.25) and (1, 0.5) ..  (1,1);
	\draw  (1,0) .. controls (1,0.5) and (0, 0.75) ..  (0,1);
	\draw [stdhl] (0.5,0)node[below]{\small $\mu_i$}  .. controls (0.5,0.25) and (0, 0.25) ..  (0,0.5)
		 	  .. controls (0,0.75) and (0.5, 0.75) ..  (0.5,1);
} 
\ = \ 
\tikzdiagh[scale=1.5,xscale=-1]{0}{
	\draw  (0,0) .. controls (0,0.25) and (1, 0.5) ..  (1,1);
	\draw (1,0)  .. controls (1,0.5) and (0, 0.75) ..  (0,1);
	\draw [stdhl]  (0.5,0) node[below]{\small $\mu_i$} .. controls (0.5,0.25) and (0, 0.25) ..  (0,0.5)
		 	  .. controls (0,0.75) and (0.5, 0.75) ..  (0.5,1);
} 
\ + \sssum{u+v=\\\mu_i-1} \ 
\tikzdiagh[scale=1.5]{0}{
	\draw  (0,0) -- (0,1) node[midway,tikzdot]{} node[midway,xshift=-1.5ex,yshift=.75ex]{\small $u$};
	\draw  (1,0) --  (1,1) node[midway,tikzdot]{} node[midway,xshift=1.5ex,yshift=.75ex]{\small $v$};
	\draw [stdhl] (0.5,0)node[below]{\small $\mu_i$}  --  (0.5,1);
}
\end{split}
\qquad \text{for $\mu_i \in \bN$,}
\end{equation}
where a non-negative label $k$ next to a dot means we put $k$ consecutive dots. 
In Webster's setting \cite{webster}, one also has to quotient by the \emph{violating condition} stating that we kill any diagram with a black strand at the left of the leftmost red strand:
\[
\tikzdiagh[yscale=1.5]{0}{
	\draw (0,0) -- (0,1);
	\draw[stdhl] (1,0) node[below]{\small $\mu_1$} -- (1,1)  ;
} 
\ = 0,
\]
which plays role of the cyclotomic quotient condition. 
Categories of (graded) modules over $T^{\und \mu}$ categorify the tensor product $V(\mu_1) \otimes \cdots \otimes V(\mu_r)$.

The dgKLRW algebras that we consider here are similar, but also adding blue strands for the non-integral weights $\mu_i$ (i.e. the Verma tensor factors). These blue strands respect degenerated braid-type relations:
\begin{align}\label{eq:R2bisintro}
	\tikzdiagl[yscale=1.5]{
		\draw (1,0) ..controls (1,.25) and (0,.25) .. (0,.5)..controls (0,.75) and (1,.75) .. (1,1)  ;
		\draw[vstdhl] (0,0)node[below]{\small $\mu_i$} ..controls (0,.25) and (1,.25) .. (1,.5) ..controls (1,.75) and (0,.75) .. (0,1)  ;
	} 
	\ &= 0,
	&
	\tikzdiagl[yscale=1.5]{
		\draw (0,0) ..controls (0,.25) and (1,.25) .. (1,.5) ..controls (1,.75) and (0,.75) .. (0,1)  ;
		\draw[vstdhl] (1,0)node[below]{\small $\mu_i$} ..controls (1,.25) and (0,.25) .. (0,.5)..controls (0,.75) and (1,.75) .. (1,1)  ;
	} 
	\ &= 0,
&
	\tikzdiagl[scale=1.5]{
		\draw  (0,0) .. controls (0,0.25) and (1, 0.5) ..  (1,1);
		\draw  (1,0) .. controls (1,0.5) and (0, 0.75) ..  (0,1);
		\draw [vstdhl] (0.5,0)node[below]{\small $\mu_i$}  .. controls (0.5,0.25) and (0, 0.25) ..  (0,0.5)
			 	  .. controls (0,0.75) and (0.5, 0.75) ..  (0.5,1);
	} 
	\ &= \ 
	\tikzdiagl[scale=1.5,xscale=-1]{
		\draw  (0,0) .. controls (0,0.25) and (1, 0.5) ..  (1,1);
		\draw (1,0)  .. controls (1,0.5) and (0, 0.75) ..  (0,1);
		\draw [vstdhl]  (0.5,0)node[below]{\small $\mu_i$}  .. controls (0.5,0.25) and (0, 0.25) ..  (0,0.5)
			 	  .. controls (0,0.75) and (0.5, 0.75) ..  (0.5,1);
	} 
\qquad \text{for $\mu_i$ non-integral.}
\end{align}
Moreover, we need to replace the violating quotient condition by a dg-enhancement, meaning we add a new generator connecting the first black strand with the first colored strand, with a differential replacing the relations implied by \cref{eq:R2intro} for the first colored strand:
\begin{align*}
d_{\mu}\left(
\tikzdiag[xscale=2]{
	 \draw (.5,-.5) .. controls (.5,-.25) .. (0,0) .. controls (.5,.25) .. (.5,.5);
          \draw[stdhl] (0,-.5) node[below]{\small $\mu_1$}-- (0,.5) node [midway,nail]{};
  }
  \right)
  &:= 
  \tikzdiag[xscale=2]{
	 \draw (.5,-.5) -- (.5,.5) node[midway,tikzdot]{} node[midway, xshift=2ex,yshift=.75ex]{$\mu_1$};
          \draw[stdhl] (0,-.5) node[below]{\small $\mu_1$}-- (0,.5);
    }
    &&\text{ if $\mu_1 \in \bN$}, 
 \\
 d_{\mu}\left(
\tikzdiag[xscale=2]{
	 \draw (.5,-.5) .. controls (.5,-.25) .. (0,0) .. controls (.5,.25) .. (.5,.5);
          \draw[vstdhl] (0,-.5) node[below]{\small $\mu_1$}-- (0,.5) node [midway,nail]{};
  }
  \right)
  &:= 0, 
  &&\text{ if $\mu_1$ is non-integral},
\end{align*}
see \cref{defn:dgKLRW} for a precise definition. 
The derived category of dg-modules over a dgKLRW algebra categorifies the corresponding tensor product of Verma and integrable modules. Moreover, it comes with a dg-categorical action of quantum $\slt$ (in the sense of \cite[\S 7]{NV3}) by the usual setup of acting by induction/restriction functor along the map that adds a vertical black strand at the right of a diagram. 

One of the difficulties in proving such statements is that one usually relies on the use of an explicit basis of the dgKLRW algebra. While finding a candidate basis and proving that it generates the algebra is not a difficult task, proving the linear independence can be more challenging. A classical way of doing this is to construct a faithful action on a polynomial space, and show that the candidate elements act by linearly independent operators. 
However, the degenerate nature of the braid-moves in \cref{eq:R2bisintro} that we need to consider for the categorification of the Verma modules prevent the construction of such an action (at least in an obvious way). 
To solve this issue, we apply tools from rewriting theory up to braid-like isotopy, as developed in \cref{sec:rewritingmethods}. We refer to Sections \ref{sec:introrewriting} and \ref{sec:prelimrewriting} for more explanation about rewriting theory.

\subsubsection{Derived standardly stratified structure}

An important ingredient in the categorification of tensor products in \cite{FKS, webster} is the notion of standardly stratified categories, which are generalizations of highest weight categories, already abstracting the structure of a BGG category $\mathcal{O}$. Indeed, the KLRW algebras are standardly stratified, and the classes of standard modules correspond to induced basis elements of the tensor product in the Grothendieck group.  Furthermore, the standardization functor can be interpreted as the categorification of the inclusion of each factor into the tensor product. This structure is also mandatory to get uniqueness results as in \cite{losevwebster}.

In the case of the dgKLRW, one does not obtain a standardly stratified category. However, the derived category shares many similarities with a standardly stratified structure: there is a stratification given by certain derived standard modules, and the (relatively) projective modules can be preordered and obtained from iterated extensions of the standard modules with lower weight. Furthermore, the classes of derived standard modules correspond with the induced basis elements in the Grothendieck group, and there is an explicit derived standardization functor categorifying the inclusion of the tensor factors.

\subsubsection{Appendix A: rewriting methods up to braid-like isotopy}
\label{sec:introrewriting}
Rewriting theory is a combinatorial theory of equivalence classes, consisting in transforming an object into another by a successive sequence of oriented moves. In an algebraic context, it consists in orienting relations of presentations by generators and relations of algebraic structures.
In particular, several tools following the principles of rewriting were developed in numerous works in linear algebra, in order to compute normal forms for different types of algebras, with applications to the decision of the ideal membership problem, and to the construction of linear bases, such as Poincar\'e-Birkhoff-Witt bases. For example, Shirshov introduced in \cite{Shirshov62} an algorithm to compute a linear basis of a Lie algebra presented by generators and relations, and deduced a constructive proof of the Poincar\'e-Birkhoff-Witt theorem, and  Gr\"obner basis theory was introduced to compute with ideals of commutative polynomial rings \cite{Buchberger65,Buchberger87}. Buchberger described an algorithm to compute Gr\"obner bases from the notion of $S$-polynomials, describing obstructions to local confluence in terms of overlappings between reductions. These approaches were extended in \cite{GuiraudHoffbeckMalbos19}, where a rewriting theoretical approach was introduced in order to study associative algebras without any assumption of compatibility of the rewriting rules with respect to a well-founded total order. This approach is based on the structure of \emph{linear polygraphs}. Polygraphs have been introduced by Burroni \cite{Bur93} and Street \cite{Street} as generating systems for higher dimensional globular strict categories, and have been extended in a linear setting in \cite{GuiraudHoffbeckMalbos19,AL16}. The computation of linear bases lay on two fundamental rewriting properties: the \emph{termination}, stating that an element can not be reduced infinitely many times, and the \emph{confluence}, stating that if an element can be reduced in two different ways, there has to exist rewriting paths starting from the two resulting elements leading to the same final result. Termination of a linear rewriting system implies that a polynomial can be reduced in finitely many steps into a linear combination of irreducible monomials, so that these latter span the presented algebra. Moreover, confluence ensures the linear independence of irreducible monomials.

 Many works studying diagrammatic presentations through rewriting techniques consist in rewriting on string diagrams in monoidal $\Bbbk$-linear categories, or $\Bbbk$-linear $2$-categories. These latter are presented by $3$-dimensional polygraphs, see for example \cite{GM09,AL16}. In this setting, the braid-like distant isotopy relations correspond to the exchange relations of the $2$-categories, and thus are structural relations that we do not need to orient.
However, if we use rewriting in the dimension of the algebras, which is needed in order to deal with the violating condition that diagrams with a leftmost strand being black are zero, these relations have to be taken into account as oriented rewriting rules. In order to mimic the well-known setting of rewriting in linear $2$-categories, we will use rewriting modulo braid-like planar isotopies. Rewriting modulo extends the usual rewriting techniques by allowing to consider a set $E$ of non-oriented equations together with a set $R$ of oriented rules. It is used mainly to split confluence proofs into many incremental steps, by first proving that the set $E$ forms a convergent rewriting system, and then study the remaining relations on $E$-equivalence classes. Following \cite{DUP19} the usual basis result given by the irreducible monomials of a convergent presentation is extended in that setting, by considering $E$-normal forms of irreducible monomials with respect to $S$.

In \cref{sec:rewritingmethods}, we develop the formalism of rewriting modulo braid-like isotopies for diagrammatic algebras. Given an algebra $\mathbf{A}$, we introduce the linear $2$-polygraph $\text{Iso}(\mathbf{A})$ containing distant isotopy relations as rewriting rules, and prove that it is \emph{convergent}, \emph{i.e.} terminating and confluent. We then describe how to prove that the linear $2$-polygraph containing the remaining relations of $\mathbf{A}$, oriented with respect to a termination order, is confluent modulo braid-like isotopies.


\subsection*{Acknowledgments}
The authors would like to thank Catharina Stroppel for interesting discussions and suggesting to consider a deformed dgKLRW algebra, which led to the proof of \cref{thm:Tbasis} in \cref{sec:basisthemtri}. 
The authors would also like to thank Philippe Malbos and St\'ephane Gaussent for helpful discussions. 
G.N. is grateful to the Max Planck Institute for Mathematics in Bonn for its hospitality and financial support.




\section{Quantum \texorpdfstring{$\slt$}{sl2} and its representations}\label{sec:quantumgroups}

Recall that \emph{quantum $\slt$} can be defined as the $\bQ\pp{q}$-algebra $U_q(\slt)$ generated by the elements $K,K^{-1}, E$ and $F$ with relations
\begin{align*}
&KE = q^2EK, &  &KK^{-1} = 1 = K^{-1}K, \\
&KF = q^{-2}FK, & &EF - FE = \frac{K-K^{-1}}{q-q^{-1}}.
\end{align*}
It becomes a bialgebra when endowed with comultiplication
\begin{align*}
\Delta(K^{\pm 1}) &:= K^{\pm 1} \otimes K^{\pm 1}, &
\Delta(F) &:= F \otimes K + 1 \otimes F, &
\Delta(E) &:= E \otimes 1 + K^{-1} \otimes E, 
\end{align*}
and with counit $\varepsilon(K^{\pm 1}) := \pm 1$, $\varepsilon(E) := \varepsilon(F) := 0$.

\begin{rem}
    Usually, one would define $U_q(\slt)$ over the rational fractions $\bQ(q)$ (or the complex numbers $\bC$) instead of the Laurent series $\bQ\pp{q}$. 
    However, $\bQ\pp{q}$ has a natural categorification by considering a certain category of graded vector spaces, while it is not clear what a categorification of $\bQ(q)$ or $\bC$ should be. 
    Therefore we always work with Laurent series in this paper. 
\end{rem}

There is a $\mathbb{Q}\pp{q}$-linear anti-involution $\antimapslt$ of $U_q(\slt)$ defined on the generators by
\begin{align*}
  \antimapslt(E) &:= q^{-1}K^{-1}F, & \antimapslt(F) &:= qEK, & \antimapslt(K) &:= K.
\end{align*}

\subsection{Integrable module \texorpdfstring{$V(N)$}{V(N)}}
For each $N \in \bN$, there is a finite dimensional irreducible $U_q(\slt)$-module $V(N)$ called \emph{integrable module}. It has a basis  $\{v_N := v_{N,0}, v_{N,1}, \dots, v_{N,N}\}$ called \emph{induced basis}, respecting 
\begin{align*}
K \cdot v_{N,i} &:= q^{N-2i} v_{N,i}, \\
F \cdot v_{N,i} &:= v_{N,i+1}, \\
E \cdot v_{N,i} &:= [i]_q [N-i+1]_q v_{N,i-1}, 
\end{align*}
where
\begin{align*}
[k]_q &:= \frac{q^{k}-q^{-k}}{q-q^{-1}}.
\end{align*}
Note that $v_{N,i} = F^{i}(v_{N})$. 
It is also common the consider the \emph{divided power basis} (or canonical basis) given by $\overline{v}_{N,i} := F^{(i)}(v_N)$ where $F^{(i)}$ is the \emph{divided power} defined as
\[
F^{(i)} := \frac{1}{[i]_q!} F^{i},
\]
where $[i]_q! := [i]_q [i-1]_q \cdots [1]_q$ and $[0]_q! := 1$. 

\smallskip

There is a unique non-degenerate bilinear form $\langle \cdot,\cdot \rangle_N : V(N) \otimes V(N) \rightarrow \bQ\pp{q}$ such that $\langle v_0, v_0 \rangle_N = 1$ and which is $\antimapslt$-Hermitian: for any $v,v' \in V(N)$ and $u \in U_q(\slt)$ we have $\langle u\cdot v, v'\rangle_N= \langle v, \antimapslt(u)\cdot v'\rangle_N$. This map is called the \emph{Shapovalov form}.

\subsection{Verma module \texorpdfstring{$M(\mu)$}{M(mu)}}

Let $\beta$ be a formal parameter and write $\lambda := q^{\beta}$ as a formal variable. Let $\bo$ be the standard Borel subalgebra of $\slt$ and $U_q(\bo)$ be its quantum version. It is the $U_q(\slt)$-subalgebra generated by $K,K^{-1}$ and $E$. 
For $\mu = \beta + z \in \beta + \bZ$, let $K_\mu$ be a 1-dimensional $\bQ\pp{q,\lambda}$-vector space with a fixed a basis element $v_\mu$. 
We endow $K_\mu$ with a $U_q(\bo)$-action by declaring that
\begin{align*}
K^{\pm 1} \cdot v_\mu &:= \lambda^{\pm 1} q^{\pm z} v_\lambda, & E \cdot v_\mu &:= 0,
\end{align*}
and extending linearly through the obvious map $\bQ\pp{q} \hookrightarrow \bQ\pp{q,\lambda}$. 
The \emph{Verma module $M(\mu)$} is the induced module
\[
M(\mu) := U_q(\slt) \otimes_{U_q(\bo)} K_\mu. 
\]
It is infinite dimensional over $\bQ\pp{q,\lambda}$ with \emph{induced basis} $\{v_{\mu,i} := F^i(v_\mu)\}_{i \geq 0}$. The action of the quantum group is explicitly given by
\begin{align*}
K \cdot v_{\mu,i} &= \lambda q^{z-2i} v_{\mu,i}, \\
F \cdot v_{\mu,i} &= v_{\mu, i+1}, \\
E \cdot v_{\mu,i} &=  [i]_q [\beta+z-i+1]_q v_{\mu, i-1},
\end{align*}
where
\begin{equation*}
[k\beta+\ell]_q :=  \frac{q^{k \beta+\ell} - q^{-k \beta -\ell}}{q-q^{-1}} = \frac{\lambda^k q^\ell - \lambda^{-k} q^{-\ell}}{q-q^{-1}},
\end{equation*}
for all $k,\ell \in \bZ$. 
One can also define the \emph{divided power basis} as $\{ \overline{v}_{\mu,i} := F^{(i)}(v_\mu) \}_{i \geq 0}$. 

\smallskip

The Verma module $M(\mu)$ can also be equipped with a Shapovalov form $(\cdot,\cdot)_\mu$, which is again the unique non-degenerate $\bQ\pp{q,\lambda}$-bilinear form such that $(v_\mu, v_\mu)_\mu = 1$ and which is $\antimapslt$-Hermitian: for any $v,v' \in M(\mu)$ and $u \in U_q(\slt)$, we have $(u\cdot v, v')_\mu = (v, \antimapslt(u)\cdot v')_\mu$.

\subsection{Tensor product}

Given two $U_q(\slt)$-modules $M$ and $M'$, one forms the \emph{tensor product representation} $M \otimes M'$ by using the action induced by the comultiplication, explicitly
\begin{align*}
K^{\pm 1} \cdot (m \otimes m') &:= (K^{\pm 1} \cdot m) \otimes (K^{\pm 1} \cdot m'), \\
F \cdot (m \otimes m') &:= (F \cdot m) \otimes (K \cdot m') + m \otimes (F \cdot m'), \\
E \cdot (m \otimes m') &:= (E \cdot m) \otimes m' + (K^{-1} \cdot m) \otimes (E \cdot m'),
\end{align*}
for all $m \in M$ and $m' \in M'$.

\smallskip

For $\mu \in \bN \sqcup (\beta + \bZ)$, we write
\[
L(\mu) := 
\begin{cases}
V(\mu), & \text{if $\mu \in \bN$,} \\
M(\mu), &\text{if $\mu \in (\beta + \bZ)$}.
\end{cases}
\]
For a string of weights $\und \mu = (\mu_1, \dots, \mu_r)$, with $\mu_i \in \bN \sqcup (\beta + \bZ)$, we write 
\[
L(\und \mu) := L(\mu_1) \otimes \cdots \otimes L(\mu_r). 
\]

\subsubsection{Weight spaces}

The module $L(\und \mu)$ decomposes into \emph{weight spaces} 
\[
 L(\und \mu)_{k \beta + \ell} := \{ v \in  L(\und \mu) | K(v) = \lambda^k q^\ell v \}.
\]
Write $|\und \mu| := \sum_{i = 1}^r \mu_i \in \bZ \beta + \bZ$. Note that $ L(\und \mu)_{k \beta + \ell} \neq 0$ only for $k \beta + \ell = |\und \mu| - 2b$ with $b \geq 0$. 
We also write $w(x) := \lambda^k q^\ell$ for $x \in  L(\und \mu)_{k\beta +\ell}$.

\subsubsection{Basis}
Let $\cP_{b}^{r}$ be the set of (weak) compositions of $b$ into $r$ parts, that is:
\[
  \cP_{b}^{r}:=\left\{ \rho = (b_1,\dots,b_r)\in\mathbb{N}^{r}\ \middle\vert\ \sum_{i=1}^r b_i = b\right\}.
\]
Consider also 
\[
\cP_{b}^{r, \und \mu} := \left\{ ( b_1, \dots, b_r)\in\mathcal{P}_{b}^{r} | b_i \leq \mu_i  \text{ for all $\mu_i \in \bN$}\right\} \subset \mathcal{P}_{b}^{r}.
\]

The module $L(\und \mu)$ admits an obvious basis induced by the ones of $L(\mu_i)$:
\[
\bigl\{ \tilde v_\rho := F^{b_1} (v_{\mu_1}) \otimes F^{b_2}(v_{\mu_2})  \otimes \cdots \otimes F^{b_r} (v_{\mu_r}) | \rho \in \cP_{b}^{r, \und \mu} \bigr\}.
\]
It also admits another basis that will reveal to be useful for categorification purposes:
\begin{equation*} 
\bigl\{
  v_{ \rho} := 
  F^{b_r} \bigl( \cdots F^{b_2}(F^{b_1}(v_{\mu_1}) \otimes v_{\mu_2})  \cdots \otimes v_{\mu_r} \bigr)
 |  \rho \in \cP_{b}^{r, \und \mu}
\bigr\}.
\end{equation*}
Indeed, in $L(\und \mu = (\mu_1, \mu_2))$, we have
\begin{equation}\label{eq:Frewriting}
x \otimes F(y) = F(x \otimes y) - w(y) F(x) \otimes y,
\end{equation}
with $x \in L(\mu_1)$ and $y \in L(\mu_2)$, by definition of $\Delta(F)$. This allows to rewrite any element $\tilde v_{\rho_0}$ in the basis of $\{v_\rho\}$ by bringing recursively all $F$'s to the left. 

\begin{lem}\label{lem:Frewriting}
Any basis element $\tilde v_{\rho_0}$ can be written as a linear combination of elements in $\{v_\rho | \rho \in \mathcal{P}_{b}^{r} \}$. 
\end{lem}

\begin{proof}
Consider an element of the form
\begin{equation}\label{eq:Frewritinglemma}
v_{\rho_1, \rho_2}^{t,\ell} := F^t\bigl( v_{\rho_1} \otimes F^{\ell} (v_{\mu}) \bigr) \otimes \tilde v_{\rho_2},
\end{equation}
where $t, \ell \geq 0$, $\rho_1 \in \bN^{r_1}$, $\rho_2 \in \bN^{r_2}$, $r_1+1+r_2 = r$, and $\mu = \mu_{r_1 + 1}$. 
If $r_1 = 0$, then it is an element of $\{\tilde v_{\rho}\}$, and if $\ell = r_2 = 0$, then of $\{v_\rho\}$. 

Applying \cref{eq:Frewriting} on \cref{eq:Frewritinglemma}, we obtain
\begin{equation}\label{eq:Frewritinglemma2}
\begin{split}
v_{\rho_1, \rho_2}^{t,\ell}
 &=
F^{t+1} \bigl( v_{\rho_1} \otimes F^{\ell-1} (v_\mu) \bigr) \otimes \tilde v_{\rho_2}
-  q^{\mu+2-2\ell} 
F^{t} \bigl( F(v_{\rho_1}) \otimes F^{\ell-1}(v_\mu) \bigr) \otimes \tilde v_{\rho_2}
\\
&=
v_{\rho_1, \rho_2}^{t+1,\ell-1}
 -   q^{\mu+2-2\ell} 
v_{F(\rho_1), \rho_2}^{t,\ell-1},
\end{split}
\end{equation}
where $F(\rho_1)$ is given by increasing the last term of $\rho_1$ by $1$. 
Furthermore, if $\ell-1 = 0$, then  they are of the form $v_{\rho_1'} \otimes \tilde v_{\rho_2'}$ for different $\rho_1' \in \bN^{r_1+1}$ and $\rho_2' \in \bN^{r_2}$. Since $\tilde v_{\rho_2'} = F^{\ell'}(v_{\mu_{r_1+2}}) \otimes v_{\rho_2''}$ with $\rho_2'' \in \bN^{r_2-1}$, we can rewrite the expression as an element of the form \cref{eq:Frewritinglemma} with $r_2$ decreased by $1$. 
In conclusion, applying \cref{eq:Frewritinglemma2} recursively allows to decrease both $\ell$ and $r_2$ to zero, giving the desired expression. 
\end{proof}

\begin{exe}
Consider $\und \mu = (\beta,\beta)$, and $\tilde v_{(0,2)} = v_\beta \otimes F^2(v_\beta)$. We compute
\begin{align*}
v_\lambda \otimes F^2(v_\lambda) &= F\bigl(v_\lambda \otimes F(v_\lambda)\bigr) - \lambda q^{-2} F(v_\lambda) \otimes F(v_\lambda), \\
F\bigl(v_\lambda \otimes F(v_\lambda)\bigr) &= F^2( v_\lambda \otimes v_\lambda) - \lambda F\bigl(F(v_\lambda) - v_{\lambda}\bigr), \\
F(v_\lambda) \otimes F(v_\lambda) &= F\bigl( F(v_\lambda) \otimes v_\lambda \bigr) - \lambda F^2(v_\lambda) \otimes v_\lambda.
\end{align*}
For another example, consider $\und \mu = (\beta,\beta,\beta)$ and $v_{(0,1,1)} = v_\lambda \otimes F(v_\lambda) \otimes F(v_\lambda)$. We compute
\begin{align*}
v_\lambda \otimes F(v_\lambda) \otimes F(v_\lambda) &= F(v_\lambda \otimes v_\lambda) \otimes F(v_\lambda) - \lambda F(v_\lambda) \otimes v_\lambda \otimes F(v_\lambda), \\
F(v_\lambda \otimes v_\lambda) \otimes F(v_\lambda)  &= F\bigl(F(v_\lambda \otimes v_\lambda) \otimes v_\lambda\bigr) - \lambda F^2(v_\lambda \otimes v_\lambda) \otimes v_\lambda, \\
F(v_\lambda) \otimes v_\lambda \otimes F(v_\lambda) &= F\bigl(F(v_\lambda) \otimes v_\lambda \otimes v_\lambda\bigr) - \lambda F\bigl(F(v_\lambda) \otimes v_\lambda\bigr) \otimes v_\lambda.
\end{align*}
\end{exe}

One can also consider the basis induced by the divided power basis
\[
\bigl\{ \tilde{\overline{v}}_\rho := F^{(b_1)} (v_{\mu_1}) \otimes F^{(b_2)}(v_{\mu_2})  \otimes \cdots \otimes F^{(b_r)} (v_{\mu_r}) | \rho \in \cP_{b}^{r, \und \mu} \bigr\},
\]
and
\[
\bigl\{
  \overline{v}_{ \rho} := 
  F^{(b_r)} \bigl( \cdots F^{(b_2)}(F^{(b_1)}(v_{\mu_1}) \otimes v_{\mu_2})  \cdots \otimes v_{\mu_r} \bigr)
 |  \rho \in \cP_{b}^{r, \und \mu}
\bigr\}.
\]

\begin{lem}
For $\rho = (b_1, \dots, b_r) \in    \cP_{b}^{r,\und \mu}$, we have
\begin{equation}\label{eq:Evkappa}
E(v_\rho) = \left(\sum_{i = 1}^{b_r} [|\und \mu|-2b+2i]_q \right) F^{b_r-1}(v_{\rho_{<r}} \otimes v_{\mu_r}) + F^{b_r}(E v_{\rho_{<r}} \otimes v_{\mu_r}),
\end{equation}
where $\rho_{<r} := (b_1, \dots, b_{r-1})$.
\end{lem}

\begin{proof}
We apply the main $\mathfrak{sl}_2$-commutator relation $b_r$ times.
\end{proof}

\subsubsection{Shapovalov forms for tensor products}\label{sec:shepfortensor}

Following \cite[\S4.7]{webster}, we consider a family of bilinear forms $(\cdot,\cdot)_{\underline{\mu}}$ on tensor products of the form $L(\und \mu)$  satisfying the following properties:
\begin{enumerate}
\item each form $(\cdot,\cdot)_{\underline{\mu}}$ is non-degenerate;
\item for any $u\in U_q(\slt)$ we have $(u \cdot v,v')_{\underline{\mu}} = (v, \antimapslt(u)\cdot v')_{\underline{\mu}}$;
\item for any $f\in \mathbb{Q}\pp{q,\lambda}$, we have $(f v,v')_{\underline{\mu}} = (v,fv')_{\underline{\mu}} = f(v,v')_{\underline{\mu}}$;
\item we have $(v,v')_{\underline{\mu}} = (v\otimes v_{\mu_{r+1}},v'\otimes v_{\mu_{r+1}})_{\underline{\mu'}}$ where $\underline{\mu'}=(\mu_1,\ldots,\mu_r, \mu_{r+1})$,
\end{enumerate}
for all $v,v' \in  L(\underline{\mu})$. 

\smallskip
Similarly to \cite[Proposition 4.33]{webster} we have:

\begin{prop}
  There exists a unique system of such bilinear forms which are given by
  \[
    (v , v')_{\underline{\mu}} =  \prod_{i=1}^r (v_i, v_i')_{\mu_i}, 
  \]
for every $v = v_1 \otimes \cdots \otimes v_r,v' = v'_1 \otimes \cdots \otimes v'_r \in  L(\underline{\mu})$.
\end{prop}


\section{Preliminaries and conventions} \label{sec:conventions}

Before defining the dgKLRW algebras, we fix some conventions, and we recall some common facts about dg-structures (classical references for this are~\cite{keller} and \cite{toen}, see also \cite[Appendix A]{NV3} for a short survey oriented towards categorification), and about rewriting methods. Since we use the same conventions as in \cite{LNV}, a part of this section is almost identical to \cite[\S3.1 and Appendix B]{LNV}. 

\subsection{Homological algebra}
First, let $\Bbbk$ be a commutative unital ring for the remaining of the paper. 

\subsubsection{Dg-algebras}

A \emph{$\bZ^n$-graded dg-($\Bbbk$-)algebra} $(A,d_A)$ is a unital $\bZ \times \bZ^n$-graded ($\Bbbk$-)algebra $A = \bigoplus_{(h,\bg) \in \bZ \times \bZ^n} A_\bg^h$, where we refer to the $\bZ$-grading as homological (or $h$-degree) and the $\bZ^n$-grading as $\bg$-degree, with a differential $d_A : A \rightarrow A$ such that:
\begin{itemize}
\item $d_A(A_\bg^h) \subset A_{\bg}^{h-1}$ for all $\bg \in \bZ^n, h \in \bZ$; \\(the differential preserves the $\bZ^n$-grading and decreases the homological grading)
\item $d_A(xy) = d_A(x)y + (-1)^{\deg_h(x)} x d_A(y)$; \\(the differential respects the \emph{graded Leibniz rule})
\item $d_A^2 = 0$. \\(the differential yields a complex) 
\end{itemize}
The \emph{homology} of $(A,d_A)$ is $H(A,d_A) := \ker(d_A)/\Image(d_A)$, 
which is a $\bZ \times \bZ^n$-graded algebra decomposing as 
\begin{align*}
    H(A,d_A) &\cong \bigoplus_{(h,\bg) \in \bZ \times \bZ^n} H_\bg^h(A,d_A),
    &
    H_\bg^h(A,d_A) &:= \frac{\ker(d_A : A_\bg^{h} \rightarrow A_\bg^{h-1})}{\Image(d_A : A_\bg^{h+1} \rightarrow A_\bg^h)}.
\end{align*}
A morphism of dg-algebras $f: (A,d_A) \rightarrow (A', d_{A'})$ is a morphism of algebras that preserves the $\bZ \times \bZ^n$-grading and commutes with the differentials. 
Such a morphism induces a morphism $f^* : H(A,d_A) \rightarrow H(A',d_{A'})$. We say that $f$ is a \emph{quasi-isomorphism} whenever $f^*$ is an isomorphism. Moreover, we say that $(A,d_A)$ is formal if there is a quasi-isomorphism $(A,d_A) \xrightarrow{\simeq} (H(A,d_A), 0)$. This happens whenever $H(A,d_A)$ is concentrated in homological degree zero. 

\begin{rem}
Note that in contrast to \cite{keller}, our differential decreases the homological degree instead of increasing it.
\end{rem}

Similarly, a $\bZ^n$-graded left dg-module over $(A,d_A)$, or simply $(A,d_A)$-module, is a $\bZ \times \bZ^{n}$-graded $A$-module $M = \bigoplus_{(h,\bg) \in \bZ \times \bZ^n} M_\bg^h$ with a differential $d_M : M \rightarrow M$ such that:
\begin{itemize}
\item $d_M(M_\bg^h) \subset M_{\bg}^{h-1}$ for all $\bg \in \bZ^n, h \in \bZ$;
\item $d_M(x \cdot m) = d_A(x) \cdot y + (-1)^{\deg_h(x)} x  \cdot d_M(y)$;
\item $d_M^2 = 0$. 
\end{itemize}
Homology, maps between dg-modules and quasi-isomorphisms are defined as above. There are similar notions of $\bZ^n$-graded right dg-modules and dg-bimodules, with only subtlety that $d_M(m \cdot x) = d_M(m) \cdot x + (-1)^{\deg_h(m)} m \cdot d_A(x)$.

\begin{rem}
For a dg-algebra $(A,d_A)$, we sometimes talk about a graded $A$-module $M$. This means we consider $M$ as a $\bZ\times\bZ^n$-graded module over $A$, forgetting about the differential $d_A$.
\end{rem}

In our convention,  a \emph{$\bZ^m$-graded category} is a category with a collection of $m$ autoequivalences, strictly commuting with each others. 
The category $(A,d_A)\amod$ of (left) $\bZ^n$-graded dg-modules over a dg-algebra $(A,d_A)$ is a $\bZ \times \bZ^n$-graded abelian category, with kernels and cokernels defined as usual. The action of $\bZ$ is given by the \emph{homological shift functor} $[1] : (A,d_A)\amod \rightarrow (A,d_A)\amod$ sending $M \mapsto M[1] := \{m[1] | m \in M\}$ and such that:
\begin{itemize}
\item $\deg_h(m[1]) := \deg_h(m) + 1$; \\(it increases the $h$-degree of all elements up by 1) 
\item $d_{M[1]} := -d_M$; \\(it switches the sign of the differential) 
\item $r \cdot (m[1])  := (-1)^{\deg_h(r)} (r \cdot m)[1]$, \\(it twists the left action)
\end{itemize}
and sending $f : M \rightarrow N$ to $f[1] : M[1] \rightarrow N[1], m[1] \mapsto f(m)[1]$.
The action of $\bg \in \bZ^n$ is given by increasing the $\bZ^n$-degree of all elements up by $\bg$, in the sense that
\[
(\bg M)_{\bg_0 + \bg} := (M)_{\bg_0},
\]
or in other terms, an element $x \in M$ with degree $\bg_0$ becomes of degree $\bg_0+\bg$ in $\bg M$.
There are similar definitions for categories of right dg-modules and dg-bimodules, with the subtlety that the homological shift functor does not twist the right-action:
\[
(m[1]) \cdot r := (m \cdot r)[1].
\]
As usual, a short exact sequence of dg-(bi)modules induces a long exact sequence in homology. 

\smallskip

Let $f : (M,d_M) \rightarrow (N,d_N)$ be a morphism of dg-(bi)modules. Then, one constructs the \emph{mapping cone} of $f$ as 
\begin{align} \label{eq:cone}
\cone(f) &:= (M[1] \oplus N, d_C), & 
d_C &:= \begin{pmatrix} -d_M & 0 \\ f & d_N \end{pmatrix}.
\end{align}
The mapping cone is a dg-(bi)module, and it fits in a short exact sequence:
\[
0 \rightarrow N \xrightarrow{\imath_N} \cone(f) \xrightarrow{\pi_{M[1]}} M[1] \rightarrow 0,
\]
where $\imath_N$ and $\pi_{M[1]}$ are the canonical inclusion and projection $N \xrightarrow{\imath_N} M[1] \oplus N \xrightarrow{\pi_{M[1]}} M[1]$. 

\subsubsection{Hom and tensor functors}\label{sec:classicalhomandtensor}

Given a left dg-module $(M,d_M)$ and a right dg-module $(N,d_N)$, one constructs the tensor product
\begin{equation}\label{eq:dgtens}
\begin{split}
(N,d_N) \otimes_{(A,d_A)} (M,d_M) &:= \bigl( (M \otimes_A N), d_{M \otimes N} \bigr), \\
d_{M \otimes N}(m \otimes n) &:= d_M(m) \otimes n + (-1)^{\deg_h(m)} m \otimes d_N(n).
\end{split}
\end{equation}
If $(N,d_N)$ (resp. $(M,d_M)$) has the structure of a dg-bimodule, then the tensor product inherits a left (resp. right) dg-module structure. 

Given a pair of left dg-modules $(M,d_M)$ and $(N,d_N)$, one constructs the dg-hom space
\begin{equation}\label{eq:dghom}
\begin{split}
\HOM_{(A,d_A)}\bigl( (M,d_M), (N,d_N) \bigr) &:= \bigl( \HOM_A(M,N), d_{\HOM(M,N)} \bigr), \\
d_{\HOM(M,N)}(f) &:= d_N \circ f - (-1)^{\deg_h(f)} f \circ d_M,
\end{split}
\end{equation}
where $\HOM_A$ is the $\bZ\times \bZ^n$-graded hom space of maps between $\bZ\times \bZ^n$-graded $A$-modules. Again, if $(M,d_M)$ (resp. $(N,d_N)$) has the structure of a dg-bimodule, then it inherits a left (resp. right) dg-module structure.

In particular, given a dg-bimodule $(B,d_B)$ over a pair of dg-algebras $(S,d_{S})$-$(R,d_R)$, we obtain tensor and hom functors
\begin{align*}
(B,d_B) \otimes_{(R,d_R)} (-) &: (R,d_R)\amod \rightarrow(S,d_S)\amod, \\
\HOM_{(S, d_{S})}((B,d_B), -) &: (S,d_{S})\amod \rightarrow(R,d_R)\amod,
\end{align*}
which form a adjoint pair $((B,d_B) \otimes_{(R,d_R)} -) \vdash \HOM_{(S,d_{S})}((B,d_B), -)$. 

\subsubsection{Derived categories} \label{ssec:derivedCat}

The \emph{derived category $\cD(A,d_A)$} of $(A,d_A)$ is the localization of the category $(A,d_A)\amod$ of  $\bZ^n$-graded $(A,d_A)$-dg-modules along quasi-isomorphisms. It is a triangulated category with translation functor induced by the homological shift functor $[1]$, and distinguished triangles are equivalent to 
\[
(M,d_N) \xrightarrow{f} (N,d_N) \xrightarrow{\imath_N} \cone(f) \xrightarrow{\pi_{M[1]}} (M,d_N)[1],
\]
for every maps of dg-modules $f : (M,d_M) \rightarrow (N,d_N)$. 

\subsubsection{Cofibrant replacements}

A \emph{cofibrant} dg-module $(P,d_P)$ is a dg-module such that $P$ is projective as $\bZ\times\bZ^n$-graded $A$-module. 
Equivalently, it is a dg-module $(P,d_P)$ such that for every surjective quasi-isomorphism $(L,d_L) \xrightarrowdbl{\simeq} (M,d_M)$, every morphism $(P,d_P) \rightarrow (M,d_M)$ factors through $(L,d_L)$.
For any dg-module $(N, d_N)$ and cofibrant dg-module $(P,d_P)$, we have
\begin{align*}
\Hom_{\cD(A,d_A)}\bigl((P,d_P), (N,d_N)\bigr) \cong H^0_0 \left(\HOM_{(A,d_A)}\bigl((P,d_P), (N,d_N) \bigr) \right).
\end{align*}
Moreover, tensoring with a cofibrant dg-module preserves quasi-isomorphisms.

Given a left 
dg-module $(M,d_M)$, there exists a cofibrant dg-module $(\br M , d_{\br M })$  
together with a surjective quasi-isomorphism $\pi_M : (\br M, d_{\br M}) \xrightarrowdbl{\simeq}  (M,d_M)$. 
Moreover, the assignment $(M,d_M) \mapsto (\br M, d_{\br M})$ 
is natural, and we refer to $(\br M, d_{\br M})$ as the \emph{cofibrant replacement} of $(M,d_M)$. Thus, we can compute $\Hom_{\cD(A,d_A)}\bigl((M,d_M), (N,d_N)\bigr)$ by taking 
\[
H^0_0\left(\HOM_{(A,d_A)}\bigl((\br M,d_{\br M}), (N,d_N) \bigr) \right) \cong \Hom_{\cD(A,d_A)}\bigl((M,d_M), (N,d_N)\bigr).
\] 

\subsubsection{Dg-derived categories}\label{sec:dgdercat}

One of the issues with triangulated categories is that the category of functors between triangulated categories is in general not triangulated. To fix this, we work with a dg-enhancement of the derived category. In particular, this allows us to talk about distinguished triangles of dg-functors.

Recall that a dg-category is a category where the hom-spaces are dg-modules over $(\Bbbk,0)$, and compositions are compatible with this structure (see \cite[\S1.2]{keller} for a precise definition). The \emph{homotopy category $H^0(\cC)$} of a dg-category $\cC$ is the category with the same objects as $\cC$ but with hom-spaces given by the degree zero homology of the dg-hom spaces of $\cC$. 

The \emph{dg-derived category $\cD_{dg}(A,d_A)$} of a $\bZ^n$-graded dg-algebra $(A,d_A)$ is the $\bZ^n$-graded dg-category with objects being cofibrant dg-modules over $(A,d_A)$, and hom-spaces being subspaces of the graded dg-spaces $\HOM_{(A,d_A)}$ from \eqref{eq:dghom}, given by maps that preserve the $\bZ^n$-grading:
\[
\Hom_{\cD_{dg}(A,d_A)}(M,N) := \HOM_{(A,d_A)}(M,N)_0^*,
\]
for $(M,d_M)$ and $(N,d_N)$ cofibrant dg-modules. 

The dg-derived category $\cD_{dg}(A,d_A)$ is a dg-triangulated category, meaning its homotopy category is canonically triangulated (see \cite{toen} for a precise definition,  or \cite[Appendix A]{NV3} for a summary oriented toward categorification). It turns out that the homotopy category of $\cD_{dg}(A,d_A)$ is triangulated equivalent to the usual derived category $\cD(A,d_A) \cong H^0(\cD_{dg}(A,d_A))$.

\subsubsection{Dg-functors}\label{sec:dgfunctors}

A \emph{dg-functor} between dg-categories is a functor commuting with the differentials. Given a dg-functor $F : \cC \rightarrow \cC'$, it induces a functor on the homotopy categories $[F] : H^0(\cC) \rightarrow H^0(\cC')$. 
We say that a dg-functor is a \emph{quasi-equivalence} if it gives quasi-isomorphisms on the hom-spaces, and induces an equivalence on the homotopy categories. 
We want to consider dg-category up to quasi-equivalences. Let $\Hqe$ be the homotopy category of dg-categories up to quasi-equivalence , and we write $\cRHom_{\Hqe}$ for the dg-space of quasi-functors between dg-categories (see \cite{toen} or  \cite{toenlectures}). These quasi-functors are not strictly speaking functors, but they induce honest functors on the homotopy categories. 
Whenever $\cC'$ is dg-triangulated, then $\cRHom_{\Hqe}(\cC,\cC')$ is dg-triangulated.

\begin{rem}
The space of quasi-functors is equivalent to the space of strictly unital $A_\infty$-functors. 
\end{rem}

It is in general a hard problem to understand the space of quasi-functors between dg-categories. However, by the results of Toen~\cite{toen}, if $\Bbbk$ is a field and $(A,d_A)$ and $(A',d_{A'})$ are dg-algebras, then it is possible to compute the space of `coproduct preserving' quasi-functors $\cRHom_{\Hqe}^{cop}(\cD_{dg}(A,d_A),\cD_{dg}(A',d_{A'}))$. Indeed, in the same way as the category of coproducts preserving functors between categories of modules is equivalent to the category of bimodules, there is a triangulated quasi-equivalence
\begin{equation}\label{eq:quasifunctequiv}
\cRHom_{\Hqe}^{cop}(\cD_{dg}(A,d_A),\cD_{dg}(A',d_{A'})) \cong \cD_{dg}((A',d_{A'}), (A,d_A)),
\end{equation}
where $ \cD_{dg}((A',d_{A'}), (A,d_A))$ is the dg-derived category of dg-bimodules. Composition of functors is equivalent to derived tensor product, and understanding the triangulated structure of $\cRHom_{\Hqe}^{cop}(\cD_{dg}(A,d_A),\cD_{dg}(A',d_{A'}))$ becomes as easy as to understand the structure of $\cD((A,d_A), (A',d_{A'}))$. 
In particular, a short exact sequence of dg-bimodules gives a distinguished triangle of dg-functors.

\subsubsection{Derived hom and tensor dg-functors}\label{sec:deriveddghomtensor}

Let $(R,d_R)$ and $(S,d_S)$ be dg-algebras. Let $(M,d_M)$ and $(N,d_N)$ be $(R,d_R)$-module and $(S,d_S)$-module respectively. Let $(B,d_B)$ be a dg-bimodule over $(S,d_S)$-$(R,d_R)$. 
The \emph{derived tensor product} is
\[
(B,d_B) \Lotimes_{(R,d_R)} (M,d_M) := (B,d_B) \otimes (\br M, d_{\br M}),
\]
and the \emph{derived hom space} is
\[
\RHOM_{(S,d_S)}((B,d_B), (N, d_N)) := \HOM_{(S,d_S)}((\br B, d_{\br B}), (N,d_N)).
\]

This defines in turns triangulated dg-functors
\begin{align*}
(B,d_B) \Lotimes_{(R,d_R)} (-) &: \cD_{dg}(R,d_R) \rightarrow \cD_{dg}(S,d_S),
\intertext{and} 
\RHOM_{(S,d_S)}((B,d_B), -) &: \cD_{dg}(S,d_S) \rightarrow \cD_{dg}(R,d_R), 
\end{align*}
which are adjoint $(B,d_B) \Lotimes_{(R,d_R)} (-)  \vdash \RHOM_{(S,d_S)}((B,d_B), -)$.

\subsection{Diagrammatic algebras}\label{ssec:diagalg}

We always read diagram from bottom to top. We say that a diagram is braid-like when it is given by strands connecting a collection of points on the bottom to a collection of points on the top, without being able to turn back. Suppose these diagrams can have singularities (like dots, 4-valent crossings, or other similar decorations). 

A \emph{braid-like planar isotopy} is an isotopy fixing the endpoints and that does not create any critical point, in particular it means we can exchange distant singularities $f$ and $g$:
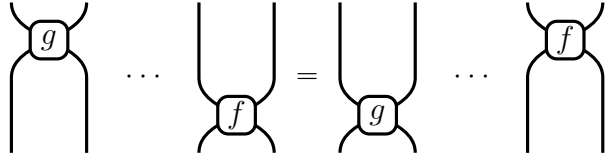
 \[
 \tikzdiag{
	\draw (0,-1) -- (0,0) ..controls (0,.5) and (1,.5) .. (1,1);
	\draw (1,-1) -- (1,0) ..controls (1,.5) and (0,.5) .. (0,1);
		\filldraw [fill=white, draw=black,rounded corners] (.5-.25,.5-.25) rectangle (.5+.25,.5+.25) node[midway] { $g$};
}
\quad
\cdots
\quad
\tikzdiag{
	\draw (0,0) ..controls (0,.5) and (1,.5) .. (1,1) -- (1,2);
	\draw (1,0) ..controls (1,.5) and (0,.5) .. (0,1) -- (0,2);
		\filldraw [fill=white, draw=black,rounded corners] (.5-.25,.5-.25) rectangle (.5+.25,.5+.25) node[midway] { $f$};
}
\ =  \ 
\tikzdiag{
	\draw (0,0) ..controls (0,.5) and (1,.5) .. (1,1) -- (1,2);
	\draw (1,0) ..controls (1,.5) and (0,.5) .. (0,1) -- (0,2);
		\filldraw [fill=white, draw=black,rounded corners] (.5-.25,.5-.25) rectangle (.5+.25,.5+.25) node[midway] { $g$};
}
\quad
\cdots
\quad
\tikzdiag{
	\draw (0,-1) -- (0,0) ..controls (0,.5) and (1,.5) .. (1,1);
	\draw (1,-1) -- (1,0) ..controls (1,.5) and (0,.5) .. (0,1);
		\filldraw [fill=white, draw=black,rounded corners] (.5-.25,.5-.25) rectangle (.5+.25,.5+.25) node[midway] { $f$};
}
 \]

\subsection{Rewriting methods}
\label{sec:prelimrewriting}

Rewriting theory is a theory of equivalences that consist in transforming algebraic objects using successive applications of oriented relations. It has been developed in linear settings to solve the problem of membership to an ideal and to compute linear bases, with the theory of Gröbner bases \cite{Buchberger65,Buchberger87}. In this context, rewriting rules are oriented with respect to an ambient monomial order on the algebra.
In this section, we recall the linear context of polygraphic rewriting for associative algebras introduced in \cite{GuiraudHoffbeckMalbos19}, where this restriction on rewriting rules is removed. The calculations lay on two fundamental rewriting properties: 
\begin{enumerate}
    \item \emph{Termination} states that an element can not be rewritten infinitely many times, and therefore reaches a linear combination of irreducible monomials (i.e. monomials that cannot be rewritten) after finitely many steps. In particular these irreducible monomials form a spanning set.
    \item \emph{Confluence} states that if a given element can be reduced in two distinct ways, there have to exist rewriting paths allowing to reduce both resulting elements into a common one. In particular the irreducible monomials are linearly independent.
\end{enumerate}
The combination of termination and confluence, called \emph{convergence}, then ensures that the set of irreducible monomials form a basis of the original algebra.  Moreover, rewriting with polygraphs allows to obtain strong local confluence criteria. In particular, one proves that if a linear polygraph is terminating, its confluence is equivalent to the confluence of the minimal overlappings between any given two relations, called \emph{critical branchings}: 
suppose there are rewriting rules $xy \Rightarrow f$ and $yz \Rightarrow g$, then there is an overlapping over $y$ and we need the check the confluence between $xyz \Rightarrow fz$ and $xyz \Rightarrow xg$. In contrast, we do not need to verify the confluence of $xyzx$ nor $xxyz$ because they are not minimal overlappings in the sense that the rightmost $x$ (resp. leftmost $x$) is never rewritten by a rule. Under the assumption of termination, confluence of a branching of the form $xyyz$ does not have to be verified as well, in the sense that it is not an overlapping but what is called a Peiffer branching, and is automatically confluent as explained below.

Rewriting modulo extends these constructions by allowing to rewrite with respect to a set of non-oriented relations, seen as axioms that one can freely use in rewriting paths. This allows in particular to split the proofs of confluence of rewriting systems into incremental steps.  
We develop in \cref{sec:rewritingmethods} rewriting methods modulo braid-like isotopy, allowing to construct bases for diagrammatic algebras defined up to braid-like isotopy. See \cref{ex:rewriteNH} for an example of this theory applied to the nilHecke algebra. 

\smallskip

In \cite{GuiraudHoffbeckMalbos19}, associative algebras over a field $\Bbbk$ are interpreted as monoidal objects in the category $\mathbf{Vect}_{\Bbbk}$ of $\Bbbk$-vector spaces and linear maps, and are presented by linear $(1-)$-polygraphs. In the sequel, in view of an extension of the constructions of Appendix \ref{sec:basisthemrewriting} towards linear $2$-categories, we interpret associative algebras as categories enriched over $\mathbf{Vect}_{\Bbbk}$ with only one $0$-cell, and in that context they are presented by \emph{linear~$2$-polygraphs}. As a consequence, there is a shift in dimensions of objects compared to \cite{GuiraudHoffbeckMalbos19}, but the terminology and constructions remain the same.
These objects are triples $(P_0,P_1,P_2)$ made of sets containing generating elements for the algebra, and the relations of the algebra. In this context, $P_0$ is always a singleton, $P_1$ contains generating $1$-cells that correspond to the generators of the algebra, so that all the $1$-cells correspond to monomials, \emph{i.e.} products of the generators, and the generating $2$-cells correspond to the relations of the algebra.
 More precisely, a linear~$2$-polygraph is a data of $P = (P_0,P_1,P_2)$ such that:
\begin{enumerate}[{\bf i)}]
\item $(P_0,P_1)$ is an oriented graph with vertices $P_0$ and edges $P_1$, equipped with source and target maps $s_0$, $t_0: P_1 \fl P_0$.
\item $P_2$ is a cellular extension of the free $1$-algebroid $P_1^\ell$, that is a set equipped with two source and target maps $
s_1,t_1: P_2 \to P_1^\ell$
such that the globular relations $s_0 s_1 (\alpha) = s_0 t_1 (\alpha)$ and $t_0 s_1 (\alpha) = t_0 t_1 (\alpha)$ hold for any $\alpha \in P_2$, where the free $1$-algebroid $P_1^\ell$ on $(P_0,P_1)$ is defined as the $1$-category enriched over $\mathbf{Vect}_{\Bbbk}$ whose objects are the elements of $P_0$, and for any $p,q$ in $P_0$, $P_1^\ell(p,q)$ is the free $\Bbbk$-vector space with basis the elements of the free $1$-category generated by $(P_0,P_1)$ with source $p$ and target $q$.
\end{enumerate}

For a linear~$2$-polygraph $P=(P_0,P_1,P_2)$, the elements of $P_i$ are called the generating $i$-cells of $P$.
When $P_0$ is a singleton, then $P_1^\ell$ corresponds to the free associative $\Bbbk$-algebra on the set $P_1$, and thus a linear~$2$-polygraph with only one $0$-cell corresponds to a presentation by generators and oriented relations of an associative algebra, where the rewriting rules are given in $P_2$. More precisely, denote by $I(P)$ the $2$-sided ideal of $P_1^\ell$ generated by the set of elements 
$ \{ s_1 (\alpha) - t_1 (\alpha) \: | \: \alpha \in P_2 \}. $

A linear $2$-polygraph $P$ \emph{presents} an algebra $A$ if $A$ is isomorphic to $P_1^\ell / I(P)$. 
The rewriting sequences will then correspond to $2$-cells in the free $2$-algebra $P_2^\ell$ on $P$, we refer to \cite{GuiraudHoffbeckMalbos19} for more details on these constructions. From now on, we will consider linear $2$-polygraphs with only one $0$-cell. A \emph{monomial} in $P_2^\ell$ is a $1$-cell of the free $1$-category $P_1^\ast$, every $1$-cell $f$ in $P_1^\ell$ can be uniquely decomposed as a linear combination of monomials $f = \lambda_1 f_1 + \dots + \lambda_p f_p$, with $\lambda_i \in \Bbbk \backslash \{ 0 \}$ for all $0 \leq i \leq p$. The set of monomials $\{ f_1,\dots,f_p \}$ is called the \emph{support} of $f$, denoted by $\text{Supp}(f)$. A linear~$2$-polygraph $P$ is called \emph{left-monomial} if, for any $\alpha$ in $P_2$, the $1$-cell $s_1(\alpha)$ is a monomial in $P_1^\ell$.

\medskip
A \emph{rewriting step} is a $2$-cell in $P_2^\ell$ with shape
\[ \lambda
\begin{tikzcd}
\bullet \arrow[r,"u"] & \bullet \arrow[r, bend left=50, ""{name=U, below},"s_1(\alpha)"]\arrow[r, bend right=50, ""{name=D},"t_1(\alpha)"']& \bullet \arrow[r,"v"] & \bullet  \arrow[Rightarrow,from=U,to=D,"\alpha"]
\end{tikzcd} \quad + \quad \begin{tikzcd} \bullet \arrow[r,"g"] & \bullet \end{tikzcd}
\]
where $\alpha \in P_2$, $\lambda \in \Bbbk$, and $g$ is a $1$-cell in $P_1^\ell$ such that the monomial $u s_1(\alpha) v$ does not belong to $\text{Supp}(g)$, see \cite{GuiraudHoffbeckMalbos19}. A \emph{rewriting sequence} is either and identity reduction $f \Rightarrow f$, or a $1$-composite
\[ f_0 \overset{\alpha_1}{\Rightarrow} f_1 \overset{}{\Rightarrow} \dots f_{k-1} \overset{\alpha_k}{\Rightarrow} f_k, \]
of rewriting steps of $P$. The linear~$2$-polygraph $P$ is said to be terminating if there is no infinite rewriting sequence in $P$. A \emph{normal form} of $P$ is a $1$-cell in $P_1^\ell$ that cannot be reduced by any rewriting step. When $P$ is terminating, any $1$-cell admits at least one normal form. A \emph{branching} of $P$ is a pair $(\alpha, \beta)$ of rewriting sequences of $P$ with a common source $s_1(\alpha) = s_1 (\beta)$. It is \emph{local} if both $\alpha$ and $\beta$ are rewriting steps of $P$. A branching $(\alpha, \beta)$ of $P$ is \emph{confluent} if there exist rewriting sequences $\alpha '$ and $\beta'$ in $P$ as in the following diagram:
\[
\begin{tikzcd}[xscale=0.9,yscale=0.5, row sep=0ex]
{} & g \arrow[dr,bend left =25,"\alpha '"] & {} \\
f \arrow[ur, bend left= 25,"\alpha"] \arrow[dr, bend right=25,"\beta"'] & & f' \\
{} & h \arrow[ur, bend right=25,"\beta '"'] & {} 
\end{tikzcd}
\]
We say that $P$ is \emph{confluent} (resp. \emph{locally confluent}) if any branching (resp. local branching) of $P$ is confluent. When $P$ is confluent, every $1$-cell in $P_1^\ell$ admits at most one normal form. When both termination and confluence properties are satisfied, we say that $P$ is \emph{convergent}, and in that case any $1$-cell $f$ in $P_1^\ell$ admits a unique normal form, denoted by $\widehat{f}$. Newman lemma \cite{Newman42} states that if $P$ is terminating and locally confluent, then it is confluent. 

We are particularly interested in convergent presentations of algebras. Indeed, \cite[Theorem 3.4.2]{GuiraudHoffbeckMalbos19} states that if an algebra $A$ is presented by a convergent linear~$2$-polygraph $P$, then the set of monomials in normal form for $P$ form a basis of the algebra $A$. Moreover, there exist some local criteria to reach confluence of a linear~$2$-polygraph.

Following \cite{GuiraudHoffbeckMalbos19}, local branchings of a linear~$2$-polygraph $P$ can be classified into $4$ families: aspherical branchings that are branchings between a rewriting step $f$ and itself, Peiffer branchings that are branchings consisting in applying two rules on a monomial at different positions with no overlapping, additive branchings that are branchigs consisting in applying two rules on two different monomials of a polynomial, and overlapping branchings that are the remaining ones. Aspherical branchings are trivially confluent, and if $P$ is terminating, Peiffer and additive branchings are confluent, \cite[Theorem 4.2.1]{GuiraudHoffbeckMalbos19}. A \emph{critical branching} of $P$ is an overlapping branching $(\alpha, \beta)$
that is minimal for the relation on branchings defined by $(\alpha,\beta) \subseteq (f \alpha f', f \beta f')$ for any monomials $f$, $f'$ in $P_1^\ast$. Following \cite[Theorem 4.2.1]{GuiraudHoffbeckMalbos19}, if $P$ is terminating it is locally confluent if and only if all its critical branchings are confluent. Thus, if $P$ is a terminating linear~$2$-polyraph, proving its confluence amounts to checking the confluence of all its critical branchings. 

In \cite{DMpp18}, a polygraphic context of rewriting modulo was introduced. Given two linear~$2$-polygraphs $(P_0,P_1,E)$ and $(P_0,P_1,R)$, one defines the cellular extension ${}_E R_E$ of $P_1^\ell$ as the set of $2$-cells that can be written as a composition $e \star_1 f \star_1 e'$, where $e$ and $e'$ are $2$-cells in $E^\ell$ and $f$ is a rewriting step of $R$. Namely, there is a rewriting step from $f$ to $g$ in ${}_E R_E$ if and only if there exists $f'$ and $g'$ in $P_1^\ell$ such that $f$ is $E$-equivalent to $f'$, $g$ is $E$-equivalent to $g'$ and there is a rewriting step for $P$ with source $f'$ and target $g'$. Explicitely, this consists in rewriting with respect to $R$ on equivalence classes modulo $E$. The data $(P_0,P_1,{}_E R_E)$ thus defines a linear~$2$-polygraph, that we denote by ${}_E R_E$. A {\emph linear $2$-polygraph modulo} is a data made of a triple $(R,E,S)$ where $R$ and $E$ are linear $2$-polygraphs with the same underlying $1$-polygraph, denoted by $P$, and $S$ is a cellular extension of $P_1^\ell$ such that $R \subseteq S \subseteq {}_E R_E$. A \emph{branching modulo} of $(R,E,S)$ is a triple $(\alpha,e,\beta)$ where $f$ and $g$ are rewriting paths of $S_2^\ell$ and $e$ is a $2$-cell of $E_2^\ell$ such that $s_1(\alpha) = s_1(e)$ and $s_1(\beta)= t_1(e)$. Such a branching is said to be confluent modulo $E$ if there exist rewriting paths $\alpha'$, $\beta'$ in $S_2^\ell$ and a $2$-cell $e'$ in $E_2^\ell$ as in the following diagram: 
\[
\begin{tikzcd}
f \arrow[r,"\alpha"] \arrow[d,"e"'] & f' \arrow[r,"\alpha'"] & f'' \arrow[d,"e'"] \\
g \arrow[r,"\gamma"'] & g' \arrow[r,"\gamma'"'] & g''.
\end{tikzcd}
\]
The linear $2$-polygraph modulo $(R,E,S)$ is said to be \emph{confluent modulo} $E$ if any of its branching modulo is confluent modulo $E$. We refer the reader to \cite{DMpp18,DUP19} for rewriting properties of polygraphs and linear polygraphs modulo. The local confluence criteria in terms of critical branchings for terminating linear rewriting systems has been extended in \cite{DUP19} in the context of linear rewriting modulo. When ${}_E R_E$ is terminating, in order to prove that the linear~$2$-polygraph ${}_E R_E$ is confluent modulo, it suffices to prove that the critical branchings modulo $(\alpha, \beta)$ where $\alpha$ is a rewriting step of $R$ and $\beta$ is a rewriting step of ${}_E R_E$ are confluent. Namely, these critical branchings modulo are given by application of a rewriting step $\alpha$ of $R$ and a rewriting step $\gamma$ of $R$ on two $1$-cells that are $E$-equivalent, with $(\alpha,e,\gamma)$ being minimal for the order $(\alpha, e , \beta) \subseteq (h \alpha h', heh', h \gamma h')$.

Moreover, following \cite{DUP19}, when the linear~$2$-polygraph $E$ is convergent, the basis theorem of \cite{GuiraudHoffbeckMalbos19} extends to that context of rewriting modulo. Explicitely, given an algebra $A$ presented by a linear~$2$-polygraph $P$ that we split into two parts $E$ (non-oriented) and $R$ (oriented), if $E$ is convergent, ${}_E R_E$ is terminating and ${}_E R_E$ is confluent modulo $E$, then the set of $E$-normal forms of monomials in normal form with respect to ${}_E R_E$ yields a basis of $A$.


\section{Dg-enhanced  KLRW algebras}\label{sec:dgWebster}

Inspired by the KLRW algebra in~\cite[\S4]{webster} (called ``tensor product algebra'' in the reference), which we think of as associated to a string of dominant integral $\bg$-weights, and generalizing the dg-enhanced KLRW algebra in~\cite{LNV}, which we think of as associated to a generic weight $\beta$ and a string of dominant integral $\bg$-weights, we introduce a dgKLRW (dg-)algebra associated to a string of weights that can each either be generic or integral.

\begin{defn}\label{defn:dgKLRW}
For $\und \mu = (\mu_1, \dots, \mu_r) \in (\bN \sqcup (\beta + \bZ))^r$,  
the \emph{dgKLRW-algebra} $\dgT_b^{\und \mu}$ is the diagrammatic $\Bbbk$-algebra defined as follows:
\begin{itemize}
\item $\dgT_b^{\und \mu}$  is generated by braid-like diagrams on $b$ black strands and $r$ colored strands. The colored strand are labeled from left to right by $\mu_1, \dots, \mu_r$, and we refer to the colored strands labeled by elements in $\bN$ as \emph{red strands}, while the ones labeled by elements in $\beta + \bZ$ are called \emph{blue strands}. We also require that the left-most strand is always colored (and thus labeled $\mu_1$). 
\item The colored strands cannot intersect each other, but the black strands can intersect all other strands (both black and colored) transversely. Moreover, black strands can carry dots, and can be `nailed' on the left-most colored strand:
\begin{align}
\label{eq:generators}
\tikzdiag{
	\draw (0,0)  ..controls (0,.5) and (1,.5) .. (1,1);
	\draw (1,0)  ..controls (1,.5) and (0,.5) .. (0,1);
          \node[align=center] at(.5,-1){black crossing\\ $q^{-2}$ };
}
&&
\tikzdiag{
	\draw (1,0)  ..controls (1,.5) and (0,.5) .. (0,1);
	\draw[pstdhl] (0,0) node[below]{\small $\mu_i$}  ..controls (0,.5) and (1,.5) .. (1,1);
	\draw (2.5,0)  ..controls (2.5,.5) and (3.5,.5) .. (3.5,1);
	\draw[pstdhl] (3.5,0) node[below]{\small $\mu_i$}   ..controls (3.5,.5) and (2.5,.5) .. (2.5,1);
          \node[align=center]  at(1.75,-1){colored crossings\\ $q^{\mu_i}$};
}
&&
\tikzdiag{
	\draw (0,0) -- (0,1) node [midway,tikzdot]{};
          \node[align=center]  at(0,-1){dot\\ $q^{2}$};
}
&&
\tikzdiag[xscale=2]{
	 \draw (.5,-.5) .. controls (.5,-.25) .. (0,0) .. controls (.5,.25) .. (.5,.5);
          \draw[pstdhl] (0,-.5) node[below]{\small $\mu_1$}-- (0,.5) node [midway,nail]{};
          \node[align=center]  at(.25,-1.5){nail\\ $hq^{2\mu_i}$};
  }
\end{align}
\item The product $x y$ of two diagrams $x$ and $y$ is given by stacking $x$ on top of $y$ if the color of the strands match, and is zero otherwise. 
\item We consider these diagrams up to braid-like planar isotopy, and subject to the following local relations:
\begin{itemize}
\item the \emph{nilHecke} relations:
\begin{align}
\label{eq:nhR2andR3}
\tikzdiag[yscale=1.5]{
	\draw (0,0) ..controls (0,.25) and (1,.25) .. (1,.5) ..controls (1,.75) and (0,.75) .. (0,1)  ;
	\draw (1,0) ..controls (1,.25) and (0,.25) .. (0,.5)..controls (0,.75) and (1,.75) .. (1,1)  ;
} 
\ &=\  
0
&
\tikzdiag[scale=1.5]{
	\draw  (0,0) .. controls (0,0.25) and (1, 0.5) ..  (1,1);
	\draw  (1,0) .. controls (1,0.5) and (0, 0.75) ..  (0,1);
	\draw  (0.5,0) .. controls (0.5,0.25) and (0, 0.25) ..  (0,0.5)
		 	  .. controls (0,0.75) and (0.5, 0.75) ..  (0.5,1);
} 
\ &= \ 
\tikzdiag[scale=1.5,xscale=-1]{
	\draw  (0,0) .. controls (0,0.25) and (1, 0.5) ..  (1,1);
	\draw  (1,0) .. controls (1,0.5) and (0, 0.75) ..  (0,1);
	\draw  (0.5,0) .. controls (0.5,0.25) and (0, 0.25) ..  (0,0.5)
		 	  .. controls (0,0.75) and (0.5, 0.75) ..  (0.5,1);
} 
\\
\label{eq:nhdotslide}
\tikzdiag{
	\draw (0,0) ..controls (0,.5) and (1,.5) .. (1,1) node [near start,tikzdot]{};
	\draw (1,0) ..controls (1,.5) and (0,.5) .. (0,1);
}
\ &= \ 
\tikzdiag{
	\draw (0,0) ..controls (0,.5) and (1,.5) .. (1,1) node [near end,tikzdot]{};
	\draw (1,0) ..controls (1,.5) and (0,.5) .. (0,1);
}
\ + \ 
\tikzdiag{
	\draw (0,0) -- (0,1)  ;
	\draw (1,0)-- (1,1)  ;
}
&
\tikzdiag{
	\draw (0,0) ..controls (0,.5) and (1,.5) .. (1,1);
	\draw (1,0) ..controls (1,.5) and (0,.5) .. (0,1) node [near end,tikzdot]{};
}
\ &= \ 
\tikzdiag{
	\draw (0,0) ..controls (0,.5) and (1,.5) .. (1,1);
	\draw (1,0) ..controls (1,.5) and (0,.5) .. (0,1) node [near start,tikzdot]{};
}
\ + \ 
\tikzdiag{
	\draw (0,0) -- (0,1)  ;
	\draw (1,0)-- (1,1)  ;
} 
\end{align}
\item the sliding relations for all $\mu_i \in \bN \sqcup (\beta + \bZ)$:
\begin{align}
\tikzdiagl[scale=1.5]{
	\draw  (0.5,0) .. controls (0.5,0.25) and (0, 0.25) ..  (0,0.5)
		 	  .. controls (0,0.75) and (0.5, 0.75) ..  (0.5,1);
	\draw (1,0)  .. controls (1,0.5) and (0, 0.75) ..  (0,1);
	\draw  [pstdhl] (0,0) node[below]{$\mu_i$}  .. controls (0,0.25) and (1, 0.5) ..  (1,1);
} 
\ &= \ 
\tikzdiagl[scale=1.5,xscale=-1]{
	\draw  (0,0) .. controls (0,0.25) and (1, 0.5) ..  (1,1);
	\draw  (0.5,0) .. controls (0.5,0.25) and (0, 0.25) ..  (0,0.5)
		 	  .. controls (0,0.75) and (0.5, 0.75) ..  (0.5,1);
	\draw [pstdhl] (1,0) node[below]{\small $\mu_i$}  .. controls (1,0.5) and (0, 0.75) ..  (0,1);
} 
&
\tikzdiagl[scale=1.5]{
	\draw  (0,0) .. controls (0,0.25) and (1, 0.5) ..  (1,1);
	\draw  (0.5,0) .. controls (0.5,0.25) and (0, 0.25) ..  (0,0.5)
		 	  .. controls (0,0.75) and (0.5, 0.75) ..  (0.5,1);
	\draw [pstdhl] (1,0) node[below]{\small $\mu_i$} .. controls (1,0.5) and (0, 0.75) ..  (0,1);
} 
\ &= \ 
\tikzdiagl[scale=1.5,xscale=-1]{
	\draw  (0.5,0) .. controls (0.5,0.25) and (0, 0.25) ..  (0,0.5)
		 	  .. controls (0,0.75) and (0.5, 0.75) ..  (0.5,1);
	\draw (1,0)  .. controls (1,0.5) and (0, 0.75) ..  (0,1);
	\draw  [pstdhl] (0,0) node[below]{\small $\mu_i$} .. controls (0,0.25) and (1, 0.5) ..  (1,1);
} 
\label{eq:crossingslidered}
\\
\tikzdiagl{
	\draw (1,0) ..controls (1,.5) and (0,.5) .. (0,1) node [near end,tikzdot]{};
	\draw[pstdhl] (0,0) node[below]{\small $\mu_i$}  ..controls (0,.5) and (1,.5) .. (1,1);
}
\ &= \ 
\tikzdiagl{
	\draw (1,0) ..controls (1,.5) and (0,.5) .. (0,1) node [near start,tikzdot]{};
	\draw[pstdhl] (0,0) node[below]{\small $\mu_i$}  ..controls (0,.5) and (1,.5) .. (1,1);
}
&
\tikzdiagl{
	\draw (0,0) ..controls (0,.5) and (1,.5) .. (1,1) node [near start,tikzdot]{};
	\draw[pstdhl] (1,0) node[below]{\small $\mu_i$}  ..controls (1,.5) and (0,.5) .. (0,1);
}
\ &= \ 
\tikzdiagl{
	\draw (0,0) ..controls (0,.5) and (1,.5) .. (1,1) node [near end,tikzdot]{};
	\draw[pstdhl] (1,0) node[below]{\small $\mu_i$}  ..controls (1,.5) and (0,.5) .. (0,1);
} 
\label{eq:dotredstrand}
\end{align}
\item the red relations for all $\mu_i \in \bN$ and $i > 1$:
\begin{align}
\tikzdiagh[yscale=1.5]{0}{
	\draw (1,0) ..controls (1,.25) and (0,.25) .. (0,.5)..controls (0,.75) and (1,.75) .. (1,1)  ;
	\draw[stdhl] (0,0) node[below]{\small $\mu_i$} ..controls (0,.25) and (1,.25) .. (1,.5) ..controls (1,.75) and (0,.75) .. (0,1)  ;
} 
\ &= \ 
\tikzdiagh[yscale=1.5]{0}{
	\draw[stdhl] (0,0) node[below]{\small $\mu_i$} -- (0,1)  ;
	\draw (1,0) -- (1,1)  node[midway,tikzdot]{}  node[midway,xshift=1.75ex,yshift=.75ex]{\small $\mu_i$} ;
} 
&
\tikzdiagh[yscale=1.5]{0}{
	\draw (0,0) ..controls (0,.25) and (1,.25) .. (1,.5) ..controls (1,.75) and (0,.75) .. (0,1)  ;
	\draw[stdhl] (1,0) node[below]{\small $\mu_i$} ..controls (1,.25) and (0,.25) .. (0,.5)..controls (0,.75) and (1,.75) .. (1,1)  ;
} 
\ &= \ 
\tikzdiagh[yscale=1.5]{0}{
	\draw (0,0) -- (0,1)  node[midway,tikzdot]{}   node[midway,xshift=1.75ex,yshift=.75ex]{\small $\mu_i$} ;
	\draw[stdhl] (1,0) node[below]{\small $\mu_i$} -- (1,1)  ;
} 
\label{eq:redR2}
\end{align}
\begin{align}
\tikzdiagh[scale=1.5]{0}{
	\draw  (0,0) .. controls (0,0.25) and (1, 0.5) ..  (1,1);
	\draw  (1,0) .. controls (1,0.5) and (0, 0.75) ..  (0,1);
	\draw [stdhl] (0.5,0)node[below]{\small $\mu_i$}  .. controls (0.5,0.25) and (0, 0.25) ..  (0,0.5)
		 	  .. controls (0,0.75) and (0.5, 0.75) ..  (0.5,1);
} 
\ &= \ 
\tikzdiagh[scale=1.5,xscale=-1]{0}{
	\draw  (0,0) .. controls (0,0.25) and (1, 0.5) ..  (1,1);
	\draw (1,0)  .. controls (1,0.5) and (0, 0.75) ..  (0,1);
	\draw [stdhl]  (0.5,0) node[below]{\small $\mu_i$} .. controls (0.5,0.25) and (0, 0.25) ..  (0,0.5)
		 	  .. controls (0,0.75) and (0.5, 0.75) ..  (0.5,1);
} 
\ + \sssum{u+v=\\\mu_i-1} \ 
\tikzdiagh[scale=1.5]{0}{
	\draw  (0,0) -- (0,1) node[midway,tikzdot]{} node[midway,xshift=-1.5ex,yshift=.75ex]{\small $u$};
	\draw  (1,0) --  (1,1) node[midway,tikzdot]{} node[midway,xshift=1.5ex,yshift=.75ex]{\small $v$};
	\draw [stdhl] (0.5,0)node[below]{\small $\mu_i$}  --  (0.5,1);
} \label{eq:redR3}
\end{align}
where a non-negative label $k$ next to a dot means we put $k$ consecutive dots,
\item the blue relations for all $\mu_i \in \beta + \bZ$ and $i > 1$:
\begin{align}\label{eq:vredR}
	\tikzdiagl[yscale=1.5]{
		\draw (1,0) ..controls (1,.25) and (0,.25) .. (0,.5)..controls (0,.75) and (1,.75) .. (1,1)  ;
		\draw[vstdhl] (0,0)node[below]{\small $\mu_i$} ..controls (0,.25) and (1,.25) .. (1,.5) ..controls (1,.75) and (0,.75) .. (0,1)  ;
	} 
	\ &= 0,
	&
	\tikzdiagl[yscale=1.5]{
		\draw (0,0) ..controls (0,.25) and (1,.25) .. (1,.5) ..controls (1,.75) and (0,.75) .. (0,1)  ;
		\draw[vstdhl] (1,0)node[below]{\small $\mu_i$} ..controls (1,.25) and (0,.25) .. (0,.5)..controls (0,.75) and (1,.75) .. (1,1)  ;
	} 
	\ &= 0,
&
	\tikzdiagl[scale=1.5]{
		\draw  (0,0) .. controls (0,0.25) and (1, 0.5) ..  (1,1);
		\draw  (1,0) .. controls (1,0.5) and (0, 0.75) ..  (0,1);
		\draw [vstdhl] (0.5,0)node[below]{\small $\mu_i$}  .. controls (0.5,0.25) and (0, 0.25) ..  (0,0.5)
			 	  .. controls (0,0.75) and (0.5, 0.75) ..  (0.5,1);
	} 
	\ &= \ 
	\tikzdiagl[scale=1.5,xscale=-1]{
		\draw  (0,0) .. controls (0,0.25) and (1, 0.5) ..  (1,1);
		\draw (1,0)  .. controls (1,0.5) and (0, 0.75) ..  (0,1);
		\draw [vstdhl]  (0.5,0)node[below]{\small $\mu_i$}  .. controls (0.5,0.25) and (0, 0.25) ..  (0,0.5)
			 	  .. controls (0,0.75) and (0.5, 0.75) ..  (0.5,1);
	} 
\end{align}
\item the nail relations:
\begin{align}  \label{eq:nailsrel} 
	\tikzdiagl[xscale=2]{
		 \draw (.5,-.5) .. controls (.5,-.25) .. (0,0) .. controls (.5,.25) .. (.5,.5)  node[midway, tikzdot]{};
	          \draw[pstdhl] (0,-.5) node[below]{\small $\mu_1$} -- (0,.5) node [midway,nail]{};
  	}
\  &= \ 
	\tikzdiagl[xscale=2]{
		 \draw (.5,-.5) .. controls (.5,-.25) .. (0,0)  node[midway, tikzdot]{} .. controls (.5,.25) .. (.5,.5);
	          \draw[pstdhl] (0,-.5) node[below]{\small $\mu_1$} -- (0,.5) node [midway,nail]{};
  	}
  &
	\tikzdiagl[xscale=-1,yscale=.75]{
		\draw (-1.5,-.75) -- (-1.5,0) .. controls (-1.5,.5) .. (0,.75) .. controls (-1.5,1) .. (-1.5,1.5);
		\draw (-.75,-.75) .. controls (-.75,-.25) .. (0,0) .. controls (-.75,.25) .. (-.75,.75) -- (-.75,1.5);
		\draw[pstdhl] (0,-.75) node[below]{\small $\mu_1$} -- (0,0) node[pos=1,nail]{} -- (0,.75) node[pos=1,nail]{}  -- (0,1.5);
	}
\ &= - \ 
	\tikzdiagl[xscale=-1,yscale=-.75]{
		\draw (-1.5,-.75) -- (-1.5,0) .. controls (-1.5,.5) .. (0,.75) .. controls (-1.5,1) .. (-1.5,1.5);
		\draw (-.75,-.75) .. controls (-.75,-.25) .. (0,0) .. controls (-.75,.25) .. (-.75,.75) -- (-.75,1.5);
		\draw[pstdhl] (0,-.75) -- (0,0) node[pos=1,nail]{} -- (0,.75) node[pos=1,nail]{}  -- (0,1.5) node[below]{$\mu_1$} ;
	}
&
	\tikzdiagl[xscale=-1,yscale=.75]{
		\begin{scope}
			\clip(0,-.75) rectangle (-.75,1.5);
			\draw (.75,-.75) .. controls (.75,0) and (-.5,0) .. (-.5,.375) 
				.. controls (-.5,.75) and (.75,.75) .. (.75,1.5);
		\end{scope}
		\draw (-.75,-.75) .. controls (-.75,-.25) .. (0,0)
			(0,.75) .. controls (-.75,1.125) .. (-.75,1.5);
		\draw[pstdhl] (0,-.75)  node[below]{\small $\mu_1$} -- (0,0) node[pos=1,nail]{} -- (0,.75) node[pos=1,nail]{}  -- (0,1.5);
	}
\ &= 0.
\end{align}
\end{itemize}
\item We endow $\dgT^{\und \mu}_b$ with a $\bZ \times \bZ^2$ grading, where the first grading is homological and denoted $h$, and the second and third one are extra grading denoted $q$ and $\lambda$ respectively. For this, we declare that the generators are in degree given by the monomial written below them in \cref{eq:generators}, where the monomial $h^a q^{b + c\beta} := h^a q^b \lambda^c$ means the element is in homological degree $a$, $q$-degree $b$ and $\lambda$-degree $c$. 
\item We turn $\dgT^{\und \mu}_b$ into a $\bZ^2$-graded dg-algebra $(\dgT^{\und \mu}_b, d_\mu)$ by defining a differential $d_{\mu}$ as being zero on the dots and crossings, and
\begin{align*}
d_{\mu}\left(
\tikzdiag[xscale=2]{
	 \draw (.5,-.5) .. controls (.5,-.25) .. (0,0) .. controls (.5,.25) .. (.5,.5);
          \draw[pstdhl] (0,-.5) node[below]{\small $\mu_1$}-- (0,.5) node [midway,nail]{};
  }
  \right)
  := \begin{cases}
  \hfill 0, \hfill & \text{if $\mu_1 \in \beta + \bZ$,} \\[.3cm]
  \tikzdiagl[xscale=2]{
	 \draw (.5,-.5) -- (.5,.5) node[midway,tikzdot]{} node[midway, xshift=2ex,yshift=.75ex]{$\mu_1$};
          \draw[stdhl] (0,-.5) node[below]{\small $\mu_1$}-- (0,.5);
  }
  & \text{if $\mu_1 \in \bN$,}
  \end{cases}
\end{align*}
and extending using the graded Leibniz rule (it is straightforward to verify that $d_\mu$ is well-defined).
\end{itemize}
\end{defn}

Note that for $\mu_1 = \beta$ and all $\mu_{i} \in \bN$ for $i > 1$, then the dg-algebra $(\dgT^{\und \mu}_b, d_{\mu})$ coincides with the dg-enhanced KLRW dg-algebra of \cite[\S3.2]{LNV}. When $\mu_1 \in \bN$, then it coincides with the dg-enhanced KLRW dg-algebra of \cite[\S3.4]{LNV} equipped with the non-trivial differential. Thus we get the following:

\begin{prop}[{\cite[Theorem 3.13]{LNV}}]\label{prop:olddgKLRW}
For a string of integral dominant weights $\und \mu \in \bN^r$, there is a quasi-isomorphism
\[
(\dgT^{\und \mu}_b, d_\mu) \xrightarrow{\simeq} (T^{\und \mu}_b,  0),
\]
where $(T^{\und \mu}_b, 0)$ is the KLRW algebra (tensor product algebra) of~\cite[\S4]{webster} viewed as a $\bZ^2$-graded dg-algebra concentrated in homological and $\lambda$-degrees zero. 
\end{prop}

For the sake of keeping notations short, we introduce the following:
\[
\tikzdiag{
	\draw (0,0) -- (0,1) node[midway, tikzdot]{} node[midway, xshift=1.5ex, yshift=.75ex]{\small $\beta$};
}
 := 0.
\]
In particular, it allows us to write in general
\[
d_{\mu}\left(
\tikzdiag[xscale=2]{
	 \draw (.5,-.5) .. controls (.5,-.25) .. (0,0) .. controls (.5,.25) .. (.5,.5);
          \draw[pstdhl] (0,-.5) node[below]{\small $\mu_1$}-- (0,.5) node [midway,nail]{};
  }
  \right)
  =
  \tikzdiag[xscale=2]{
	 \draw (.5,-.5) -- (.5,.5) node[midway,tikzdot]{} node[midway, xshift=2ex,yshift=.75ex]{$\mu_1$};
          \draw[pstdhl] (0,-.5) node[below]{\small $\mu_1$}-- (0,.5);
  }
\]
and rewrite the relations \eqref{eq:redR2}-\eqref{eq:vredR} as
\begin{align*}
\tikzdiagh[yscale=1.5]{0}{
	\draw (1,0) ..controls (1,.25) and (0,.25) .. (0,.5)..controls (0,.75) and (1,.75) .. (1,1)  ;
	\draw[pstdhl] (0,0) node[below]{\small $\mu_i$} ..controls (0,.25) and (1,.25) .. (1,.5) ..controls (1,.75) and (0,.75) .. (0,1)  ;
} 
\ &= \ 
\tikzdiagh[yscale=1.5]{0}{
	\draw[pstdhl] (0,0) node[below]{\small $\mu_i$} -- (0,1)  ;
	\draw (1,0) -- (1,1)  node[midway,tikzdot]{}  node[midway,xshift=1.75ex,yshift=.75ex]{\small $\mu_i$} ;
} 
&
\tikzdiagh[yscale=1.5]{0}{
	\draw (0,0) ..controls (0,.25) and (1,.25) .. (1,.5) ..controls (1,.75) and (0,.75) .. (0,1)  ;
	\draw[pstdhl] (1,0) node[below]{\small $\mu_i$} ..controls (1,.25) and (0,.25) .. (0,.5)..controls (0,.75) and (1,.75) .. (1,1)  ;
} 
\ &= \ 
\tikzdiagh[yscale=1.5]{0}{
	\draw (0,0) -- (0,1)  node[midway,tikzdot]{}   node[midway,xshift=1.75ex,yshift=.75ex]{\small $\mu_i$} ;
	\draw[pstdhl] (1,0) node[below]{\small $\mu_i$} -- (1,1)  ;
} 
\end{align*}
\begin{align*}
\tikzdiagh[scale=1.5]{0}{
	\draw  (0,0) .. controls (0,0.25) and (1, 0.5) ..  (1,1);
	\draw  (1,0) .. controls (1,0.5) and (0, 0.75) ..  (0,1);
	\draw [pstdhl] (0.5,0)node[below]{\small $\mu_i$}  .. controls (0.5,0.25) and (0, 0.25) ..  (0,0.5)
		 	  .. controls (0,0.75) and (0.5, 0.75) ..  (0.5,1);
} 
\ &= \ 
\tikzdiagh[scale=1.5,xscale=-1]{0}{
	\draw  (0,0) .. controls (0,0.25) and (1, 0.5) ..  (1,1);
	\draw (1,0)  .. controls (1,0.5) and (0, 0.75) ..  (0,1);
	\draw [pstdhl]  (0.5,0) node[below]{\small $\mu_i$} .. controls (0.5,0.25) and (0, 0.25) ..  (0,0.5)
		 	  .. controls (0,0.75) and (0.5, 0.75) ..  (0.5,1);
} 
\ + \sssum{u+v=\\\mu_i-1} \ 
\tikzdiagh[scale=1.5]{0}{
	\draw  (0,0) -- (0,1) node[midway,tikzdot]{} node[midway,xshift=-1.5ex,yshift=.75ex]{\small $u$};
	\draw  (1,0) --  (1,1) node[midway,tikzdot]{} node[midway,xshift=1.5ex,yshift=.75ex]{\small $v$};
	\draw [pstdhl] (0.5,0)node[below]{\small $\mu_i$}  --  (0.5,1);
} 
\end{align*}
where the sum is zero whenever $\mu_i \in \beta + \bZ$ since there are no pair of non-negative integer $u$ and $v$ such that $u+v = \beta + z -1$. 

\subsection{Basis}

For any $\rho=(b_1,\dots,b_r)\in \mathcal{P}_b^r$, define the idempotent
\[
1_{\rho} := \tikzdiagh{0}{
	\draw[pstdhl] (2,0)  node[below]{\small $\mu_1$} --(2,1);
	\draw (2.5,0) -- (2.5,1);
	\node at(3,.5) {\tiny$\dots$};
	\draw (3.5,0) -- (3.5,1);
	\draw[decoration={brace,mirror,raise=-8pt},decorate]  (2.4,-.35) -- node {$b_1$} (3.6,-.35);
	\draw[pstdhl] (4,0)  node[below]{\small $\mu_2$} --(4,1);
	\node[pcolor] at  (5,.5) {\dots};
	\draw[pstdhl] (6,0)  node[below]{\small ${\mu_r}$} --(6,1);
	\draw (6.5,0) -- (6.5,1);
	\node at(7,.5) {\tiny$\dots$};
	\draw (7.5,0) -- (7.5,1);
	\draw[decoration={brace,mirror,raise=-8pt},decorate]  (6.4,-.35) -- node {$b_r$} (7.6,-.35);
}
\]
of $\muT_{b}$. 
We will construct a $\bZ\times\bZ^2$-graded $\Bbbk$-basis ${}_\kappa B_\rho$ for $1_\kappa \muT_{b}1_\rho$,  similarly as in \cite[Section 3.2.3]{NV3}.

\subsubsection{Left-adjusted expressions}
Let $S_n$ be the symmetric group viewed as a Coxeter group generated by the simple transpositions $\sigma_1, \dots, \sigma_{n-1}$. 
Recall the notion of left-adjusted expressions of \cite[Section 2.2.1]{NV2}: a reduced expression $\sigma_{i_1}\cdots\sigma_{i_k}$ of an element $w\in S_{n}$ is said to be left-adjusted if $i_1+\cdots + i_k$ is minimal. One can obtain a left-adjusted expression of any element of $S_{n}$ by taking recursively its representative in the left coset decomposition
\[
  S_n = \bigsqcup_{t=1}^{n}S_{n-1}\sigma_{n-1}\cdots\sigma_{t}.
\]

If we think of permutations as string diagrams, a left-adjusted reduced expression is obtained by pulling every string as far as possible to the left.

\subsubsection{A basis of $\muT_{b}$}\label{sec:Tbasis}
For an element $\rho \in \mathcal{P}^r_b$ and $1 \leq k \leq b$, we define the tightened nail $\theta_{k,\rho} \in 1_\rho \muT_{b} 1_\rho$ as the following element:
\[
\theta_{k,\rho} :=\tikzdiagh[xscale=1.25]{-1.5ex}{
	\draw (0,-1) -- (0,1);
	\node at(.25,-.85) {\tiny $\dots$};
	\node at(.25,.85) {\tiny $\dots$};
	\draw (.5,-1) -- (.5,1);
	%
	%
	\node[pcolor] at(1.125,-.85) { $\dots$};
	\node[pcolor] at(1.125,.85) { $\dots$};
	\draw (1.75,-1) -- (1.75,1);
	\node at(2,-.85) {\tiny $\dots$};
	\node at(2,.85) {\tiny $\dots$};
	\draw (2.25,-1) -- (2.25,1);
	\draw (2.5,-1) 
	    .. controls (2.5,-.25) and (-.5,-.25) ..  (-.25,0) 
	    .. controls (-.5,.25) and (2.5,.25) ..  (2.5,1);
	\draw (2.75,-1) -- (2.75,1);
	\node at(3,-.85) {\tiny $\dots$};
	\node at(3,.85) {\tiny $\dots$};
	\draw (3.25,-1) -- (3.25,1);
	\draw [pcolor]  (3.5,-1)  node[below]{\small $\mu_{i{+}1}$} -- (3.5,1);
	\node[pcolor] at(3.875,-.85) { $\dots$};
	\node[pcolor] at(3.875,.85) { $\dots$};
	\draw [pstdhl]  (4.25,-1)  node[below]{\small $\mu_{r}$} -- (4.25,1);
	\draw (4.5,-1) -- (4.5,1);
	\node at(4.75,-.85) {\tiny $\dots$};
	\node at(4.75,.85) {\tiny $\dots$};
	\draw (5,-1) -- (5,1);
	%
	\draw [pstdhl] (1.5,-1)  node[below]{\small $\mu_i$} -- (1.5,1);
	\draw [pstdhl] (.75,-1)  node[below]{\small $\mu_2$} -- (.75,1);
	\draw [pstdhl] (-.25,-1) node[below]{\small $\mu_1$} -- (-.25,1)  node[midway,nail]{};
	\tikzbraceop{-.25}{2.5}{1}{\small $k+i$};
	}
\]
where the nail is on the $k$-th black strand from the left. This element is of degree $\deg (\theta_{k,\rho}) = q^{2(\mu_1 + \cdots + \mu_i)-4(k-1)}$.

\begin{lem}
  \label{lem:anticom_theta}
  Tightened nails anticommute with each other:
  \begin{align*}
      \theta_{k,\rho} \theta_{\ell,\rho} &= -\theta_{\ell,\rho} \theta_{k,\rho}, & \theta_{k,\rho}^2 &= 0,
  \end{align*}
for all $1 \leq k, \ell \leq b$. 
\end{lem}

\begin{proof}
 It follows from \cref{lem:bignailcommutes} and \cref{prop:doublenailsamestrandrewritestozero}. 
\end{proof}

Fix $\kappa,\rho\in\mathcal{P}^r_b$ and consider the subset of permutations $w\in{}_\kappa S_\rho$ of $S_{r+b}$, viewed as strand diagrams with $b$ black strands and $r$ colored strands, such that:
\begin{itemize}
\item there are no black strand on the left,
\item the strands are ordered at the bottom by $1_\kappa$ and at the top by $1_\rho$,
\item for any reduced expression of $w$, there is no crossing between colored strands.
\end{itemize}

\begin{ex}
  If $\kappa=\rho=(0,1,1)$, the set ${}_\kappa S_\rho$ has two elements, namely
  \[
    \tikzdiagh[xscale=2]{-1.5ex}{
      \draw[pstdhl] (0,0) -- (0,1);
      \draw[pstdhl] (.25,0) -- (.25,1);
      \draw (0.5,0) -- (.5,1);
      \draw[pstdhl] (.75,0) -- (.75,1);
      \draw (1,0) -- (1,1);
    }
    \quad\text{and}\quad
    \tikzdiagh[xscale=2]{-1.5ex}{
      \draw[pstdhl] (0,0) -- (0,1);
      \draw[pstdhl] (.25,0) -- (.25,1);
      \draw (.5,0) ..controls (.5,.5) and (1,.5) .. (1,1);
      \draw (1,0) ..controls (1,.5) and (.5,.5) .. (.5,1);
      \draw[pstdhl] (.75,0) ..controls (.75,.35) and (1,.15) .. (1,.5);
      \draw[pstdhl] (1,.5) ..controls (1,.85) and (.75,.65) ..(.75,1);
    }
  \]
  Note that the second element is not left-adjusted.
\end{ex}

For each $w\in {}_\kappa S_\rho,\ \underline{l}=(l_1,\ldots,l_b)\in \{0,1\}^b$ and $\underline{a}=(a_1,\ldots,a_b)\in\mathbb{N}^b$ we define an element $b_{w,\underline{l},\underline{a}}\in 1_\kappa \muT_{b} 1_\rho$ as follows:
\begin{enumerate}
\item choose a left-adjusted reduced expression of $w$ in terms of diagrams as above,
\item for each $1\leq i \leq b$, if $l_i=1$, nail the $i$-th black strand (counting at the top, from the left) on the left-most colored strand by pulling it from its leftmost position,
\item for each $1\leq i \leq b$, add $a_i$ dots on the $i$-th black strand at the top.
\end{enumerate}

Let ${}_\kappa B_\rho$ be the set of all $b_{w,\underline{l},\underline{a}}$ for $w\in {}_\kappa S_\rho,\ \underline{l}\in \{0,1\}^b$ and $\underline{a}\in\mathbb{N}^b$, where we also assume that the tightened floating dots are ordered such that whenever we have $\theta_{k,\rho} \theta_{\ell,\rho}$, then $\ell > k$. 

\begin{ex}
  We continue the example of $\kappa=\rho=(0,1,1)$. If we choose for $w$ the permutation with a black/black crossing, $\underline{l}=(1,0)$ and $\underline{a}=(0,1)$ we have
  \[
    b_{w,\underline{l},\underline{a}}  =
      \tikzdiagh[xscale=2]{-1.5ex}{
      \draw (.5,0) ..controls (.5,.5) and (1,.5) .. (1,1) node [pos=.85,tikzdot]{};
      \draw (.5,1) ..controls (.5,.85) and (0,.90) .. (0,.75);
      \draw (0,.75) ..controls (0,.60) and (1,.75) .. (1,0);
      \draw[pstdhl] (.25,0) -- (.25,1);
      \draw[pstdhl] (.75,0) ..controls (.75,.35) and (.5,.15) .. (.5,.5);
      \draw[pstdhl] (.5,.5) ..controls (.5,.85) and (.75,.65) ..(.75,1);
      \draw[pstdhl] (0,0) -- (0,1)  node [pos=.75,nail]{};
    }
  \] 
 Note that we added the nail at the top and not the bottom because that is where the black strand is at its left-most position. 
\end{ex}

\begin{thm}\label{thm:Tbasis}
As a $\bZ \times \bZ^2$-graded $\Bbbk$-module, $1_\kappa \muT_{b}  1_\rho$ is free with basis given by ${}_\kappa B_\rho$ .
\end{thm}

\begin{proof}
The statement is given by \cref{cor:basis} in the next section. 
\end{proof}

\subsubsection{Left decomposition}
In the following, we draw $\muT_b 1_{\rho}$ with $\rho = (b_0, \dots, b_r)$ as a box diagram
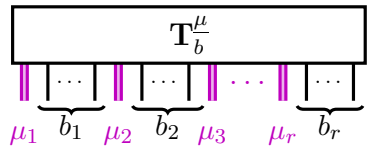
\[
\tikzdiag[xscale=1.25]{
	\draw [pstdhl] (-.25,0) node[below,yshift={-1ex}]{\small $\mu_1$} -- (-.25,1);
	\draw (0,0) -- (0,1);
	\node at(.25,.25) {\tiny $\dots$};
	\draw (.5,0) -- (.5,1);
	\draw[decoration={brace,mirror,raise=-8pt},decorate]  (-.1,-.35) -- node {\small $b_1$} (.6,-.35);
	\draw [pstdhl] (.75,0)  node[below,yshift={-1ex}]{\small $\mu_2$} -- (.75,1);
	\draw (1,0) -- (1,1);
	\node at(1.25,.25) {\tiny $\dots$};
	\draw (1.5,0) -- (1.5,1);
	\draw[decoration={brace,mirror,raise=-8pt},decorate]  (.9,-.35) -- node {\small $b_{2}$} (1.65,-.35);
	\draw [pstdhl] (1.75,0)  node[below,yshift={-1ex}]{\small $\mu_{3}$} -- (1.75,1);
	\node[pcolor] at(2.15,.25) { $\dots$};
	\draw [pstdhl] (2.5,0)  node[below,yshift={-1ex}]{\small $\mu_{r}$} -- (2.5,1);
	\draw (2.75,0) -- (2.75,1);
	\node at(3,.25) {\tiny $\dots$};
	\draw (3.25,0) -- (3.25,1);
	\draw[decoration={brace,mirror,raise=-8pt},decorate]  (2.65,-.35) -- node {\small $b_r$} (3.35,-.35);
	\filldraw [fill=white, draw=black] (-.375,.5) rectangle (3.375,1.25) node[midway] { $\muT_b$};
}
\]

Let $\rho_{\hat i} := (b_1, \dots, b_{i-1}, b_i -1, b_{i+1}, \dots, b_r)$. 
When we draw a box in a diagram as follows:
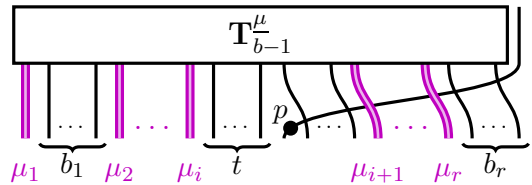
\[
\tikzdiag[xscale=1.25]{
	\draw[fill=white, color=white] (-.35,0) circle (.15cm);
	\draw [pstdhl] (-.25,-.5) node[below,yshift={-1ex}]{\small $\mu_1$} -- (-.25,1);
	\draw (0,-.5) -- (0,1);
	\node at(.25,-.35) {\tiny $\dots$};
	\draw (.5,-.5) -- (.5,1);
	\draw[decoration={brace,mirror,raise=-8pt},decorate]  (-.1,-.85) -- node {\small $b_1$} (.6,-.85);
	\draw [pstdhl] (.75,-.5)  node[below,yshift={-1ex}]{\small $\mu_2$} -- (.75,1);
	\node[pcolor] at(1.125,-.35) { $\dots$};
	\draw [pstdhl] (1.5,-.5)  node[below,yshift={-1ex}]{\small $\mu_i$} -- (1.5,1);
	\draw (1.75,-.5) -- (1.75,1);
	\node at(2,-.35) {\tiny $\dots$};
	\draw (2.25,-.5) -- (2.25,1);
	\draw[decoration={brace,mirror,raise=-8pt},decorate]  (1.65,-.85) -- node {\small $t$} (2.35,-.85);
	\draw (2.5,-.5) .. controls (2.5,0) and (5,0) ..  (5,.5) node[pos=.1,tikzdot]{} node[pos=.1, xshift=-.75ex, yshift=1ex]{\small $p$} -- (5,1.25);
	\draw (2.75,-.5) .. controls (2.75,0) and (2.5,0) .. (2.5,.5);
	\node at(3,-.35) {\tiny $\dots$};
	\draw (3.25,-.5) .. controls (3.25,0) and (3,0) .. (3,.5);
	\draw [pstdhl]  (3.5,-.5)  node[below,yshift={-1ex}]{\small $\mu_{i+1}$} .. controls (3.5,0) and (3.25,0) .. (3.25,.5);
	\node[pcolor] at(3.875,-.35) { $\dots$};
	\draw [pstdhl]  (4.25,-.5)  node[below,yshift={-1ex}]{\small $\mu_{r}$} .. controls (4.25,0) and (4,0) .. (4,.5);
	\draw (4.5,-.5) .. controls (4.5,0) and (4.25,0) .. (4.25,.5);
	\node at(4.75,-.35) {\tiny $\dots$};
	\draw (5,-.5) .. controls (5,0) and (4.75,0) .. (4.75,.5);
	\draw[decoration={brace,mirror,raise=-8pt},decorate]  (4.4,-.85) -- node {\small $b_r$} (5.1,-.85);
	\filldraw [fill=white, draw=black] (-.375,.5) rectangle (4.875,1.25) node[midway] { $\muT_{b-1}$};
}
\]
with $p \geq 0$ and $0 \leq t < b_i$, it means we consider the subset of $\muT_b 1_{\rho}$ isomorphic to a grading shift of $\muT_{b-1}1_{\rho_{\hat i}}$ given by replacing the box labeled $\muT_{b-1}$ with any diagram of $\muT_{b-1}$ in the diagram above, and consider it as a diagram of $\muT_b 1_{\rho}$. 
We also write
\[
\tikzdiagl[xscale=2]{
	 \draw (.5,-.5) .. controls (.5,-.25) .. (0,0) .. controls (.5,.25) .. (.5,.5);
          \draw[pstdhl] (0,-.5) node[below]{\small $\mu_1$}-- (0,.5) node [midway,nail]{} node[midway, xshift=-1.25ex,yshift=.75ex, black]{\small $p$};
  }
\ :=\ 
\tikzdiagl[xscale=2]{
	 \draw (.5,-.5) .. controls (.5,-.25) .. (0,0) .. controls (.5,.25) .. (.5,.5) node[pos=.5,tikzdot]{} node[pos=.5, xshift=-1ex, yshift=.75ex]{\small $p$};
          \draw[pstdhl] (0,-.5) node[below]{\small $\mu_1$}-- (0,.5) node [midway,nail]{};
  }
  \ = \ 
\tikzdiagl[xscale=2]{
	 \draw (.5,-.5) .. controls (.5,-.25) .. (0,0) node[pos=.5,tikzdot]{} node[pos=.5, xshift=-1ex, yshift=-.75ex]{\small $p$} .. controls (.5,.25) .. (.5,.5) ;
          \draw[pstdhl] (0,-.5) node[below]{\small $\mu_1$}-- (0,.5) node [midway,nail]{};
  }
\]
for all $p \geq 0$, and
\[
\theta_{k,\rho}(p) :=\tikzdiagh[xscale=1.25]{-1.5ex}{
	\draw (0,-1) -- (0,1);
	\node at(.25,-.85) {\tiny $\dots$};
	\node at(.25,.85) {\tiny $\dots$};
	\draw (.5,-1) -- (.5,1);
	%
	%
	\node[pcolor] at(1.125,-.85) { $\dots$};
	\node[pcolor] at(1.125,.85) { $\dots$};
	\draw (1.75,-1) -- (1.75,1);
	\node at(2,-.85) {\tiny $\dots$};
	\node at(2,.85) {\tiny $\dots$};
	\draw (2.25,-1) -- (2.25,1);
	\draw (2.5,-1) 
	    .. controls (2.5,-.25) and (-.5,-.25) ..  (-.25,0) 
	    .. controls (-.5,.25) and (2.5,.25) ..  (2.5,1);
	\draw (2.75,-1) -- (2.75,1);
	\node at(3,-.85) {\tiny $\dots$};
	\node at(3,.85) {\tiny $\dots$};
	\draw (3.25,-1) -- (3.25,1);
	\draw [pcolor]  (3.5,-1)  node[below]{\small $\mu_{i{+}1}$} -- (3.5,1);
	\node[pcolor] at(3.875,-.85) { $\dots$};
	\node[pcolor] at(3.875,.85) { $\dots$};
	\draw [pstdhl]  (4.25,-1)  node[below]{\small $\mu_{r}$} -- (4.25,1);
	\draw (4.5,-1) -- (4.5,1);
	\node at(4.75,-.85) {\tiny $\dots$};
	\node at(4.75,.85) {\tiny $\dots$};
	\draw (5,-1) -- (5,1);
	%
	\draw [pstdhl] (1.5,-1)  node[below]{\small $\mu_i$} -- (1.5,1);
	\draw [pstdhl] (.75,-1)  node[below]{\small $\mu_2$} -- (.75,1);
	\draw [pstdhl] (-.25,-1) node[below]{\small $\mu_1$} -- (-.25,1)  node[midway,nail]{} node[midway, xshift=-1.25ex,yshift=.75ex, black]{\small $p$};
	\tikzbraceop{-.25}{2.5}{1}{\small $k+i$};
	}
\]
Note that $\theta_{k,\rho}(0) = \theta_{k,\rho}$.

\begin{prop}\label{prop:Tdecomp}
As a $\bZ\times \bZ^2$-graded $\Bbbk$-module, $\muT_{b} 1_{\rho}$ decomposes as a direct sum
\begin{align}\label{eq:Tdecomp}
\begin{split}
\tikzdiag[xscale=1.25]{
	\draw [pstdhl] (-.25,-.5) node[below,yshift={-1ex}]{\small $\mu_1$} -- (-.25,1);
	\draw (0,-.5) -- (0,1);
	\node at(.25,0) {\tiny $\dots$};
	\draw (.5,-.5) -- (.5,1);
	\draw[decoration={brace,mirror,raise=-8pt},decorate]  (-.1,-.85) -- node {\small $b_1$} (.6,-.85);
	\draw [pstdhl] (.75,-.5)  node[below,yshift={-1ex}]{\small $\mu_2$} -- (.75,1);
	\draw (1,-.5) -- (1,1);
	\node at(1.25,0) {\tiny $\dots$};
	\draw (1.5,-.5) -- (1.5,1);
	\draw[decoration={brace,mirror,raise=-8pt},decorate]  (.9,-.85) -- node {\small $b_{2}$} (1.65,-.85);
	\draw [pstdhl] (1.75,-.5)  node[below,yshift={-1ex}]{\small $\mu_{3}$} -- (1.75,1);
	\node[pcolor] at(2.15,0) { $\dots$};
	\draw [pstdhl] (2.5,-.5)  node[below,yshift={-1ex}]{\small $\mu_{r}$} -- (2.5,1);
	\draw (2.75,-.5) -- (2.75,1);
	\node at(3,0) {\tiny $\dots$};
	\draw (3.25,-.5) -- (3.25,1);
	\draw[decoration={brace,mirror,raise=-8pt},decorate]  (2.65,-.85) -- node {\small $b_r$} (3.35,-.85);
	\filldraw [fill=white, draw=black] (-.375,.5) rectangle (3.375,1.25) node[midway] { $\muT_b$};
}
\ \cong& \ 
\tikzdiag[xscale=1.25]{
	\draw (0,-.5) -- (0,1);
	\node at(.25,-.35) {\tiny $\dots$};
	\draw (.5,-.5) -- (.5,1);
	%
	\draw [pstdhl] (.75,-.5)  node[below]{\small $\mu_2$} -- (.75,1);
	\node[pcolor] at(1.125,-.35) { $\dots$};
	\draw [pstdhl] (1.5,-.5)  node[below]{\small $\mu_{r{-}1}$} -- (1.5,1);
	\draw (1.75,-.5) -- (1.75,1);
	\node at(2,-.35) {\tiny $\dots$};
	\draw (2.25,-.5) -- (2.25,1);
	%
	\draw (2.75,-.5) .. controls (2.75,0) and (2.5,0) .. (2.5,.5) -- (2.5,1);
	\node at(3,-.35) {\tiny $\dots$};
	\draw (3.25,-.5) .. controls (3.25,0) and (3,0) .. (3,.5) -- (3,1);
	%
	\draw [pstdhl] (2.5,-.5)  node[below]{\small $\mu_r$} .. controls (2.5,0) and (3.5,0) .. (3.5,.5) -- (3.5,1.25);
	\draw [pstdhl] (-.25,-.5) node[below]{\small $\mu_1$} -- (-.25,1);
	\filldraw [fill=white, draw=black] (-.375,.5) rectangle (3.125,1.25) node[midway] { $\dgT^{\und \mu'}_{b}$};
}
\\
&\oplus
\bigoplus_{i=1}^r
\ssbigoplus{0 \leq t < b_i \\ p \geq 0}
\tikzdiag[xscale=1.25]{
	\draw[fill=white, color=white] (-.35,0) circle (.15cm);
	\draw [pstdhl] (-.25,-.5) node[below]{\small $\mu_1$} -- (-.25,1)  node[pos=.33, xshift=-1.25ex,yshift=.75ex, white]{\small $p$};
	\draw (0,-.5) -- (0,1);
	\node at(.25,-.35) {\tiny $\dots$};
	\draw (.5,-.5) -- (.5,1);
	%
	\draw [pstdhl] (.75,-.5)  node[below]{\small $\mu_2$} -- (.75,1);
	\node[pcolor] at(1.125,-.35) { $\dots$};
	\draw [pstdhl] (1.5,-.5)  node[below]{\small $\mu_i$} -- (1.5,1);
	\draw (1.75,-.5) -- (1.75,1);
	\node at(2,-.35) {\tiny $\dots$};
	\draw (2.25,-.5) -- (2.25,1);
	\draw[decoration={brace,mirror,raise=-8pt},decorate]  (1.65,-.85) -- node {\small $t$} (2.35,-.85);
	%
	\draw (2.5,-.5) .. controls (2.5,0) and (5,0) ..  (5,.5) node[pos=.1,tikzdot]{} node[pos=.1, xshift=-.75ex, yshift=1ex]{\small $p$} -- (5,1.25);
	\draw (2.75,-.5) .. controls (2.75,0) and (2.5,0) .. (2.5,.5);
	\node at(3,-.35) {\tiny $\dots$};
	\draw (3.25,-.5) .. controls (3.25,0) and (3,0) .. (3,.5);
	\draw [pstdhl]  (3.5,-.5)  node[below]{\small $\mu_{i{+}1}$} .. controls (3.5,0) and (3.25,0) .. (3.25,.5);
	\node[pcolor] at(3.875,-.35) { $\dots$};
	\draw [pstdhl]  (4.25,-.5)  node[below]{\small $\mu_{r}$} .. controls (4.25,0) and (4,0) .. (4,.5);
	\draw (4.5,-.5) .. controls (4.5,0) and (4.25,0) .. (4.25,.5);
	\node at(4.75,-.35) {\tiny $\dots$};
	\draw (5,-.5) .. controls (5,0) and (4.75,0) .. (4.75,.5);
	%
	\filldraw [fill=white, draw=black] (-.375,.5) rectangle (4.875,1.25) node[midway] { $\muT_{b-1}$};
}
\\
&\oplus
\bigoplus_{i=1}^r
\ssbigoplus{0 \leq t < b_i \\ p \geq 0}
\tikzdiag[xscale=1.25]{
	\draw (0,-.5) -- (0,1);
	\node at(.25,-.35) {\tiny $\dots$};
	\draw (.5,-.5) -- (.5,1);
	%
	%
	\node[pcolor] at(1.125,-.35) { $\dots$};
	\draw (1.75,-.5) -- (1.75,1);
	\node at(2,-.35) {\tiny $\dots$};
	\draw (2.25,-.5) -- (2.25,1);
	\draw[decoration={brace,mirror,raise=-8pt},decorate]  (1.65,-.85) -- node {\small $t$} (2.35,-.85);
	%
	\draw (2.5,-.5) .. controls (2.5,-.25) ..  (-.5,0) .. controls (5,.25) ..  (5,.5) -- (5,1.25);
	\draw (2.75,-.5) .. controls (2.75,0) and (2.5,0) .. (2.5,.5);
	\node at(3,-.35) {\tiny $\dots$};
	\draw (3.25,-.5) .. controls (3.25,0) and (3,0) .. (3,.5);
	\draw [pstdhl]  (3.5,-.5)  node[below]{\small $\mu_{i{+}1}$} .. controls (3.5,0) and (3.25,0) .. (3.25,.5);
	\node[pcolor] at(3.875,-.35) { $\dots$};
	\draw [pstdhl]  (4.25,-.5)  node[below]{\small $\mu_{r}$} .. controls (4.25,0) and (4,0) .. (4,.5);
	\draw (4.5,-.5) .. controls (4.5,0) and (4.25,0) .. (4.25,.5);
	\node at(4.75,-.35) {\tiny $\dots$};
	\draw (5,-.5) .. controls (5,0) and (4.75,0) .. (4.75,.5);
	%
	\draw [pstdhl] (1.5,-.5)  node[below]{\small $\mu_i$} -- (1.5,1);
	\draw [pstdhl] (.75,-.5)  node[below]{\small $\mu_2$} -- (.75,1);
	\draw[fill=white, color=white] (-.35,0) circle (.15cm);
	\draw [pstdhl] (-.25,-.5) node[below]{\small $\mu_1$} -- (-.25,1)  node[nail,pos=.33]{} node[pos=.33, xshift=-1.25ex,yshift=.75ex, black]{\small $p$};
	\filldraw [fill=white, draw=black] (-.375,.5) rectangle (4.875,1.25) node[midway] { $\muT_{b-1}$};
}
\end{split}
\end{align}
where $\und \mu' = (\mu_1,\dots, \mu_{r-1})$, and the isomorphism is given by inclusion. 
\end{prop}

\begin{proof}
By \cref{thm:Tbasis}, we get a similar decomposition as in \cref{eq:Tdecomp}, but where we put the $p$ dots on the upper-right part of the black strand. Since we can slides dots up to adding terms with a lower number of crossings using \cref{eq:nhdotslide} and \cref{eq:dotredstrand}, it means we get the decomposition of the statement by a diagonal change of basis. 
\end{proof}

Let $1_{b,1} \in \muT_{b+1}$ be the idempotent given by
\[
1_{b,1} := \sum_{\rho \in  \mathcal{P}_b^r}  \ 
\tikzdiagh{0}{
	\draw[pstdhl] (2,0)  node[below]{\small $\mu_1$} --(2,1);
	\draw (2.5,0) -- (2.5,1);
	\node at(3,.5) {\tiny$\dots$};
	\draw (3.5,0) -- (3.5,1);
	\draw[decoration={brace,mirror,raise=-8pt},decorate]  (2.4,-.35) -- node {$b_1$} (3.6,-.35);
	\draw[pstdhl] (4,0)  node[below]{\small $\mu_2$} --(4,1);
	\node[pcolor] at  (5,.5) {\dots};
	\draw[pstdhl] (6,0)  node[below]{\small $\mu_{r}$} --(6,1);
	\draw (6.5,0) -- (6.5,1);
	\node at(7,.5) {\tiny$\dots$};
	\draw (7.5,0) -- (7.5,1);
	\draw[decoration={brace,mirror,raise=-8pt},decorate]  (6.4,-.35) -- node {$b_r$} (7.6,-.35);
	\draw (8,0) -- (8,1);
}
\]
We define
\begin{align*}
G_1(i,t,p) &:= q^{2(b_i - t - 1 + p) + \sum_{s > i} (\mu_s + 2b_s)  } \bigl(\muT_b1_{\rho_{\hat i}}, d_\mu \bigr), \\
G_2(i,t,p) &:= q^{|\und \mu| - 2b + \mu_i + 2(p-t) + \sum_{s < i} (\mu_s -2b_s) } \bigl( \widetilde \muT_b 1_{\rho_{\hat i}},- d_\mu \bigr), 
\end{align*}
where $\widetilde M$ is defined as $M$ but with twisted left-action: $x \cdot \widetilde m := (-1)^{\deg_h(x)} \widetilde{(x \cdot m)}$.
Note that $G_1(i,t,p)$ is isomorphic as $\bZ \times \bZ^2$-graded $\muT_b$-module 
to the subset of $\muT_{b+1} 1_\rho$ given by the diagrams pictured at the second line of \cref{eq:Tdecomp}, and $G_2(i,t,p) [1]$ to the ones at the third line. 

\begin{rem}
    We need to introduce some twist in the definition of $G_2(i,t,p)$ to get the correct signs because in our convention the homological shift twists the left-action, while the inclusion $G_2(i,t,p) \hookrightarrow \muT_{b} 1_{\rho}$ is given by adding diagrams below (i.e. multiplication on the right).  
\end{rem}

Moreover, $d_\mu(\theta_{k,\rho}(p))$ is either $0$ (if $\mu_1 \in \beta + \bZ$) or can be rewritten as a combination of diagrams with dots and crossings only involving the first $i$ colored strands and $k$ black strands. 
Therefore, we obtain an isomorphism of left $(\muT_{b-1}, d_\mu)$-modules:
\begin{equation}\label{eq:dgTdecomp}
(1_{b,1} \muT_{b+1}1_\rho, d_\mu) \cong \cone\left(
\bigoplus_{i=1}^r
\ssbigoplus{0 \leq t < b_i \\ p \geq 0}
G_2(i,t,p)
\xrightarrow{L_\mu}
\bigoplus_{i=1}^r
\ssbigoplus{0 \leq t < b_i \\ p \geq 0}
G_1(i,t,p)
\right), 
\end{equation}
for some morphism $L_\mu$ of left $(\muT_b, d_\mu)$-modules determined by $d_\mu$ and using \cref{prop:Tdecomp}. More precisely, for $x \in G_2(i,t,p)$ we set $L_\mu(x) := (-1)^{\deg_h(x)} x\cdot d_\mu(\theta_{(b_1+\cdots+b_{i-1}+t),\rho}(p))$. The sign is due to the fact we have twisted the left action in the definition of $G_2(i,t,p)$.

\begin{exe}
    Consider $r=2, \und \mu = (\mu_1, \mu_2), b=2, \rho = (2,0)$. \cref{prop:Tdecomp} gives:
    \[
        \tikzdiag[xscale=1.25]{
        	\draw [pstdhl] (-.25,0) node[below]{\small $\mu_1$} -- (-.25,1);
        	\draw (0,0) -- (0,1);
        	\draw (.25,0) -- (.25,1);
        	\draw [pstdhl] (.5,0) node[below]{\small $\mu_2$} -- (.5,1);
        	\filldraw [fill=white, draw=black] (-.375,.5) rectangle (.625,1.25) node[midway] { $\muT_2$};
        }
        \cong 
        \tikzdiag[xscale=1.25]{
        	\draw [pstdhl] (-.25,0) node[below]{\small $\mu_1$} -- (-.25,1);
        	\draw (0,0) -- (0,1);
        	\draw (.25,0) -- (.25,1);
        	\draw [pstdhl] (.5,0) node[below]{\small $\mu_2$} -- (.5,1.25);
        	\filldraw [fill=white, draw=black] (-.375,.5) rectangle (.375,1.25) node[midway] { $\dgT^{\und \mu'}_{2}$};
        }
        \oplus 
        \bigoplus_{p \geq 0}
        \tikzdiag[xscale=1.25]{
        	\draw [pstdhl] (-.25,0) node[below]{\small $\mu_1$} -- (-.25,1);
        	\draw (0,0) .. controls (0,.25) and (.5,.25) .. (.5,.5) node[pos=.2, tikzdot]{} node[pos=.2, xshift=-.85ex, yshift=.5ex]{\small $p$} -- (.5,1.25) ;
        	\draw (.25,0) .. controls (.25,.25) and (0,.25) ..  (0,.5) -- (0,1);
        	\draw [pstdhl] (.5,0) node[below]{\small $\mu_2$} .. controls (.5,.25) and (.25,.25) .. (.25,.5) -- (.25,1);
        	\filldraw [fill=white, draw=black] (-.375,.5) rectangle (.375,1.25) node[midway] { $\muT_1$};
        }
        \oplus 
        \bigoplus_{p \geq 0}
        \tikzdiag[xscale=1.25]{
        	\draw [pstdhl] (-.25,0) node[below]{\small $\mu_1$} -- (-.25,1);
        	\draw (0,0) -- (0,1);
        	\draw (.25,0) .. controls (.25,.25) and (.5,.25) .. (.5,.5)node[pos=.2, tikzdot]{} node[pos=.2, xshift=-.85ex, yshift=.5ex]{\small $p$} -- (.5,1.25);
        	\draw [pstdhl] (.5,0) node[below]{\small $\mu_2$} .. controls (.5,.25) and (.25,.25) .. (.25,.5) -- (.25,1);
        	\filldraw [fill=white, draw=black] (-.375,.5) rectangle (.375,1.25) node[midway] { $\muT_1$};
        }
        \oplus 
        \bigoplus_{p \geq 0}
        \tikzdiag[xscale=1.25]{
        	\draw (0,0) .. controls (0,.125) .. (-.25,.25) .. controls (.5,.375) .. (.5,.5) -- (.5,1.25);
        	\draw (.25,0) .. controls (.25,.25) and (0,.25) ..  (0,.5) -- (0,1);
        	\draw [pstdhl] (.5,0) node[below]{\small $\mu_2$} .. controls (.5,.25) and (.25,.25) .. (.25,.5) -- (.25,1);
        	\draw [pstdhl] (-.25,0) node[below]{\small $\mu_1$} -- (-.25,.5) node[midway, nail]{}  node[midway, xshift=-1.25ex,yshift=.15ex, black]{\small $p$}  -- (-.25,1);
        	\filldraw [fill=white, draw=black] (-.375,.5) rectangle (.375,1.25) node[midway] { $\muT_1$};
        }
        \oplus 
        \bigoplus_{p \geq 0}
        \tikzdiag[xscale=1.25]{
        	\draw (0,0) -- (0,1);
        	\draw (.25,0) .. controls (.25,.125) .. (-.25,.25) .. controls (.5,.375) .. (.5,.5) -- (.5,1.25);
        	\draw [pstdhl] (.5,0) node[below]{\small $\mu_2$} .. controls (.5,.25) and (.25,.25) .. (.25,.5) -- (.25,1);
        	\draw [pstdhl] (-.25,0) node[below]{\small $\mu_1$} --  (-.25,.5) node[midway, nail]{}  node[midway, xshift=-1.25ex,yshift=.15ex, black]{\small $p$} -- (-.25,1);
        	\filldraw [fill=white, draw=black] (-.375,.5) rectangle (.375,1.25) node[midway] { $\muT_1$};
        }
    \]
    where $\und \mu' = (\mu_1)$. 
    Then we have
    \begin{align*}
        G_1(1,0,p) &\cong 
        \tikzdiag[xscale=1.25]{
        	\draw [pstdhl] (-.25,0) node[below]{\small $\mu_1$} -- (-.25,1);
        	\draw (0,0) .. controls (0,.25) and (.5,.25) .. (.5,.5) node[pos=.2, tikzdot]{} node[pos=.2, xshift=-.85ex, yshift=.5ex]{\small $p$} -- (.5,1.25) ;
        	\draw (.25,0) .. controls (.25,.25) and (0,.25) ..  (0,.5) -- (0,1);
        	\draw [pstdhl] (.5,0) node[below]{\small $\mu_2$} .. controls (.5,.25) and (.25,.25) .. (.25,.5) -- (.25,1);
        	\filldraw [fill=white, draw=black] (-.375,.5) rectangle (.375,1.25) node[midway] { $\muT_1$};
        }
        &
        G_1(1,1,p) &\cong 
        \tikzdiag[xscale=1.25]{
        	\draw [pstdhl] (-.25,0) node[below]{\small $\mu_1$} -- (-.25,1);
        	\draw (0,0) -- (0,1);
        	\draw (.25,0) .. controls (.25,.25) and (.5,.25) .. (.5,.5)node[pos=.2, tikzdot]{} node[pos=.2, xshift=-.85ex, yshift=.5ex]{\small $p$} -- (.5,1.25);
        	\draw [pstdhl] (.5,0) node[below]{\small $\mu_2$} .. controls (.5,.25) and (.25,.25) .. (.25,.5) -- (.25,1);
        	\filldraw [fill=white, draw=black] (-.375,.5) rectangle (.375,1.25) node[midway] { $\muT_1$};
        }
        \\
        G_2(1,0,p){[1]} &\cong 
        \tikzdiag[xscale=1.25]{
        	\draw (0,0) .. controls (0,.125) .. (-.25,.25) .. controls (.5,.375) .. (.5,.5) -- (.5,1.25);
        	\draw (.25,0) .. controls (.25,.25) and (0,.25) ..  (0,.5) -- (0,1);
        	\draw [pstdhl] (.5,0) node[below]{\small $\mu_2$} .. controls (.5,.25) and (.25,.25) .. (.25,.5) -- (.25,1);
        	\draw [pstdhl] (-.25,0) node[below]{\small $\mu_1$} -- (-.25,.5) node[midway, nail]{}  node[midway, xshift=-1.25ex,yshift=.15ex, black]{\small $p$}  -- (-.25,1);
        	\filldraw [fill=white, draw=black] (-.375,.5) rectangle (.375,1.25) node[midway] { $\muT_1$};
        }
        &
        G_2(1,1,p){[1]} &\cong 
        \tikzdiag[xscale=1.25]{
        	\draw (0,0) -- (0,1);
        	\draw (.25,0) .. controls (.25,.125) .. (-.25,.25) .. controls (.5,.375) .. (.5,.5) -- (.5,1.25);
        	\draw [pstdhl] (.5,0) node[below]{\small $\mu_2$} .. controls (.5,.25) and (.25,.25) .. (.25,.5) -- (.25,1);
        	\draw [pstdhl] (-.25,0) node[below]{\small $\mu_1$} --  (-.25,.5) node[midway, nail]{}  node[midway, xshift=-1.25ex,yshift=.15ex, black]{\small $p$} -- (-.25,1);
        	\filldraw [fill=white, draw=black] (-.375,.5) rectangle (.375,1.25) node[midway] { $\muT_1$};
        }
    \end{align*}
    In order to compute $L_\mu$, we compute
    \begin{align*}
        d_{\mu}\left(
        \tikzdiag[xscale=2,yscale=2]{
        	\draw (0,0) .. controls (0,.125) .. (-.25,.25) .. controls (.5,.375) .. (.5,.5);
        	\draw (.25,0) .. controls (.25,.25) and (0,.25) ..  (0,.5);
        	\draw [pstdhl] (.5,0) node[below]{\small $\mu_2$} .. controls (.5,.25) and (.25,.25) .. (.25,.5);
        	\draw [pstdhl] (-.25,0) node[below]{\small $\mu_1$} -- (-.25,.5) node[midway, nail]{}  node[midway, xshift=-1.25ex,yshift=.15ex, black]{\small $p$};
        }
        \right)
        &= 
        \tikzdiag[xscale=2,yscale=2]{
        	\draw (0,0) .. controls (0,.25) and (.5,.25) .. (.5,.5) node[pos=.2, tikzdot]{} node[pos=.2, xshift=-.85ex, yshift=.5ex]{\small $p'$} ;
        	\draw (.25,0) .. controls (.25,.25) and (0,.25) ..  (0,.5);
        	\draw [pstdhl] (.5,0) node[below]{\small $\mu_2$} .. controls (.5,.25) and (.25,.25) .. (.25,.5);
        	\draw [pstdhl] (-.25,0) node[below]{\small $\mu_1$} -- (-.25,.5);
        }
        \quad \text{with $p' := p+\mu_1$,}
        \\
        d_{\mu}\left(
        \tikzdiag[xscale=2,yscale=2]{
        	\draw (0,0) -- (0,.5);
        	\draw (.25,0) .. controls (.25,.125) .. (-.25,.25) .. controls (.5,.375) .. (.5,.5);
        	\draw [pstdhl] (.5,0) node[below]{\small $\mu_2$} .. controls (.5,.25) and (.25,.25) .. (.25,.5);
        	\draw [pstdhl] (-.25,0) node[below]{\small $\mu_1$} --  (-.25,.5) node[midway, nail]{}  node[midway, xshift=-1.25ex,yshift=.15ex, black]{\small $p$};
        }
        \right)
        &= 
        \sssum{u+v\\=p'-1}
        \tikzdiag[xscale=2,yscale=2]{
        	\draw (0,0) .. controls (0,.25) and (.5,.25) .. (.5,.5) node[pos=.2, tikzdot]{} node[pos=.2, xshift=-1.5ex, yshift=0ex]{\small $v$} ;
        	\draw (.25,0) .. controls (.25,.25) and (0,.25) ..  (0,.5) node[pos=.7, tikzdot]{} node[pos=.7, xshift=-1.5ex, yshift=0ex]{\small $u$};
        	\draw [pstdhl] (.5,0) node[below]{\small $\mu_2$} .. controls (.5,.25) and (.25,.25) .. (.25,.5);
        	\draw [pstdhl] (-.25,0) node[below]{\small $\mu_1$} -- (-.25,.5);
        }
        \ - \sssum{u+v\\=p'-1} \quad \sssum{a+b\\=v-1} \ 
        \tikzdiag[xscale=2,yscale=2]{
        	\draw [pstdhl] (-.25,0) node[below]{\small $\mu_1$} -- (-.25,.5);
        	\draw (0,0) -- (0,.5) node[pos=.33,tikzdot]{} node[pos=.33,xshift=-1.25ex]{\small $a$} node[pos=.67,tikzdot]{} node[pos=.67,xshift=-1.25ex]{\small $u$} ;
        	\draw (.25,0) .. controls (.25,.25) and (.5,.25) .. (.5,.5)node[pos=.2, tikzdot]{} node[pos=.2, xshift=-1.25ex, yshift=.5ex]{\small $b$};
        	\draw [pstdhl] (.5,0) node[below]{\small $\mu_2$} .. controls (.5,.25) and (.25,.25) .. (.25,.5);
        }
    \end{align*}
    In particular, note that $G_2(1,0,p)$ has its image only in $G_1(1,0,p')$, while $G_2(1,1,p)$ has its image in both $G_1(1,0,p'') $ and $G_1(1,1,p'')$ for $p'' \leq p'-1$.
\end{exe}



\section{Basis theorem} \label{sec:basisthemtri}

The goal of this section is to prove \cref{thm:Tbasis}. Usually with KLR-like algebra, one proves such a statement by constructing a faithful action on a polynomial space. However, the degenerate nature of the relations in \cref{eq:vredR} make the construction of such an action a non-obvious problem. To get around this issue, we define a new parametrized algebra $\dgT_b^{\und \mu}(\param)$ where the degenerate relations are replaced by non-degenerate ones, and which gives back $\dgT_b^{\und \mu}$ when specializing the parameter $\param$ to zero. We show that $\dgT_b^{\und \mu}(\param)$ comes with a faithful polynomial action, and use it to prove \cref{thm:Tbasis} through rewriting methods.

\begin{defn}
Let $\dgT_b^{\und \mu}(\param)$ be the $\bZ \times \bZ^2$-graded diagrammatic $\Bbbk[\param]$-algebra defined exactly as $\dgT_b^{\und \mu}$ in \cref{defn:dgKLRW} except that the relations in \cref{eq:vredR} are replaced by
\begin{align}
	\tikzdiagl[yscale=1.5]{
		\draw (1,0) ..controls (1,.25) and (0,.25) .. (0,.5)..controls (0,.75) and (1,.75) .. (1,1)  ;
		\draw[vstdhl] (0,0)node[below]{\small $\mu_i$} ..controls (0,.25) and (1,.25) .. (1,.5) ..controls (1,.75) and (0,.75) .. (0,1)  ;
	} 
	\ &= \param 
	\tikzdiagl[yscale=1.5]{
		\draw (1,0) -- (1,1)  ;
		\draw[vstdhl] (0,0)node[below]{\small $\mu_i$}-- (0,1)  ;
	}
	&
	\tikzdiagl[yscale=1.5]{
		\draw (0,0) ..controls (0,.25) and (1,.25) .. (1,.5) ..controls (1,.75) and (0,.75) .. (0,1)  ;
		\draw[vstdhl] (1,0)node[below]{\small $\mu_i$} ..controls (1,.25) and (0,.25) .. (0,.5)..controls (0,.75) and (1,.75) .. (1,1)  ;
	} 
	\ &= \param \ 
	\tikzdiagl[yscale=1.5]{
		\draw (0,0) -- (0,1)  ;
		\draw[vstdhl] (1,0)node[below]{\small $\mu_i$}-- (1,1)  ;
	}
	\label{eq:vredR2param}
\end{align}
\begin{align}
	\tikzdiagl[scale=1.5]{
		\draw  (0,0) .. controls (0,0.25) and (1, 0.5) ..  (1,1);
		\draw  (1,0) .. controls (1,0.5) and (0, 0.75) ..  (0,1);
		\draw [vstdhl] (0.5,0)node[below]{\small $\mu_i$}  .. controls (0.5,0.25) and (0, 0.25) ..  (0,0.5)
			 	  .. controls (0,0.75) and (0.5, 0.75) ..  (0.5,1);
	} 
	\ &= \ 
	\tikzdiagl[scale=1.5,xscale=-1]{
		\draw  (0,0) .. controls (0,0.25) and (1, 0.5) ..  (1,1);
		\draw (1,0)  .. controls (1,0.5) and (0, 0.75) ..  (0,1);
		\draw [vstdhl]  (0.5,0)node[below]{\small $\mu_i$}  .. controls (0.5,0.25) and (0, 0.25) ..  (0,0.5)
			 	  .. controls (0,0.75) and (0.5, 0.75) ..  (0.5,1);
	} 
	\label{eq:vredR3param}
\end{align}
\end{defn}

Note that if we specialize $\param = 0$, then we obtain $\dgT_b^{\und \mu}(0) \cong \dgT_b^{\und \mu}$. 

Our goal is to equip $\muT_b$ with a rewriting system up to braid-like isotopy in the sense of \cref{sec:rewritingmethods}, and then specialize it to the case $\param = 0$ in order to prove \cref{thm:Tbasis}.

\subsection{Rewriting rules}\label{ssec:rewritingrulesparam}

Let $\muDiag_{b}(\param)$ be the set of diagrams of the same form as in the definition of $\muT_{b}(\param)$, up to braid-like planar isotopy (see \cref{ssec:diagalg}).

We define a weight function $w : \muDiag_{b}(\param) \rightarrow \bZ^3$ that takes a diagram to the element of $\bZ^3$ given by starting at $(0,0,0)$ and applying the following procedure:
\begin{itemize}
\item for each black or colored crossing, count the number $i$ of strands at its left and add $(i,0,0)$;
\item for each dot, follow the strand above and sum $k$ the amount of crossings and nails involving the strand, then add $(0,k,0)$;
\item for each nail, count the number $\ell$ of crossings in the region at the bottom left delimited by following the nailed strand from the nail to the bottom, then add $(0,0,\ell)$.
\end{itemize}
Clearly, this weight function is well-defined as it is stable under braid-like planar isotopy. Therefore this gives a preorder $\preceq$ on $\muDiag_{b}(\param)$ by saying  $D \preceq D'$ whenever $w(D) \leq w(D')$ for the lexicographic order on $\bZ^3$.

\begin{exe}
    Consider the following diagram:
    \[
        \tikzdiag{
            \draw (.5,0) .. controls (.5,.5) and (1.5,.5) .. (1.5,1) 
                         .. controls (1.5,1.5) and (.5,1.5) .. (.5,2);
            \draw (1.5,0) .. controls (1.5,.6) .. (0,1.2) node[pos=.7, tikzdot]{}
                          .. controls (1,1.6) .. (1,2);
            \draw[pstdhl] (1,0)node[below]{\small $\mu_2$}  .. controls (1,.4) and (.5,.4) .. (.5,.8)
                                .. controls (.5,1.4) and (1.5,1.4) .. (1.5,2);
            \draw[pstdhl] (0,0) node[below]{\small $\mu_1$} -- (0,2) node[nail, pos=.6]{};
        }
    \]
    We obtain that its weight is $(7,3,1)$.
\end{exe}

Following Appendix \ref{sec:rewritingmethods}, we will rewrite in the algebras $\dgT_b^{\und \mu}(\delta)$ modulo braid-like isotopies. 
Let $\mathbb{T}_b^{\und \mu}(\delta)$ be the  linear $2$-polygraph
having one $0$-cell, with generating $1$-cells given by 
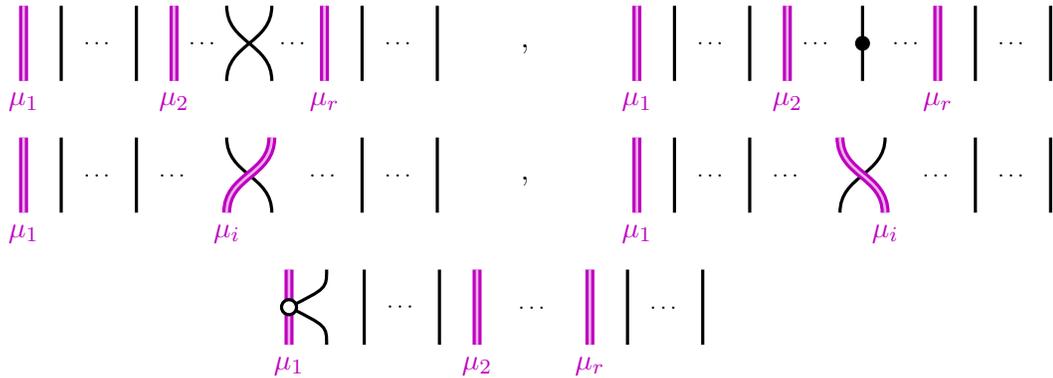
\begin{gather*}
\begin{align*}
     \tikzdiagh{0}{
    	\draw[pstdhl] (2,0)  node[below]{\small $\mu_1$} --(2,1);
    	%
    	\draw (2.5,0) -- (2.5,1);
    	\node at(3,.5) {\tiny$\dots$};
    	\draw (3.5,0) -- (3.5,1);
    	\draw[pstdhl] (4,0)  node[below]{\small $\mu_2$} --(4,1);
        \draw (4.7,0) ..controls (4.7,.5) and (5.3,.5) .. (5.3,1);
    	\draw (5.3,0) ..controls (5.3,.5) and (4.7,.5) .. (4.7,1);
        \node at (5.6,0.5) {\tiny $\dots$};
        \node at (4.4,0.5) {\tiny $\dots$};
    	%
    	%
    	\draw[pstdhl] (6,0)  node[below]{\small ${\mu_r}$} --(6,1);
    	\draw (6.5,0) -- (6.5,1);
    	\node at(7,.5) {\tiny$\dots$};
    	\draw (7.5,0) -- (7.5,1); (7.6,-.35);
    }
    && , &&
     \tikzdiagh{0}{
    	\draw[pstdhl] (2,0)  node[below]{\small $\mu_1$} --(2,1);
    	\draw (2.5,0) -- (2.5,1);
    	\node at(3,.5) {\tiny$\dots$};
    	\draw (3.5,0) -- (3.5,1);
    	\draw[pstdhl] (4,0)  node[below]{\small $\mu_2$} --(4,1);
    	\draw (5,0) -- (5,1) node[midway, tikzdot]{};
        \node at (5.6,0.5) {\tiny $\dots$};
        \node at (4.4,0.5) {\tiny $\dots$};
    	%
    	%
    	\draw[pstdhl] (6,0)  node[below]{\small ${\mu_r}$} --(6,1);
    	\draw (6.5,0) -- (6.5,1);
    	\node at(7,.5) {\tiny$\dots$};
    	\draw (7.5,0) -- (7.5,1); (7.6,-.35);
    }
    \\
     \tikzdiagh{0}{
    	\draw[pstdhl] (2,0)  node[below]{\small $\mu_1$} --(2,1);
    	%
    	\draw (2.5,0) -- (2.5,1);
    	\node at(3,.5) {\tiny$\dots$};
    	\draw (3.5,0) -- (3.5,1);
    	\draw (5.3,0) ..controls (5.3,.5) and (4.7,.5) .. (4.7,1);
        \draw[pstdhl] (4.7,0) node[below]{\small $\mu_i$} ..controls (4.7,.5) and (5.3,.5) .. (5.3,1);
        \node at (6,0.5) {\tiny $\dots$};
        \node at (4,0.5) {\tiny $\dots$};
    	%
    	%
    	\draw (6.5,0) -- (6.5,1);
    	\node at(7,.5) {\tiny$\dots$};
    	\draw (7.5,0) -- (7.5,1); (7.6,-.35);
    }
    &&,&&
    \tikzdiagh{0}{
    	\draw[pstdhl] (2,0)  node[below]{\small $\mu_1$} --(2,1);
    	%
    	\draw (2.5,0) -- (2.5,1);
    	\node at(3,.5) {\tiny$\dots$};
    	\draw (3.5,0) -- (3.5,1);
        \draw (4.7,0) ..controls (4.7,.5) and (5.3,.5) .. (5.3,1);
    	\draw[pstdhl] (5.3,0)  node[below]{\small $\mu_i$}  ..controls (5.3,.5) and (4.7,.5) .. (4.7,1);
        \node at (6,0.5) {\tiny $\dots$};
        \node at (4,0.5) {\tiny $\dots$};
    	%
    	%
    	\draw (6.5,0) -- (6.5,1);
    	\node at(7,.5) {\tiny$\dots$};
    	\draw (7.5,0) -- (7.5,1); (7.6,-.35);
    }
\end{align*}
\\
    \tikzdiag[xscale=2]{
        \draw (.25,-.5) .. controls (.25,-.25) .. (0,0) .. controls (.25,.25) .. (.25,.5);
        \draw[pstdhl] (0,-.5) node[below]{\small $\mu_1$}-- (0,.5)  node[midway, nail]{};
        \draw (.5,-.5) -- (.5,.5);
        \node at(0.75,0) {\tiny $\cdots$};
        \draw (1,-0.5) -- (1,0.5);
    	\draw[pstdhl] (1.25,-0.5)  node[below]{\small ${\mu_2}$} --(1.25,0.5);
    	\node at (1.625,0){\tiny $\dots$};
    	\draw[pstdhl] (2,-0.5)  node[below]{\small ${\mu_r}$} --(2,0.5);
    	\draw (2.25,-.5) -- (2.25,.5);
    	\node at (2.5,0){\tiny $\dots$};
    	\draw (2.75,-.5) -- (2.75,.5);
  }
\end{gather*}
and containing the following rewriting rules as generating $2$-cells:
\begin{align}
\label{eq:nhR2andR3rewrite}
\tikzdiag[yscale=1.5]{
	\draw (0,0) ..controls (0,.25) and (1,.25) .. (1,.5) ..controls (1,.75) and (0,.75) .. (0,1)  ;
	\draw (1,0) ..controls (1,.25) and (0,.25) .. (0,.5)..controls (0,.75) and (1,.75) .. (1,1)  ;
} 
\ &\Rightarrow\  
0,
&
\tikzdiag[scale=1.5,xscale=-1]{
	\draw  (0,0) .. controls (0,0.25) and (1, 0.5) ..  (1,1);
	\draw  (1,0) .. controls (1,0.5) and (0, 0.75) ..  (0,1);
	\draw  (0.5,0) .. controls (0.5,0.25) and (0, 0.25) ..  (0,0.5)
		 	  .. controls (0,0.75) and (0.5, 0.75) ..  (0.5,1);
} 
\ &\Rightarrow \ 
\tikzdiag[scale=1.5]{
	\draw  (0,0) .. controls (0,0.25) and (1, 0.5) ..  (1,1);
	\draw  (1,0) .. controls (1,0.5) and (0, 0.75) ..  (0,1);
	\draw  (0.5,0) .. controls (0.5,0.25) and (0, 0.25) ..  (0,0.5)
		 	  .. controls (0,0.75) and (0.5, 0.75) ..  (0.5,1);
} 
\\
\label{eq:nhdotsliderewrite}
\tikzdiag{
	\draw (0,0) ..controls (0,.5) and (1,.5) .. (1,1) node [near start,tikzdot]{};
	\draw (1,0) ..controls (1,.5) and (0,.5) .. (0,1);
}
\ &\Rightarrow \ 
\tikzdiag{
	\draw (0,0) ..controls (0,.5) and (1,.5) .. (1,1) node [near end,tikzdot]{};
	\draw (1,0) ..controls (1,.5) and (0,.5) .. (0,1);
}
\ + \ 
\tikzdiag{
	\draw (0,0) -- (0,1)  ;
	\draw (1,0)-- (1,1)  ;
}
&
\tikzdiag{
	\draw (0,0) ..controls (0,.5) and (1,.5) .. (1,1);
	\draw (1,0) ..controls (1,.5) and (0,.5) .. (0,1) node [near start,tikzdot]{};
}
\ &\Rightarrow \ 
\tikzdiag{
	\draw (0,0) ..controls (0,.5) and (1,.5) .. (1,1);
	\draw (1,0) ..controls (1,.5) and (0,.5) .. (0,1) node [near end,tikzdot]{};
}
\ - \ 
\tikzdiag{
	\draw (0,0) -- (0,1)  ;
	\draw (1,0)-- (1,1)  ;
} 
\end{align}
\begin{align}
\tikzdiagl[scale=1.5,xscale=-1]{
	\draw  (0,0) .. controls (0,0.25) and (1, 0.5) ..  (1,1);
	\draw  (0.5,0) .. controls (0.5,0.25) and (0, 0.25) ..  (0,0.5)
		 	  .. controls (0,0.75) and (0.5, 0.75) ..  (0.5,1);
	\draw [pstdhl] (1,0) node[below]{\small $\mu_i$}  .. controls (1,0.5) and (0, 0.75) ..  (0,1);
} 
\ &\Rightarrow \ 
\tikzdiagl[scale=1.5]{
	\draw  (0.5,0) .. controls (0.5,0.25) and (0, 0.25) ..  (0,0.5)
		 	  .. controls (0,0.75) and (0.5, 0.75) ..  (0.5,1);
	\draw (1,0)  .. controls (1,0.5) and (0, 0.75) ..  (0,1);
	\draw  [pstdhl] (0,0) node[below]{$\mu_i$}  .. controls (0,0.25) and (1, 0.5) ..  (1,1);
} 
&
\tikzdiagl[scale=1.5,xscale=-1]{
	\draw  (0.5,0) .. controls (0.5,0.25) and (0, 0.25) ..  (0,0.5)
		 	  .. controls (0,0.75) and (0.5, 0.75) ..  (0.5,1);
	\draw (1,0)  .. controls (1,0.5) and (0, 0.75) ..  (0,1);
	\draw  [pstdhl] (0,0) node[below]{\small $\mu_i$} .. controls (0,0.25) and (1, 0.5) ..  (1,1);
} 
\ &\Rightarrow \ 
\tikzdiagl[scale=1.5]{
	\draw  (0,0) .. controls (0,0.25) and (1, 0.5) ..  (1,1);
	\draw  (0.5,0) .. controls (0.5,0.25) and (0, 0.25) ..  (0,0.5)
		 	  .. controls (0,0.75) and (0.5, 0.75) ..  (0.5,1);
	\draw [pstdhl] (1,0) node[below]{\small $\mu_i$} .. controls (1,0.5) and (0, 0.75) ..  (0,1);
} 
\label{eq:crossingslideredrewrite}
\\
\tikzdiagl{
	\draw (1,0) ..controls (1,.5) and (0,.5) .. (0,1) node [near start,tikzdot]{};
	\draw[pstdhl] (0,0) node[below]{\small $\mu_i$}  ..controls (0,.5) and (1,.5) .. (1,1);
}
\ &\Rightarrow \ 
\tikzdiagl{
	\draw (1,0) ..controls (1,.5) and (0,.5) .. (0,1) node [near end,tikzdot]{};
	\draw[pstdhl] (0,0) node[below]{\small $\mu_i$}  ..controls (0,.5) and (1,.5) .. (1,1);
}
&
\tikzdiagl{
	\draw (0,0) ..controls (0,.5) and (1,.5) .. (1,1) node [near start,tikzdot]{};
	\draw[pstdhl] (1,0) node[below]{\small $\mu_i$}  ..controls (1,.5) and (0,.5) .. (0,1);
}
\ &\Rightarrow \ 
\tikzdiagl{
	\draw (0,0) ..controls (0,.5) and (1,.5) .. (1,1) node [near end,tikzdot]{};
	\draw[pstdhl] (1,0) node[below]{\small $\mu_i$}  ..controls (1,.5) and (0,.5) .. (0,1);
} 
\label{eq:dotredstrandrewrite}
\end{align}
\begin{align}
\tikzdiagh[yscale=1.5]{0}{
	\draw (1,0) ..controls (1,.25) and (0,.25) .. (0,.5)..controls (0,.75) and (1,.75) .. (1,1)  ;
	\draw[stdhl] (0,0) node[below]{\small $\mu_i$} ..controls (0,.25) and (1,.25) .. (1,.5) ..controls (1,.75) and (0,.75) .. (0,1)  ;
} 
\ &\Rightarrow \ 
\tikzdiagh[yscale=1.5]{0}{
	\draw[stdhl] (0,0) node[below]{\small $\mu_i$} -- (0,1)  ;
	\draw (1,0) -- (1,1)  node[midway,tikzdot]{}  node[midway,xshift=1.75ex,yshift=.75ex]{\small $\mu_i$} ;
} 
&
\tikzdiagh[yscale=1.5]{0}{
	\draw (0,0) ..controls (0,.25) and (1,.25) .. (1,.5) ..controls (1,.75) and (0,.75) .. (0,1)  ;
	\draw[stdhl] (1,0) node[below]{\small $\mu_i$} ..controls (1,.25) and (0,.25) .. (0,.5)..controls (0,.75) and (1,.75) .. (1,1)  ;
} 
\ &\Rightarrow \ 
\tikzdiagh[yscale=1.5]{0}{
	\draw (0,0) -- (0,1)  node[midway,tikzdot]{}   node[midway,xshift=1.75ex,yshift=.75ex]{\small $\mu_i$} ;
	\draw[stdhl] (1,0) node[below]{\small $\mu_i$} -- (1,1)  ;
} 
&&
\text{if $\mu _i \in \bN$,}
\label{eq:redR2rewrite}
\\
\tikzdiagh[yscale=1.5]{0}{
	\draw (1,0) ..controls (1,.25) and (0,.25) .. (0,.5)..controls (0,.75) and (1,.75) .. (1,1)  ;
	\draw[vstdhl] (0,0) node[below]{\small $\mu_i$} ..controls (0,.25) and (1,.25) .. (1,.5) ..controls (1,.75) and (0,.75) .. (0,1)  ;
} 
\ &\Rightarrow \param \ 
\tikzdiagh[yscale=1.5]{0}{
	\draw[vstdhl] (0,0) node[below]{\small $\mu_i$} -- (0,1)  ;
	\draw (1,0) -- (1,1) ;
} 
&
\tikzdiagh[yscale=1.5]{0}{
	\draw (0,0) ..controls (0,.25) and (1,.25) .. (1,.5) ..controls (1,.75) and (0,.75) .. (0,1)  ;
	\draw[vstdhl] (1,0) node[below]{\small $\mu_i$} ..controls (1,.25) and (0,.25) .. (0,.5)..controls (0,.75) and (1,.75) .. (1,1)  ;
} 
\ &\Rightarrow \param \ 
\tikzdiagh[yscale=1.5]{0}{
	\draw (0,0) -- (0,1)  ;
	\draw[vstdhl] (1,0) node[below]{\small $\mu_i$} -- (1,1)  ;
} 
&&
\text{if $\mu_i \in \beta + \bZ$, }
\label{eq:blueR2rewrite}
\end{align}
\begin{align}
\tikzdiagh[scale=1.5,xscale=-1]{0}{
	\draw  (0,0) .. controls (0,0.25) and (1, 0.5) ..  (1,1);
	\draw (1,0)  .. controls (1,0.5) and (0, 0.75) ..  (0,1);
	\draw [pstdhl]  (0.5,0) node[below]{\small $\mu_i$} .. controls (0.5,0.25) and (0, 0.25) ..  (0,0.5)
		 	  .. controls (0,0.75) and (0.5, 0.75) ..  (0.5,1);
} 
\ &\Rightarrow \ 
\tikzdiagh[scale=1.5]{0}{
	\draw  (0,0) .. controls (0,0.25) and (1, 0.5) ..  (1,1);
	\draw  (1,0) .. controls (1,0.5) and (0, 0.75) ..  (0,1);
	\draw [pstdhl] (0.5,0)node[below]{\small $\mu_i$}  .. controls (0.5,0.25) and (0, 0.25) ..  (0,0.5)
		 	  .. controls (0,0.75) and (0.5, 0.75) ..  (0.5,1);
} 
\ - \sssum{u+v=\\\mu_i-1} \ 
\tikzdiagh[scale=1.5]{0}{
	\draw  (0,0) -- (0,1) node[midway,tikzdot]{} node[midway,xshift=-1.5ex,yshift=.75ex]{\small $u$};
	\draw  (1,0) --  (1,1) node[midway,tikzdot]{} node[midway,xshift=1.5ex,yshift=.75ex]{\small $v$};
	\draw [pstdhl] (0.5,0)node[below]{\small $\mu_i$}  --  (0.5,1);
} \label{eq:redR3rewrite}
\end{align}
where we recall the sum is $0$ by convention whenever $\mu_i \in \beta + \bZ$,
\begin{align}  \label{eq:nailsrelrewrite} 
	\tikzdiagl[xscale=2]{
		 \draw (.5,-.5) .. controls (.5,-.25) .. (0,0)  node[midway, tikzdot]{} .. controls (.5,.25) .. (.5,.5);
	          \draw[pstdhl] (0,-.5) node[below]{\small $\mu_1$} -- (0,.5) node [midway,nail]{};
  	}
\  &\Rightarrow \ 
	\tikzdiagl[xscale=2]{
		 \draw (.5,-.5) .. controls (.5,-.25) .. (0,0) .. controls (.5,.25) .. (.5,.5)  node[midway, tikzdot]{};
	          \draw[pstdhl] (0,-.5) node[below]{\small $\mu_1$} -- (0,.5) node [midway,nail]{};
  	}
  &
	\tikzdiagl[xscale=-1,yscale=.75]{
		\draw (-1.5,-.75) -- (-1.5,0) .. controls (-1.5,.5) .. (0,.75) .. controls (-1.5,1) .. (-1.5,1.5);
		\draw (-.75,-.75) .. controls (-.75,-.25) .. (0,0) .. controls (-.75,.25) .. (-.75,.75) -- (-.75,1.5);
		\draw[pstdhl] (0,-.75) node[below]{\small $\mu_1$} -- (0,0) node[pos=1,nail]{} -- (0,.75) node[pos=1,nail]{}  -- (0,1.5);
	}
\ &\Rightarrow - \ 
	\tikzdiagl[xscale=-1,yscale=-.75]{
		\draw (-1.5,-.75) -- (-1.5,0) .. controls (-1.5,.5) .. (0,.75) .. controls (-1.5,1) .. (-1.5,1.5);
		\draw (-.75,-.75) .. controls (-.75,-.25) .. (0,0) .. controls (-.75,.25) .. (-.75,.75) -- (-.75,1.5);
		\draw[pstdhl] (0,-.75) -- (0,0) node[pos=1,nail]{} -- (0,.75) node[pos=1,nail]{}  -- (0,1.5) node[below]{$\mu_1$} ;
	}
&
	\tikzdiagl[xscale=-1,yscale=.75]{
		\begin{scope}
			\clip(0,-.75) rectangle (-.75,1.5);
			\draw (.75,-.75) .. controls (.75,0) and (-.5,0) .. (-.5,.375) 
				.. controls (-.5,.75) and (.75,.75) .. (.75,1.5);
		\end{scope}
		\draw (-.75,-.75) .. controls (-.75,-.25) .. (0,0)
			(0,.75) .. controls (-.75,1.125) .. (-.75,1.5);
		\draw[pstdhl] (0,-.75)  node[below]{\small $\mu_1$} -- (0,0) node[pos=1,nail]{} -- (0,.75) node[pos=1,nail]{}  -- (0,1.5);
	}
\ &\Rightarrow 0,
\end{align}
and finally the collections of local rewriting rules
{
\allowdisplaybreaks
\begin{align}
\label{eq:crossingnailrewrite1}
\tikzdiagh[xscale=1.25]{0}{
	\draw (1, -1) .. controls(1,-.75) and (.75,-.75) .. (.75,-.5) -- (.75,1);
	\draw (.75,-1)  .. controls(.75,-.75) and (1,-.75) .. (1,-.5)   -- (1,1);
	%
	\draw (0,-1) -- (0,1);
	\draw[dashed, pstdhl] (0,-1) -- (0,1);
	\node at(.25,-.85) {\tiny $\dots$};
	\node at(.25,.85) {\tiny $\dots$};
	\draw (.5,-1) -- (.5,1);
	\draw[dashed, pstdhl] (.5,-1) -- (.5,1);
	\draw (1.25,-1)  .. controls (1.25,-.25) and (-.5,-.25) ..  (-.25,0) .. controls (-.5,.25) and (1.25,.25) ..  (1.25,1);
	\draw [pstdhl] (-.25,-1)  node[below]{\small $\mu_1$}  -- (-.25,1)  node[midway,nail]{};	
	\tikzbrace{0}{.5}{-1}{$\ell$};
}
&\Rightarrow
\tikzdiagh[xscale=1.25]{0}{
	\draw (.75, -1) -- (.75,.5) .. controls (.75,.75) and (1,.75) .. (1,1);
	\draw (1,-1)  -- (1,.5) .. controls (1,.75) and (.75,.75) .. (.75,1);
	%
	\draw (0,-1) -- (0,1);
	\draw[dashed, pstdhl] (0,-1) -- (0,1);
	\node at(.25,-.85) {\tiny $\dots$};
	\node at(.25,.85) {\tiny $\dots$};
	\draw (.5,-1) -- (.5,1);
	\draw[dashed, pstdhl] (.5,-1) -- (.5,1);
	\draw (1.25,-1)  .. controls (1.25,-.25) and (-.5,-.25) ..  (-.25,0) .. controls (-.5,.25) and (1.25,.25) ..  (1.25,1);
	\draw [pstdhl] (-.25,-1) node[below]{\small $\mu_1$} -- (-.25,1)  node[midway,nail]{};	
	\tikzbrace{0}{.5}{-1}{$\ell$};
}
\\
\label{eq:crossingnailrewrite2}
\tikzdiagh[xscale=1.25]{0}{
	\draw (1.25,-1)  .. controls (1.25,-.25) and (-.5,-.25) ..  (-.25,0) .. controls (-.5,.25) and (1.25,.25) ..  (1.25,1);
	%
	\draw (.75,-1)  .. controls(.75,-.75) and (1,-.75) .. (1,-.5)   -- (1,1);
	%
	\draw (0,-1) -- (0,1);
	\draw[dashed, pstdhl] (0,-1) -- (0,1);
	\node at(.25,-.85) {\tiny $\dots$};
	\node at(.25,.85) {\tiny $\dots$};
	\draw (.5,-1) -- (.5,1);
	\draw[dashed, pstdhl] (.5,-1) -- (.5,1);
	\draw [pstdhl] (-.25,-1)  node[below]{\small $\mu_1$}  -- (-.25,1)  node[midway,nail]{};	
	\tikzbrace{0}{.5}{-1}{$\ell$};
	\draw[pstdhl] (1, -1)  node[below]{\small $\mu_i$} .. controls(1,-.75) and (.75,-.75) .. (.75,-.5) -- (.75,1);
}
&\Rightarrow
\tikzdiagh[xscale=1.25]{0}{
	\draw (1.25,-1)  .. controls (1.25,-.25) and (-.5,-.25) ..  (-.25,0) .. controls (-.5,.25) and (1.25,.25) ..  (1.25,1);
	\draw (.75, -1) -- (.75,.5) .. controls (.75,.75) and (1,.75) .. (1,1);
	%
	\draw (0,-1) -- (0,1);
	\draw[dashed, pstdhl] (0,-1) -- (0,1);
	\node at(.25,-.85) {\tiny $\dots$};
	\node at(.25,.85) {\tiny $\dots$};
	\draw (.5,-1) -- (.5,1);
	\draw[dashed, pstdhl] (.5,-1) -- (.5,1);
	\draw [pstdhl] (-.25,-1) node[below]{\small $\mu_1$} -- (-.25,1)  node[midway,nail]{};	
	\tikzbrace{0}{.5}{-1}{$\ell$};
	\draw[pstdhl]  (1,-1) node[below]{\small $\mu_i$}  -- (1,.5) .. controls (1,.75) and (.75,.75) .. (.75,1);
}
\ + \sssum{u+v=\\ \mu_i - 1} \ 
\tikzdiagh[xscale=1.25]{0}{
	\draw (.75,-1)  .. controls (.75,-.25) and (-.5,-.25) ..  (-.25,0) 
		node[pos=.1,tikzdot]{} node[pos=.1, xshift=.85ex,yshift=.85ex]{\small $u$}
		.. controls (-.5,.25) and (1.25,.25) ..  (1.25,1);
	\draw (1.25, -1) -- (1.25,.5)
		node[pos=.15,tikzdot]{} node[pos=.15, xshift=.85ex,yshift=.85ex]{\small $v$}
		 .. controls (1.25,.75) and (1,.75) .. (1,1);
	%
	\draw (0,-1) -- (0,1);
	\draw[dashed, pstdhl] (0,-1) -- (0,1);
	\node at(.25,-.85) {\tiny $\dots$};
	\node at(.25,.85) {\tiny $\dots$};
	\draw (.5,-1) -- (.5,1);
	\draw[dashed, pstdhl] (.5,-1) -- (.5,1);
	\draw [pstdhl] (-.25,-1) node[below]{\small $\mu_1$} -- (-.25,1)  node[midway,nail]{};	
	\tikzbrace{0}{.5}{-1}{$\ell$};
	\draw[pstdhl]  (1,-1)  node[below]{\small $\mu_i$}  -- (1,.5) .. controls (1,.75) and (.75,.75) .. (.75,1);
}
\\
\label{eq:crossingnailrewrite3}
\tikzdiagh[xscale=1.25]{0}{
	\draw (1.25,-1)  .. controls (1.25,-.25) and (-.5,-.25) ..  (-.25,0)
		 .. controls (-.5,.25) and (1.25,.25) ..  (1.25,1);
	\draw (1, -1) .. controls(1,-.75) and (.75,-.75) .. (.75,-.5) -- (.75,1);
	%
	%
	\draw (0,-1) -- (0,1);
	\draw[dashed, pstdhl] (0,-1) -- (0,1);
	\node at(.25,-.85) {\tiny $\dots$};
	\node at(.25,.85) {\tiny $\dots$};
	\draw (.5,-1) -- (.5,1);
	\draw[dashed, pstdhl] (.5,-1) -- (.5,1);
	\draw [pstdhl] (-.25,-1)  node[below]{\small $\mu_1$}  -- (-.25,1)  node[midway,nail]{};	
	\tikzbrace{0}{.5}{-1}{$\ell$};
	\draw[pstdhl] (.75,-1)   node[below]{\small $\mu_i$} .. controls(.75,-.75) and (1,-.75) .. (1,-.5)   -- (1,1);
}
&\Rightarrow
\tikzdiagh[xscale=1.25]{0}{
	\draw (1.25,-1)  .. controls (1.25,-.25) and (-.5,-.25) ..  (-.25,0) .. controls (-.5,.25) and (1.25,.25) ..  (1.25,1);
	%
	%
	\draw (1,-1)  -- (1,.5) .. controls (1,.75) and (.75,.75) .. (.75,1);
	\draw (0,-1) -- (0,1);
	\draw[dashed, pstdhl] (0,-1) -- (0,1);
	\node at(.25,-.85) {\tiny $\dots$};
	\node at(.25,.85) {\tiny $\dots$};
	\draw (.5,-1) -- (.5,1);
	\draw[dashed, pstdhl] (.5,-1) -- (.5,1);
	\draw [pstdhl] (-.25,-1) node[below]{\small $\mu_1$} -- (-.25,1)  node[midway,nail]{};	
	\tikzbrace{0}{.5}{-1}{$\ell$};
	\draw[pstdhl]  (.75, -1)   node[below]{\small $\mu_i$} -- (.75,.5) .. controls (.75,.75) and (1,.75) .. (1,1);
}
\ - \sssum{u+v=\\ \mu_i - 1} \ 
\tikzdiagh[xscale=1.25]{0}{
	\draw (1.25,-1)  .. controls (1.25,-.25) and (-.5,-.25) ..  (-.25,0) .. controls (-.5,.25) and (.75,.25) ..  (.75,1)
		node[pos=.9,tikzdot]{} node[pos=.9, xshift=.85ex,yshift=-.85ex]{\small $u$};
	%
	\draw (1,-1)  .. controls (1,-.75) and (1.25,-.75) .. (1.25,-.5) -- (1.25,1)
		node[pos=.85,tikzdot]{} node[pos=.85, xshift=.85ex,yshift=-.85ex]{\small $v$};
	\draw (0,-1) -- (0,1);
	\draw[dashed, pstdhl] (0,-1) -- (0,1);
	\node at(.25,-.85) {\tiny $\dots$};
	\node at(.25,.85) {\tiny $\dots$};
	\draw (.5,-1) -- (.5,1);
	\draw[dashed, pstdhl] (.5,-1) -- (.5,1);
	\draw [pstdhl] (-.25,-1) node[below]{\small $\mu_1$} -- (-.25,1)  node[midway,nail]{};	
	\tikzbrace{0}{.5}{-1}{$\ell$};
	\draw[ pstdhl]  (.75, -1)   node[below]{\small $\mu_i$} .. controls (.75,-.75) and (1,-.75) .. (1,-.5) -- (1,1);
}
\end{align}
}
for all $\ell \geq 0$ and where a dashed strand mean it can either be black or colored. 
Note that going from left to right strictly decreases the weight. Also note that all these relations holds in $\muT_{b}(\param)$, and together they present $\muT_{b}(\param)$. 

In the sequel, we rewrite with the rewriting rules above modulo braid-like planar isotopies. As a consequence, we consider rewriting with respect to the linear $2$-polygraph modulo ${}_{\text{Iso}(\dgT_b^{\und \mu}(\delta))} \mathbb{T}_b^{\und \mu}(\delta)_{\text{Iso}(\dgT_b^{\und \mu}(\delta))}$ consisting in applying the rewriting rules of $\mathbb{T}_b^{\und \mu}(\delta)$ on diagrams of $\muDiag_{b}(\param)$ that are defined up to braid-like planar isotopies. In order to shorten the notations, we will denote by $\widetilde{\mathbb{T}}_b^{\und \mu}(\delta)$ the linear $2$-polygraph modulo ${}_{\text{Iso}(\dgT_b^{\und \mu}(\delta))} \mathbb{T}_b^{\und \mu}(\delta)_{\text{Iso}(\dgT_b^{\und \mu}(\delta))}$.

\begin{rem}
    Note that we added the rewriting rules \cref{eq:crossingnailrewrite1}, \cref{eq:crossingnailrewrite2} and \cref{eq:crossingnailrewrite3},  which do not come from orienting defining relations of the algebra, in order to reach confluence modulo of the linear $2$-polygraph modulo $\widetilde{\mathbb{T}}_b^{\und \mu}(\delta)$. Indeed, there are indexed critical branchings of the form
    
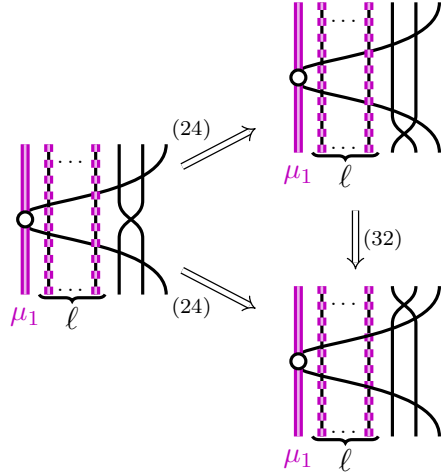
\begin{equation*}
        \begin{tikzcd}[row sep=-6ex]
        &
        \tikzdiagh[xscale=1.25]{0}{
        	\draw (1.25,-1)  .. controls (1.25,-.25) and (-.5,-.25) ..  (-.25,0) .. controls (-.5,.25) and (1.25,.25) ..  (1.25,1);
        	\draw (1, -1) .. controls(1,-.75) and (.75,-.75) .. (.75,-.5) -- (.75,1);
        	\draw (.75,-1)  .. controls(.75,-.75) and (1,-.75) .. (1,-.5)   -- (1,1);
        	%
        	\draw (0,-1) -- (0,1);
        	\draw[dashed, pstdhl] (0,-1) -- (0,1);
        	\node at(.25,-.85) {\tiny $\dots$};
        	\node at(.25,.85) {\tiny $\dots$};
        	\draw (.5,-1) -- (.5,1);
        	\draw[dashed, pstdhl] (.5,-1) -- (.5,1);
        	\draw [pstdhl] (-.25,-1) node[below]{\small $\mu_1$} -- (-.25,1)  node[midway,nail]{};	
        	\tikzbrace{0}{.5}{-1}{$\ell$};
        }
        \ar[Rightarrow]{dd}{\eqref{eq:crossingnailrewrite1}}
        &
        \\
        \tikzdiagh[xscale=1.25]{0}{
        	\draw (1.25,-1)  .. controls (1.25,-.25) and (-.5,-.25) ..  (-.25,0) .. controls (-.5,.25) and (1.25,.25) ..  (1.25,1);
        	\draw (1,-1) -- (1, -.25) .. controls(1,0) and (.75,0) .. (.75,.25) -- (.75,1);
        	\draw (.75,-1) -- (.75,-.25)  .. controls(.75,0) and (1,0) .. (1,.25)   -- (1,1);
        	%
        	\draw (0,-1) -- (0,1);
        	\draw[dashed, pstdhl] (0,-1) -- (0,1);
        	\node at(.25,-.85) {\tiny $\dots$};
        	\node at(.25,.85) {\tiny $\dots$};
        	\draw (.5,-1) -- (.5,1);
        	\draw[dashed, pstdhl] (.5,-1) -- (.5,1);
        	\draw [pstdhl] (-.25,-1) node[below]{\small $\mu_1$} -- (-.25,1)  node[midway,nail]{};	
        	\tikzbrace{0}{.5}{-1}{$\ell$};
        }
        \ar[Rightarrow]{ur}{\eqref{eq:nhR2andR3rewrite}}
        \ar[Rightarrow,swap]{dr}{\eqref{eq:nhR2andR3rewrite}}
        &
        &
        \\
        &
        \tikzdiagh[xscale=1.25]{0}{
        	\draw (.75, -1) -- (.75,.5) .. controls (.75,.75) and (1,.75) .. (1,1);
        	\draw (1,-1)  -- (1,.5) .. controls (1,.75) and (.75,.75) .. (.75,1);
        	%
        	\draw (0,-1) -- (0,1);
        	\draw[dashed, pstdhl] (0,-1) -- (0,1);
        	\node at(.25,-.85) {\tiny $\dots$};
        	\node at(.25,.85) {\tiny $\dots$};
        	\draw (.5,-1) -- (.5,1);
        	\draw[dashed, pstdhl] (.5,-1) -- (.5,1);
        	\draw (1.25,-1)  .. controls (1.25,-.25) and (-.5,-.25) ..  (-.25,0) .. controls (-.5,.25) and (1.25,.25) ..  (1.25,1);
        	\draw [pstdhl] (-.25,-1) node[below]{\small $\mu_1$} -- (-.25,1)  node[midway,nail]{};	
        	\tikzbrace{0}{.5}{-1}{$\ell$};
        }
        &
        \end{tikzcd}
    \end{equation*}
   that is not confluent if we don't add the relation \cref{eq:crossingnailrewrite1}. Other shapes of indexed critical branchings also impose to add the relations \cref{eq:crossingnailrewrite2} and \cref{eq:crossingnailrewrite3}. Moreover, without these relations we sould still have a terminating rewriting system, but some normal forms would not be basis elements. 
\end{rem}

The rewriting rules above terminate on diagrams up to braid-like isotopy, \emph{i.e.} we have the following proposition:
\begin{prop}\label{prop:rewriteterminates}
The linear $2$-polygraph modulo 
$\widetilde{\mathbb{T}}_b^{\und \mu}(\delta)$
is terminating.
\end{prop}

\begin{proof}
Note that for any $D \in \muDiag_{b}(\param)$, then $w(D) \geq (0,0,0)$. Moreover, we have the following:
\begin{itemize}
    \item the $2$-cells above strictly decrease the weight, that is $w(s_2(\alpha)) > w(h)$ for any $h$ in $\text{Supp}(t_2(\alpha))$.
    \item the weight function is stable under multiplication, that is for any monomials $D$,$D'$, \linebreak $D_1$,$D_2$ of $\dgT_b^{\und \mu}$, $w(D) > w(D')$ implies that $w(D_1 D D_2) > w(D_1 D' D_2)$ since we add to the triples $w(D)$ and $w(D')$ the same elements in each entry.
\end{itemize}
Therefore, the preorder $\preceq$ defines a termination order for the linear $2$-polygraph $R$. As it is stable under the application of braid-like isotopy $2$-cells, it extends to a termination order for the linear $2$-polygraph modulo ${}_E R_E$.
\end{proof}

\subsection{Polynomial action}

Our goal is to construct a faithful action of $\muT_b(\param)$ on a polynomial ring. The construction is similar to \cite[\S 3.3.1]{LNV}. 
Let $R := \Bbbk[\param]$, and let $\Pol_b^{\und \mu} := \bigoplus_{\rho \in   \cP_{b}^{r}} \Pol_b \varepsilon_\rho$ be the free module over the ring $\Pol_b := R[x_1, \dots, x_b] \otimes  \bV^{\bullet}(\omega_1,\dots,\omega_b)$ generated by $\varepsilon_\rho$ for each $\rho \in   \cP_{b}^{r}$. 

There is an $R$-linear action of the symmetric group $S_b$ on $\Pol_b$, similar to the one already used in \cite[\S2.2]{NV2}. For each simple transposition $\sigma_i$ we put
\begin{align*}
  \sigma_i(x_j) &:= x_{\sigma_i(j)},\\
  \sigma_i(\omega_j) &:= \omega_j + \delta_{i,j}(x_i-x_{i+1})\omega_{i+1},
\end{align*}
where $\delta_{i,j} := 1$ if $i=j$ and $\delta_{i,j} := 0$ if $i \neq j$.

For $\kappa,\rho \in  \cP_{b}^{r}$, we let any element of $1_{\kappa}\muT_b(\param)1_{\rho}$ act as zero on $\Pol_b\varepsilon_{\rho'}$ for $\rho'\neq \rho$ and sends elements in $\Pol_b\varepsilon_{\rho}$ to elements in $\Pol_b\varepsilon_{\kappa}$. We now describe the action of the local generators of $\muT_b(\param)$ on a polynomial $f \varepsilon_\rho\in \Pol_b \varepsilon_\rho$. First, similarly as in~\cite[Lemma 4.12]{webster}, we put
\begin{align*}
  \tikzdiagh{-1.5ex}{
  \node at(0,.5) {\small$\dots$};
  \draw (0.5,0) -- (0.5,1) node [midway,tikzdot]{};
  \node at(1,.5) {\small$\dots$};
  }\cdot f \varepsilon_{\rho} &:= x_if \varepsilon_{\kappa},
             &
  \tikzdiagh{-1.5ex}{
  \node at(0,.5) {\small$\dots$};
  \draw (0.5,0) ..controls (0.5,.5) and (1.5,.5) .. (1.5,1);
  \draw (1.5,0) ..controls (1.5,.5) and (0.5,.5) .. (0.5,1);
  \node at(2,.5) {\small$\dots$};
  }\cdot f  \varepsilon_{\rho} &:= \frac{f-\sigma_i(f)}{x_i-x_{i+1}}  \varepsilon_{\kappa},\\
  \tikzdiagh{0}{
  \node at(0,.5) {\small$\dots$};
  \draw (0.5,0) ..controls (0.5,.5) and (1.5,.5) .. (1.5,1);
  \draw[stdhl] (1.5,0) node[below]{\small $N$} ..controls (1.5,.5) and (0.5,.5) .. (0.5,1);
  \node at(2,.5) {\small$\dots$};
  }\cdot f  \varepsilon_{\rho} &:= f  \varepsilon_{\kappa},
             &
  \tikzdiagh{0}{
  \node at(0,.5) {\small$\dots$};
  \draw (1.5,0) ..controls (1.5,.5) and (0.5,.5) .. (0.5,1);
  \draw[stdhl] (0.5,0) node[below]{\small $N$} ..controls (0.5,.5) and (1.5,.5) .. (1.5,1);
  \node at(2,.5) {\small$\dots$};
  }\cdot f  \varepsilon_{\rho} &:= x_i^{N}f  \varepsilon_{\kappa},
\end{align*}
where we only have drawn the $i$-th or the $i$-th and $(i+1)$-th black strands, counting from left to right. We also put
\begin{align*}
  \tikzdiagh{0}{
  \node at(0,.5) {\small$\dots$};
  \draw (0.5,0) ..controls (0.5,.5) and (1.5,.5) .. (1.5,1);
  \draw[vstdhl] (1.5,0) node[below]{\small $\beta+N$} ..controls (1.5,.5) and (0.5,.5) .. (0.5,1);
  \node at(2,.5) {\small$\dots$};
  }\cdot f  \varepsilon_{\rho} &:= f  \varepsilon_{\kappa},
             &
  \tikzdiagh{0}{
  \node at(0,.5) {\small$\dots$};
  \draw (1.5,0) ..controls (1.5,.5) and (0.5,.5) .. (0.5,1);
  \draw[vstdhl] (0.5,0) node[below]{\small $\beta+N$} ..controls (0.5,.5) and (1.5,.5) .. (1.5,1);
  \node at(2,.5) {\small$\dots$};
  }\cdot f  \varepsilon_{\rho} &:= \param f  \varepsilon_{\kappa},
\end{align*}
 Finally, as in \cite[\S2.2]{NV2} we put
\[
  \\
  \tikzdiagh{0}{
  \draw (.5,0) .. controls (.5,.25) .. (0,0.5) .. controls (.5,.75)  .. (.5,1);
  \draw[pstdhl] (0,0) node[below]{\small $\mu_1$} -- (0,1) node [midway,nail]{};
  \node at(1,.5) {\small$\dots$};
}\cdot f  \varepsilon_{\rho} := \omega_1 f  \varepsilon_{\kappa}.
\]

\begin{prop}
The rules above define an action of $\muT_b(\param)$ on $\Pol_b^{\und \mu}$.
\end{prop}

\begin{proof}
One easily checks that the defining relations of  $\muT_b(\param)$  are satisfied, similarly as in \cite[Proposition 3.7]{NV3}. We leave the details to the reader. 
\end{proof}

Note that the elements in ${}_\kappa B_\rho$ can all be seen as elements in $1_\kappa \muT_b(\param) 1_\rho$.

\begin{thm}\label{thm:basisparam}
As a $\bZ \times \bZ^2$-graded $\Bbbk$-module, $1_\kappa \muT_{b}(\param)  1_\rho$ is free with basis given by ${}_\kappa B_\rho$ .
\end{thm}

\begin{proof}
First, we observe that the elements in ${}_\kappa B_\rho$ are the normal forms for the rewriting rules of \cref{ssec:rewritingrulesparam}. Thus, \cref{prop:rewriteterminates} shows that ${}_\kappa B_\rho$ generates $1_\kappa \muT_{b}(\param)  1_\rho$. 
We obtain linear independence by observing that elements in ${}_\kappa B_\rho$ act by linearly independent elements on $\Pol_b^{\und \mu}$ as in \cite[Theorem 3.11]{LNV}. 
\end{proof}

\begin{cor}
The action of $\muT_b(\param)$ on $\Pol_b^{\und \mu}$ described above is faithful.
\end{cor}

\begin{rem}
Note that the action of $\muT_b(0)$ on $\Pol_b^{\und \mu}$ after specializing $\param  = 0$ is no longer faithful since 
\[
  \tikzdiagh{0}{
  \node at(0,.5) {\small$\dots$};
  \draw (1.5,0) ..controls (1.5,.5) and (0.5,.5) .. (0.5,1);
  \draw[vstdhl] (0.5,0) node[below]{\small $\beta+N$} ..controls (0.5,.5) and (1.5,.5) .. (1.5,1);
  \node at(2,.5) {\small$\dots$};
  }
  \]
  acts as zero. 
\end{rem}

\subsection{Basis for \texorpdfstring{$\param = 0$}{param=zero}}

The rewriting rules on $\muDiag_{b}(\param)$ defined above are confluent modulo braid-like isotopies:

\begin{prop}
The linear $2$-polygraph modulo $\widetilde{\mathbb{T}}_b^{\und \mu}(\delta)$ is confluent modulo $\text{Iso}(\dgT_b^{\und \mu}(\delta))$.
\end{prop}

\begin{proof}
By \cref{thm:basisparam}, we know that the normal forms are linearly independent. Therefore the rewriting rules are confluent. 
\end{proof}

\begin{rem}
One can also verify by hand that all the regular critical branchings modulo of $\widetilde{\mathbb{T}}_b^{\und \mu}(\delta)$ are confluent modulo braid-like isotopies. However, indexed critical branchings given by overlappings of the rewriting rules \eqref{eq:crossingnailrewrite1}, \eqref{eq:crossingnailrewrite2} and \eqref{eq:crossingnailrewrite3} produce infinitely many cases to check, which can be unwieldy in practice. We show that they are confluent in the case of tensor products of Verma modules (i.e. $\mu_i \in \beta + \bZ$ for all$i$) but the general case is more complicated. Since we find this to be an interesting problem in terms of confluence, we describe this in details in \cref{sec:basisthemrewriting}. 
\end{rem}

\begin{cor}\label{cor:specializedconfluent}
After specializing $\param = 0$, the linear $2$-polygraph modulo $\widetilde{\mathbb{T}}_b^{\und \mu}(0)$ is confluent modulo braid-like isotopies.
\end{cor}

\begin{proof}
If an equation holds for generic $\param$, then it holds for $\param = 0$. 
\end{proof}

\begin{cor}\label{cor:basis}
As a $\bZ \times \bZ^2$-graded $\Bbbk$-module, $1_\kappa \muT_{b} 1_\rho$ is free with basis given by ${}_\kappa B_\rho$ .
\end{cor}



\section{Categorification of \texorpdfstring{$L(\und \mu)$}{L(undmu)}}\label{sec:catTensProd}

In this section, we explain how derived categories of $(\muT,d_\mu)$-modules categorify $L(\und \mu)$. Since the categorical action is similar to \cite{NV2} and \cite{NV3}, we will rely heavily on the references for the details.

Recall we write $\cD_{dg}(\muT_{b}, d_\mu)$ for the (dg-enhanced) derived dg-category of $\bZ^2$-graded $(\muT_{b}, d_\mu)$-dg-modules, see \cref{sec:dgdercat}. 
We will also write $\otimes$ for $\otimes_\Bbbk$ and $\otimes_b$ for $\otimes_{(\muT_b, d_\mu)}$. Similarly $\RHOM_b$ denotes $\RHOM_{(\muT_b, d_\mu)}$

\subsection{Categorical action}\label{sec:catAction}

There is a (non-unital) map 
$
(\muT_{b}, d_\mu) \rightarrow(\muT_{b+1}, d_\mu)
$
 given by adding a vertical black strand on the right:
\begin{equation}\label{eq:addblackstrand}
\tikzdiag[xscale=1.25]{
	\draw [pstdhl] (-.25,0) -- (-.25,1);
	\draw (0,0) -- (0,1);
	\node at(.25,.125) {\tiny $\dots$};
	\node at(.25,.875) {\tiny $\dots$};
	\draw (.5,0) -- (.5,1);
	\draw [pstdhl] (.75,0)   -- (.75,1);
	\node[pcolor] at(1.125,.125) { $\dots$};
	\node[pcolor] at(1.125,.875) { $\dots$};
	\draw [pstdhl] (1.5,0) -- (1.5,1);
	\draw (1.75,0) -- (1.75,1);
	\node at(2,.125) {\tiny $\dots$};
	\node at(2,.875) {\tiny $\dots$};
	\draw (2.25,0) -- (2.25,1);
	\draw [pstdhl] (2.5,0)   -- (2.5,1);
	\draw (2.75,0) -- (2.75,1);
	\node at(3,.125) {\tiny $\dots$};
	\node at(3,.875) {\tiny $\dots$};
	\draw (3.25,0) -- (3.25,1);
	\filldraw [fill=white, draw=black] (-.375,.25) rectangle (3.375,.75) node[midway] { $D$};
}
\ \mapsto \ 
\tikzdiag[xscale=1.25]{
	\draw [pstdhl] (-.25,0) -- (-.25,1);
	\draw (0,0) -- (0,1);
	\node at(.25,.125) {\tiny $\dots$};
	\node at(.25,.875) {\tiny $\dots$};
	\draw (.5,0) -- (.5,1);
	\draw [pstdhl] (.75,0)  -- (.75,1);
	\node[pcolor] at(1.125,.125) { $\dots$};
	\node[pcolor] at(1.125,.875) { $\dots$};
	\draw [pstdhl] (1.5,0)  -- (1.5,1);
	\draw (1.75,0) -- (1.75,1);
	\node at(2,.125) {\tiny $\dots$};
	\node at(2,.875) {\tiny $\dots$};
	\draw (2.25,0) -- (2.25,1);
	\draw [pstdhl] (2.5,0)  -- (2.5,1);
	\draw (2.75,0) -- (2.75,1);
	\node at(3,.125) {\tiny $\dots$};
	\node at(3,.875) {\tiny $\dots$};
	\draw (3.25,0) -- (3.25,1);
	\filldraw [fill=white, draw=black] (-.375,.25) rectangle (3.375,.75) node[midway] { $D$};
	\draw (3.5,0) -- (3.5,1);
}
\end{equation}
It sends $1 \in \muT_{b}$ to $1_{b,1} \in \muT_{b+1}$. 
Moreover, it gives rise to derived induction and restriction dg-functors
\begin{align*}
\Ind_b^{b+1} &: \cD_{dg}(\muT_{b},d_\mu) \rightarrow \cD_{dg}(\muT_{b+1},d_\mu), \\ 
 &\Ind_b^{b+1}(-) \cong (\muT_{b+1} 1_{b,1} ,d_\mu)\Lotimes_b -,\\
\Res_b^{b+1} &: \cD_{dg}(\muT_{b+1},d_\mu) \rightarrow \cD_{dg}(\muT_{b},d_\mu), \\
 &\Res_b^{b+1}(-) \cong  \RHOM_{b+1}((\muT_{b+1}1_{b,1},d_\mu),-), 
\end{align*}
which are adjoint. By \cref{prop:Tdecomp}, we know that $(\muT_{b+1} 1_{\rho,1} ,d_\mu)$ is a cofibrant right dg-module over $(\muT_{b},d_\mu)$, so that we can replace derived tensor product (resp. derived hom) by usual tensor products:
\begin{align*}
\Ind_b^{b+1}(-) &\cong (\muT_{b+1}1_{b,1},d_\mu)  \otimes_b -,
&
\Res_b^{b+1}(-) &\cong  (1_{b,1}\muT_{b+1},d_\mu) \otimes_{b+1} -.
\end{align*}
 Then we define
\begin{align*}
\F_b &:= \Ind_b^{b+1},
&
\E_b &:=  q^{2b+1-|\und \mu|} \Res_b^{b+1},
\end{align*}
and $\id_b$ is the identity dg-functor on $\cD_{dg}(\muT_{b},d_\mu)$.

Consider the map 
\[
\psi : q^{-2} (\muT_{b}  1_{b-1,1}\otimes_{b-1} 1_{b-1,1} \muT_b) \rightarrow1_{b,1} \muT_{b+1} 1_{b,1},
\]
given by 
\[
x \otimes_{b-1} y \mapsto x \tau_b y,
\]
where $\tau_b$ is a crossing between the $b$-th and $(b+1)$-th black strands. Diagrammatically, one can picture it as
\[
\tikzdiag[xscale=.75,yscale=.75]{
	\draw (0,-1.25) -- (0,1.25);
	\draw (.5,-1.25) -- (.5,1.25);
	\draw (1.5,-1.25) -- (1.5,1.25);
	\draw (2,-1.25) -- (2,-.5) .. controls (2,-.25) .. (2.25,-.25);
	\draw (2,1.25) -- (2,.5) .. controls (2,.25) .. (2.25,.25);
	\node at(1,1.2) {\small $\dots$};
	\filldraw [fill=white, draw=black] (-.25,-1) rectangle (2.25,-.5);
	\node at(1,0) {\small $\dots$}; 
	\filldraw [fill=white, draw=black] (-.25,.5) rectangle (2.25,1);
	\node at(1,-1.2) {\small $\dots$};
}
\ \mapsto \  
\tikzdiag[xscale=.75,yscale=.75]{
	\draw (0,-1.25) -- (0,1.25);
	\draw (.5,-1.25) -- (.5,1.25);
	\draw (1.5,-1.25) -- (1.5,1.25);
	\draw (2,-1.25) -- (2,-.5) .. controls (2,0) and (2.5,0) .. (2.5,.5) -- (2.5,1.25);
	\draw (2,1.25) -- (2,.5) .. controls (2,0) and (2.5,0) .. (2.5,-.5) -- (2.5,-1.25);
	\node at(1,1.2) {\small $\dots$};
	\filldraw [fill=white, draw=black] (-.25,-1) rectangle (2.25,-.5);
	\node at(1,0) {\small $\dots$};
	\filldraw [fill=white, draw=black] (-.25,.5) rectangle (2.25,1);
	\node at(1,-1.2) {\small $\dots$};
}
\]
where the bent black strands depict the induction/restriction functors. 
Consider also the map
\[
\phi : 1_{b,1} \muT_{b+1} 1_{b,1} \rightarrow  \bigoplus_{p\geq 0} q^{2p}  (\muT_{b}) \oplus q^{2p+2|\underline{\mu}|-4b} (\muT_b)[1],
\]
given by projection onto the following summands 
\begin{align*}
\bigoplus_{p \geq 0} \ 
\tikzdiag[xscale=1.25]{
	\draw (0,-.5) -- (0,1);
	\node at(.25,-.35) {\tiny $\dots$};
	\draw (.5,-.5) -- (.5,1);
	\draw[decoration={brace,mirror,raise=-8pt},decorate]  (-.1,-.85) -- node {\small $b_1$} (.6,-.85);
	\draw [pstdhl] (.75,-.5)  node[below,yshift={-1ex}]{\small $\mu_2$} -- (.75,1);
	\node[pcolor] at(1.125,-.35) { $\dots$};
	\draw [pstdhl] (1.5,-.5)  node[below,yshift={-1ex}]{\small $\mu_{r}$} -- (1.5,1);
	\draw (1.75,-.5) -- (1.75,1);
	\node at(2,-.35) {\tiny $\dots$};
	\draw (2.25,-.5) -- (2.25,1);
	\draw[decoration={brace,mirror,raise=-8pt},decorate]  (1.65,-.85) -- node {\small $b_{r}$} (2.35,-.85);
	\draw (2.5, -.5) -- (2.5, 1.25) node[pos=.25,tikzdot]{} node[pos=.25, xshift=1.5ex, yshift=1ex]{\small $p$};
	%
	\draw [pstdhl] (-.25,-.5) node[below,yshift={-1ex}]{\small $\mu_1$} -- (-.25,1);
	\filldraw [fill=white, draw=black] (-.375,.5) rectangle (2.375,1.25) node[midway] { $\muT_{b}$};
}
\oplus
 \ 
\tikzdiag[xscale=1.25]{
	\draw (0,-.5) -- (0,1);
	\node at(.25,-.35) {\tiny $\dots$};
	\draw (.5,-.5) -- (.5,1);
	\draw[decoration={brace,mirror,raise=-8pt},decorate]  (-.1,-.85) -- node {\small $b_1$} (.6,-.85);
	\draw (2.5, -.5) .. controls (2.5,-.25) .. (-.5,0) .. controls (2.5,.25) ..  (2.5, .75) ;
	\draw [pstdhl] (.75,-.5)  node[below,yshift={-1ex}]{\small $\mu_2$} -- (.75,1);
	\node[pcolor] at(1.125,-.35) { $\dots$};
	\draw [pstdhl] (1.5,-.5)  node[below,yshift={-1ex}]{\small $\mu_{r}$} -- (1.5,1);
	\draw (1.75,-.5) -- (1.75,1);
	\node at(2,-.35) {\tiny $\dots$};
	\draw (2.25,-.5) -- (2.25,1);
	\draw[decoration={brace,mirror,raise=-8pt},decorate]  (1.65,-.85) -- node {\small $b_{r}$} (2.35,-.85);
	%
	\draw (2.5, .75) -- (2.5, 1.25);
	%
	\draw[fill=white, color=white] (-.35,0) circle (.15cm);
	\draw [pstdhl] (-.25,-.5) node[below,yshift={-1ex}]{\small $\mu_1$} -- (-.25,1) node[pos=.33, nail]{} node[black, pos=.33,xshift=-1.25ex, yshift=.75ex]{\small $p$};
	\filldraw [fill=white, draw=black] (-.375,.5) rectangle (2.375,1.25) node[midway] { $\muT_{b}$};
}
\end{align*}
of \cref{prop:Tdecomp} (i.e. when $i=r$ and $t=b_r$). 

\begin{lem}\label{lem:SES}
There is a short exact sequence
\begin{align*}
0 \rightarrow q^{-2} (\muT_{b}  1_{b-1,1}\otimes_{b-1} 1_{b-1,1} \muT_b) &\xrightarrow{\ \psi\ }1_{b,1}  \muT_{b+1} 1_{b,1}
\\
 &\xrightarrow{\ \phi\ } \bigoplus_{p\geq 0} q^{2p}  (\muT_{b}) \oplus  q^{2p+2|\underline{\mu}|-4b} (\muT_b)[1] \rightarrow 0,
\end{align*}
of $\bZ \times \bZ^2$-graded $\muT_b$-$\muT_b$-bimodules.
\end{lem}

\begin{proof}
The map $\psi$ is clearly a morphism of graded bimodules, while the map $\phi$ is clearly a morphism of graded left modules. 
By similar computations as in~\cite[Lemma 5.4]{NV3}, one can show that $\phi$ defines a map of bimodules, and we omit the details. 
Exactness follows from an immediate dimensional argument using \cref{prop:Tdecomp}. 
\end{proof}

We observe that $\psi$ lift immediately to a map of dg-bimodules
\[
\hat \psi : q^{-2} (\muT_{b} 1_{b-1,1}, d_\mu)  \otimes_{b-1} (1_{b-1,1} \muT_b,d_\mu) \xrightarrow{\psi} (1_{b,1} \muT_{b+1} 1_{b,1}, d_\mu).
\]
Define
\[
h_\mu : \bigoplus_{p \geq 0}  q^{2p+2|\underline{\mu}|-4b} (\muT_b) \rightarrow \bigoplus_{p\geq 0} q^{2p}  (\muT_{b}),
\]
as the morphism of left $(\muT_b,d_\mu)$-modules
\[
h_\mu(x) := \phi \circ L_\mu \circ \phi^{-1}(x),
\]
where we recall $L_\mu$ is defined in \cref{eq:dgTdecomp}. 

\begin{lem}\label{lem:hmudg}
The map $h_\mu$ defined above is a morphism of graded dg-bimodules.
\end{lem}

\begin{proof}
There is a similar decomposition as in \cref{prop:Tdecomp} of $1_{b,1}\muT_{b+1}1_{b,1}$, but flipped vertically, yielding a decomposition as right $\muT_b$-module. We denote the decomposition summand as $\tilde G_1(i,t,p)$ and $\tilde G_2(i,t,p)$. 
Then, we get an isomorphism of right $(\muT_b, d_\mu)$-modules
\[
(1_{b,1} \muT_{b+1}1_{b,1}, d_\mu) \cong 
\cone\left(
\bigoplus_{i=1}^r
\ssbigoplus{0 \leq t < b_i \\ p \geq 0}
\tilde G_2(i,t,p)
\xrightarrow{R_\mu}
\bigoplus_{i=1}^r
\ssbigoplus{0 \leq t < b_i \\ p \geq 0}
\tilde G_1(i,t,p)
\right), 
\]
for a certain map of right modules $R_\mu$ defined similarly as $L_\mu$. 
Since $\phi$ is a map of bimodules, it appears that the projections on $G_k(r,b_r,p)$ and on $\tilde G_k(r,b_r,p)$ coincides for all $k \in \{1,2\}$ and $p \geq 0$. 
Finally, we observe that $L_\mu(G_2(r,b_r,p))|_{\oplus_{\ell \geq 0}G_1(r,b_r,\ell)} \cong R_\mu(\tilde G_2(r,b_r,p))|_{\oplus_{\ell \geq 0}\tilde G_1(r,b_r,\ell)}$ under the above mentioned identification, because all defining relations of $\muT_b$ are symmetric with respect to a vertical flip, and so is $d_\mu(\theta_{\rho,b})$. 
Therefore we have $h_\mu = \phi \circ R_\mu \circ \phi^{-1}$ as well, and we conclude that $h_\mu$ is a morphism of right modules. 
\end{proof}

Consequently, we get an induced morphism 
\[
\hat \phi : 
(1_{b,1} \muT_{b+1} 1_{b,1}, d_\mu) \xrightarrow{\phi} 
\cone\left(  
\bigoplus_{p\geq 0}  q^{2p+2|\underline{\mu}|-4b} (\muT_b) \xrightarrow{h_\mu} \bigoplus_{p\geq 0}  q^{2p}  (\muT_{b}) 
\right),
\]
of dg-bimodules.

\begin{exe}
Take $\und \mu = (N,\beta)$. 
We compute
\[
h_\mu\left(
\tikzdiag{
	\draw (.5,0) -- (.5,1);
	%
	\draw[vstdhl] (1,0) node[below]{\small $\beta$} -- (1,1);
	\draw[stdhl] (0,0) node[below]{\small $N$} -- (0,1);
}
\in q^{2p +2|\und \mu|-4b}(\muT_b)
\right)
=
\phi \circ L_\mu\left(
\tikzdiag{
	\draw (.5,0) -- (.5,1);
	\draw (1.5, 0) .. controls (1.5,.25) .. (0,.5) .. controls (1.5,.75) .. (1.5,1); 
	\draw[vstdhl] (1,0) node[below]{\small $\beta$} -- (1,1);
	\draw[stdhl] (0,0) node[below]{\small $N$} -- (0,1) node[midway, nail]{} node[midway, xshift=-1.25ex, yshift=.75ex,black]{\small $p$};
}
\right)
= \phi \left(
\sssum{k+\ell=\\p+N-1}\ 
\tikzdiag{
	\draw (.5,0) .. controls (.5,.5) and (1.5,.5) .. (1.5,1) node[pos=.1, tikzdot]{} node[pos=.1, xshift=-1.25ex, yshift=-.1ex]{\small$k$};
	\draw (1.5, 0) .. controls (1.5,.5) and (.5,.5) .. (.5,1) node[pos=.9, tikzdot]{} node[pos=.9, xshift=-1.25ex, yshift=.1ex]{\small$\ell$};
	\draw[vstdhl] (1,0) node[below]{\small $\beta$} .. controls (1,.25) and (.5,.25) .. (.5,.5) .. controls (.5,.75) and (1,.75) .. (1,1);
	\draw[stdhl] (0,0) node[below]{\small $N$} -- (0,1);
}
\right)
= 0.
\]
For $\und \mu = (N,1)$, we compute
\begin{align*}
h_\mu&\left(
\tikzdiag{
	\draw (.5,0) -- (.5,1);
	%
	\draw[stdhl] (1,0) node[below]{\small $1$} -- (1,1);
	\draw[stdhl] (0,0) node[below]{\small $N$} -- (0,1);
}
\in q^{2p +2|\und \mu|-4b}(\muT_b)
\right)
\\
&= \phi \left(
\sssum{k+\ell=\\p+N-1}\ 
\tikzdiagl{
	\draw (.5,0) .. controls (.5,.5) and (1.5,.5) .. (1.5,1) node[pos=.1, tikzdot]{} node[pos=.1, xshift=-1.25ex, yshift=-.1ex]{\small$k$};
	\draw (1.5, 0) .. controls (1.5,.5) and (.5,.5) .. (.5,1) node[pos=.9, tikzdot]{} node[pos=.9, xshift=-1.25ex, yshift=.1ex]{\small$\ell$};
	\draw[stdhl] (1,0) node[below]{\small $1$} .. controls (1,.25) and (.5,.25) .. (.5,.5) .. controls (.5,.75) and (1,.75) .. (1,1);
	\draw[stdhl] (0,0) node[below]{\small $N$} -- (0,1) ;
}
\ - \ 
\tikzdiagl{
	\draw (.5,0) -- (.5,1) 
	    node[pos=.33, tikzdot]{} node[pos=.33, xshift=1.25ex, yshift=1ex]{\small$k$}
	    node[pos=.66, tikzdot]{} node[pos=.66, xshift=1.25ex, yshift=1ex]{\small$\ell$};
	\draw (1.5, 0)  -- (1.5,1);
	\draw[stdhl] (1,0) node[below]{\small $1$} -- (1,1);
	\draw[stdhl] (0,0) node[below]{\small $N$} -- (0,1);
}
\ - \sssum{s+t=\\ \ell-1} \ 
\tikzdiagl{
	\draw (.5,0) -- (.5,1) 
	    node[pos=.33, tikzdot]{} node[pos=.33, xshift=1.25ex, yshift=1ex]{\small$k$}
	    node[pos=.66, tikzdot]{} node[pos=.66, xshift=1.25ex, yshift=1ex]{\small$s$};
	\draw (1.5, 0)  -- (1.5,1) node[pos=.5, tikzdot]{} node[pos=.5, xshift=3ex, yshift=1ex]{\small$t+1$};
	\draw[stdhl] (1,0) node[below]{\small $1$} -- (1,1);
	\draw[stdhl] (0,0) node[below]{\small $N$} -- (0,1);
}
\right)
\\
&= 
- \sum_{\ell = 0}^{p+N-1} (\ell+1) \left( 
\tikzdiag{
	\draw (.5,0) -- (.5,1) node[midway, tikzdot]{} node[midway, xshift=1ex, yshift=.75ex]{\small $\ell$};
	%
	\draw[stdhl] (1,0) node[below]{\small $1$} -- (1,1);
	\draw[stdhl] (0,0) node[below]{\small $N$} -- (0,1);
}
\in q^{2(p+N-1-\ell)}(\muT_b)
\right)
\end{align*}
\end{exe}

\begin{prop}
If $|\und \mu| \notin \bN$, then $h_\mu = 0$ and we obtain an isomorphism
\[
\cone\left(  
\bigoplus_{p\geq 0}  q^{2p+2|\underline{\mu}|-4b} (\muT_b) \xrightarrow{h_\mu} \bigoplus_{p\geq 0}  q^{2p}  (\muT_{b}) 
\right)
\cong 
\bigoplus_{p\geq 0}  q^{2p+2|\underline{\mu}|-4b} (\muT_b) [1] \oplus q^{2p}  (\muT_{b}).
\]
of dg-bimodules. 
If $|\und \mu| \in \bN$, then we have a quasi-isomorphism
\[
\cone\left(  
\bigoplus_{p\geq 0}  q^{2p+2|\underline{\mu}|-4b} (\muT_b) \xrightarrow{h_\mu} \bigoplus_{p\geq 0}  q^{2p}  (\muT_{b}) 
\right)
\xrightarrow{\simeq}
\begin{cases} 
\bigoplus_{p = 0}^{|\und \mu|-2b-1} q^{2p}\muT_b, &\text{if $|\und \mu|-2b \geq 0$}, \\
\bigoplus_{p = 0}^{2b-|\und \mu|-1} q^{2p}\muT_b[1], &\text{if $|\und \mu|-2b \leq 0$}, 
\end{cases}
\]
of dg-bimodules.
\end{prop}

\begin{proof}
If $|\und \mu| \in \bN$, then it is \cite[Proposition 4.3]{LNV}. 
Suppose $|\und \mu| \notin \bN$.
Since $h_\mu(1_\rho)$ is symmetric w.r.t. vertical flip of diagrams, and commutes with dots, we can conclude it is given by a linear combination of diagram without black crossing, and thus also without colored crossing. Therefore $h_\mu(1_\rho)$ is a polynomial of dots on $1_\rho$.
By \cref{prop:Tdecomp}, adding crossings at the top or bottom of the subset of polynomials of dots in $\muT_{b+1}$ is an injective operation. 
Let $\rho_0 := (0, \dots, 0, b)$, and $w = \sigma_{i_k} \cdots \sigma_{i_1} \in S^{r+n}$ be a reduced expression such that $w(\rho) = \rho_0$: 
\[
1_{\rho_0}\tau_w 1_{\rho} := \ 
\tikzdiagl[xscale=1.25, yscale=1.5]{
	\draw (.25, 0) .. controls (.25,.5) and (1.25,.5) .. (1.25,1);
	\node at(.575,.15) {\tiny $\dots$};
	\node at(1.5,.85) {\tiny $\dots$};
	\draw (.75,0) .. controls (.75,.5) and (1.75,.5) .. (1.75,1);
	\draw (2.25, 0) .. controls (2.25, .5) and (2.5,.5) .. (2.5,1);
	\node at(2.5,.15) {\tiny $\dots$};
	\node at(2.75,.85) {\tiny $\dots$};
	\draw (2.75,0) .. controls (2.75,.5) and (3,.5) .. (3,1);
	\draw (3.25, 0) -- (3.25,1);
	\node at(3.5,.15) {\tiny $\dots$};
	\node at(3.5,.85) {\tiny $\dots$};
	\draw (3.75,0) -- (3.75,1);
	\node[pcolor] at(1.5,.15) { $\dots$};
	\node at(2.125,.85) { $\dots$};
	\draw[pstdhl] (3,0) node[below]{\small $\mu_r$} .. controls (3,.5) and (1,.5) .. (1,1);
	\draw[pstdhl] (2,0) node[below]{\small $\mu_{r-1}$} .. controls (2,.5) and (.75,.5) .. (.75,1);
	\node[pcolor] at(.575,.85) { \tiny $\dots$};
	\draw[pstdhl] (1,0) node[below]{\small $\mu_2$} .. controls (1,.5) and (.25,.5) .. (.25,1);
	\draw[pstdhl] (0,0) node[below]{\small $\mu_1$} -- (0,1);
}
\]
Then we have $\tau_w h_\mu(1_\rho) = h_\mu(1_{\rho_0}) \tau_w $.  We obviously have $h_\mu(1_{\rho_0}) = 0$ by \cref{eq:vredR}, thus $h_\mu(1_\rho) = 0$, and we conclude that $h_\mu = 0$. 
\end{proof}

Let $\muT := \bigoplus_{b \geq 0} \muT_{b}$, and $\F := \bigoplus_{b \geq 0} \F_b$, $\E := \bigoplus_{b \geq 0} \E_b$.  
Let $\K : \cD_{dg}(\muT, d_\mu) \rightarrow \cD_{dg}(\muT, d_\mu) $ denotes the auto-equivalence functor given by the grading shift
\[
\K  \muT_{b}  :=  q^{|\und \mu|-2b}  (\muT_{b}). 
\]
Let $[\K]_q$ denotes 
\[
[\K]_q := \cone\left( \bigoplus_{p \geq 0} q^{2p+1} \K \xrightarrow{ \ h_\mu \ } \bigoplus_{p \geq 0} q^{2p+1} \K^{-1} \right),
\]
which we think of as a categorification of $(K^{-1}-K)/(q^{-1}-q)$. 

\begin{thm}\label{thm:sl2comqi}
There is a quasi-isomorphism
\[
\cone(\F\E \xrightarrow{\hat \psi} \E\F) \xrightarrow{\simeq} [\K]_q,
\]
of dg-functors. 
\end{thm}

\begin{proof}
The statement follows from \cref{lem:SES} and \cref{lem:hmudg}.
\end{proof}

We also obtain the following immediately from the induction/restriction adjunction:

\begin{prop}\label{prop:adjEF}
The dg-functor $\F$ is left-adjoint to $q\E\K$. 
\end{prop}

\subsubsection{Induction along colored strands}\label{sec:redind}

Take $\underline{\mu} = (\mu_1, \dots, \mu_r)$ and $\underline \mu ' = (\underline \mu, \mu_{r+1})$. 
Consider the (non-unital) map of dg-algebras $(\muT_b,d_\mu) \rightarrow (\dgT^{\und \mu'}_b, d_{\mu'})$ that consists in adding a vertical colored strand labeled $\mu_{r+1}$ at the right of a diagram:
\begin{equation*}
\tikzdiagh[xscale=1.25]{0}{
	\draw [pstdhl] (-.25,0) node[below]{\small $\mu_1$} -- (-.25,1);
	\draw (0,0) -- (0,1);
	\node at(.25,.125) {\tiny $\dots$};
	\node at(.25,.875) {\tiny $\dots$};
	\draw (.5,0) -- (.5,1);
	\draw [pstdhl] (.75,0)  node[below]{\small $\mu_2$} -- (.75,1);
	\node[pcolor] at(1.125,.125) { $\dots$};
	\node[pcolor] at(1.125,.875) { $\dots$};
	\draw [pstdhl] (1.5,0)  node[below]{\small $\mu_{r-1}$} -- (1.5,1);
	\draw (1.75,0) -- (1.75,1);
	\node at(2,.125) {\tiny $\dots$};
	\node at(2,.875) {\tiny $\dots$};
	\draw (2.25,0) -- (2.25,1);
	\draw [pstdhl] (2.5,0)  node[below]{\small $\mu_{r}$} -- (2.5,1);
	\draw (2.75,0) -- (2.75,1);
	\node at(3,.125) {\tiny $\dots$};
	\node at(3,.875) {\tiny $\dots$};
	\draw (3.25,0) -- (3.25,1);
	\filldraw [fill=white, draw=black] (-.375,.25) rectangle (3.375,.75) node[midway] { $D$};
}
\ \mapsto \ 
\tikzdiagh[xscale=1.25]{0}{
	\draw [pstdhl] (-.25,0) node[below]{\small $\mu_1$} -- (-.25,1);
	\draw (0,0) -- (0,1);
	\node at(.25,.125) {\tiny $\dots$};
	\node at(.25,.875) {\tiny $\dots$};
	\draw (.5,0) -- (.5,1);
	\draw [pstdhl] (.75,0)  node[below]{\small $\mu_2$} -- (.75,1);
	\node[pcolor] at(1.125,.125) { $\dots$};
	\node[pcolor] at(1.125,.875) { $\dots$};
	\draw [pstdhl] (1.5,0)  node[below]{\small $\mu_{r-1}$} -- (1.5,1);
	\draw (1.75,0) -- (1.75,1);
	\node at(2,.125) {\tiny $\dots$};
	\node at(2,.875) {\tiny $\dots$};
	\draw (2.25,0) -- (2.25,1);
	\draw [pstdhl] (2.5,0)  node[below]{\small $\mu_{r}$} -- (2.5,1);
	\draw (2.75,0) -- (2.75,1);
	\node at(3,.125) {\tiny $\dots$};
	\node at(3,.875) {\tiny $\dots$};
	\draw (3.25,0) -- (3.25,1);
	\filldraw [fill=white, draw=black] (-.375,.25) rectangle (3.375,.75) node[midway] { $D$};
	\draw [pstdhl]  (3.5,0)  node[below]{\small $\mu_{r+1}$}  -- (3.5,1);
}
\end{equation*}
Let $\mathfrak{I} : \cD_{dg}(\muT_{b},d_\mu) \rightarrow \cD_{dg}(\dgT^{\und \mu'}_b,d_{\mu'})$ be the corresponding induction dg-functor, and let $ \mathfrak{\bar I} : \cD_{dg}(\dgT^{\und \mu'}_b, d_{\mu'}) \rightarrow \cD_{dg}(\muT_b,d_\mu) $ be the restriction dg-functor. 

\begin{prop}\label{prop:indresredstrand}
There is a natural isomorphism $ \mathfrak{\bar I} \circ \mathfrak{I} \cong \id$. 
\end{prop}

\begin{proof}
The statement follows from \cref{prop:Tdecomp}. 
\end{proof}

\subsection{Categorification theorem}

In this section, we suppose $\Bbbk$ is a field. 
Recall that $\bZ\pp{\lambda,q}$ is given by Laurent series with non-zero coefficients contained in certain cones of $\bZ^2$ (see \cite{laurent} for a nice exposition, or \cite[\S5]{asympK0} for categorification). 
For a $\bZ^2$-graded dg-algebra $(A,d)$, let $\cD_{dg}^{cblf}(A,d)$ be its c.b.l.f. derived category, that is the full sub-category of dg-modules having a cone bounded, locally finite dimensional homology, or in other words having graded Euler characteristic contained in $\bZ\pp{\lambda,q}$. 
We denote by $\bKO^\Delta(A,d)$ the asymptotic Grothendieck group  (it is a version of Grothendieck group where we mod out relations coming from infinite iterated extensions, see \cite{asympK0} for details) of  $\cD_{dg}^{cblf}(A,d)$. 
Since $(\muT_b,d_\mu)$ is a positive c.b.l.f. dimensional $\bZ^2$-graded dg-algebra (in the sense of~\cite[\S9]{asympK0}), we know that $\bKO^\Delta(\muT_b,d_\mu)$ is a free $\bZ\pp{q,\lambda}$-module and is spanned by the classes of indecomposable relatively projective $(\muT_b,d_\mu)$-modules (i.e. direct summands of $(\muT_b,d_\mu)$). The action of $q$ (resp. $\lambda$) is given by a grading shift up in the $q$-degree (resp. $\lambda$-degree). 
We also write ${}_\bQ\bKO^\Delta(\muT_b,d_\mu) := \bKO^\Delta(\muT_b,d_\mu) \otimes_{\bZ\pp{q,\lambda}} \bQ\pp{q,\lambda}$.

\smallskip 

For an element $f = \sum_{a,b} \alpha_{a,b} q^a \lambda^b \in \bZ\pp{q,\lambda}$ where $\alpha_{a,b} \geq 0$, we write
\[
\oplus_f (M) := \bigoplus_{a,b} q^a \lambda^b  (\underbrace{M \oplus M \oplus \cdots \oplus M}_{\alpha_{a,b}}),
\]
for any module $M$. Therefore we have in  $\bKO^\Delta(\muT_b,d_\mu)$ that 
$[\oplus_f(M)] = f [M]$.

\smallskip

For each $\rho \in \mathcal{P}_{b}^{r}$ there is a relatively projective $(\muT_b,d_\mu)$-module given by $(\muP_{\rho}, d_\mu)$ where
\[
\muP_{\rho} := 
\muT_b 1_\rho = \ 
\tikzdiag[xscale=1.25]{
	\draw [pstdhl] (-.25,0)  node[below]{\small $\mu_1$} -- (-.25,1);
	\draw (0,0) -- (0,1);
	\node at(.25,.25) {\tiny $\dots$};
	\draw (.5,0) -- (.5,1);
	\tikzbrace{0}{.5}{0}{\small $b_1$};
	\draw [pstdhl] (.75,0) node[below]{\small $\mu_2$} -- (.75,1);
	\node[pcolor] at(1.125,.25) { $\dots$};
	\draw [pstdhl] (1.5,0)   -- (1.5,1);
	\draw (1.75,0) -- (1.75,1);
	\node at(2,.25) {\tiny $\dots$};
	\draw (2.25,0) -- (2.25,1);
	\tikzbrace{1.75}{2.25}{0}{\small $b_{r-1}$};
	\draw [pstdhl] (2.5,0) node[below]{\small $\mu_r$}   -- (2.5,1);
	\draw (2.75,0) -- (2.75,1);
	\node at(3,.25) {\tiny $\dots$};
	\draw (3.25,0) -- (3.25,1);
	\tikzbrace{2.75}{3.25}{0}{\small $b_{r}$};
	\filldraw [fill=white, draw=black] (-.375,.5) rectangle (3.375,1.25) node[midway] { $\muT_b$};
}
\]
Let $\nh_n$ be the nilHecke algebra on $n$-strands (it is presented as a diagrammatic algebra with only black strands and dots, subject to the relations \cref{eq:nhR2andR3} and \cref{eq:nhdotslide}). There is an inclusion (because of \cref{thm:Tbasis})
\begin{equation}\label{eq:nhinclusion}
\imath : \nh_{b_1} \otimes \nh_{b_2} \otimes \cdots \otimes \nh_{b_r} \hookrightarrow \muT_{b},
\end{equation}
given by
\[
\tikzdiag[xscale=1.5]{
	\draw (0,-.5) -- (0,.5);
	\node at(.25,-.4) {\tiny $\dots$};
	\node at(.25,.4) {\tiny $\dots$};
	\draw (.5,-.5) -- (.5,.5);
	\filldraw [fill=white, draw=black] (-.1,-.25) rectangle (.6,.25) node[midway] { $\nh_{b_1}$};
}
\otimes
\tikzdiag[xscale=1.5]{
	\draw (0,-.5) -- (0,.5);
	\node at(.25,-.4) {\tiny $\dots$};
	\node at(.25,.4) {\tiny $\dots$};
	\draw (.5,-.5) -- (.5,.5);
	\filldraw [fill=white, draw=black] (-.1,-.25) rectangle (.6,.25) node[midway] { $\nh_{b_{2}}$};
}
\otimes
\cdots
\otimes
\tikzdiag[xscale=1.5]{
	\draw (0,-.5) -- (0,.5);
	\node at(.25,-.4) {\tiny $\dots$};
	\node at(.25,.4) {\tiny $\dots$};
	\draw (.5,-.5) -- (.5,.5);
	\filldraw [fill=white, draw=black] (-.1,-.25) rectangle (.6,.25) node[midway] { $\nh_{b_r}$};
}
\ \mapsto \ 
\tikzdiagl[xscale=1.5]{
	\draw [pstdhl] (-.25,-.5) node[below]{\small $\mu_1$} -- (-.25,.5);
	\draw (0,-.5) -- (0,.5);
	\node at(.25,-.4) {\tiny $\dots$};
	\node at(.25,.4) {\tiny $\dots$};
	\draw (.5,-.5) -- (.5,.5);
	%
	\draw [pstdhl] (.75,-.5) node[below]{\small $\mu_2$} -- (.75,.5);
	\draw (1,-.5) -- (1,.5);
	\node at(1.25,-.4) {\tiny $\dots$};
	\node at(1.25,.4) {\tiny $\dots$};
	\draw (1.5,-.5) -- (1.5,.5);
	%
	\draw [pstdhl] (1.75,-.5) node[below]{\small $\mu_3$}   -- (1.75,.5);
	\node[pcolor] at(2.125,0) { $\dots$};
	\draw [pstdhl] (2.5,-.5) node[below]{\small $\mu_r$}  -- (2.5,.5);
	\draw (2.75,-.5) -- (2.75,.5);
	\node at(3,-.4) {\tiny $\dots$};
	\node at(3,.4) {\tiny $\dots$};
	\draw (3.25,-.5) -- (3.25,.5);
	%
	\filldraw [fill=white, draw=black] (-.1,-.25) rectangle (.6,.25) node[midway] { $\nh_{b_1}$};
	\filldraw [fill=white, draw=black] (1-.1,-.25) rectangle (1+.6,.25) node[midway] { $\nh_{b_2}$};
	\filldraw [fill=white, draw=black] (2.75-.1,-.25) rectangle (2.75+.6,.25) node[midway] { $\nh_{b_r}$};
}
\]
Furthermore, it is well-known (see for example~\cite[Section~2.2]{KL1}) that $\nh_n$ admits a unique primitive idempotent up to equivalence given by
\[
e_{n} := \tau_{\vartheta_n} x_1^{n-1} x_2^{n-2} \cdots x_{n-1} \in \nh_n,
\]
where $\vartheta_n \in S_n$ is the longest element, $\tau_{w_1w_2\cdots w_k} := \tau_{w_1}\tau_{w_2}\cdots\tau_{w_k}$, with $\tau_i$ being a crossing between the $i$-th and $(i+1)$-th strands, and $x_i$ is a dot on the $i$-th strand. 
Moreover, for degree reasons and using \cite[Lemma 4.37]{webster}, any primitive idempotent of $\muT_b$ is equivalent to the image of a collection of idempotents under the inclusion \cref{eq:nhinclusion}, and thus is of the form
\[
e_{\rho} := \imath\left( e_{b_1} \otimes \cdots \otimes e_{b_n} \right).
\]

We say that a $\bZ^2$-graded dg-category $\cC$ is \emph{c.b.l.f. generated} by a collection of objects $\{X_j\}_{j \in J}$ if any object in $\cC$ is isomorphic to an iterated extensions of shifted copies of elements from a finite subset of $\{X_j\}_{j \in J}$, with coefficients contained in $\bZ\pp{q,\lambda}$ (see \cite[Appendix B]{LNV} for a precise definition). In this case, we also have that $\bKO^\Delta(\cC) $ is spanned as $\bZ\pp{q,\lambda}$ by the classes of  $[X_j]$ for all $j \in J$. 
As a consequence of the explanations above, we obtain the following:

\begin{prop}
The dg-category $\cD_{dg}^{cblf}(\muT_b, d_\mu)$ is c.b.l.f. generated by $\{(\muT_b e_\rho, d_\mu) | \rho \in \cP_b^{r}\}$. 
\end{prop}

It is also well-known (see \cite[\S2.2.3]{KL1} for example) that there is a decomposition
\[
\nh_n \cong q^{n(n-1)/2} \bigoplus_{[n]_q!} \nh_n e_n,
\]
as left $\nh_n$-modules. For the same reasons, we obtain
\begin{equation}\label{eq:dividedsummands}
\muP_\rho \cong q^{\sum_{i=1}^r b_i(b_i-1)/2}  \bigoplus_{\prod_{i=1}^r [b_i]_q!} \muT_b e_\rho.
\end{equation}
In the other direction, one can construct a free resolution of $\nh_n e_n$ over $\nh_n$ with coefficients (i.e. grading shifts) corresponding to $1/(q^{n(n-1)/2}[n]_q!)$ and contained in $\bZ\pp{q}$. Similarly, we one can construct a c.b.l.f. resolution of $\muT_b e_\rho$ over $\muP_\rho$, and thus we obtain the following:

\begin{cor}\label{cor:muPgenerates}
The dg-category $\cD_{dg}^{cblf}(\muT_b, d_\mu)$ is c.b.l.f. generated by $\{(\muP_\rho, d_\mu) | \rho \in \cP_b^{r}\}$. 
\end{cor}

In particular, we have that $\bKO^\Delta(\muT_b, d_\mu) $ is spanned either by the classes of  $[(\muT_b e_\rho, d_\mu)]$ for all $\rho \in \cP_b^r$, or by the classes of $[(\muP_\rho, d_\mu)]$. 
The following lemma is well-known, and one can find a proof of it for example in \cite[Proposition 3.17]{NV2}.

\begin{lem}
\label{lem:NdotsleftNH}
For $k > n$ we have 
\[
\tikzdiagl{
	\draw (0,0) -- (0,1);
	\draw (.5,0) -- (.5,1);
	\node at(1,.5) {\small $\dots$};
	\draw (1.5,0) -- (1.5,1);
	\tikzbrace{0}{1.5}{0}{\small $k$};
}
\ =
\sum_{i} \ 
\tikzdiag{
	\draw (0,-.5) -- (0,1.5) node[midway, tikzdot]{} node[midway, xshift=-1.5ex, yshift=.75ex]{\small $n$};
	\draw (.5,-.5) -- (.5,1.5);
	\node at(1,.5) {\small $\dots$};
	\draw (1.5,-.5) -- (1.5,1.5);
	\filldraw [fill=white, draw=black] (-.1,-.25) rectangle (1.6,.25) node[midway] { $v_i$};
	\filldraw [fill=white, draw=black] (-.1,.75) rectangle (1.6,1.25) node[midway] { $u_i$};
}
\]
for a certain finite collection of elements $u_i, v_i \in \nh_k$.
\end{lem}

\begin{lem}\label{lem:K0surjection}
There is a surjection
\[
L(\und \mu)_{|\und \mu| - 2b} \twoheadrightarrow {}_\bQ\bKO^\Delta(\muT_b,d_\mu), \quad
v_\rho \mapsto [(\muP_{\rho}, d_\mu)],
\]
of $\bQ\pp{q,\lambda}$-modules. 
\end{lem}

\begin{proof}
We want to show that $\bKO^\Delta(\muT_b,d_\mu)$ is spanned by the classes of $ [(\muP_{\rho}, d_\mu)]$ for all $\rho \in \cP_{b}^{r, \und \mu}$. 
Take any $\rho \in \cP_{b}^{r}$, and also assume that $b_1 \leq \mu_1$ if $\mu_1 \in \bN$. Because of \cref{lem:NdotsleftNH}, we have that $1_\rho$ can be rewritten as a sum of elements factorizing through $1_{\rho'}$ for various $\rho' \in \cP_{b}^{r, \und \mu}$  by \cref{eq:redR2}. 
Then $(\muP_{\rho}, d_\mu)$ is isomorphic to a direct sum of shifted copies of $(\muP_{\rho'}, d_\mu)$  for various $\rho' \in \cP_{b}^{r, \und \mu}$.  
If $\mu_1 \in \beta + \bZ$ we are done. Suppose $\mu_1 \in \bN$ and $b_1 > \mu_1$. Then  
$(\muP_{\rho}, d_\mu)$ is acyclic by 
 \cref{lem:NdotsleftNH}, concluding the proof. 
\end{proof}

\begin{exe}
We consider $\und \mu = (\mu_1, 1)$ and $\rho = (b_1, 2)$. We have
\begin{align*}
\tikzdiagl{
	\draw [pstdhl] (0,0) node[below]{\small $\mu_1$} -- (0,1);
	\draw (.5,0) -- (.5,1);
	\draw (1.5,0) -- (1.5,1);
	\node at(1,.5) {\small $\dots$};
	\tikzbrace{.5}{1.5}{-.1}{\small $b_1$};
	\draw [stdhl] (2,0) node[below]{\small $1$} -- (2,1);
	\draw (2.5,0) -- (2.5,1);
	\draw (3,0) -- (3,1);
}
\ &= \  
\tikzdiagl{
	\draw [pstdhl] (0,0) node[below]{\small $\mu_1$} -- (0,1);
	\draw (.5,0) -- (.5,1);
	\draw (1.5,0) -- (1.5,1);
	\node at(1,.5) {\small $\dots$};
	\tikzbrace{.5}{1.5}{-.1}{\small $b_1$};
	\draw [stdhl] (2,0) node[below]{\small $1$} -- (2,1);
	\draw (2.5,0) .. controls (2.5,.25) and (3,.25) .. (3,.5) .. controls (3,.75) and (2.5,.75) .. (2.5,1)
		node[pos=.75, tikzdot]{};
	\draw (3,0)  .. controls (3,.25) and (2.5,.25) .. (2.5,.5)
		node[pos=1, tikzdot]{} 
		.. controls (2.5,.75) and (3,.75) .. (3,1);
}
\ - \ 
\tikzdiagl{
	\draw [pstdhl] (0,0) node[below]{\small $\mu_1$} -- (0,1);
	\draw (.5,0) -- (.5,1);
	\draw (1.5,0) -- (1.5,1);
	\node at(1,.5) {\small $\dots$};
	\tikzbrace{.5}{1.5}{-.1}{\small $b_1$};
	\draw [stdhl] (2,0) node[below]{\small $1$} -- (2,1);
	\draw (2.5,0) .. controls (2.5,.25) and (3,.25) .. (3,.5) .. controls (3,.75) and (2.5,.75) .. (2.5,1);
	\draw (3,0)  .. controls (3,.25) and (2.5,.25) .. (2.5,.5)
		node[pos=.25, tikzdot]{}
		node[pos=1, tikzdot]{} 
		.. controls (2.5,.75) and (3,.75) .. (3,1);
}
\\
\ &= \ 
\tikzdiagl{
	\draw [pstdhl] (0,0) node[below]{\small $\mu_1$} -- (0,1);
	\draw (.5,0) -- (.5,1);
	\draw (1.5,0) -- (1.5,1);
	\node at(1,.5) {\small $\dots$};
	\tikzbrace{.5}{1.5}{-.1}{\small $b_1$};
	\draw (2.5,0) .. controls (2.5,.25) and (3,.25) .. (3,.5) .. controls (3,.75) and (2.5,.75) .. (2.5,1)
		node[pos=.8, tikzdot]{};
	\draw (3,0)  .. controls (3,.25) and (2,.25) .. (2,.5)
		.. controls (2,.75) and (3,.75) .. (3,1);
	\draw [stdhl] (2,0) node[below]{\small $1$}   .. controls (2,.25) and (2.5,.25) .. (2.5,.5) .. controls (2.5,.75) and (2,.75) ..  (2,1);
}
\ - \ 
\tikzdiagl{
	\draw [pstdhl] (0,0) node[below]{\small $\mu_1$} -- (0,1);
	\draw (.5,0) -- (.5,1);
	\draw (1.5,0) -- (1.5,1);
	\node at(1,.5) {\small $\dots$};
	\tikzbrace{.5}{1.5}{-.1}{\small $b_1$};
	\draw (2.5,0) .. controls (2.5,.25) and (3,.25) .. (3,.5) .. controls (3,.75) and (2.5,.75) .. (2.5,1);
	\draw (3,0)  .. controls (3,.25) and (2,.25) .. (2,.5)
		node[pos=.2, tikzdot]{}
		.. controls (2,.75) and (3,.75) .. (3,1);
	\draw [stdhl] (2,0) node[below]{\small $1$}   .. controls (2,.25) and (2.5,.25) .. (2.5,.5) .. controls (2.5,.75) and (2,.75) ..  (2,1);
}
\end{align*}
If $\mu_1 = 1 \in \bN$, then we have similarly that
\[
d_\mu\left(
\tikzdiag{
	\draw (.5,0) .. controls (.5,.25) and (1,.25) .. (1,.5) .. controls (1,.75) and (.5,.75) ..(.5,1) node[pos=.8, tikzdot]{};
	\draw (1,0) .. controls (1,.25) .. (0,.5) .. controls (1,.75) .. (1,1);
	\draw [stdhl] (0,0) node[below]{\small $1$} -- (0,1) node[midway, nail]{};
}
\ - \ 
\tikzdiag{
	\draw (.5,0) .. controls (.5,.25) and (1,.25) .. (1,.5) .. controls (1,.75) and (.5,.75) ..(.5,1);
	\draw (1,0) .. controls (1,.25) .. (0,.5) node[pos=.2, tikzdot]{} .. controls (1,.75) .. (1,1);
	\draw [stdhl] (0,0) node[below]{\small $1$} -- (0,1) node[midway, nail]{};
}
\right)
= \ 
\tikzdiag{
	\draw (.5,0) -- (.5,1);
	\draw (1,0) -- (1,1);
	\draw [stdhl] (0,0) node[below]{\small $1$} -- (0,1);
}
\]
and thus $\muP_\rho$ is acyclic whenever $b_1 \geq 2$. 
\end{exe}

\subsubsection{Categorifed Shapovalov form} 
As in \cite[\S2.5]{KL1}, let $\psi : \muT \rightarrow \opalg{(\muT)}$ be the map that takes the mirror image of diagrams along the horizontal axis.
Given a left $(\muT,d_\mu)$-module $M$, we obtain a right $(\muT,d_\mu)$-module $M^\psi$ with action given by 
\[
m^\psi \cdot r := (-1)^{\deg_h(r) \deg_h(m)} \psi(r) \cdot m,
\] 
for $m \in M$ and $r \in \muT$. 
Then we define the dg-bifunctor
\begin{align*}
(-,-) &:  \cD_{dg}(\muT,d_\mu)  \times \cD_{dg}(\muT, d_\mu)   \rightarrow \cD_{dg}(\Bbbk, 0), 
&
(W,W') := W^\psi \Lotimes_{(\muT, d_\mu)} W'.
\end{align*}

\begin{prop}\label{prop:catshap}
The dg-bifunctor defined above satisfies:
\begin{itemize}
\item $((\muT_0, d_\mu),(\muT_0, d_\mu)) \cong (\Bbbk,0)$;
\item $(\Ind_b^{b+1} M,M') \cong (M, \Res_b^{b+1} M')$ for all $M,M' \in \cD_{dg}(\muT, d_\mu)$; 
\item $(\oplus_f M,M') \cong (M, \oplus_f M') \cong \oplus_f (M,M')$ for all $f \in \bZ\pp{q,\lambda}$;
\item $(M,M') \cong (\mathfrak{I}(M), \mathfrak{I}(M'))$.
\end{itemize}
\end{prop}

\begin{proof}
Straightforward, except for the last point which follows from:
\[
    (\mathfrak{I}(M), \mathfrak{I}(M')) \cong ( M,\mathfrak{\bar I}\circ \mathfrak{I}(M')) \cong (M,M'),
\]
using \cref{prop:indresredstrand} together with the adjunction $\mathfrak{I} \vdash  \mathfrak{\bar I}$.
\end{proof}

Comparing \cref{prop:catshap} to \cref{sec:shepfortensor}, we deduce 
that $(-,-)$ has the same properties on the asymptotic Grothendieck group of $(\muT,d_\mu)$ as the Shapovalov form on $L(\und \mu)$. 

\subsubsection{Categorification theorem}

Because of \cref{thm:sl2comqi}, we know that the functors $\E$ and $\F$ induce an $U_q(\slt)$-action on ${}_\bQ\bKO^\Delta(\muT, d_\mu) \cong \bigoplus_{b \geq 0} {}_\bQ\bKO^\Delta(\muT_b, d_\mu)$.

\begin{thm}\label{thm:isocat}
There is an isomorphism of $U_q(\slt)$-modules
\begin{equation*}
\gamma : L(\und \mu) \xrightarrow{\simeq} {}_\bQ\bKO^\Delta(\muT, d_\mu), \quad v_\rho \mapsto  [(\muP_\rho, d_\mu)].
\end{equation*}
Moreover the divided power basis elements are sent to $\overline{v}_\rho \mapsto [(\muT_b e_\rho, d_\mu)]$.
\end{thm}

\begin{proof}
The argument is similar as in \cite[Theorem 4.7]{LNV}. 
By \cref{lem:K0surjection}, we know that the $\bQ\pp{q,\lambda}$-linear map $\gamma$ 
 is surjective. 
Moreover, the map $\gamma$ clearly commutes with the action of $K^{\pm 1}$, and with $E$ because of \cref{prop:Tdecomp} together with \cref{eq:Evkappa}.  
By \cref{prop:catshap}, $\gamma$ intertwines the Shapovalov form with the bilinear form induced by the bifunctor $(-,-)$ on ${}_\bQ\bKO^\Delta(\muT,d_\mu)$. 
Therefore $\gamma$ is a $\bQ\pp{q,\lambda}$-linear isomorphism by non-degeneracy of the Shapovalov form. 
Since the map $\gamma$ intertwines the Shapovalov form with the bifunctor $(-,-)$, and commutes with the action of $E$ and $K^{\pm 1}$, we also deduce by non-degeneracy of the Shapovalov form that $\gamma$ commutes with the action of $F$. In conclusion, $\gamma$ is an isomorphism of $U_q(\slt)$-modules. 

The statement with the divided power basis elements is immediate from \cref{eq:dividedsummands}. 
\end{proof}



\section{Derived standard stratification}\label{sec:standard}

In \cite{webster}, the change of basis corresponding to \cref{lem:Frewriting} is categorified by introducing a standard module (with respect to some standard stratification on $T_b^\mu\amod$) for each $\rho$. This standard module categorifies the basis elements $\tilde v_{\rho}$. The change of basis is encoded in the fact that the projective module $T_b^\mu 1_\rho$ that categorifies the basis element $v_{\rho}$ admits a filtration with quotient being the standard modules. 
We introduce similar modules for $\muT$ that play the role of the standard modules. Strictly speaking, they do not give a standard stratification of $(\muT, d_\mu)\amod$, but they do have a similar behavior in a derived way, see \cref{sec:stratification} below.

\subsection{Standard modules}

There are two ways to construct the standard modules: either directly, or as an iterated mapping cone construction. We describe both constructions in this order. 

\subsubsection{Definition of standard modules}

Fix $\rho = (b_1, \dots, b_r) \in \cP_b^r$. Let
\begin{align*}
J_\rho &:= \bigsqcup_{\ell = 2}^r  J_{\ell, \rho}, 
&
J_{\ell, \rho} &:= \{1, \dots, b_\ell\}.
\end{align*}
For $\bj \subset J_\rho$ we write $\bj_\ell = \{\bj_{\ell,1}, \dots, \bj_{\ell, |\bj_\ell|} \} := \bj \cap J_{\ell,\rho}$ with $\bj_{\ell,1} < \cdots < \bj_{\ell, |\bj_\ell|}$. We define
\begin{align*}
\rho_\bj &:= (b_1 + |\bj_2|, b_2 - |\bj_2| + |\bj_3|, \dots,  b_r - |\bj_{r-1}| + |\bj_{r}|, b_r - |\bj_{r}|),
\end{align*}
or in others words we obtain $\rho_\bj$ from $\rho$ by increasing $b_{j-1}$ by $1$ and decreasing $b_j$ by $1$ for each $j \in \bj \cap J_{\ell, \rho}$. Then we define
\[
\muS_{\rho,\bj} := 
q^{\sum_{\ell = 2}^{r}\sum_{t \in \bj_\ell }(\mu_\ell - 2 t +2)}
\muP_{\rho_\bj} [|\bj|].
\]
Consider $\bj' \subset \bj$ such that $|\bj| = |\bj'| + 1$. 
We have $\bj' = \bj \setminus \{ b' \in  J_{\ell,\rho}  \} $ for some $b'$ and $\ell$.
We obtain a map of left $(\muT,d_\mu)$-modules $ (\muS_{\rho,\bj}, d_\mu) \rightarrow (\muS_{\rho,\bj'}, d_\mu)$ by gluing on the bottom the element:
\[
\tau_{\bj,\bj'} :=
\tikzdiagh{0}{
	\node at(.5,.5) {$\dots$};
	\draw 		(1.5,0) .. controls (1.5,.5) and (2,.5) .. (2,1);
	\node at (2,.15){$\dots$};
	\node at (2.5,.85){$\dots$};
	\draw 		(2.5,0)  .. controls (2.5,.5) and (3,.5) .. (3,1);
	\draw 		(3.5,0)  .. controls (3.5,.5) and (4,.5) .. (4,1);
	\node at (4,.15){$\dots$};
	\node at (4.5,.85){$\dots$};
	\draw 		(4.5,0)  .. controls (4.5,.5) and (5,.5) .. (5,1);
	\draw		(5,0) .. controls (5,.5) and (1.5,.5)  .. (1.5,1);
	\tikzbrace{1.5}{2.5}{-.15}{\small $p_1$};
	\tikzbrace{3.5}{4.5}{-.15}{\small $p_2$};
	\node at(6,.5) {$\dots$};
	\draw[pstdhl] 	(3,0) node[below]{\small $\mu_\ell$}  .. controls (3,.5) and (3.5,.5) .. (3.5,1);
}
\]
where $p_1 + p_2 = b'-1$ and $p_1 = \#\{ j \in \bj \cap J_{\ell,\rho} | j < b' \}$, and extending on the left and right with vertical strands with color and label matching $1_{\rho_\bj'}$.

\begin{lem}\label{lem:stdtaucom}
Consider $\bj''' \subset \bj' \subset \bj$ and $\bj''' \subset \bj'' \subset \bj$ such that $|\bj| = |\bj'| + 1  = |\bj''| + 1  = |\bj'''| +2 $ and $\bj' \neq \bj''$. We have
\[
\tau_{\bj,\bj'} \tau_{\bj',\bj'''} = \tau_{\bj,\bj''} \tau_{\bj'',\bj'''}.
\]
\end{lem}

\begin{proof}
We first assume that $\bj' = \bj  \setminus \{ b' \in  J_{\ell,\rho} \} $ and  $\bj'' = \bj  \setminus \{ b'' \in  J_{\ell,\rho} \}$ for the same $\ell$, and thus $b' \neq b''$. Without loss of generality, we can also assume that $b' < b''$. 
Then we obtain
\[
\tau_{\bj,\bj''} \tau_{\bj'',\bj'''}  =
\tikzdiagh{0}{
	\draw (0,0) .. controls (0,.5) and (.5,.5) .. (.5,1) 
			.. controls (.5,1.5) and (1,1.5) .. (1,2);
	\draw (1,0) .. controls (1,.5) and (1.5,.5) .. (1.5,1)
			.. controls (1.5,1.5) and (2,1.5) .. (2,2);
	\node at (.5,.1){\small $\dots$}; \node at (1,1){\small $\dots$}; \node at (1.5,1.9){\small $\dots$};
	%
	%
	\draw (2,0) .. controls (2,.5) and (2.5,.5) .. (2.5,1)
		.. controls (2.5,1.5) and (3,1.5) .. (3,2);
	\draw (3,0) .. controls (3,.5) and (3.5,.5) .. (3.5,1)
		.. controls (3.5,1.5) and (4,1.5) .. (4,2);
	\node at (2.5,.1){\small $\dots$}; \node at (3,1){\small $\dots$}; \node at (3.5,1.9){\small $\dots$};
	\draw (3.5,0) .. controls (3.5,.5) and (4,.5) .. (4,1)
		.. controls (4,1.5) and (4.5,1.5) .. (4.5,2);
	\draw (4,0) .. controls (4,.5) and (0,.5) .. (0,1)
		.. controls (0,1.5) and (.5,1.5) .. (.5,2);
	\draw (4.5,0) -- (4.5,1) 
		.. controls (4.5,1.5) and (5,1.5) .. (5,2);
	\draw (5.5,0) -- (5.5,1)
		.. controls (5.5,1.5) and (6,1.5) .. (6,2);
	\node at (5,.1){\small $\dots$}; \node at (5,1){\small $\dots$}; \node at (5.5,1.9){\small $\dots$};
	\draw (6,0) -- (6,1)
		.. controls (6,1.5) and (0,1.5) .. (0,2);
	\tikzbrace{0}{1}{-.15}{\small $p_1$};
	\tikzbrace{2}{3}{-.15}{\small $p_2$};
	\tikzbrace{4.5}{5.5}{-.15}{\small $b''-b'-1$};
	\node at(-1,1){$\dots$};
	\node at(7,1){$\dots$};
	\draw [pstdhl] (1.5,0) node[below]{\small $\mu_\ell$} .. controls (1.5,.5) and (2,.5) .. (2,1)
		.. controls (2,1.5) and (2.5,1.5) .. (2.5,2);
}
\]
where $p_1 + p_2 = b'-1$ and $p_1 = \#\{ j \in \bj \cap J_{\ell,\rho} | j < b' \}$, 
and
\[
\tau_{\bj,\bj'} \tau_{\bj',\bj'''}  =
\tikzdiagh{0}{
	\draw (0,0) .. controls (0,.5) and (.5,.5) .. (.5,1) 
			.. controls (.5,1.5) and (1,1.5) .. (1,2);
	\draw (1,0) .. controls (1,.5) and (1.5,.5) .. (1.5,1)
			.. controls (1.5,1.5) and (2,1.5) .. (2,2);
	\node at (.5,.1){\small $\dots$}; \node at (1,1){\small $\dots$}; \node at (1.5,1.9){\small $\dots$};
	%
	%
	\draw (2,0) .. controls (2,.5) and (2.5,.5) .. (2.5,1)
		.. controls (2.5,1.5) and (3,1.5) .. (3,2);
	\draw (3,0) .. controls (3,.5) and (3.5,.5) .. (3.5,1)
		.. controls (3.5,1.5) and (4,1.5) .. (4,2);
	\node at (2.5,.1){\small $\dots$}; \node at (3,1){\small $\dots$}; \node at (3.5,1.9){\small $\dots$};
	\draw (3.5,0) .. controls (3.5,.5) and (4,.5) .. (4,1)
		.. controls (4,1.5) and (4.5,1.5) .. (4.5,2);
	\draw (4,0) .. controls (4,.5) and (4.5,.5) .. (4.5,1)
		.. controls (4.5,1.5) and (.5,1.5) .. (.5,2);
	\draw (4.5,0) .. controls (4.5,.5) and (5,.5) .. (5,1) 
		-- (5,2);
	\draw (5.5,0) .. controls (5.5,.5) and (6,.5) .. (6,1)
		-- (6,2);
	\node at (5,.1){\small $\dots$}; \node at (5.5,1){\small $\dots$}; \node at (5.5,1.9){\small $\dots$};
	\draw (6,0) .. controls (6,.5) and (0,.5) .. (0,1)
		-- (0,2);
	\tikzbrace{0}{1}{-.15}{\small $p_1$};
	\tikzbrace{2}{3}{-.15}{\small $p_2$};
	\tikzbrace{4.5}{5.5}{-.15}{\small $b''-b'-1$};
	\node at(-1,1){$\dots$};
	\node at(7,1){$\dots$};
	\draw [pstdhl] (1.5,0) node[below]{\small $\mu_\ell$} .. controls (1.5,.5) and (2,.5) .. (2,1)
		.. controls (2,1.5) and (2.5,1.5) .. (2.5,2);
}
\]
Thus we have $\tau_{\bj,\bj'} \tau_{\bj',\bj'''} = \tau_{\bj,\bj''} \tau_{\bj'',\bj'''}$ by the braid moves in \cref{eq:nhR2andR3} and \cref{eq:crossingslidered}.

We now assume that $\bj' = \bj  \setminus \{ b' \in  J_{\ell',\rho} \}$ and $\bj' = \bj  \setminus  \{ b'' \in  J_{\ell'',\rho}  \} $ for $\ell' \neq \ell''$. Then we have 
$\tau_{\bj,\bj'} \tau_{\bj',\bj'''} = \tau_{\bj,\bj''} \tau_{\bj'',\bj'''}$ by a braid-like planar isotopy, exchanging distant crossings. 
\end{proof}

We extend the natural order on each $J_{\ell,\rho}$ to a total order on $J_\rho$ by declaring that $b' < b''$ whenever $b' \in J_{\ell',\rho}$ and $b'' \in J_{\ell'',\rho}$ and $\ell' < \ell''$. 

\begin{defn}
The \emph{standard module $(\muS_{\rho}, d_{\dgS})$} is defined as 
\[
\muS_{\rho} := \bigoplus_{\bj \subset J_\rho} \muS_{\rho,\bj},
\]
with
\begin{align*}
d_{\dgS} &:= \sum_{\bj \subset J_\rho} (-1)^{|\bj|}(d_\mu : \muS_{\rho,\bj}  \rightarrow \muS_{\rho,\bj} ) +  \bigl( d_{\dgS,\bj} : \muS_{\rho,\bj} \rightarrow \muS_\rho  \bigr),
\\
d_{\dgS,\bj} &:= \sssum{
\bj' = \bj \setminus \{b'\} 
} (-1)^{\#\{ b'' \in \bj | b'' > b' \}} \tau_{\bj,\bj'}.
\end{align*}
We have $d_{\dgS}^2 = 0$ by \cref{lem:stdtaucom}.
\end{defn}

\begin{exe}
We take $\und \mu = (\mu_1, \mu_2)$ and $\rho = (0,2)$. We have $J_\rho = J_{2,\rho}$ with $J_{2,\rho} = \{1,2\}$. We draw all possible $\bj \subset J_{\rho}$ as
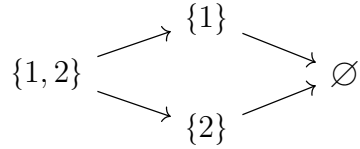
\[
\begin{tikzcd}[row sep=0ex]
&\{1\} \ar{dr}&
\\
\{1,2\} \ar{ur} \ar{dr} && \emptyset
\\
&\{2\} \ar{ur}&
\end{tikzcd}
\]
where the arrows represent the $\tau_{\bj,\bj'}$.
Then we can picture $\muS_{(0,2)}$ as the complex
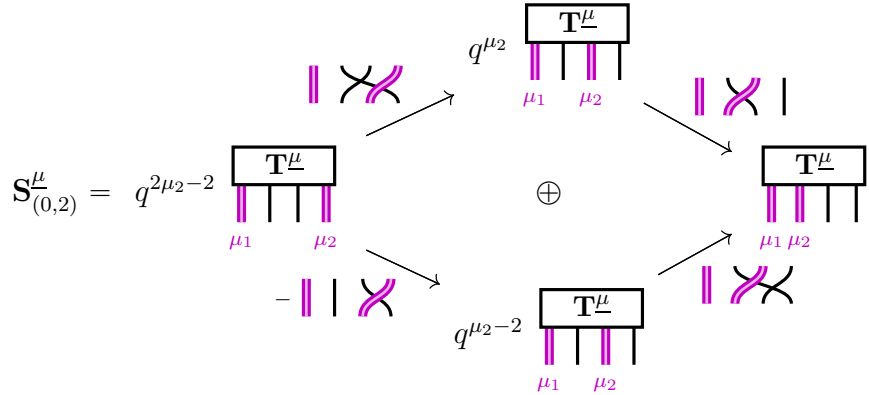
\[
\muS_{(0,2)} = 
\begin{tikzcd}[row sep=0ex]
&
 q^{\mu_2}
\tikzdiagh[xscale=.75]{-1ex}{
	\draw[pstdhl]	(0,0) node[below]{\tiny $\mu_1$} -- (0,1);
	\draw 	 	(.5,0) -- (.5,1);
	\draw[pstdhl]	(1,0) node[below]{\tiny $\mu_2$} -- (1,1);
	\draw		(1.5,0) -- (1.5,1);
	\filldraw [fill=white, draw=black] (-.15,.5) rectangle (1.65,1) node[midway]{\small $\muT$};
}
\ar{dr}
\ar[dash]{dr}{
\tikzdiag[xscale=.75,yscale=.5]{
	\draw[pstdhl]	(0,0) -- (0,1);
	\draw		(1,0) .. controls (1,.5) and (.5,.5) .. (.5,1);
	\draw[pstdhl]	(.5,0) .. controls (.5,.5) and (1,.5) .. (1,1);
	\draw 	 	(1.5,0) -- (1.5,1);
}
}
&
\\
 q^{2\mu_2-2}
 \tikzdiagh[xscale=.75]{-1ex}{
	\draw[pstdhl]	(0,0) node[below]{\tiny $\mu_1$} -- (0,1);
	\draw 	 	(.5,0) -- (.5,1);
	\draw		(1,0) -- (1,1);
	\draw[pstdhl]	(1.5,0) node[below]{\tiny $\mu_2$} -- (1.5,1);
	\filldraw [fill=white, draw=black] (-.15,.5) rectangle (1.65,1) node[midway]{\small $\muT$};
}
\ar{ur}
\ar[dash]{ur}{
\tikzdiag[xscale=.75,yscale=.5]{
	\draw[pstdhl]	(0,0) -- (0,1);
	\draw		(.5,0) .. controls (.5,.5) and (1,.5) .. (1,1);
	\draw 	 	(1.5,0)  .. controls (1.5,.5) and (.5,.5) .. (.5,1);
	\draw[pstdhl]	(1,0) .. controls (1,.5) and (1.5,.5) .. (1.5,1);
}
}
\ar{dr}
\ar[dash,swap]{dr}{
- \ 
\tikzdiag[xscale=.75,yscale=.5]{
	\draw[pstdhl]	(0,0) -- (0,1);
	\draw 	 	(.5,0) -- (.5,1);
	\draw		(1.5,0) .. controls (1.5,.5) and (1,.5) .. (1,1);
	\draw[pstdhl]	(1,0) .. controls (1,.5) and (1.5,.5) .. (1.5,1);
}
}
&\oplus&
\tikzdiagh[xscale=.75]{-1ex}{
	\draw[pstdhl]	(0,0) node[below]{\tiny $\mu_1$} -- (0,1);
	\draw[pstdhl]	(.5,0) node[below]{\tiny $\mu_2$} -- (.5,1);
	\draw		(1,0) -- (1,1);
	\draw		(1.5,0) -- (1.5,1);
	\filldraw [fill=white, draw=black] (-.15,.5) rectangle (1.65,1) node[midway]{\small $\muT$};
}
\\
&
 q^{\mu_2-2}
\tikzdiagh[xscale=.75]{-1ex}{
	\draw[pstdhl]	(0,0) node[below]{\tiny $\mu_1$} -- (0,1);
	\draw 	 	(.5,0) -- (.5,1);
	\draw[pstdhl]	(1,0) node[below]{\tiny $\mu_2$} -- (1,1);
	\draw		(1.5,0) -- (1.5,1);
	\filldraw [fill=white, draw=black] (-.15,.5) rectangle (1.65,1) node[midway]{\small $\muT$};
}
\ar{ur}
\ar[dash,swap]{ur}{
\tikzdiag[xscale=.75,yscale=.5]{
	\draw[pstdhl]	(0,0) -- (0,1);
	\draw 	 	(1.5,0)  .. controls (1.5,.5) and (.5,.5) .. (.5,1);
	\draw		(1,0) .. controls (1,.5) and (1.5,.5) .. (1.5,1);
	\draw[pstdhl]	(.5,0) .. controls (.5,.5) and (1,.5) .. (1,1);
}
}
&
\end{tikzcd}
\]
where the $d_\mu$ part of the differential is implicit. 

As another example, take $\und \mu= (\mu_1,\mu_2,\mu_3)$ and $\rho = (0,1,1)$. We have $J_{\rho} = J_{2,\rho} \sqcup J_{3,\rho}$ with $J_{2,\rho} = \{1\}$ and $J_{3,\rho} = \{1\}$. Similarly as above, we draw $\bj \subset J_{\rho}$ as
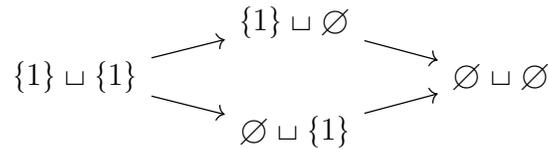
\[
\begin{tikzcd}[row sep=0ex]
&\{1\} \sqcup \emptyset \ar{dr}&
\\
\{1\} \sqcup \{1\} \ar{ur} \ar{dr} && \emptyset \sqcup \emptyset
\\
&\emptyset \sqcup \{1\} \ar{ur}&
\end{tikzcd}
\]
Then we picture $\muS_{(0,1,1)}$ as
\[
\muS_{(0,1,1)} = 
\begin{tikzcd}[row sep=0ex]
&
 q^{\mu_3}
\tikzdiagh[xscale=.75]{-1ex}{
	\draw[pstdhl]	(0,0) node[below]{\tiny $\mu_1$} -- (0,1);
	\draw[pstdhl] 	(.5,0) node[below]{\tiny $\mu_2$} -- (.5,1);
	\draw		(1,0) -- (1,1);
	\draw		(1.5,0) -- (1.5,1);
	\draw[pstdhl]	(2,0) node[below]{\tiny $\mu_3$} -- (2,1);
	\filldraw [fill=white, draw=black] (-.15,.5) rectangle (2.15,1) node[midway]{\small $\muT$};
}
\ar{dr}
\ar[dash]{dr}{
\tikzdiag[xscale=.75,yscale=.5]{
	\draw[pstdhl]	(0,0) -- (0,1);
	\draw[pstdhl] 	 	(.5,0) -- (.5,1);
	\draw		(1,0) -- (1,1);
	\draw		(2,0) .. controls (2,.5) and (1.5,.5) .. (1.5,1);
	\draw[pstdhl]	(1.5,0) .. controls (1.5,.5) and (2,.5) .. (2,1);
}
}
&
\\
 q^{\mu_2+\mu_3}
\tikzdiagh[xscale=.75]{-1ex}{
	\draw[pstdhl]	(0,0) node[below]{\tiny $\mu_1$} -- (0,1);
	\draw 	 	(.5,0) -- (.5,1);
	\draw[pstdhl]	(1,0) node[below]{\tiny $\mu_2$} -- (1,1);
	\draw		(1.5,0) -- (1.5,1);
	\draw[pstdhl]	(2,0) node[below]{\tiny $\mu_3$} -- (2,1);
	\filldraw [fill=white, draw=black] (-.15,.5) rectangle (2.15,1) node[midway]{\small $\muT$};
}
\ar{ur}
\ar[dash]{ur}{
\tikzdiag[xscale=.75,yscale=.5]{
	\draw[pstdhl]	(0,0) -- (0,1);
	\draw 	 	(1,0)  .. controls (1,.5) and (.5,.5) .. (.5,1);
	\draw		(1.5,0)  -- (1.5,1);
	\draw[pstdhl]	(.5,0) .. controls (.5,.5) and (1,.5) .. (1,1);
	\draw[pstdhl]	(2,0) -- (2,1);
}
}
\ar{dr}
\ar[dash,swap]{dr}{
- \ 
\tikzdiag[xscale=.75,yscale=.5]{
	\draw[pstdhl]	(0,0) -- (0,1);
	\draw 	 	(.5,0) -- (.5,1);
	\draw[pstdhl]		(1,0) -- (1,1);
	\draw		(2,0) .. controls (2,.5) and (1.5,.5) .. (1.5,1);
	\draw[pstdhl]	(1.5,0) .. controls (1.5,.5) and (2,.5) .. (2,1);
}
}
&\oplus&
\tikzdiagh[xscale=.75]{-1ex}{
	\draw[pstdhl]	(0,0) node[below]{\tiny $\mu_1$} -- (0,1);
	\draw[pstdhl] 	(.5,0) node[below]{\tiny $\mu_2$} -- (.5,1);
	\draw		(1,0) -- (1,1);
	\draw[pstdhl]	(1.5,0) node[below]{\tiny $\mu_3$} -- (1.5,1);
	\draw		(2,0) -- (2,1);
	\filldraw [fill=white, draw=black] (-.15,.5) rectangle (2.15,1) node[midway]{\small $\muT$};
}
\\
&
 q^{\mu_2}
\tikzdiagh[xscale=.75]{-1ex}{
	\draw[pstdhl]	(0,0) node[below]{\tiny $\mu_1$} -- (0,1);
	\draw 	 	(.5,0) -- (.5,1);
	\draw[pstdhl]	(1,0) node[below]{\tiny $\mu_2$} -- (1,1);
	\draw[pstdhl]	(1.5,0)  node[below]{\tiny $\mu_3$}-- (1.5,1);
	\draw		(2,0) -- (2,1);
	\filldraw [fill=white, draw=black] (-.15,.5) rectangle (2.15,1) node[midway]{\small $\muT$};
}
\ar{ur}
\ar[dash,swap]{ur}{
\tikzdiag[xscale=.75,yscale=.5]{
	\draw[pstdhl]	(0,0) -- (0,1);
	\draw 	 	(1,0)  .. controls (1,.5) and (.5,.5) .. (.5,1);
	\draw[pstdhl]		(1.5,0)  -- (1.5,1);
	\draw[pstdhl]	(.5,0) .. controls (.5,.5) and (1,.5) .. (1,1);
	\draw	(2,0) -- (2,1);
}
}
&
\end{tikzcd}
\]
\end{exe}

\subsubsection{Standard modules as iterated mapping cones}

Alternatively, we can build the standard modules recursively as iterated mapping cones by categorifying the following equation from \cref{lem:Frewriting}:
\begin{equation}\tag{\ref{eq:Frewritinglemma2}}
v_{\rho_1, \rho_2}^{t,\ell}
=
v_{\rho_1, \rho_2}^{t+1,\ell-1}
 -   q^{\mu+2-2\ell} 
v_{F(\rho_1), \rho_2}^{t,\ell-1},
\end{equation}
where $\mu := \mu_{r_1 + 1}$. 
In particular, we will lift all the intermediate elements 
\[
v_{\rho_1, \rho_2}^{t,\ell} := F^t\bigl( v_{\rho_1} \otimes F^{\ell} (v_{\mu}) \bigr) \otimes \tilde v_{\rho_2},
\]
 with $\rho_1 \in \bN^{r_1}$ and $\rho_2 \in \bN^{r_2}$, $r = r_1 + 1 + r_2$. 

\smallskip

Define the element
\[
\tau_{\rho_1, \rho_2}^{t,\ell}
:=
\sum_{\rho_2' \in \cP_{b_2}^{r_2}} \ \sssum{\ell_1 + \ell_2 \\= \ell-1}
1_{\rho_1} \boxtimes \ 
\tikzdiagh{-2.5ex}{
	\draw 		(1.5,0) .. controls (1.5,.5) and (2,.5) .. (2,1);
	\node at (2,.15){$\dots$};
	\node at (2.5,.85){$\dots$};
	\draw 		(2.5,0)  .. controls (2.5,.5) and (3,.5) .. (3,1);
	\draw 		(3.5,0)  .. controls (3.5,.5) and (4,.5) .. (4,1);
	\node at (4,.15){$\dots$};
	\node at (4.5,.85){$\dots$};
	\draw 		(4.5,0)  .. controls (4.5,.5) and (5,.5) .. (5,1);
	\tikzbraceop{2}{3}{1.1}{\small $\ell_1$};
	\tikzbraceop{4}{5}{1.1}{\small $\ell_2$};
	\draw		(5,0) .. controls (5,.5) and (1.5,.5)  .. (1.5,1);
	\draw		(5.5,0) -- (5.5,1);
	\node at(6,.5) {$\dots$};
	\draw		(6.5,0) -- (6.5,1);
	\tikzbraceop{5.5}{6.5}{1.1}{\small $t$};
	\draw[pstdhl] 	(3,0) node[below]{\small $\mu$}  .. controls (3,.5) and (3.5,.5) .. (3.5,1);
}
\ \boxtimes 1_{\rho_2'}
\]
where $\boxtimes$ means we put diagrams next to each other.

\begin{defn}\label{def:recdefstd}
We define recursively $(\BV_{\rho_1,\rho_2}^{t,\ell}, d_\BV)$ as 
\begin{align*}
(\BV_{\rho_1,\emptyset}^{t,0}, d_\BV) &:= (\muP_{(\rho_1,t)}, d_\mu),
\qquad
(\BV_{\rho_1, \rho_2 = (\ell',\rho_2')}^{t,0}, d_\BV) := (\BV_{(\rho_1,t),\rho_2'}^{0,\ell'}, d_\BV), 
\\
(\BV_{\rho_1,\rho_2}^{t,\ell}, d_\BV) &:=
 \cone\left(q^{\mu-2\ell+2} (\BV_{F(\rho_1), \rho_2}^{t,\ell-1}, d_\BV) \xrightarrow{\tau_{\rho_1, \rho_2}^{t,\ell}}  (\BV_{\rho_1, \rho_2}^{t+1,\ell-1}, d_\BV) \right). 
\end{align*}
for $\ell > 0$ and $\rho_2 \neq \emptyset$, and where $\tau_{\rho_1, \rho_2}^{t,\ell}$ defines a map of left $(\muT, d_\mu)$-modules for the same reasons as in the proof of \cref{lem:stdtaucom}.
\end{defn}

Note that we have $ (\muS_{\rho = (b_1, \rho')}, d_\dgS) \cong (\BV_{\emptyset, \rho'}^{b_1,0}, d_\BV) $. Moreover  
$[(\muS_\rho, d_\dgS)]$ (resp. $[(\BV_{\rho_1,\rho_2}^{t,\ell}, d_\BV) ]$) coincides with $\tilde v_\rho$ (resp. $v_{\rho_1, \rho_2}^{t,\ell}$) under the isomorphism of \cref{thm:isocat}. 

\begin{exe}\label{ex:stdmodule}
We take $\und \mu = (\mu_1, \mu_2)$. We have 
\begin{align*}
\muS_{(0,2)} \cong \BV_{\emptyset,(2)}^{0,0} = \BV_{(0),\emptyset}^{0,2} &=  \cone(q^{\mu_2-2}\BV_{(1),\emptyset}^{0,1} \xrightarrow{\tau_{(0),\emptyset}^{0,2}} \BV_{(0),\emptyset}^{1,1}),
\\
\BV_{(1),\emptyset}^{0,1} &= \cone( q^{\mu_2} \BV_{(2),\emptyset}^{0,0} = \muP_{(2,0)} \xrightarrow{\tau_{(1),\emptyset}^{0,1}} \BV^{1,0}_{(1),\emptyset} = \muP_{(1,1)}  ),
\\
\BV_{(0),\emptyset}^{1,1} &= \cone(  q^{\mu_2} \BV_{(1),\emptyset}^{1,0} = \muP_{(1,1)} \xrightarrow{\tau_{(0),\emptyset}^{1,1}} \BV_{(0),\emptyset}^{2,0} = \muP_{(0,2)} ),
\end{align*}
 which we can picture as
\[
\muS_{(0,2)} \cong 
\cone\left(
 q^{\mu_2-2}
\left(
\begin{tikzcd}[row sep=4ex]
q^{\mu_2}\ 
\tikzdiagh[xscale=.75]{-1ex}{
	\draw[pstdhl]	(0,0) node[below]{\tiny $\mu_1$} -- (0,1);
	\draw 	 	(.5,0) -- (.5,1);
	\draw		(1,0) -- (1,1);
	\draw[pstdhl]	(1.5,0) node[below]{\tiny $\mu_2$} -- (1.5,1);
	\filldraw [fill=white, draw=black] (-.15,.5) rectangle (1.65,1) node[midway]{\small $\muT$};
}
\ar{d}
\ar[dash, "{
\tikzdiag[xscale=.75,yscale=.5]{
	\draw[pstdhl]	(0,0) -- (0,1);
	\draw 	 	(.5,0) -- (.5,1);
	\draw		(1.5,0) .. controls (1.5,.5) and (1,.5) .. (1,1);
	\draw[pstdhl]	(1,0) .. controls (1,.5) and (1.5,.5) .. (1.5,1);
} \ 
}"' 
]{d}
\\
\phantom{q^{\mu_2}}\ 
\tikzdiagh[xscale=.75]{-1ex}{
	\draw[pstdhl]	(0,0) node[below]{\tiny $\mu_1$} -- (0,1);
	\draw 	 	(.5,0) -- (.5,1);
	\draw[pstdhl]	(1,0) node[below]{\tiny $\mu_2$} -- (1,1);
	\draw		(1.5,0) -- (1.5,1);
	\filldraw [fill=white, draw=black] (-.15,.5) rectangle (1.65,1) node[midway]{\small $\muT$};
}
\end{tikzcd}
\right)
\begin{tikzcd}[column sep=10ex]
{}
\ar{r}
\ar[dash]{r}{
\tikzdiag[xscale=.75,yscale=.5]{
	\draw[pstdhl]	(0,0) -- (0,1);
	\draw		(.5,0) .. controls (.5,.5) and (1,.5) .. (1,1);
	\draw 	 	(1.5,0)  .. controls (1.5,.5) and (.5,.5) .. (.5,1);
	\draw[pstdhl]	(1,0) .. controls (1,.5) and (1.5,.5) .. (1.5,1);
}
}
&
{}
\\
{}
\ar{r}
\ar[dash,swap]{r}{
\tikzdiag[xscale=.75,yscale=.5]{
	\draw[pstdhl]	(0,0) -- (0,1);
	\draw 	 	(1.5,0)  .. controls (1.5,.5) and (.5,.5) .. (.5,1);
	\draw		(1,0) .. controls (1,.5) and (1.5,.5) .. (1.5,1);
	\draw[pstdhl]	(.5,0) .. controls (.5,.5) and (1,.5) .. (1,1);
}
}
&
{}
\end{tikzcd}
\left(
\begin{tikzcd}[row sep=4ex]
q^{\mu_2}\ 
\tikzdiagh[xscale=.75]{-1ex}{
	\draw[pstdhl]	(0,0) node[below]{\tiny $\mu_1$} -- (0,1);
	\draw 	 	(.5,0) -- (.5,1);
	\draw[pstdhl]	(1,0) node[below]{\tiny $\mu_2$} -- (1,1);
	\draw		(1.5,0) -- (1.5,1);
	\filldraw [fill=white, draw=black] (-.15,.5) rectangle (1.65,1) node[midway]{\small $\muT$};
}
\ar{d}
\ar[dash, "\ {
\tikzdiag[xscale=.75,yscale=.5]{
	\draw[pstdhl]	(0,0) -- (0,1);
	\draw		(1,0) .. controls (1,.5) and (.5,.5) .. (.5,1);
	\draw[pstdhl]	(.5,0) .. controls (.5,.5) and (1,.5) .. (1,1);
	\draw 	 	(1.5,0) -- (1.5,1);
} \ 
}" 
]{d}
\\
\phantom{\lambda}\ 
\tikzdiagh[xscale=.75]{-1ex}{
	\draw[pstdhl]	(0,0) node[below]{\tiny $\mu_1$} -- (0,1);
	\draw[pstdhl]	(.5,0) node[below]{\tiny $\mu_2$} -- (.5,1);
	\draw		(1,0) -- (1,1);
	\draw		(1.5,0) -- (1.5,1);
	\filldraw [fill=white, draw=black] (-.15,.5) rectangle (1.65,1) node[midway]{\small $\muT$};
}
\end{tikzcd}
\right)
\right)
\]
As another example, take $\und \mu= (\mu_1,\mu_2,\mu_3)$ and we obtain
\begin{align*}
\muS_{(0,1,1)} \cong 
\BV_{\emptyset,(1,1)}^{0,0} =
\BV_{(0),(1)}^{0,1} &=
\cone(
q^{\mu_2} \BV_{(1),(1)}^{0,0}
\xrightarrow{\tau_{\emptyset,(1,1)}^{0,0} }
\BV_{(0),(1)}^{1,0}
),
\\
 \BV_{(1),(1)}^{0,0} = \BV_{(1,0),\emptyset}^{0,1} &= \cone(q^{\mu_3} \BV_{(1,1),\emptyset}^{0,0} = \muP_{(1,1,0)} \xrightarrow{\tau_{(1,0),\emptyset}^{0,1} } \BV_{(1,0),\emptyset}^{1,0} = \muP_{(1,0,1)} ),
 \\
 \BV_{(0),(1)}^{1,0} = \BV_{(0,1),\emptyset}^{0,1} &= \cone(q^{\mu_3} \BV_{(0,2),\emptyset}^{0,0} = \muP_{(0,2,0)} \xrightarrow{\tau_{(0,1),\emptyset}^{0,1}} \BV_{(0,1),\emptyset}^{1,0} = \muP_{(0,1,1)} ),
\end{align*}
which we picture as
\[
\muS_{(0,1,1)} \cong
\cone\left(
q^{\mu_2} \left(
\begin{tikzcd}[row sep=4ex]
q^{\mu_3}\ 
\tikzdiagh[xscale=.75]{-1ex}{
	\draw[pstdhl]	(0,0) node[below]{\tiny $\mu_1$} -- (0,1);
	\draw 	 	(.5,0) -- (.5,1);
	\draw[pstdhl]	(1,0) node[below]{\tiny $\mu_2$} -- (1,1);
	\draw		(1.5,0) -- (1.5,1);
	\draw[pstdhl]	(2,0) node[below]{\tiny $\mu_3$} -- (2,1);
	\filldraw [fill=white, draw=black] (-.15,.5) rectangle (2.15,1) node[midway]{\small $\muT$};
}
\ar{d}
\ar[dash, "{
\tikzdiag[xscale=.75,yscale=.5]{
	\draw[pstdhl]	(0,0) -- (0,1);
	\draw 	 	(.5,0) -- (.5,1);
	\draw[pstdhl]		(1,0) -- (1,1);
	\draw		(2,0) .. controls (2,.5) and (1.5,.5) .. (1.5,1);
	\draw[pstdhl]	(1.5,0) .. controls (1.5,.5) and (2,.5) .. (2,1);
} \ 
}"' 
]{d}
\\
\phantom{q^{\mu_3}}\ 
\tikzdiagh[xscale=.75]{-1ex}{
	\draw[pstdhl]	(0,0) node[below]{\tiny $\mu_1$} -- (0,1);
	\draw 	 	(.5,0) -- (.5,1);
	\draw[pstdhl]	(1,0) node[below]{\tiny $\mu_2$} -- (1,1);
	\draw[pstdhl]	(1.5,0)  node[below]{\tiny $\mu_3$}-- (1.5,1);
	\draw		(2,0) -- (2,1);
	\filldraw [fill=white, draw=black] (-.15,.5) rectangle (2.15,1) node[midway]{\small $\muT$};
}
\end{tikzcd}
\right)
\begin{tikzcd}[column sep=12ex]
{}
\ar{r}
\ar[dash]{r}{
\tikzdiag[xscale=.75,yscale=.5]{
	\draw[pstdhl]	(0,0) -- (0,1);
	\draw 	 	(1,0)  .. controls (1,.5) and (.5,.5) .. (.5,1);
	\draw		(1.5,0)  -- (1.5,1);
	\draw[pstdhl]	(.5,0) .. controls (.5,.5) and (1,.5) .. (1,1);
	\draw[pstdhl]	(2,0) -- (2,1);
}
}
&
{}
\\
{}
\ar{r}
\ar[dash,swap]{r}{
\tikzdiag[xscale=.75,yscale=.5]{
	\draw[pstdhl]	(0,0) -- (0,1);
	\draw 	 	(1,0)  .. controls (1,.5) and (.5,.5) .. (.5,1);
	\draw[pstdhl]		(1.5,0)  -- (1.5,1);
	\draw[pstdhl]	(.5,0) .. controls (.5,.5) and (1,.5) .. (1,1);
	\draw	(2,0) -- (2,1);
}
}
&
{}
\end{tikzcd}
\left(
\begin{tikzcd}[row sep=4ex]
q^{\mu_3}\ 
\tikzdiagh[xscale=.75]{-1ex}{
	\draw[pstdhl]	(0,0) node[below]{\tiny $\mu_1$} -- (0,1);
	\draw[pstdhl] 	(.5,0) node[below]{\tiny $\mu_2$} -- (.5,1);
	\draw		(1,0) -- (1,1);
	\draw		(1.5,0) -- (1.5,1);
	\draw[pstdhl]	(2,0) node[below]{\tiny $\mu_3$} -- (2,1);
	\filldraw [fill=white, draw=black] (-.15,.5) rectangle (2.15,1) node[midway]{\small $\muT$};
}
\ar{d}
\ar[dash, "{\ 
\tikzdiag[xscale=.75,yscale=.5]{
	\draw[pstdhl]	(0,0) -- (0,1);
	\draw[pstdhl] 	 	(.5,0) -- (.5,1);
	\draw		(1,0) -- (1,1);
	\draw		(2,0) .. controls (2,.5) and (1.5,.5) .. (1.5,1);
	\draw[pstdhl]	(1.5,0) .. controls (1.5,.5) and (2,.5) .. (2,1);
} \ 
}"
]{d}
\\
\phantom{\lambda}\ 
\tikzdiagh[xscale=.75]{-1ex}{
	\draw[pstdhl]	(0,0) node[below]{\tiny $\mu_1$} -- (0,1);
	\draw[pstdhl] 	(.5,0) node[below]{\tiny $\mu_2$} -- (.5,1);
	\draw		(1,0) -- (1,1);
	\draw[pstdhl]	(1.5,0) node[below]{\tiny $\mu_3$} -- (1.5,1);
	\draw		(2,0) -- (2,1);
	\filldraw [fill=white, draw=black] (-.15,.5) rectangle (2.15,1) node[midway]{\small $\muT$};
}
\end{tikzcd}
\right)
\right)
\]
\end{exe}

\begin{rem}
If $\und \mu$ contains only integral weights, then the underlying complex of the standard module is exact everywhere except in the last rightmost term. In this case we can replace it by the quotient of $\muP_{\rho}$ by the ideal given by diagrams with a black/colored crossing of the type:
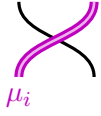
\[
\tikzdiag{
	\draw (1,0)  ..controls (1,.5) and (0,.5) .. (0,1);
	\draw[pstdhl] (0,0) node[below]{\small $\mu_i$}  ..controls (0,.5) and (1,.5) .. (1,1);
}
\]
This coincides up to quasi-isomorphism with the standard modules in \cite{webster} (viewed as dg-modules concentrated in homological and $\lambda$-degrees zero).
\end{rem}

\subsubsection{Preorder}

Inspired by~\cite{webster}, we say that there is an arrow $\rho \leftarrow \rho'$ for $\rho, \rho' \in \cP_b^r$ whenever there is some $1 \leq j \leq r$ such that $b_i = b_i'$ for all $i \neq j,j+1$ and $b_j = b_j' + 1$ and $b_{j+1} = b'_{j+1} - 1$. 
Consider the preorder on $ \cP_b^r$ given by $\rho \leq \rho'$ whenever there is a chain of arrows $\rho = \rho_0 \leftarrow \rho_1 \leftarrow \cdots \leftarrow \rho_t = \rho'$. Note that there is a maximal element given by $(0,0,\dots,b)$ and a minimal element given by $(b,\dots,0,0)$. 
If we think in terms of idempotents $1_\rho$, then $\rho \leq \rho'$ whenever we can obtain $1_{\rho'}$ from $1_{\rho}$ by sliding colored strands to the left. 

\begin{exe}
Writing the idempotent $1_\rho$ to picture the element $\rho$, we have the following arrows:
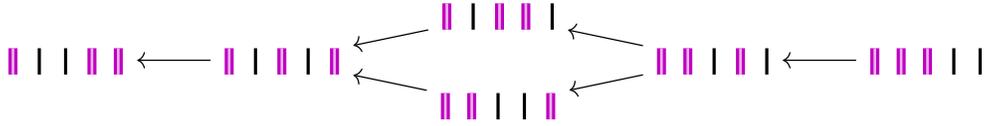
\[
\begin{tikzcd}[row sep=-1ex]
&&
\tikzdiag[scale=.35]{
	\draw[pstdhl] (0,0) -- (0,1);
	\draw		(1,0) -- (1,1);
	\draw[pstdhl] (2,0) -- (2,1);
	\draw[pstdhl] (3,0) -- (3,1);
	\draw	 	(4,0) -- (4,1);
}
\ar[leftarrow]{dr}
&&
\\
\tikzdiag[scale=.35]{
	\draw[pstdhl] (0,0) -- (0,1);
	\draw		(1,0) -- (1,1);
	\draw 		(2,0) -- (2,1);
	\draw[pstdhl] (3,0) -- (3,1);
	\draw[pstdhl] 	(4,0) -- (4,1);
}
\ar[leftarrow]{r}
&
\tikzdiag[scale=.35]{
	\draw[pstdhl] (0,0) -- (0,1);
	\draw		(1,0) -- (1,1);
	\draw[pstdhl] (2,0) -- (2,1);
	\draw	 	(3,0) -- (3,1);
	\draw[pstdhl]	(4,0) -- (4,1);
}
\ar[leftarrow]{ur}
\ar[leftarrow]{dr}
&&
\tikzdiag[scale=.35]{
	\draw[pstdhl] (0,0) -- (0,1);
	\draw[pstdhl]	(1,0) -- (1,1);
	\draw 		(2,0) -- (2,1);
	\draw[pstdhl] (3,0) -- (3,1);
	\draw	 	(4,0) -- (4,1);
}
\ar[leftarrow]{r}
&
\tikzdiag[scale=.35]{
	\draw[pstdhl] (0,0) -- (0,1);
	\draw[pstdhl]	(1,0) -- (1,1);
	\draw[pstdhl] (2,0) -- (2,1);
	\draw		(3,0) -- (3,1);
	\draw	 	(4,0) -- (4,1);
}
\\
&&
\tikzdiag[scale=.35]{
	\draw[pstdhl] (0,0) -- (0,1);
	\draw[pstdhl]	(1,0) -- (1,1);
	\draw 		(2,0) -- (2,1);
	\draw 		(3,0) -- (3,1);
	\draw[pstdhl] 	(4,0) -- (4,1);
}
\ar[leftarrow]{ur}
&&
\end{tikzcd}
\]
\end{exe}

\begin{prop}\label{prop:Pstrat}
The dg-module $(\muP_\rho, d_\mu)$ can be obtained as a mapping cone 
\[
(\muP_\rho, d_\mu) \cong \cone((\muS_\rho, d_\dgS)[-1] \rightarrow (\BQ_{< \rho}, d_\BQ)),
\]
where $(\BQ_{< \rho}, d_\BQ)$ is a finite iterated extension of shifted copies of elements in the set $\{ (\muS_{\rho'}, d_\dgS) | \rho' < \rho \}$.
\end{prop}

\begin{proof}
If $\rho = (b,0,\dots,0)$ is minimal, then $\muS_{\rho} \cong \muP_\rho$, and we are done by setting $\BQ_{> \rho} := 0$. Suppose by induction that the theorem is true for $\rho' < \rho$. 

We have an injection of $(\muT, d_\mu)$-modules
\[
f_\rho : (\muP_\rho, d_\mu) = (\muS_{\rho,\emptyset}, d_\dgS) \hookrightarrow (\muS_\rho, d_\dgS),
\]
and we define $(\BQ_{<\rho}, d_\BQ) := \cok f_\rho$, so that we get a distinguished triangle
\[
 (\muP_\rho, d_\mu) \rightarrow (\muS_\rho, d_\dgS) \rightarrow (\BQ_{<\rho}, d_\BQ) \rightarrow
\]
implying that
\[
(\muP_\rho, d_\mu) \cong \cone((\muS_\rho, d_\dgS)[-1] \rightarrow (\BQ_{< \rho}, d_\BQ)).
\]
 We observe that $\BQ_{<\rho} \cong  \ssbigoplus{\bj \subset J_\rho \\ \bj \neq \emptyset} \muP_{\rho_\bj}$, and $\rho_\bj < \rho$ for $\bj \neq \emptyset$. 
 Therefore, by induction hypothesis,  $(\BQ_{<\rho}, d_\BQ)$ is isomorphic to an iterated extension of various shifted $(\muS_{\rho''}, d_\dgS)$ with $\rho'' < \rho$. 
\end{proof}

\begin{cor}\label{cor:stdcblfgen}
The dg-category $\cD_{dg}^{cblf}(\muT_b, d_\mu)$ is c.b.l.f. generated by $\{(\muS_\rho, d_\dgS) | \rho \in \cP_b^{r, \und \mu}\}$. 
\end{cor}

\begin{proof}
 This is immediate by \cref{cor:muPgenerates} and \cref{prop:Pstrat}. 
\end{proof}

\subsection{Standardization functor}

We want to construct a standardization functor
\[
\stdFunct : \cD_{dg}(\dgT^{\mu_1}\otimes\cdots\otimes \dgT^{\mu_r}, d_{\mu_1} + \cdots + d_{\mu_r}) \rightarrow \cD_{dg}(\muT,d_\mu),
\]
such that it is exact and sends
\[
\dgP^{\mu_1}_{b_1} \otimes \cdots \otimes \dgP^{\mu_r}_{b_r} \mapsto \muS_\rho.
\]
In order to do this, we endow $\muS_\rho$ with a dg-bimodule structure. 

\subsubsection{Bimodule structure on $\muS_\rho$}
We start by defining the right action of $x \otimes 1 \otimes \cdots \otimes 1 \in \dgT^{\mu_1}_{b_1} \otimes \cdots \otimes  \dgT^{\mu_r}_{b_r}$ as gluing diagrams at the bottom of each summand $\muS_{\rho,\bj} \subset \muS_\rho$:
\[
\tikzdiagh[xscale=.5,yscale=1]{0}{
	\draw[pstdhl] (0,0) node[below]{\small $\mu_1$} -- (0,1);
	\draw	 	(1,0) -- (1,1);
	\node at(2,.1){\small $\cdots$};
	\node at(2,.85){\small $\cdots$};
	\draw	 	(3,0) -- (3,1);
	\filldraw [fill=white, draw=black] (-.25,.25) rectangle (3.25,.75) node[midway] { $x$};
}
\otimes 1 \otimes \cdots \otimes 1 \ : 
\tikzdiagh[xscale=.5,yscale=1]{0}{
	\draw[pstdhl] (0,0) node[below]{\small $\mu_1$} -- (0,.5);
	\draw	 	(1,0) -- (1,.5);
	\node at(2,.25){\small $\cdots$};
	\draw	 	(3,0) -- (3,.5);
	\draw[pstdhl] (4,0) node[below]{\small $\mu_2$} -- (4,.5);
	\node at(6,.25){\small $\cdots$};
	%
	\filldraw [fill=white, draw=black] (-.25,.5) rectangle (7.25,1.25) node[midway] { $\muS_{\rho,\bj}$};
}
\ \mapsto \
\tikzdiagh[xscale=.5,yscale=1]{0}{
	\draw[pstdhl] (0,-.5) node[below]{\small $\mu_1$} -- (0,.5);
	\draw	 	(1,-.5) -- (1,.5);
	\node at(2,0){\small $\cdots$};
	\draw	 	(3,-.5) -- (3,.5);
	\draw	 	(4,-.5) -- (4,.5);
	\node at(5,0){\small $\cdots$};
	\draw	 	(6,-.5) -- (6,.5);
	\draw[pstdhl] (7,-.5) node[below]{\small $\mu_2$} -- (7,.5);
	\node at(9,0){\small $\cdots$};
	%
	\node at(2,.1-.5){\small $\cdots$};
	\node at(2,.85-.5){\small $\cdots$};
	\filldraw [fill=white, draw=black] (-.25,-.25) rectangle (3.25,.25) node[midway] { $x$};
	\filldraw [fill=white, draw=black] (-.25,.5) rectangle (10.25,1.25) node[midway] { $\muS_{\rho,\bj}$};
}
\]
Since the differential $d_{\dgS}$ only touches the strands on the right, except for the $d_\mu$ part which is already in $(\dgT^{\mu_1}_{b_1}, d_\mu)$, the action of $\dgT_{b_1}^{\mu_1}$ respects the graded Leibniz rule.

Since we want to define a bimodule structure, it is enough to define the right action on each generating elements of $\muS_\rho$ as left-module. We fix $\ell$ and we describe below the right action of $\dgT^{\mu_\ell}_{b_\ell}$ on $\muS_\rho$. 

We need some preparation. For $j \in J_{\ell,\rho}$ we define
\[
\omega_j := \ 
\tikzdiagl{
	\draw (.5, 0) .. controls (.5,.5) and (1,.5) .. (1, 1);
	\draw (1.5, 0) .. controls (1.5,.5) and (2,.5) .. (2, 1);
	\node at(1,.1) {\small $\dots$}; \node at(1.5,.9) {\small $\dots$};
	\tikzbrace{.5}{1.5}{-.1}{\small $j-1$};
	\draw (2,0) .. controls (2,.5) .. (0,.75) .. controls (.5,.825) .. (.5,1);
	\draw (2.5,0) -- (2.5,1);
	\draw (3.5,0) -- (3.5,1);
	\node at(3,.5) {\small $\dots$};
	\draw [pstdhl] (0,0) node[below]{\small $\mu_\ell$} -- (0,1) node[pos=.75, nail]{};
}
\ \in \dgT^{\mu_\ell}_{b_\ell}.
\]
and for $\bj_\ell = \{\bj_{\ell,1}, \dots, \bj_{\ell, |\bj|_\ell} \} :=  \bj \cap J_{\ell,\rho}$ with $\bj_{\ell,1} < \cdots < \bj_{\ell,|\bj_\ell|}$ we put
\[
\omega_{\bj_\ell} := \omega_{ \bj_{\ell,|\bj_\ell|}} \cdots \omega_{\bj_{\ell, 1}}.
\]
In terms of pictures, we can draw this as
\begin{equation}\label{eq:omegapicture}
\omega_{\bj_\ell} = 
\tikzdiag{
	\draw (.5,0) .. controls (.5,.5) and (1,.5) .. (1,1)
		.. controls (1,1.5) and (1.5,1.5) .. (1.5,2);
	\draw [dashed] (1.5,2) .. controls (1.5,3) and (3,3) .. (3,4);
	\draw (3,4) .. controls (3,4.5) and (3.5,4.5) .. (3.5,5);
	\draw (1.5,0) .. controls (1.5,.5) and (2,.5) .. (2,1)
		.. controls (2,1.5) and (2.5,1.5) .. (2.5,2);
	\draw [dashed] (2.5,2) .. controls (2.5,3) and (4,3) .. (4,4);
	\draw (4,4) .. controls (4,4.5) and (4.5,4.5) .. (4.5,5);
	\node at(1,.1) {\small $\dots$}; 
	\node at(1.5,1) {\small $\dots$};  
	\node at(2,2) {\small $\dots$};
	\node at(3.5,4) {\small $\dots$};
	\node at(4,4.9) {\small $\dots$};
	\tikzbrace{.5}{1.5}{-.1}{\tiny $\bj_{\ell,1}-1$};
	\draw (2,0) .. controls (2,.25) .. (0,.5) .. controls (.5,.75) .. (.5,1)
		.. controls (.5,1.5) and (1,1.5) .. (1,2);
	\draw [dashed] (1,2) .. controls (1,3) and (2.5,3) .. (2.5,4);
	\draw (2.5,4) .. controls (2.5,4.5) and (3,4.5) .. (3,5);
	\draw (2.5,0) -- (2.5,1) .. controls (2.5,1.5) and (3,1.5) .. (3,2);
	\draw [dashed] (3,2) .. controls (3,3) and (4.5,3) .. (4.5,4);
	\draw (4.5,4) .. controls (4.5,4.5) and (5,4.5) .. (5,5);
	\draw (3.5,0) -- (3.5,1) .. controls (3.5,1.5) and (4,1.5) .. (4,2);
	\draw [dashed] (4,2) .. controls (4,3) and (5.5,3) .. (5.5,4);
	\draw (5.5,4) .. controls (5.5,4.5) and (6,4.5) .. (6,5);
	\node at(3,.1) {\small $\dots$}; 
	\node at(3,1) {\small $\dots$};  
	\node at(3.5,2) {\small $\dots$};
	\node at(5,4) {\small $\dots$};
	\node at(5.5,4.9) {\small $\dots$};
	\tikzbrace{2.5}{3.5}{-.1}{\tiny $\bj_{\ell,2}-\bj_{\ell,1}-1$};
	\draw  (4,0) -- (4,1) .. controls (4,1.25) .. (0,1.5) .. controls (.5,1.75) .. (.5,2);
	\draw [dashed] (.5,2) .. controls (.5,3) and (2,3) .. (2,4);
	\draw (2,4) .. controls (2,4.5) and (2.5,4.5) .. (2.5,5);
	\node at(5,1){\dots};
	\node at(1.5,4.9){\dots};
	\draw (6,0) -- (6,4)
		.. controls (6,4.5) and (6.5,4.5) .. (6.5,5);
	\draw (7,0) -- (7,4)
		.. controls (7,4.5) and (7.5,4.5) .. (7.5,5);
	\node at(6.5,.1){\small $\dots$};
	\node at(6.5,1){\small $\dots$};
	\node at(6.5,2){\small $\dots$};
	\node at(6.5,4){\small $\dots$};
	\node at(7,4.9){\small $\dots$};
	\tikzbrace{6}{7}{-.1}{\tiny $\bj_{\ell,|\bj_\ell|}-\bj_{\ell,|\bj_\ell|-1}-1$};
	\draw (7.5,0) -- (7.5,4)
		.. controls (7.5,4.25) .. (0,4.5) .. controls (.5,4.75) .. (.5,5);
	\draw (8,0) -- (8,5);
	\draw (9,0) -- (9,5);
	\node at(8.5,.1){\small $\dots$};
	\node at(8.5,1){\small $\dots$};
	\node at(8.5,2){\small $\dots$};
	\node at(8.5,4){\small $\dots$};
	\node at(8.5,4.9){\small $\dots$};
	\draw [pstdhl] (0,0) node[below]{\small $\mu_\ell$} 
		-- (0,1) node[pos=.5,nail]{} 
		-- (0,2) node[pos=.5,nail]{} 
		-- (0,4)
		-- (0,5) node[pos=.5,nail]{};
	\tikzbraceop{.5}{3}{5.1}{\small $|\bj_\ell|$};
	\tikzbraceop{3.5}{9}{5.1}{\small $b_\ell - |\bj_\ell|$};
}
\end{equation}
By \cref{thm:Tbasis} we have that $\dgT^{\mu_\ell}_{b_\ell}$ decomposes as a graded $\Bbbk$-module as
\begin{equation}\label{eq:dgnhdecomp}
\tikzdiag{
	\draw[pstdhl] (0,0) node[below]{\small $\mu_\ell$} -- (0,1);
	\draw	 	(.5,0) -- (.5,1);
	\node at(1,.1){\small $\cdots$};
	\draw	 	(1.5,0) -- (1.5,1);
	\filldraw [fill=white, draw=black] (-.15,.25) rectangle (1.65,1) node[midway] { $\dgT^{\mu_\ell}_{b_\ell}$};
}
\ \cong
\bigoplus_{\bj_\ell \subset J_{\ell, \rho} }
\tikzdiag{
	\draw[pstdhl] (0,0)  node[below]{\small $\mu_\ell$} -- (0,2);
	\draw	 	(.5,0) -- (.5,2);
	\node at(1,.1){\small $\cdots$};
	\draw	 	(1.5,0) -- (1.5,2);
	\filldraw [fill=white, draw=black] (-.15,.25) rectangle (1.65,1) node[midway] { $\omega_{\bj_\ell}$};
	\filldraw [fill=white, draw=black] (.35,1.25) rectangle (1.65,2) node[midway] { $\nh_{b_\ell}$};
}
\end{equation}
where $\nh_{b_\ell}$ is the nilHecke algebra on $b_\ell$ strands, that is the diagrammatic algebra on $b_\ell$ black strands with dots subject to the relations in \cref{eq:nhR2andR3} and \cref{eq:nhdotslide}.

\begin{exe}
We have
\[
\tikzdiag{
	\draw[pstdhl] (0,0) node[below]{\small $\mu_\ell$} -- (0,1);
	\draw	 	(.5,0) -- (.5,1);
	\draw	 	(1,0) -- (1,1);
	\filldraw [fill=white, draw=black] (-.15,.25) rectangle (1.15,1) node[midway] { $\dgT^{\mu_\ell}_{2}$};
}
\ \cong \ 
\tikzdiag{
	\draw[pstdhl] (0,0) node[below]{\small $\mu_\ell$} -- (0,2);
	\draw	 	(.5,0) -- (.5,2);
	\draw	 	(1,0) -- (1,2);
	\filldraw [fill=white, draw=black] (.3,1.25) rectangle (1.2,2) node[midway] { $\nh_{2}$};
}
\ \oplus \ 
\tikzdiag{
	\draw	 	(.5,0) .. controls (.5,.4) .. (0,.5) .. controls (.5,.6) .. (.5,1) -- (.5,2);
	\draw	 	(1,0) -- (1,2);
	\filldraw [fill=white, draw=black] (.3,1.25) rectangle (1.2,2) node[midway] { $\nh_{2}$};
	\draw[pstdhl] (0,0) node[below]{\small $\mu_\ell$} -- (0,2) node[pos=.25,nail]{};
}
\ \oplus \ 
\tikzdiag{
	\draw	 	(.5,0) .. controls (.5,.5) and (1,.5) .. (1,1) -- (1,2);
	\draw	 	(1,0)  .. controls (1,.4) .. (0,.5) .. controls (.5,.6) .. (.5,1) -- (.5,2);
	\filldraw [fill=white, draw=black] (.3,1.25) rectangle (1.2,2) node[midway] { $\nh_{2}$};
	\draw[pstdhl] (0,0) node[below]{\small $\mu_\ell$} -- (0,2) node[pos=.25,nail]{};
}
\ \oplus \ 
\tikzdiag{
	\draw	 	(.5,0) .. controls (.5,.2) .. (0,.4) .. controls (1,.8) .. (1,1) -- (1,2);
	\draw	 	(1,0)  .. controls (1,.5) .. (0,.8) .. controls (.5,.9) .. (.5,1.2) -- (.5,2);
	\filldraw [fill=white, draw=black] (.3,1.25) rectangle (1.2,2) node[midway] { $\nh_{2}$};
	\draw[pstdhl] (0,0) node[below]{\small $\mu_\ell$} -- (0,2) node[pos=.2,nail]{}  node[pos=.4,nail]{};
}
\]
\end{exe}

To define the action of $x \in \dgT^{\mu_\ell}_{b_\ell}$ on 
\[
1_{\rho_\bj} = 
\tikzdiagl{
	\node at(-1, .5){\small $\dots$};
	\draw (0,0) -- (0,1);
	\draw (1,0) -- (1,1);
	\node at(.5,.5) {\small $\dots$};
	\tikzbrace{0}{1}{-.1}{\small $|\bj_\ell|$};
	\draw [pstdhl] (1.5,0) node[below]{\small $\mu_\ell$} -- (1.5,1);
	\draw (2,0) -- (2,1);
	\draw (3,0) -- (3,1);
	\node at(2.5,.5) {\small $\dots$};
	\tikzbrace{2}{3}{-.1}{\small $b_\ell - |\bj_\ell|$};
	\node at(4, .5){\small $\dots$};
}
\in \muS_{\rho,\bj},
\]
we consider the collection of unique $x_{\bj'_\ell} \in \nh_{b_\ell}$ such that
\begin{equation}\label{eq:omegaxdecomp}
\omega_{\bj_\ell} x = \sum_{\bj'_\ell} x_{\bj'_\ell} \omega_{\bj'_\ell}, 
\end{equation}
given by the decomposition in \cref{eq:dgnhdecomp}. Note that $x_{\bj'_\ell} = 0$ whenever $|\bj'_\ell| < |\bj_\ell|$.

\begin{lem}\label{lem:omegaexchange}
We have 
\[
\omega_i \omega_j = 
\begin{cases}
0, & \text{if $i = 1$}, \\
-\tau_1 \omega_j \omega_{i-1} & \text{if $i \leq j$ and $i > 1$}.
\end{cases}
\]
\end{lem}

\begin{proof}
This is a straightforward computation using the nilHecke relations in \cref{eq:nhR2andR3} and \cref{eq:nhdotslide} together with the nail relations in \cref{eq:nailsrel}. We leave the details to the reader. 
\end{proof}

\begin{lem}\label{lem:xjdecomp}
In \cref{eq:omegaxdecomp}, we have that 
\[
x_{\bj'_\ell} = x_{\bj'_\ell}^1 \boxtimes x_{\bj'_\ell}^2
 \in 
\tikzdiagl[xscale=1.25]{
	\draw [pstdhl] (0,0) node[below]{\small $\mu_\ell$} -- (0,1);
	\draw (.5,0) -- (.5,1);
	\draw (1.5,0) -- (1.5,1);
	\node at(1,.1) {\small $\dots$}; 
	\filldraw [fill=white, draw=black] (.35,.25) rectangle (1.65,1) node[midway] { $\nh_{|\bj'_\ell|}$};
	\draw (2,0) -- (2,1);
	\draw (3,0) -- (3,1);
	\node at(2.5,.1) {\small $\dots$}; 
	\filldraw [fill=white, draw=black] (1.85,.25) rectangle (3.15,1) node[midway] { $\nh_{b_\ell - |\bj'_\ell|}$};
}
\]
for some $x_{\bj'_\ell}^1 \in \nh_{|\bj'_\ell|}$ and $x_{\bj'_\ell}^2 \in \nh_{b_\ell - |\bj'_\ell|}$.
\end{lem}

\begin{proof}
We need to investigate how $\omega_{\bj_\ell} x$ decomposes when $x$ is either a dot, a crossing or a nail. 

First assume that $x$ is a nail. By looking at the diagram in \cref{eq:omegapicture}, we observe that adding a nail at the bottom gives $0$ by  \cref{eq:nailsrel} if $1 \in \bj_\ell$, and we get $\omega_\bj x = \omega_{\bj \sqcup \{1 \in J_{\ell, \rho}\}}$ otherwise. Thus, $x_{\bj'_\ell} $ is either $0$ or $1_{\bj'_\ell}$. 

Suppose $x$ is a crossing between the $i$-th and $(i+1)$-th black strands. Looking at the diagram in \cref{eq:omegapicture}, if the crossing is below one of the horizontal brackets at the bottom, then we can use the braid move in \cref{eq:nhR2andR3} to slide it to the top right, so that $x^1_{\bj_\ell} = 1$ and $x^2_{\bj_\ell}$ is a crossing. If $i+1 = \bj_{\ell,t}$ for some $t$, then $\omega_\bj x = 0$ by \cref{eq:nhR2andR3}. For the remaining cases, suppose $i = \bj_{\ell,t}$ for some $t$. If $i+1 = \bj_{\ell,t+1}$ then we can use the braid move to bring the crossing to the levels of nail, slide it through the nails using \cref{eq:nailsrel}, and finally slide it to the top. Thus we obtain that $x^1_{\bj_\ell}$ is a crossing and $x^2_{\bj_\ell} = 1$. Otherwise, $\omega_{\bj_\ell} x = \omega_{\bj_\ell \setminus \{t\} \sqcup \{t+1\}}$, and thus $x_{\bj'_\ell} $ is either $0$ or $1_{\bj'_\ell}$. 

Finally, suppose $x$ is a dot on the $i$-th black strand. We can slide the dot to the top using the nilHecke relations in \cref{eq:nhdotslide} at the cost of adding diagrams with one fewer crossings. Therefore, we consider what happens whenever we remove a crossing from the diagram in \cref{eq:omegapicture}. 
If we remove a crossing situated in the upper left triangle below the bracket $|\bj_\ell|$, then we obtain zero because we would have two nails on the same black strand. 
If we remove a crossing elsewhere, we can first slide to the top right all crossings at the bottom right of the crossing we removed using the braid move in \cref{eq:nhR2andR3}, giving an element $x^2_{\bj_\ell'}$. Then we observe that having removed a crossing turned some $\omega_t$ to $\omega_{t'}$ with $t' < t$. Thus we use \cref{lem:omegaexchange} to reorder the $\omega_t$'s, at the cost of adding crossings that can be slided to the top left part, giving the elements $x^1_{\bj_\ell'}$. In particular, we never obtain a crossing at the top between the $|\bj_\ell|$-th and $(|\bj_\ell|+1)$-th black strands, concluding the proof.
\end{proof}

Because of \cref{lem:xjdecomp}, we can define
\[
1_{\rho_\bj} \bullet x := 
(-1)^{\deg_h(x) (\sum_{t > \ell} |\bj_t|)}
\sum_{\bj'_\ell} 
\ 
\tikzdiag{
	\node at(-2.5, 1){\small $\dots$};
	%
	\draw (-1.5,0) -- (-1.5,2);
	\draw (-.5,0) -- (-.5,2);
	\node at(-1,.1) {\small $\dots$};
	\node at(-1,1.1) {\small $\dots$};
	\node at(-1,1.9) {\small $\dots$};
	%
	%
	\draw (0,0) -- (0,1).. controls (0,1.5) and (.5,1.5) .. (.5,2);
	\draw (1,0) -- (1,1)  .. controls (1,1.5) and (1.5,1.5) .. (1.5,2);
	\tikzbraceop{.5}{1.5}{2}{$|\bj'_\ell| - |\bj_\ell|$};
	\node at(.5,.1) {\small $\dots$};
	\node at(.5,1.1) {\small $\dots$};
	\node at(1,1.9) {\small $\dots$};
	\filldraw [fill=white, draw=black] (-1.5-.15,.25) rectangle (1.15,1) node[midway] { $x^1_{\bj'_\ell}$};
	%
	\draw [pstdhl] (1.5,0) node[below]{\small $\mu_\ell$} -- (1.5,1) .. controls (1.5,1.5) and (0,1.5) ..  (0,2);
	\draw (2,0) -- (2,2);
	\draw (3,0) -- (3,2);
	\node at(2.5,.1) {\small $\dots$};
	\node at(2.5,1.1) {\small $\dots$};
	\node at(2.5,1.9) {\small $\dots$};
	\filldraw [fill=white, draw=black] (2-.15,.25) rectangle (3+.15,1) node[midway] { $x^2_{\bj'_\ell}$};
	\node at(4, 1){\small $\dots$};
}
\ 
\in \muS_{\rho,\bj'}, 
\]
where $\bj'$ is obtain from $\bj$ by replacing $\bj_\ell$ with $\bj'_\ell$.  
Note that this is well-defined because of the isomorphism in \cref{eq:dgnhdecomp}. 
Moreover, it is homogeneous because $q^{|\bj_\ell| \mu_\ell + \sum_{t \in \bj_\ell} \mu_\ell -2t + 2} h^{|\bj_\ell|} = \deg(\omega_{\bj_\ell}) $.

\begin{exe}
Take $b_1 \geq 0,  b_2 = 3$. We have 
\begin{align*}
\tikzdiagl[xscale=.5]{
	\draw (1,0) .. controls (1,1) and (3,1) .. (3,2);
	\draw (2,0) .. controls (2,.6) .. (0,.8) .. controls (2,1) .. (2,2);
	\draw (3,0) -- (3,1) node[near start, tikzdot]{} .. controls (3,1.4) .. (0,1.6) .. controls (1,1.8) .. (1,2);
	\draw[pstdhl] (0,0) node[below]{\small $\mu_2$} -- (0,2) node[pos=.8, nail]{} node[pos=.4, nail]{};
}
=
\tikzdiagl[xscale=.5]{
	\draw (1,0) .. controls (1,1) and (3,1) .. (3,2);
	\draw (2,0) .. controls (2,.6) .. (0,.8) .. controls (2,1) .. (2,2);
	\draw (3,0) -- (3,1) .. controls (3,1.4) .. (0,1.6) .. controls (1,1.8) .. (1,2)  node[pos=.5, tikzdot]{};
	\draw[pstdhl] (0,0)  node[below]{\small $\mu_2$}  -- (0,2) node[pos=.8, nail]{} node[pos=.4, nail]{};
}
+ 
\tikzdiagl[xscale=.5,yscale=-1]{
	\draw (1,0) .. controls (1,.4) and (2,.4) .. (2,.8) .. controls (2,1.2) .. (0,1.6) .. controls (1,1.8) .. (1,2);
	\draw (2,0) .. controls (2,.6) .. (0,.8) .. controls (2,1) .. (2,2);
	\draw (3,0) -- (3,2);
	\draw[pstdhl] (0,0) -- (0,2)   node[below]{\small $\mu_2$} node[pos=.8, nail]{} node[pos=.4, nail]{};
}
\end{align*}
and thus we obtain
\[
1_{\rho_{(2,3)}} \bullet \left(1_{(b_1)} \otimes\tikzdiagl[xscale=1]{
	\draw[pstdhl] (0,0)  node[below]{\small $\mu_1$}  -- (0,1);
	\draw (.5,0) -- (.5,1);
	\draw (1,0) -- (1,1);
	\draw (1.5,0) -- (1.5,1) node[midway,tikzdot]{};
} \right) = 
\tikzdiagl[xscale=.5]{
	\draw[pstdhl] (0,0)  node[below]{\small $\mu_1$}  -- (0,1);
	\draw		(1,0) -- (1,1);
	\node at (1.75,.5) {\small $\dots$};
	\draw		(2.5,0) -- (2.5,1);
	\draw		(3.5,0) -- (3.5,1) node[midway, tikzdot]{};
	\draw		(4.5,0) -- (4.5,1);
	\draw[pstdhl]	(5.5,0)  node[below]{\small $\mu_2$}  -- (5.5,1);
	\draw		(6.5,0) -- (6.5,1);
	\tikzbrace{1}{2.5}{-.15}{\small $b_1$};
}
\ .1_{\rho_{(2,3)}} + 
\tikzdiagl[xscale=.5]{
	\draw[pstdhl] (0,0)  node[below]{\small $\mu_1$}  -- (0,1);
	\draw		(1,0) -- (1,1);
	\node at (1.75,.5) {\small $\dots$};
	\draw		(2.5,0) -- (2.5,1);
	\draw		(3.5,0) .. controls(3.5,.5) and (4.5,.5) .. (4.5,1);
	\draw		(4.5,0) .. controls(4.5,.5) and (3.5,.5) .. (3.5,1);
	\draw[pstdhl]	(5.5,0)  node[below]{\small $\mu_2$}  -- (5.5,1);
	\draw		(6.5,0) -- (6.5,1);
	\tikzbrace{1}{2.5}{-.15}{\small $b_1$};
}
\ .1_{\rho_{(1,2)}}.
\]
\end{exe}

\begin{lem}\label{lem:stdNdotsaction}
We have
\begin{align*}
&1_{\rho_{\bj}} \bullet \left( 1 \otimes
\tikzdiagl[xscale=1]{
	\draw[pstdhl] (0,0)  node[below]{\small $\mu_\ell$}  -- (0,1);
	\draw (.5,0) -- (.5,1)  node[midway,tikzdot]{} node[midway, xshift=1.5ex, yshift=.75ex]{\small $N$};
	\draw (1,0) -- (1,1);
	\node at(1.5,.5){\small $\dots$};
	\draw (2,0) -- (2,1);
 }
 \otimes 1
 \right) \\
& = 
\begin{cases}
   \tikzdiagl[yscale=1, xscale=.75]{
	\draw (1,0) -- (1,1)  node[midway,tikzdot]{} node[midway,xshift=2ex,yshift=.75ex]{\small $N$};
	\draw[pstdhl] (0,0)node[below]{\small $\mu_\ell$}-- (0,1)  ;
} \in \muS_{\rho,\bj},
& \text{if $1 \in \bj_\ell$}, \\
  \tikzdiagl[yscale=1, xscale=.75]{
	\draw (1,0) -- (1,1)  node[midway,tikzdot]{} node[midway,xshift=2ex,yshift=.75ex]{\small $N$};
	\draw[pstdhl] (0,0)node[below]{\small $\mu_\ell$}-- (0,1)  ;
} \in \muS_{\rho,\bj}
-
(-1)^{|\bj_\ell|}
\sssum{\bj' = \\ \bj \setminus \{t \in \bj_\ell\} \sqcup \{1\}} \ 
\alpha_{\bj,t}
\sssum{u+v=\\ N -1} \ 
\tikzdiagl[xscale=.9]{
	\draw (-.5,0) .. controls (-.5,.5) and (-2,.5) .. (-2,1) node[pos=.2,tikzdot]{} node[pos=.2,xshift=1.5ex,yshift=.75ex]{\small $u$};
	%
	\draw (-2,0) .. controls (-2,.5) and (-1.5,.5) .. (-1.5,1);
	\draw (-1,0) .. controls (-1,.5) and (-.5,.5) .. (-.5,1);
	\node at(-1.5,.1) {\small $\dots$};
	\node at(-1,.9) {\small $\dots$};
	\draw (.5,0) .. controls (.5,.5) and (1,.5) .. (1,1);
	\draw (1.5,0) .. controls (1.5,.5) and (2,.5) .. (2,1);
	\node at(1,.1) {\small $\dots$};
	\node at(1.5,.9) {\small $\dots$};
	\draw (2,0) .. controls (2,.5) and (.5,.5) .. (.5,1)  node[pos=.8,tikzdot]{} node[pos=.8,xshift=-1.5ex,yshift=-.75ex]{\small $v$};
	\draw [pstdhl] (0,0) node[below]{\small $\mu_\ell$} -- (0,1);
}
\in \muS_{\rho,\bj'},
& \text{if $1 \notin \bj_\ell$}.
 \end{cases}
\end{align*}
\end{lem}

\begin{proof}
The case $1 \in \bj_{\ell}$ is immediate by looking at \cref{eq:omegapicture} and observing that sliding the dots to the top using \cref{eq:nhdotslide} produces diagrams with fewer crossings in the top left region, so that they all have two nails on a single black strand, and are zero. 

The case $1 \notin  \bj_{\ell}$ follows immediately from \cref{prop:Ndotsslideovernails}.
\end{proof}

\begin{rem}\label{rem:diffomegaj}
Using the diagrams $\omega_{\bj_\ell}$, there is a convenient way to write how $d_{\dgS, \bj}$ acts on $\muS_{\rho,\bj}$. For each $\ell$, there is a differential $d_0$ (not preserving the degree) on $\dgT^{\mu_\ell}_{b_\ell}$ given by 
\[
d_{0}\left(
\tikzdiag[xscale=2]{
	 \draw (.5,-.5) .. controls (.5,-.25) .. (0,0) .. controls (.5,.25) .. (.5,.5);
          \draw[pstdhl] (0,-.5) node[below]{\small $\mu_\ell$}-- (0,.5) node [midway,nail]{};
  }
  \right)
  :=
  \tikzdiag[xscale=2]{
	 \draw (.5,-.5) -- (.5,.5);
          \draw[pstdhl] (0,-.5) node[below]{\small $\mu_\ell$}-- (0,.5);
  }
\]
and $d_0$ is zero on the other generators (note that it coincides with $d_0  = d_\mu$ for $\und \mu = (0)$).
Let $\omega_\bj := \omega_{\bj_2} \otimes \cdots \otimes \omega_{\bj_r}$ and we extend $d_0$ by the graded Leibniz rule to the tensor product $\dgT^{\mu_2}_{b_2} \otimes \cdots \otimes \dgT^{\mu_r}_{b_r}$. 
  By the decomposition in \cref{eq:dgnhdecomp} we have 
\[
d_0(\omega_{\bj}) = \sssum{\bj' = \bj \setminus \{b'\} } y_{b'} \omega_{\bj'},
\]
where 
\[
y_{b'} =  \pm
1 \otimes \ 
\tikzdiagh{0}{
	\draw [pstdhl] (0,0)   node[below]{\small $\mu_\ell$}  -- (0,1);
	\draw (.5,0) -- (.5,1);
	\node at(1,.5) {$\dots$};
	\draw (1.5,0) -- (1.5,1);
	\draw 		(2,0) .. controls (2,.5) and (2.5,.5) .. (2.5,1);
	\node at (2.5,.15){$\dots$};
	\node at (3,.85){$\dots$};
	\draw 		(3,0)  .. controls (3,.5) and (3.5,.5) .. (3.5,1);
	%
	\draw		(3.5,0) .. controls (3.5,.5) and (2,.5)  .. (2,1);
	\draw (4,0) -- (4,1);
	\node at(4.5,.5) {$\dots$};
	\draw (5,0) -- (5,1);
	\tikzbrace{2}{3}{-.15}{\small $b'-1$};
	%
}
\ \otimes 1
\]
for $b' \in \bj_\ell$. 
Then if we define 
\begin{align*}
\overline{\omega_{\bj'}}& := 1_{\rho_{\bj'}},
&
\overline{y_{b'}} &:= \pm \ 
\tikzdiag{
	\node at(-.75,.5) {$\dots$};
	\draw (0,0) -- (0,1);
	\node at(.5,.5) {$\dots$};
	\draw (1,0) -- (1,1);
	\draw 		(1.5,0) .. controls (1.5,.5) and (2,.5) .. (2,1);
	\node at (2,.15){$\dots$};
	\node at (2.5,.85){$\dots$};
	\draw 		(2.5,0)  .. controls (2.5,.5) and (3,.5) .. (3,1);
	\draw 		(3.5,0)  .. controls (3.5,.5) and (4,.5) .. (4,1);
	\node at (4,.15){$\dots$};
	\node at (4.5,.85){$\dots$};
	\draw 		(4.5,0)  .. controls (4.5,.5) and (5,.5) .. (5,1);
	%
	\draw		(5,0) .. controls (5,.5) and (1.5,.5)  .. (1.5,1);
	\tikzbrace{1.5}{2.5}{-.1}{\small $p_1$};
	\tikzbrace{3.5}{4.5}{-.1}{\small $p_2$};
	\tikzbraceop{0}{3}{1.1}{\small $|\bj_\ell|$};
	\draw (5.5,0) -- (5.5,1);
	\node at(6,.5) {$\dots$};
	\draw (6.5,0) -- (6.5,1);
	\node at(7.25,.5) {$\dots$};
	%
	%
	\draw[pstdhl] 	(3,0) node[below]{\small $\mu_\ell$}  .. controls (3,.5) and (3.5,.5) .. (3.5,1);
}
\end{align*}
where $p_1 + p_2 = b'-1$, 
we have 
\[
d_{\dgS}(1_{\rho_\bj}) = \overline{d_0(\omega_\bj)}.
\]
For example consider $\und \mu = (\mu_1, \mu_2)$ and $\rho = (0, 2)$,  and we compute
\begin{align*}
d_0( \omega_{\{1,2\}}) &= d_0\left(
\tikzdiagl{
	\draw	 	(.5,0) .. controls (.5,.125) .. (0,.25) .. controls (1,.625) .. (1,1);
	\draw	 	(1,0)  .. controls (1,.375) .. (0,.75) .. controls (.5,.875) .. (.5,1);
	\draw[pstdhl] (0,0) node[below]{\small $\mu_2$} -- (0,1) node[pos=.25,nail]{}  node[pos=.75,nail]{};
}
\right)
= \ 
\tikzdiagl[yscale=-1]{
	\draw	 	(.5,0) .. controls (.5,.5) and (1,.5) .. (1,1);
	\draw	 	(1,0)  .. controls (1,.4) .. (0,.5) .. controls (.5,.6) .. (.5,1);
	\draw[pstdhl] (0,0) -- (0,1) node[pos=.5,nail]{}  node[below]{\small $\mu_2$};
}
\ - \ 
\tikzdiagl{
	\draw	 	(.5,0) .. controls (.5,.5) and (1,.5) .. (1,1);
	\draw	 	(1,0)  .. controls (1,.4) .. (0,.5) .. controls (.5,.6) .. (.5,1);
	\draw[pstdhl] (0,0) node[below]{\small $\mu_2$} -- (0,1) node[pos=.5,nail]{};
}
=
\tikzdiagl{
	\draw	 	(.5,0) .. controls (.5,.5) and (1,.5) .. (1,1);
	\draw	 	(1,0)  .. controls (1,.5) and (.5,.5) ..  (.5,1);
	\draw[pstdhl] (0,0) node[below]{\small $\mu_2$} -- (0,1)   ;
}
\ \omega_{\{1\}} - \omega_{\{2\}},
\\
d_\dgS(1_{\rho_{\{1,2\}}})
 &= 
\overline{d_0( \omega_{\{1,2\}})}
= 
\tikzdiagl{
	\draw[pstdhl]	(0,0)  node[below]{\small $\mu_1$}-- (0,1);
	\draw		(.5,0) .. controls (.5,.5) and (1,.5) .. (1,1);
	\draw 	 	(1.5,0)  .. controls (1.5,.5) and (.5,.5) .. (.5,1);
	\draw[pstdhl]	(1,0)  node[below]{\small $\mu_2$}.. controls (1,.5) and (1.5,.5) .. (1.5,1);
}
\ 1_{\rho_{\{1\}}} - \ 
\tikzdiagl{
	\draw[pstdhl]	(0,0) node[below]{\small $\mu_1$} -- (0,1);
	\draw 	 	(.5,0) -- (.5,1);
	\draw		(1.5,0) .. controls (1.5,.5) and (1,.5) .. (1,1);
	\draw[pstdhl]	(1,0) node[below]{\small $\mu_2$} .. controls (1,.5) and (1.5,.5) .. (1.5,1);
}
\ 1_{\rho_{\{2\}}},
\end{align*}
which agrees with \cref{ex:stdmodule}. 
\end{rem}

\begin{prop}
The construction described above gives $(\muS_\rho, d_\dgS)$ the structure of a $(\muT, d_\mu)$-$(\dgT^{\mu_1} \otimes \cdots \otimes \dgT^{\mu_r}, d_{\mu_1} + \cdots + d_{\mu_r})$-bimodule.
\end{prop}

\begin{proof}
Clearly, the action of each $(\dgT^{\mu_\ell}, d_\mu)$ (graded) commute with each other, and with the left-action of $(\muT, d_\mu)$. Thus we only need to check that it respects the differentials. In particular, we need to verify that
\begin{equation}\label{eq:gradedLeibniz}
d_{\dgS}(m \bullet x) = d_{\dgS}(m) \bullet x + (-1)^{|m|} m \bullet d_{\mu_\ell}(x),
\end{equation}
for all homogeneous $m \in \muS_\rho$ and $x \in \dgT_{b_\ell}^{\mu_\ell}$. We can assume $m = 1_{\rho_\bj}$ and $x$ is either a nail, a crossing or a dot. 
If $x$ is a nail, we compute 
\begin{align*}
d_{\dgS}(1_{\rho_\bj} \bullet x)  &=
\begin{cases}
 0, \hfill & \text{if $1 \in \bj_\ell$}, \\
 (-1)^{|\bj_\ell|}
  \tikzdiagl[yscale=1, xscale=.7]{
	\draw (1,0) ..controls (1,.25) and (0,.25) .. (0,.5)..controls (0,.75) and (1,.75) .. (1,1)  ;
	\draw[pstdhl] (0,0)node[below]{\small $\mu_\ell$} ..controls (0,.25) and (1,.25) .. (1,.5) ..controls (1,.75) and (0,.75) .. (0,1)  ;
} \in \muS_{\rho,\bj}
+ \sssum{\bj' = \\ \bj \setminus \{t \in \bj_\ell\} \sqcup \{1\}} 
\alpha_{\bj,t}\ 
\tikzdiagl[xscale=.8]{
	\draw (-.5,0) .. controls (-.5,.5) and (.5,.5) .. (.5,1);
	%
	\draw (-2,0) .. controls (-2,.5) and (-1.5,.5) .. (-1.5,1);
	\draw (-1,0) .. controls (-1,.5) and (-.5,.5) .. (-.5,1);
	\node at(-1.5,.1) {\small $\dots$};
	\node at(-1,.9) {\small $\dots$};
	\draw (.5,0) .. controls (.5,.5) and (1,.5) .. (1,1);
	\draw (1.5,0) .. controls (1.5,.5) and (2,.5) .. (2,1);
	\node at(1,.1) {\small $\dots$};
	\node at(1.5,.9) {\small $\dots$};
	\draw (2,0) .. controls (2,.5) and (-2,.5) .. (-2,1);
	\draw [pstdhl] (0,0) node[below]{\small $\mu_\ell$} .. controls (0,.25) and (.5,.25) .. (.5,.5) .. controls (.5,.75) and (0,.75) .. (0,1);
}
\in \muS_{\rho,\bj'},
& \text{if $1 \notin \bj_\ell$},
\end{cases}
\\
d_{\dgS}(1_{\rho_\bj}) \bullet x &= 
\begin{cases}
 (-1)^{|\bj_\ell|-1}
   \tikzdiagl[yscale=1, xscale=-.75]{
	\draw (1,0) ..controls (1,.25) and (0,.25) .. (0,.5)..controls (0,.75) and (1,.75) .. (1,1)  ;
	\draw[pstdhl] (0,0)node[below]{\small $\mu_\ell$} ..controls (0,.25) and (1,.25) .. (1,.5) ..controls (1,.75) and (0,.75) .. (0,1)  ;
} \in \muS_{\rho,\bj}, 
& \text{if $1 \in \bj_\ell$}, \\
\sssum{\bj' = \\ \bj \setminus \{t \in \bj_\ell\} \sqcup \{1\}} 
\alpha_{\bj,t}\ 
\tikzdiagl{
	\draw (-.5,0) .. controls (-.5,.5) and (.5,.5) .. (.5,1);
	%
	\draw (-2,0) .. controls (-2,.5) and (-1.5,.5) .. (-1.5,1);
	\draw (-1,0) .. controls (-1,.5) and (-.5,.5) .. (-.5,1);
	\node at(-1.5,.1) {\small $\dots$};
	\node at(-1,.9) {\small $\dots$};
	\draw (.5,0) .. controls (.5,.5) and (1,.5) .. (1,1);
	\draw (1.5,0) .. controls (1.5,.5) and (2,.5) .. (2,1);
	\node at(1,.1) {\small $\dots$};
	\node at(1.5,.9) {\small $\dots$};
	\draw (2,0) .. controls (2,.5) and (-2,.5) .. (-2,1);
	\draw [pstdhl] (0,0) node[below]{\small $\mu_\ell$} .. controls (0,.25) and (-.5,.25) .. (-.5,.5) .. controls (-.5,.75) and (0,.75) .. (0,1);
}
\in \muS_{\rho,\bj'}, 
& \text{if $1 \notin \bj_\ell$},
\end{cases}
\end{align*}
and using \cref{lem:stdNdotsaction} we also have
\begin{align*}
 &1_{\rho_\bj} \bullet d_{\mu_\ell}(x) \\ &= 
 \begin{cases}
   \tikzdiagl[yscale=1, xscale=.75]{
	\draw (1,0) -- (1,1)  node[midway,tikzdot]{} node[midway,xshift=1.5ex,yshift=.75ex]{\small $\mu_\ell$};
	\draw[pstdhl] (0,0)node[below]{\small $\mu_\ell$}-- (0,1)  ;
} \in \muS_{\rho,\bj},
& \text{if $1 \in \bj_\ell$}, \\
  \tikzdiagl[yscale=1, xscale=.75]{
	\draw (1,0) -- (1,1)  node[midway,tikzdot]{} node[midway,xshift=1.5ex,yshift=.75ex]{\small $\mu_\ell$};
	\draw[pstdhl] (0,0)node[below]{\small $\mu_\ell$}-- (0,1)  ;
} \in \muS_{\rho,\bj}
-
(-1)^{|\bj_\ell|}
\sssum{\bj' = \\ \bj \setminus \{t \in \bj_\ell\} \sqcup \{1\}} \ 
\alpha_{\bj,t}
\sssum{u+v=\\ \mu_\ell -1} \ 
\tikzdiagl[xscale=.9]{
	\draw (-.5,0) .. controls (-.5,.5) and (-2,.5) .. (-2,1) node[pos=.2,tikzdot]{} node[pos=.2,xshift=1.5ex,yshift=.75ex]{\small $u$};
	%
	\draw (-2,0) .. controls (-2,.5) and (-1.5,.5) .. (-1.5,1);
	\draw (-1,0) .. controls (-1,.5) and (-.5,.5) .. (-.5,1);
	\node at(-1.5,.1) {\small $\dots$};
	\node at(-1,.9) {\small $\dots$};
	\draw (.5,0) .. controls (.5,.5) and (1,.5) .. (1,1);
	\draw (1.5,0) .. controls (1.5,.5) and (2,.5) .. (2,1);
	\node at(1,.1) {\small $\dots$};
	\node at(1.5,.9) {\small $\dots$};
	\draw (2,0) .. controls (2,.5) and (.5,.5) .. (.5,1)  node[pos=.8,tikzdot]{} node[pos=.8,xshift=-1.5ex,yshift=-.75ex]{\small $v$};
	\draw [pstdhl] (0,0) node[below]{\small $\mu_\ell$} -- (0,1);
}
\in \muS_{\rho,\bj'},
& \text{if $1 \notin \bj_\ell$}.
 \end{cases}
\end{align*}
where each one of the diagrams are embedded in bigger diagrams with only vertical strand whose colors are determined by the idempotents $1_{\rho_\bj}$ and  $1_{\rho_{\bj'}}$,
and $\alpha_{\bj,t} := (-1)^{\#\{t' \in \bj_\ell | t' > t\}}$. Then \cref{eq:gradedLeibniz} follows from Eq. (\ref{eq:redR2}-\ref{eq:vredR}).

If $x$ is a dot or a crossing, then we obtain immediately $d_{\dgS}(m \bullet x) = d_{\dgS}(m) \bullet x $ by \cref{rem:diffomegaj}, since $d_0$ is well-defined and thus pushing $x$ to the top and then applying $d_0$ is the same as applying $d_0$ and then pushing $x$ to the top. 
\end{proof}

\subsubsection{Standardization functor}

\begin{defn}
We define the \emph{standardization functor} as
\[
\stdFunct : \cD_{dg}(\dgT^{\mu_1}\otimes\cdots\otimes \dgT^{\mu_r}, d_{\mu_1} + \cdots + d_{\mu_r}) \rightarrow \cD_{dg}(\muT,d_\mu), \quad M \mapsto \muS \Lotimes M,
\]
where $\muS := \bigoplus_{\rho \in \bN^r} \muS_\rho$. 
\end{defn}

For $1 \leq i \leq r$, let $\E^{[i]}$, $\F^{[i]}$ and $\K^{\pm [i]}$ denotes the categorical action of $U_q(\slt)$ on each $\dgT^{\mu_i}_{b_i}$ in $ \cD_{dg}(\dgT^{\mu_1}\otimes\cdots\otimes \dgT^{\mu_r}, d_{\mu_1} + \cdots + d_{\mu_r}) $, defined by induction/restriction along a black strand as in \cref{sec:catAction}. Let us write $\id_\rho$ with $\rho = (b_1,\dots, b_r)$ for the functor given by tensoring with $(\dgT^{\mu_1}_{b_1} \otimes \cdots \otimes \dgT^{\mu_r}_{b_r} )$. 

\begin{prop}
    The standardization functor is exact and essentially surjective. 
    In particular, it induces a surjection
    \[
      {}_\bQ\bKO^\Delta(\dgT^{\mu_1}, d_{\mu_1})
      \otimes \cdots \otimes
      {}_\bQ\bKO^\Delta(\dgT^{\mu_r}, d_{\mu_r}) 
      \twoheadrightarrow {}_\bQ\bKO^\Delta(\muT, d_\mu),
    \]
    which sends $v_{\mu_1,b_1} \otimes \cdots \otimes v_{\mu_{r}, b_r} = [\dgP^{\mu_1}_{b_1}] \otimes  \cdots \otimes [\dgP^{\mu_r}_{b_r}] \mapsto [(\muS_\rho, d_\dgS)] = v_\rho$ under the isomorphism of \cref{thm:isocat}.
\end{prop}

\begin{proof}
    This follows immediately from the fact that $\stdFunct(\dgT^{\mu_1}_{b_1}\otimes\cdots\otimes \dgT^{\mu_r}_{b_r}, d_{\mu_1} + \cdots + d_{\mu_r}) \cong (\muS_\rho, d_\dgS)$ together with \cref{cor:stdcblfgen}. 
\end{proof}

Note that we have
\[
\K^{\pm 1} \stdFunct \cong \stdFunct \K^{\pm [1]} \cdots \K^{\pm[r]},
\]
which lift the equality $\Delta(K^{\pm 1}) = K^{\pm 1} \otimes K^{\pm 1}$. 
Furthermore, as in \cite[Proposition 5.5]{webster}, we can lift the equality 
\begin{align*}
F(1 \otimes 1 \otimes \cdots \otimes 1) =& F \otimes K \otimes K \otimes \cdots \otimes K + 1 \otimes F \otimes K \otimes \cdots \otimes K + \cdots
\\
&\cdots + 1  \otimes \cdots \otimes 1 \otimes F \otimes K +  1 \otimes \cdots \otimes 1 \otimes 1 \otimes F,
\end{align*}
to the categorical setting as follows:

\begin{prop}\label{prop:stratFfunct}
There is a natural isomorphism 
$\F\stdFunct  \cong Q_r$ 
with $Q_r$ being obtained as an iterated extension
\[ 
\begin{tikzcd}[column sep=1ex]
0 = Q_0 \ar{rr}  && Q_1 \ar{rr} \ar{dl} && Q_2 \ar{rr} \ar{dl} 
&&  \ar{dl} \cdots  \ar{rr} && Q_{r} \cong \F\stdFunct, 
\ar{dl}
\\
&Q_1/Q_0 \ar{ul}{[1]} && Q_2/Q_1 \ar{ul}{[1]}  && Q_3/Q_2 \ar{ul}{[1]} && Q_r/Q_{r-1} \ar{ul}{[1]}  &
\end{tikzcd}
\]
where
\[
Q_\ell/Q_{\ell-1} \cong \stdFunct\F^{[\ell]} \K^{[\ell+1]} \cdots \K^{[r]},
\]
for $1 \leq \ell \leq r$.
\end{prop}

\begin{proof}
Take $\rho = (b_1, \dots, b_r)$ with $\sum b_i = b$. 
Since the functor $\stdFunct$ is given by derived tensor product with a bimodule which is cofibrant as left module, we have 
\[
\F\stdFunct\id_\rho(-) \cong \bigl((\muT_{b+1},d_\mu) \otimes_b \muS_\rho \bigr) \Lotimes - \cong (\F\muS_\rho) \Lotimes -.
\]
Similarly, we have 
\[
\stdFunct\F^{[\ell]} \id_\rho(-) 
 \cong
 \muS_{F^{[\ell]}(\rho)} \Lotimes -,
\] 
where $F^{[\ell]}(\rho) := (b_1,\dots,b_{\ell-1}, b_{\ell} +1, b_{\ell+1}, \dots, b_r)$. 

We want to construct categorifications of the elements $F(\tilde v_{\rho_1}) \otimes \tilde v_{\rho_2}$ for various decompositions $\rho = (\rho_1, \rho_2)$, and these will give the functors $Q_\ell$. 

Let 
\[
 \widetilde{\BQ}_\ell  := \bigoplus_{\bj \subset J_\rho} q^{\sum_{\ell = 2}^{r}\sum_{t \in \bj_\ell }(\mu_\ell - 2 t +2)}
\muP_{F^{[\ell]}(\rho_\bj)} [|\bj|],
\]
and define $d_{\BQ}$ similarly as $d_{\dgS}$ but using 
\[
\tikzdiagh{0}{
	\node at(.5,.5) {$\dots$};
	%
	\draw 		(1.5,0) .. controls (1.5,.5) and (2,.5) .. (2,1);
	\node at (2,.15){$\dots$};
	\node at (2.5,.85){$\dots$};
	\draw 		(2.5,0)  .. controls (2.5,.5) and (3,.5) .. (3,1);
	\draw 		(3.5,0)  .. controls (3.5,.5) and (4,.5) .. (4,1);
	\node at (4,.15){$\dots$};
	\node at (4.5,.85){$\dots$};
	\draw 		(4.5,0)  .. controls (4.5,.5) and (5,.5) .. (5,1);
	%
	\draw		(5,0) .. controls (5,.5) and (1.5,.5)  .. (1.5,1);
	\tikzbrace{1.5}{2.5}{-.15}{\small $p_1$};
	\tikzbrace{3.5}{4.5}{-.15}{\small $p_2$};
	\draw (5.5,0) -- (5.5,1);
	%
	\node at(6.5,.5) {$\dots$};
	%
	%
	\draw[pstdhl] 	(3,0) node[below]{\small $\mu_\ell$}  .. controls (3,.5) and (3.5,.5) .. (3.5,1);
}
\]
instead of $\tau_{\bj,\bj'}$ whenever $\bj$ differs from $\bj'$ by an element in $\bj_\ell$. For the same reasons as $(\muS, d_\dgS)$,  $( \widetilde{\BQ}_\ell, d_\BQ)$ is a dg-bimodule. 
Note that $\widetilde{\BQ}_1 = \muS_{F^{[1]}(\rho)}$ and $\widetilde{\BQ}_r \cong \F  \muS_{\rho}$. 
Moreover, we have a map of dg-bimodules  
\[
\tau_{\ell-1, \ell} : q^{\mu_{\ell}-2b_{\ell}} \widetilde{\BQ}_{\ell-1} \rightarrow \widetilde{\BQ}_{\ell}
\]
given by gluing on the bottom
\begin{equation*} 
\tikzdiagh{0}{
	\node at(.5,.5) {$\dots$};
	%
	\draw 		(1.5,0) .. controls (1.5,.5) and (2,.5) .. (2,1);
	\node at (2,.15){$\dots$};
	\node at (2.5,.85){$\dots$};
	\draw 		(2.5,0)  .. controls (2.5,.5) and (3,.5) .. (3,1);
	\draw 		(3.5,0)  .. controls (3.5,.5) and (4,.5) .. (4,1);
	\node at (4,.15){$\dots$};
	\node at (4.5,.85){$\dots$};
	\draw 		(4.5,0)  .. controls (4.5,.5) and (5,.5) .. (5,1);
	\tikzbrace{1.5}{2.5}{-.15}{\small $|\bj_{\ell}|$};
	\tikzbrace{3.5}{4.5}{-.15}{\small $(b_{\ell}-|\bj_{\ell}|)$};
	\draw		(5,0) .. controls (5,.5) and (1.5,.5)  .. (1.5,1);
	%
	\node at(6,.5) {$\dots$};
	%
	%
	\draw[pstdhl] 	(3,0)  node[below]{\small $\mu_{\ell}$} .. controls (3,.5) and (3.5,.5) .. (3.5,1);
}
\end{equation*}
By construction of $\muS_{F^{[\ell+1]}\rho}$, we have
\[
\muS_{F^{[\ell]}\rho} \cong \cone\left(
q^{\mu_{\ell}-2b_{\ell}} \widetilde{\BQ}_{\ell-1} 
\xrightarrow{\tau_{\ell-1,\ell}}
\widetilde{\BQ}_{\ell}
\right).
\]
Thus, putting $
\BQ_\ell :=  q^{\sum_{t > \ell} \mu_t - 2b_t} \widetilde{\BQ}_\ell
$ and $Q_\ell := \BQ_\ell \Lotimes -$ concludes the proof. 
\end{proof}

\subsection{Stratification}\label{sec:stratification}

Fix $b \geq 0$. 
Let $\cD := \cD_{dg}^{cblf}(\muT_b, d_\mu)$. Define $\cD_{\succeq \rho}$ as the full subcategory of $\cD$ c.b.l.f. generated by $\{\muS_{\rho'} | \rho' \succeq \rho\}$. 
Define similarly $\cD_{\succ \rho} \subset \cD_{\succeq \rho}$. 

Consider the exact sequence
\[
    \cD_{\succ \rho} \rightarrow \cD_{\succeq \rho} \rightarrow \cD_{\rho},
\]
of dg-categories where $\cD_\rho$ is Verdier  dg-quotient (see \cite{dgquotient1, dgquotient2}) of $\cD_{\succeq \rho}$ by $\cD_{\succ \rho}$.

\begin{lem}\label{lem:derivedPtostdhom}
    We have 
    $ 
        \RHOM_{(\muT_b,d_\mu)}((\muP_\rho,d_\mu), (\muS_{\rho'},d_\dgS)) \cong 0
    $ 
    whenever $\rho' \npreceq \rho$. 
\end{lem}

\begin{proof}
    We know that $(\muS_{\rho'},d_\dgS)$ can be constructed as an iterated mapping cone, and thus takes the form of a hypercube of $\muP_{\rho''}$ for $\rho'' \preceq \rho'$. We can reaarrange the hypercube so that the first mapping cones are all of the form
    \[
        (\BS, d_S) :=
        \cone\left( (\muP_{\rho_2},d_\mu) \xrightarrow{
        \tikzdiag[scale=.5]{
            \node at(-.5,.5) {\small $\dots$};
            \draw (1,0) .. controls (1,.5) and (0,.5) .. (0,1);
            \draw[pstdhl] (0,0) node[below]{\small $\mu_i$} .. controls (0,.5) and (1,.5) .. (1,1);
            \node at(1.5,.5) {\small $\dots$};
        }
        } (\muP_{\rho_1},d_\mu) \right),
    \]
    for various $i$ and $\rho_1, \rho_2$ such that $\rho_1 \npreceq \rho$. 
    We claim that $\RHOM_{(\muT_b,d_\mu)}((\muP_\rho,d_\mu), (\BS,d_S)) \cong 0$ and then the statement of the lemma follows from exactness of the derived hom functor. 

    Since $(\muP_\rho, d_\mu)$ is cofibrant, we can replace the derived hom-space by the dg-hom-space. We only need to show that the homology of $\HOM_{(\muT_b,d_\mu)}((\muP_\rho,d_\mu), (\BS,d_S))$ is zero. 
    Recall that a map in the dg-hom-space is in the kernel of the differential if and only if it graded commutes with the differentials of the target and source dg-modules. 
    All these maps are generated by the map $(\muP_\rho,d_\mu) \rightarrow (\muP_{\rho_1},d_\mu)$ that sends $1_\rho$ to the diagrams with the least number of crossings $1_\rho W_1 1_{\rho_1}$. Then we can consider the map $(\muP_\rho,d_\mu) \rightarrow (\muP_{\rho_2},d_\mu)$ (that does not commute with the differentials) that sends $1_\rho$ to the diagram with the least number of crossings $1_\rho W_2 1_{\rho_2}$. But then we have $d_S(1_\rho W_2 1_{\rho_2}) = 1_\rho W_1 1_{\rho_1}$. Therefore the dg-hom-space is acyclic, concluding the proof. 
\end{proof}

\begin{lem}\label{lem:derivedstdhom}
    We have 
    $ 
        \RHOM_{(\muT_b,d_\mu)}((\muS_\rho,d_\dgS), (\muS_{\rho'},d_\dgS)) \cong 0
    $ 
    whenever $\rho' \succ \rho$.
\end{lem}

\begin{proof}
    It follows by exactness of the derived hom functor together with \cref{lem:derivedPtostdhom} and the fact that $(\muS_\rho,d_\dgS)$ is an iterated extension of $(\muP_{\rho''},d_\mu)$ for various $\rho'' \preceq \rho$.
\end{proof}

\begin{prop}
    There is a quasi-equivalence $\cD_{\rho} \cong \cD_{dg}^{cblf}(\dgT_{b_1}^{\mu_1}\otimes\cdots\otimes \dgT_{b_r}^{\mu_r}, d_{\mu_1} + \cdots + d_{\mu_r})$. Moreover the projection $ \cD_{\succeq \rho} \rightarrow \cD_{\rho}$ is equivalent to the dg-functor 
    \[
        \RHOM_{(\muT_b,d_\mu)}((\muS_\rho,d_\dgS), -) : \cD_{\succeq \rho} \rightarrow \cD_{dg}^{cblf}(\dgT_{b_1}^{\mu_1}\otimes\cdots\otimes \dgT_{b_r}^{\mu_r}, d_{\mu_1} + \cdots + d_{\mu_r}),
    \]
    which is right adjoint to the standardization functor $\stdFunct$. 
\end{prop}

\begin{proof}
    It follows from \cref{lem:derivedstdhom} and exactness of the derived hom functor. 
\end{proof}


\appendix

\section{Rewriting methods} \label{sec:rewritingmethods}

\subsection{Diagrammatic rewriting}
Let $\mathbf{A}$ be a diagrammatic algebra presented by generators and relations. It is defined by a set of generators, denoted by $\mathbf{A}_g$, containing diagrams that are of the form 
\begin{equation}
\label{eq:shapegenerators}
\tikzdiag{
\draw (-1.25,-1) to (-1.25,2);
\draw (-2,-1) to (-2,2);
\node[scale=0.75] at (-1.6,0.5) {$\dots$};
\node[scale=0.75] at (-1.25,-1.2) {$\lambda_k$};
\node[scale=0.75] at (-2,-1.2) {$\lambda_1$};
	\draw (0,-1) -- (0,0) ..controls (0.2,.5) and (1,.5) .. (1,1);
	\draw (-0.5,-1) -- (-0.5,0) ..controls (0,0.5) and (1,.5) .. (1,1);
	\draw (1,-1) -- (1,0) ..controls (1,.5) and (0,.5) .. (0,1);
		\draw (1.5,-1) -- (1.5,0) ..controls (1,.5) and (0,.5) .. (0,1);
		\draw (1,1) to (1,2);
		\draw (0,1) to (0,2);
			\draw (-0.5,0) ..controls (0,.5) and (1.5,.5) .. (1.5,1) -- (1.5,2);
	\draw (1.5,0) ..controls (1,.5) and (-0.5,.5) .. (-0.5,1) -- (-0.5,2);
		\filldraw [fill=white, draw=black,rounded corners] (.5-.25,.5-.25) rectangle (.5+.25,.5+.25) node[midway] { $x$};
		\node[scale=0.75] at (0.5,1.5) {$\cdots$};
		\node[scale=0.75] at (0.5,-0.5) {$\cdots$};
		\node[scale=0.75] at (-0.5,2.15) {$\mu_1$};
		\node[scale=0.75] at (0,2.15) {$\mu_2$};
		\node[scale=0.75] at (1.5,2.15) {$\mu_m$};
		\node[scale=0.75] at (-0.5,-1.15) {$\eta_1$};
		\node[scale=0.75] at (0,-1.15) {$\eta_2$};
		\node[scale=0.75] at (1.5,-1.15) {$\eta_n$};
\draw (2.25,-1) to (2.25,2);
\draw (3,-1) to (3,2);
\node[scale=0.75] at (2.6,0.5) {$\dots$};
\node[scale=0.75] at (2.25,-1.2) {$\lambda'_1$};
\node[scale=0.75] at (3,-1.2) {$\lambda'_\ell$};
} 
\end{equation}
where $m$,$n$, $k$, $\ell$ are integers, and $\lambda_1 \dots, \lambda_k, \lambda'_1, \dots, \lambda'_\ell, \mu_1, \dots, \mu_m, \eta_1, \dots, \eta_n$ are labels (or colors) that belong to an indexing set $I_{\mathbf{A}}$. Such a diagram can be considered \emph{locally}, by forgetting the vertical strands on the left and on the right, and we say that a diagram $x$ as in \cref{eq:shapegenerators} has {\em arity} $n$ and {\em coarity} $m$. To simplify the notations, we will write this as $x: \eta_1 \dots \eta_n \fl \mu_1 \dots \mu_m$. In other words, the generators of $\mathbf{A}$ are represented by diagrams, with vertical labelled strands in the leftmost and the rightmost region, and in between such a diagram with arity $n$ and coarity $m$, corresponding to a diagram that has $n$ labelled strands as input and $m$ labelled strand as output. We allow $m$ and $n$ to be $0$, however we assume in the sequel that any generator $x$ in $\mathbf{A}_g$ has same arity and coarity, that can be $0$. Therefore, we have the following disjoint decomposition for $\mathbf{A}_g$:
\[ \mathbf{A}_g = \underset{n \in \mathbb{N}}{\sqcup} \mathbf{A}_g(n) \]
where $\mathbf{A}_g(n)$ denotes the set of generators with arity and coarity $n$. Moreover, we assume that $\mathbf{A}_g^0$ is equipped with a total order $\prec_0$.  We also assume that the diagrams in an algebra $\mathbf{A}$ admit a constant number of strands, so that the sum $k + n + \ell$ for a diagram $x$ as in \cref{eq:shapegenerators} is constant, equal to a fixed number $s(\mathbf{A})$ giving the number of strands of $\mathbf{A}$.

The product of two generators $x : \eta_1 \dots \eta_n \fl \mu_1 \dots \mu_n$ and $y : \mu_1 \dots \mu_n \fl \delta_1 \dots \delta_n$ (that can admit vertical strands) is obtained by vertically composing the two diagrams, from bottom to top. It is zero if the common sequence of labels $\mu_1 \dots \mu_n$ do not match. A {\emph monomial} of $\mathbf{A}$ is a product in the elements of $\mathbf{A}_g$, that is a diagram containing layers of generating pieces, in which any generator has a given {\em height}. Explicitly, a generator $x_i$ in a monomial $x_1 \dots x_{i-1} x_i x_{i+1} \dots x_n$ admits a diagrammatic height, denoted by $h(x_i) := i$. This extends to monomials of $\mathbf{A}$: if $x_k \dots x_{k+m}$ is a monomial dividing a monomial $x_1 \dots x_n$, then we set $h(x_k \dots x_{k+m}) := k$.

The presentation of a diagrammatic algebra is then given by choosing a set of diagrammatic relations between polynomials made of these monomials, with common source and target labels. As a consequence, the algebra $\mathbf{A}_g$ can be presented by a linear $2$-polygraph $P_{\mathbf{A}}$ with only one $0$-cells, whose generating $1$-cells are given by the elements of $\mathbf{A}_g$ and whose generating $2$-cells correspond to a fixed orientation of these relations. The generating $1$-cells of $P_{\mathbf{A}}$ are thus also equipped with an arity and coarity, that extends to the monomials of $P_1^\ell$. We denote by $P_1^\ell[n,m]$ the set of monomials with arity $n$ and coarity $m$.

\begin{exe}
For the nilHecke algebra $\nh_n$ of degree $n$, the set $I_{\nh_n}$ is a singleton, so that we may omit labels in the diagrams, $s(\nh_n) = n$ and the set of generators is given by $(\nh_n)_g:= \{ x_i \: \mid \: 1 \leq i \leq n \} \cup \{ \tau_k \: \mid 1 \leq k \leq n-1 \}$ of respective (co)arity $1$ and $2$ that are diagrammatically depicted as follows:
\begin{align}
\label{eq:Gensnilhecke}
x_i &:= \tikzdiagh{0}{
	\draw(2,0) -- (2,1);
		\node at(2.5,.5) {\tiny$\dots$};
		\draw (3,0) -- (3,1);
		\node at (3,0.5) {$\bullet$};
		\node[scale=0.8] at (3,-0.3) {$i$};
		\node at (3.5,0.5) {\tiny$\dots$};
		\draw (4,0) -- (4,1);
}
&
\tau_k &:= \tikzdiagh{0}{
	\draw (2,0) -- (2,1);
	\node at(2.5,.5) {\tiny$\dots$};
\draw (2.7,0) ..controls (2.7,.5) and (3.3,.5) .. (3.3,1);
	\draw (3.3,0) ..controls (3.3,.5) and (2.7,.5) .. (2.7,1);
	\node[scale=0.8] at (2.7,-0.3) {$k$};
    \node at (3.5,0.5) {\tiny $\dots$};
	\draw (4,0) -- (4,1);
}
\end{align}
where the label $i$ indicates that this is the $i$-th strand at the bottom from left to right.
\end{exe}

\subsection{The linear~2-polygraph of distant isotopies}

Given a linear $2$-polygraph $P_{\mathbf{A}}$ presenting a diagrammatic algebra $\mathbf{A}$ with set of generators $\mathbf{A}_g$ and indexing set $I_{\mathbf{A}}$, we define the linear~$2$-polygraph $\text{Iso}(\mathbf{A})$ of planar isotopies of $\mathbf{A}$ that has only one $0$-cell and whose:
\begin{enumerate}[{\bf i)}]
\item generating $1$-cells are given by the $1$-cells of $(P_\mathbf{A})_1^\ast$, that correspond to the monomials of $\mathbf{A}$,
\item generating $2$-cells are given by the following local relations:
\[
\tikzdiag[yscale=1]{
\draw (0,0.5) to (1,0.5) to (1,1) to (0,1) to (0,0.5);
\draw[dashed] (0.2,1) to (0.2,1.4);
\node[scale=0.75] at (0.2,1.6) {$\eta_1$};
\node[scale=0.75] at (0.5,1.6) {$\cdots$};
\node[scale=0.75] at (0.8,1.6) {$\eta_k$};
\draw[dashed] (0.8,1) to (0.8,1.4);
\node at (0.5,0.75) {$D$};
\draw[dashed] (0.2,0.5) to (0.2,-0.5);
\draw[dashed] (0.8,0.5) to (0.8,-0.5);
\node[scale=0.75] at (0.2,-0.7) {$\mu_1$};
\node[scale=0.75] at (0.5,-0.7) {$\cdots$};
\node[scale=0.75] at (0.8,-0.7) {$\mu_k$};
\node[scale=0.75] at (1.5,0.25) {$\cdots$};
\draw (2,-0.1) to (2,0.4) to (3,0.4) to (3,-0.1) to (2,-0.1);
\node at (2.5, 0.15) {$D'$};
\draw[dashed] (2.2,-0.1) to (2.2,-0.5);
\draw[dashed] (2.8,-0.1) to (2.8,-0.5);
\draw[dashed] (2.2,0.4) to (2.2,1.4);
\draw[dashed] (2.8,0.4) to (2.8,1.4);
\node[scale=0.75] at (2.2,1.6) {$\eta'_1$};
\node[scale=0.75] at (2.5,1.6) {$\cdots$};
\node[scale=0.75] at (2.8,1.6) {$\eta'_m$};
\node[scale=0.75] at (2.2,-0.7) {$\mu'_1$};
\node[scale=0.75] at (2.5,-0.7) {$\cdots$};
\node[scale=0.75] at (2.8,-0.7) {$\mu'_m$};
}  \: \overset{E_{D,D'}}{\Rrightarrow} \: \tikzdiag[yscale=1]{
\draw (2,0.5) to (3,0.5) to (3,1) to (2,1) to (2,0.5);
\draw[dashed] (2.2,1) to (2.2,1.4);
\draw[dashed] (2.8,1) to (2.8,1.4);
\draw[dashed] (2.2,0.5) to (2.2,-0.5);
\draw[dashed] (2.8,0.5) to (2.8,-0.5);
\node[scale=0.75] at (1.5,0.25) {$\cdots$};
\draw (0,-0.1) to (0,0.4) to (1,0.4) to (1,-0.1) to (0,-0.1);
\node at (0.5, 0.15) {$D$};
\node at (2.5,0.75) {$D'$};
\draw[dashed] (0.1,-0.1) to (0.1,-0.5);
\draw[dashed] (0.8,-0.1) to (0.8,-0.5);
\draw[dashed] (0.2,0.4) to (0.2,1.4);
\draw[dashed] (0.8,0.4) to (0.8,1.4);
\node[scale=0.75] at (0.2,1.6) {$\eta_1$};
\node[scale=0.75] at (0.5,1.6) {$\cdots$};
\node[scale=0.75] at (0.8,1.6) {$\eta_k$};
\node[scale=0.75] at (0.2,-0.7) {$\mu_1$};
\node[scale=0.75] at (0.5,-0.7) {$\cdots$};
\node[scale=0.75] at (0.8,-0.7) {$\mu_k$};
\node[scale=0.75] at (2.2,1.6) {$\eta'_1$};
\node[scale=0.75] at (2.5,1.6) {$\cdots$};
\node[scale=0.75] at (2.8,1.6) {$\eta'_m$};
\node[scale=0.75] at (2.2,-0.7) {$\mu'_1$};
\node[scale=0.75] at (2.5,-0.7) {$\cdots$};
\node[scale=0.75] at (2.8,-0.7) {$\mu'_m$};
}
\]
for any monomials $D: \mu_1 \dots \mu_k \to \eta_1 \dots \eta_k$ and $D': \mu'_1 \dots \mu'_m \to \eta'_1 \dots \eta'_n$ in $P_1^\ell$ of respective heights $i$ and $j$, with $i > j$, provided that $D \prec_0 D'$ if $D$ and $D$' are both of arity and coarity $0$, and for any number of strands with any label bewteen $D$ and $D'$.
\end{enumerate}

In the sequel, we will prove rewriting properties of the linear $2$-polygraph $\text{Iso}(A)$ that are independant of the labels of the generators. Therefore, we omit the labels in the diagrams in the proofs of termination and confluence for $\text{Iso}(A)$. Let us first prove the following statement:

\begin{prop}
Given a diagrammatic algebra $\mathbf{A}$ with the above assumptions, the linear $2$-polygraph $\text{Iso}(\mathbf{A})$ is terminating.
\end{prop}

\begin{proof} Consider the mapping 
\[
\delta : (P_\mathbf{A})_1^\ast \to \mathbb{Z}^{s(\mathbf{A})}
\]
that sends any monomial $D$ onto $(\delta_1(D), \dots, \delta_{s(\mathbf{A})}(D))$ where 
$\delta_i(D)$ is computed as follows: follow the $i$-strand (counted from left) from the bottom to the top, and each time we encounter a generator that intersects this line, add the number of generators (intersecting or not) that are below. 
One may check that for any $2$-cell $\alpha$ of $\text{Iso}(\mathbf{A})$, the inequality
$\delta(s_1(\alpha)) >_{\text{lex}} \delta(t_1(\alpha))$ for the lexicographic order on $\mathbb{Z}^{s(A)}$. Moreover, this order is admissible, that is $\delta(D) >_{\text{lex}} \delta(D')$ implies that $\delta(D_1 D D_2) >_{\text{lex}} \delta(D_1 D' D_2)$ for any monomials $D$,$D'$, $D_1$, $D_2$ such that the products are well-defined, since we add on bottom and top of $D$ and $D'$ a constant number of generators below any height. Therefore, the order on $P_1^\ell$ defined by $D < D'$ if and only if $\delta(D) <_{\text{lex}} \delta(D')$ defines a termination order for $Iso(\mathbf{A})$.
\end{proof}

\begin{exe} 
Consider the nilHecke algebra $\nh_6$ on $6$ strands, we have the following:
\[ \delta\left( \: \raisebox{-6mm}{$\begin{tikzpicture}[scale=1.2] \begin{scope} [ x = 10pt, y = 10pt, join = round, cap = round, thick ]  \draw (0.00,2.75)--(0.00,2.25) (1.00,2.75)--(1.00,2.50) (2.00,2.75)--(2.00,2.50) (3.00,2.75)--(3.00,2.25) (4.00,2.75)--(4.00,2.25) (5.00,2.75)--(5.00,2.25) ; 
\draw (1.00,2.50) ..controls (1,2.25) and (2,2.25) .. (2,2);
\draw (2.00,2.50) ..controls (2,2.25) and (1,2.25) .. (1.00,2.00) ;
  \draw (0.00,2.25)--(0.00,1.75) (1.00,2.00)--(1.00,1.75) (2.00,2.00)--(2.00,1.50) (3.00,2.25)--(3.00,1.50) (4.00,2.25)--(4.00,1.50) (5.00,2.25)--(5.00,1.50) ; \draw (0.00,1.75) ..controls (0,1.5) and (1,1.5) .. (1.00,1.25);
  \draw (1.00,1.75) ..controls(1,1.5) and (0,1.5) .. (0.00,1.25); 
  \draw (0.00,1.25)--(0.00,0.75) (1.00,1.25)--(1.00,0.75) (2.00,1.50)--(2.00,0.75) (3.00,1.50)--(3.00,1.00) (4.00,1.50)--(4.00,1.00) (5.00,1.50)--(5.00,0.75) ; 
  \draw (3.00,1.00) ..controls (3,0.75) and (4,0.75) .. (4.00,0.50);
  \draw (4.00,1.00) ..controls (4,0.75) and (3,0.75) .. (3.00,0.50) ;
  \draw (0.00,0.75)--(0.00,0.00) (1.00,0.75)--(1.00,0.00) (2.00,0.75)--(2.00,0.00) (3.00,0.50)--(3.00,0.00) (4.00,0.50)--(4.00,0.25) (5.00,0.75)--(5.00,0.25) ; 
  \draw (4.00,0.25) ..controls (4,0) and (5,0) .. (5.00,-0.25);
  \draw (5.00,0.25) ..controls (5,0) and (4,0) .. (4.00,-0.25) ; \draw (0.00,0.00)--(0.00,-0.50) (1.00,0.00)--(1.00,-0.50) (2.00,0.00)--(2.00,-0.50) (3.00,0.00)--(3.00,-0.50) (4.00,-0.25)--(4.00,-0.50) (5.00,-0.25)--(5.00,-0.50) ; 
\node at (2,1.1) {$\bullet$};
\end{scope} \end{tikzpicture}$} \: \right) =  \raisebox{0mm}{$ (3,7,6,1,1,0)$}, 
\qquad \qquad 
\delta\left( \: \raisebox{-7mm}{\begin{tikzpicture}[scale=1.2] \begin{scope} [ x = 10pt, y = 10pt, join = round, cap = round, thick ]  \draw (0.00,2.75)--(0.00,2.25) (1.00,2.75)--(1.00,2.50) (2.00,2.75)--(2.00,2.50) (3.00,2.75)--(3.00,2.25) (4.00,2.75)--(4.00,2.25) (5.00,2.75)--(5.00,2.25) ; 
\draw (1.00,2.50) ..controls (1,2.25) and (2,2.25) .. (2,2);
\draw (2.00,2.50) ..controls (2,2.25) and (1,2.25) .. (1.00,2.00) ;
  \draw (0.00,2.25)--(0.00,1.75) (1.00,2.00)--(1.00,1.75) (2.00,2.00)--(2.00,1.50) (3.00,2.25)--(3.00,1.50) (4.00,2.25)--(4.00,1.50) (5.00,2.25)--(5.00,1.50) ; \draw (0.00,1.75) ..controls (0,1.5) and (1,1.5) .. (1.00,1.25);
  \draw (1.00,1.75) ..controls(1,1.5) and (0,1.5) .. (0.00,1.25); 
  \draw (0.00,1.25)--(0.00,0.75) (1.00,1.25)--(1.00,0.75) (2.00,1.50)--(2.00,0.75) (3.00,1.50)--(3.00,1.00) (4.00,1.50)--(4.00,1.00) (5.00,1.50)--(5.00,0.75) ; 
  \draw (3.00,1.00) ..controls (3,0.75) and (4,0.75) .. (4.00,0.50);
  \draw (4.00,1.00) ..controls (4,0.75) and (3,0.75) .. (3.00,0.50) ;
  \draw (0.00,0.75)--(0.00,0.00) (1.00,0.75)--(1.00,0.00) (2.00,0.75)--(2.00,0.00) (3.00,0.50)--(3.00,0.00) (4.00,0.50)--(4.00,0.25) (5.00,0.75)--(5.00,0.25) ; 
  \draw (4.00,0.25) ..controls (4,0) and (5,0) .. (5.00,-0.25);
  \draw (5.00,0.25) ..controls (5,0) and (4,0) .. (4.00,-0.25) ; \draw (0.00,0.00)--(0.00,-0.50) (1.00,0.00)--(1.00,-0.50) (2.00,0.00)--(2.00,-0.50) (3.00,0.00)--(3.00,-0.50) (4.00,-0.25)--(4.00,-0.50) (5.00,-0.25)--(5.00,-0.50) ; 
\node at (2,-0.3) {$\bullet$};
 \end{scope} 
 \end{tikzpicture}} \: \right) = \raisebox{0mm}{$(3,7,4,2,3,1)$}
 \]
\vspace{-0.2cm}
\[ \delta\left( \: \raisebox{-8mm}{$\begin{tikzpicture}[scale=1.2] \begin{scope} [ x = 10pt, y = 10pt, join = round, cap = round, thick ] \draw (0.00,3.00)--(0.00,2.50) (1.00,3.00)--(1.00,2.50) (2.00,3.00)--(2.00,2.50) (3.00,3.00)--(3.00,2.75) (4.00,3.00)--(4.00,2.75) (5.00,3.00)--(5.00,2.50) ; 
\draw (3.00,2.75) ..controls (3,2.5) and (4,2.5) .. (4.00,2.25);
\draw (4.00,2.75) ..controls (4,2.5) and (3,2.5) .. (3.00,2.25) ;  \draw (0.00,2.50)--(0.00,1.75) (1.00,2.50)--(1.00,1.75) (2.00,2.50)--(2.00,1.75) (3.00,2.25)--(3.00,1.75) (4.00,2.25)--(4.00,2.00) (5.00,2.50)--(5.00,2.00) ;   
\draw (4.00,2.00) ..controls (4,1.75) and (5,1.75) .. (5.00,1.50);
\draw (5.00,2.00) ..controls (5,1.75) and (4,1.75) .. (4.00,1.50); \draw (0.00,1.75)--(0.00,1.00) (1.00,1.75)--(1.00,1.25) (2.00,1.75)--(2.00,1.25) (3.00,1.75)--(3.00,1.00) (4.00,1.50)--(4.00,1.00) (5.00,1.50)--(5.00,1.00) ;  
\draw (1.00,1.25) ..controls (1,1) and (2,1) .. (2.00,0.75);
\draw (2.00,1.25) ..controls (2,1) and (1,1) .. (1.00,0.75) ;  \draw (0.00,1.00)--(0.00,0.50) (1.00,0.75)--(1.00,0.50) (2.00,0.75)--(2.00,0.25) (3.00,1.00)--(3.00,0.25) (4.00,1.00)--(4.00,0.25) (5.00,1.00)--(5.00,0.25) ; 
\draw (0.00,0.50) ..controls (0,0.25) and (1,0.25) .. (1.00,0.00);
\draw (1.00,0.50) ..controls (1,0.25) and (0,0.25) .. (0.00,0.00) ;  \draw (0.00,0.00)--(0.00,-0.25) (1.00,0.00)--(1.00,-0.25) (2.00,0.25)--(2.00,-0.25) (3.00,0.25)--(3.00,-0.25) (4.00,0.25)--(4.00,-0.25) (5.00,0.25)--(5.00,-0.25) ;
\node at (2,-0.3) {$\bullet$};
 \end{scope} \end{tikzpicture}$} \: \right) = \raisebox{0mm}{$(1,3,2,4,7,3)$}, 
 \qquad \qquad  
 \delta \left( \: \raisebox{-6mm}{$\begin{tikzpicture}[scale=1.2] \begin{scope} [ x = 10pt, y = 10pt, join = round, cap = round, thick ] \draw (0.00,3.00)--(0.00,2.50) (1.00,3.00)--(1.00,2.50) (2.00,3.00)--(2.00,2.50) (3.00,3.00)--(3.00,2.75) (4.00,3.00)--(4.00,2.75) (5.00,3.00)--(5.00,2.50) ; 
\draw (3.00,2.75) ..controls (3,2.5) and (4,2.5) .. (4.00,2.25);
\draw (4.00,2.75) ..controls (4,2.5) and (3,2.5) .. (3.00,2.25) ;  \draw (0.00,2.50)--(0.00,1.75) (1.00,2.50)--(1.00,1.75) (2.00,2.50)--(2.00,1.75) (3.00,2.25)--(3.00,1.75) (4.00,2.25)--(4.00,2.00) (5.00,2.50)--(5.00,2.00) ;   
\draw (4.00,2.00) ..controls (4,1.75) and (5,1.75) .. (5.00,1.50);
\draw (5.00,2.00) ..controls (5,1.75) and (4,1.75) .. (4.00,1.50); \draw (0.00,1.75)--(0.00,1.00) (1.00,1.75)--(1.00,1.25) (2.00,1.75)--(2.00,1.25) (3.00,1.75)--(3.00,1.00) (4.00,1.50)--(4.00,1.00) (5.00,1.50)--(5.00,1.00) ;  
\draw (1.00,1.25) ..controls (1,1) and (2,1) .. (2.00,0.75);
\draw (2.00,1.25) ..controls (2,1) and (1,1) .. (1.00,0.75) ;  \draw (0.00,1.00)--(0.00,0.50) (1.00,0.75)--(1.00,0.50) (2.00,0.75)--(2.00,0.25) (3.00,1.00)--(3.00,0.25) (4.00,1.00)--(4.00,0.25) (5.00,1.00)--(5.00,0.25) ; 
\draw (0.00,0.50) ..controls (0,0.25) and (1,0.25) .. (1.00,0.00);
\draw (1.00,0.50) ..controls (1,0.25) and (0,0.25) .. (0.00,0.00) ;  \draw (0.00,0.00)--(0.00,-0.25) (1.00,0.00)--(1.00,-0.25) (2.00,0.25)--(2.00,-0.25) (3.00,0.25)--(3.00,-0.25) (4.00,0.25)--(4.00,-0.25) (5.00,0.25)--(5.00,-0.25) ;
\node at (2,0.6) {$\bullet$};
 \end{scope} \end{tikzpicture}$} \: \right) = \raisebox{0mm}{$ (0,2,3,4,7,3)$} \]
 On this example, the last element is the normal form of the corresponding diagram with respect to $\text{Iso}(\nh_6)$.
\end{exe}

The linear~$2$-polygraph $Iso(\mathbf{A})$ is also confluent, since all the critical branchings of $\text{Iso}(\mathbf{A})$ are given by local overlappings of the form
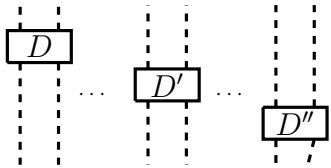
\[
\tikzdiag[yscale=0.85,xscale=0.85]{
\draw (0,0.5) to (1,0.5) to (1,1) to (0,1) to (0,0.5);
\draw[dashed] (0.2,1) to (0.2,1.4);
\draw[dashed] (0.8,1) to (0.8,1.4);
\draw[dashed] (0.2,0.5) to (0.2,-1.1);
\draw[dashed] (0.8,0.5) to (0.8,-1.1);
\node at (0.5,0.75) {$D$};
\node[scale=0.75] at (1.35,0) {$\cdots$};
\draw (2,-0.1) to (2,0.4) to (3,0.4) to (3,-0.1) to (2,-0.1);
\node at (2.5, 0.15) {$D'$};
\draw[dashed] (2.2,-0.1) to (2.2,-1.1);
\draw[dashed] (2.8,-0.1) to (2.8,-1.1);
\draw[dashed] (2.2,0.4) to (2.2,1.4);
\draw[dashed] (2.8,0.4) to (2.8,1.4);
\node[scale=0.75] at (3.5,0) {$\cdots$};
\draw (4,-0.2) to (4,-0.7) to (5,-0.7) to (5,-0.2) to (4,-0.2);
\node at (4.5,-0.45) {$D''$};
\draw[dashed] (4.2,-0.7) to (4.2,-1.1);
\draw[dashed] (4.8,-0.7) to (4.7,-1.1);
\draw[dashed] (4.2,-0.2) to (4.2,1.4);
\draw[dashed] (4.8,-0.2) to (4.8,1.4);
} \]
where $h(D) > h(D') > h(D'')$, for any labels of the strands provided the products are well-defined. They are proved confluent as follows:
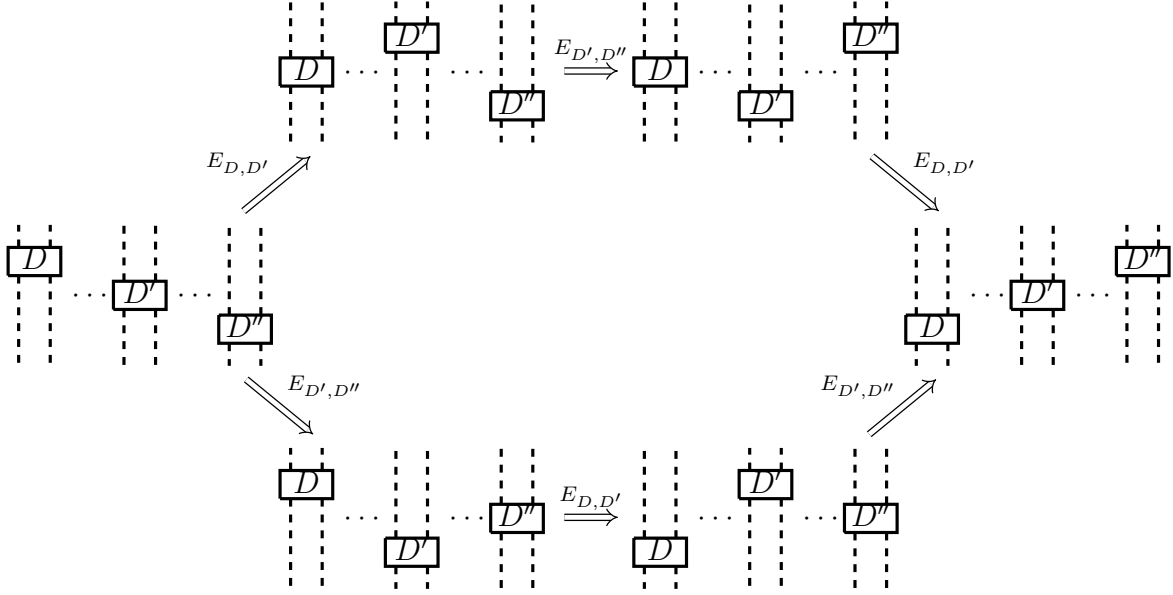
\[
\hspace{-0.5cm}
\begin{tikzcd}[column sep=-2ex]
{} & \criticalbdash \arrow[r,Rightarrow,"E_{D',D''}"] & [6ex] \criticalcdash \arrow[dr,Rightarrow,"E_{D,D'}"] & \\
\criticaladash \arrow[ur,Rightarrow,"E_{D,D'}"] \arrow[dr,Rightarrow,"E_{D',D''}"] &  & & \criticalenddash\\
 & \bcriticaldash \arrow[r,Rightarrow,"E_{D,D'}"] & \ccriticaldash \arrow[ur,Rightarrow,"E_{D',D''}"] & {} 
\end{tikzcd}
\]

We then rewrite with the linear $2$-polygraph $P$ modulo the convergent linear $2$-polygraph $\text{Iso}(\mathbf{A})$. Therefore, it is similar to the usual rewriting context on string diagrams in the monoidal category (seen as a $2$-category with only one object) admitting as generating $1$-cells the elements of $I_{\mathbf{A}}$, so that the $1$-cells of $\mathcal{C}$ are words of the form $\mu_1 \mu_2 \dots \mu_n$ for any $\mu_i \in I_{\mathbf{A}}$, and as generating $2$-cells the generating diagrams of $\mathbf{A}_g$ considered locally, that is by forgetting the vertical strands on the left and on the right.

\begin{exe}
For the nilHecke algebra $\nh_n$, rewriting modulo $\text{Iso}(\nh_n)$ is similar to rewriting in the monoidal category whose $1$-cells are generated by $1$, and thus isomorphic to $N$, whose generating $2$-cells are given by
\[ 
\tikzdiag{
\draw (2.7,0) ..controls (2.7,.5) and (3.3,.5) .. (3.3,1);
	\draw (3.3,0) ..controls (3.3,.5) and (2.7,.5) .. (2.7,1);} : 2 \to 2, \qquad \qquad 
	\tikzdiag{
	\draw (1,0) -- (1,1);
	\node at (1,0.5) {$\bullet$};} : 1 \to 1
	\]
	and are subject to the relations \eqref{eq:nhR2andR3} and \eqref{eq:nhdotslide}.
\end{exe}

As a consequence, the classification of critical branchings modulo in that context is the same as in the case of rewriting in string diagrams in the monoidal category $\curl{C}$, and most of them can be considered locally. Following \cite{GM09}, there are $3$ different forms of critical branchings in that context. For $2$-cells $\alpha, \beta$ of $P_2^\ell$, any $1$-cells $f$,$g$,$h$ of $P_1^\ell$ and any context $C$ of $P_1^\ast$, as defined in \cite{GM09}, there are:
\begin{itemize}
\item Regular critical branchings of the form
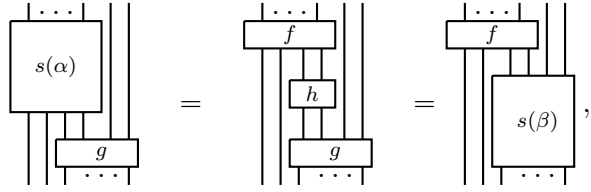
\begin{center}
$ \regcba{s(\alpha)}{g} \quad = \quad \regcbb{f}{h}{g} \quad = \regcbc{f}{s(\beta)} \: ,$ 
\end{center}
These amount to application on two local relations overlapping on the central part $h$ of the diagram. Since we rewrite modulo distant isotopies, these can be considered locally as in the $2$-category case, and one may forget about the diagrams that are on the left and on the right of this overlapping.
\item Inclusion critical branchings of the form
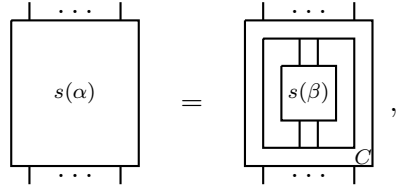
\begin{center}
$ \inccba{s(\alpha)} \quad = \quad \inccbb{s(\beta)}{C} \: , $
\end{center}
These branchings are given by application of a relation $\beta$ inside a diagram that is also reducible by a rule $\alpha$. There is no such example of branching for the linear $2$-polygraph modulo $(R,E,{}_E R_E)$, and one may in general avoid these branchings, since there always exist a linear~$2$-polygraph that does not contain such branchings and present the same $2$-category.
\item Left-indexed critical branchings (also right-indexed, multi-indexed) of the form
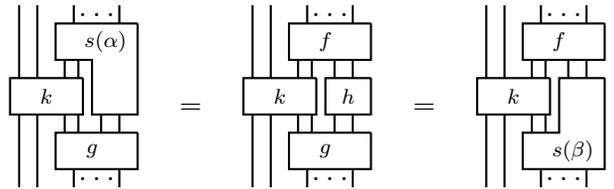
\begin{center}
$ \leftindcba{s(\alpha)}{k}{g} \quad = \quad \leftindcbb{f}{h}{k}{g} \quad = \quad \leftindcbc{f}{k}{s(\beta)} \: .$
\end{center}
These branchings come from the overlapping of two rewriting rules $\alpha$ and $\beta$ with an identity strand in the middle, in which we can plug new diagrams, giving new critical branchings to consider. Following \cite{GM09}, it suffices to check the confluence of the indexed branchings for the instance $k$ being in normal form.
\end{itemize}

\begin{exe}\label{ex:rewriteNH}
Let us consider the nilHecke algebra $\nh_n$ on $n$ strands, presented by the linear $2$-polygraph $P$ having as generating $1$-cells the elements $\tau_i$ and $x_l$ for $1 \leq i \leq n$ and $1 \leq l \leq n-1$ as in \eqref{eq:Gensnilhecke}, and as generating $2$-cells the relations \eqref{eq:nhR2andR3rewrite} and \eqref{eq:nhdotsliderewrite}. One might prove that $P$ is convergent modulo braid-like isotopies. Indeed, it is terminating using the weight order introduced in Section \ref{ssec:rewritingrulesparam}. Moreover, one might check its confluence modulo by examining its critical branching. It has regular critical branchings whose sources are given by:
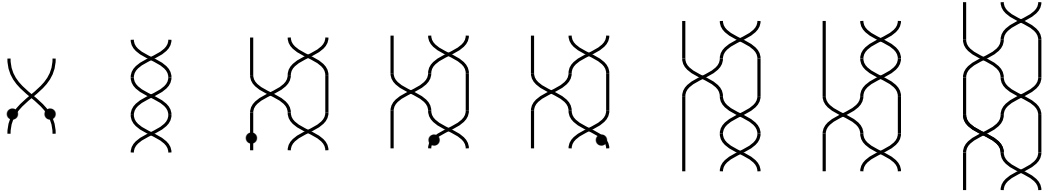
\[ \tikzdiag{
\draw (2.7,0) ..controls (2.7,.5) and (3.3,.5) .. (3.3,1);
	\draw (3.3,0) ..controls (3.3,.5) and (2.7,.5) .. (2.7,1);
	\node at (3.25,0.25) {$\bullet$};
	\node at (2.75,0.25) {$\bullet$};} \qquad  
\tikzdiag{	\draw (0,0) ..controls (0,0.25) and (0.5,0.25) .. (0.5,0.5);
\draw (0,0.5) ..controls (0,0.25) and (0.5,0.25) .. (0.5,0);
\draw (0,0.5) ..controls (0,0.75) and (0.5,0.75) .. (0.5,1);
\draw (0,1) ..controls (0,0.75) and (0.5,0.75) .. (0.5,0.5);
\draw (0,1) ..controls (0,1.25) and (0.5,1.25) .. (0.5,1.5);
\draw (0,1.5) ..controls (0,1.25) and (0.5,1.25) .. (0.5,1);}
	\qquad
	\tikzdiag{
\draw (0,0) ..controls (0,0.25) and (0.5,0.25) .. (0.5,0.5);
\draw (0,0.5) ..controls (0,0.25) and (0.5,0.25) .. (0.5,0);
\draw (-0.5,0) to (-0.5,0.5);
\draw (0.5,0.5) to (0.5,1);
\draw (-0.5,0.5) ..controls (-0.5,0.75) and (0,0.75) .. (0,1);
\draw (0,0.5) ..controls (0,0.75) and (-0.5,0.75) .. (-0.5,1);
\draw (-0.5,1) to (-0.5,1.5);
\draw (0,1) ..controls (0,1.25) and (0.5,1.25) .. (0.5,1.5);
\draw (0.5,1) ..controls (0.5,1.25) and (0,1.25) .. (0,1.5);
\node at (-0.5,0.15) {$\bullet$};
	} \qquad  
	\tikzdiag{
\draw (0,0) ..controls (0,0.25) and (0.5,0.25) .. (0.5,0.5);
\draw (0,0.5) ..controls (0,0.25) and (0.5,0.25) .. (0.5,0);
\draw (-0.5,0) to (-0.5,0.5);
\draw (0.5,0.5) to (0.5,1);
\draw (-0.5,0.5) ..controls (-0.5,0.75) and (0,0.75) .. (0,1);
\draw (0,0.5) ..controls (0,0.75) and (-0.5,0.75) .. (-0.5,1);
\draw (-0.5,1) to (-0.5,1.5);
\draw (0,1) ..controls (0,1.25) and (0.5,1.25) .. (0.5,1.5);
\draw (0.5,1) ..controls (0.5,1.25) and (0,1.25) .. (0,1.5);
\node at (0.06,0.10) {$\bullet$};
	} \qquad 
	\tikzdiag{
\draw (0,0) ..controls (0,0.25) and (0.5,0.25) .. (0.5,0.5);
\draw (0,0.5) ..controls (0,0.25) and (0.5,0.25) .. (0.5,0);
\draw (-0.5,0) to (-0.5,0.5);
\draw (0.5,0.5) to (0.5,1);
\draw (-0.5,0.5) ..controls (-0.5,0.75) and (0,0.75) .. (0,1);
\draw (0,0.5) ..controls (0,0.75) and (-0.5,0.75) .. (-0.5,1);
\draw (-0.5,1) to (-0.5,1.5);
\draw (0,1) ..controls (0,1.25) and (0.5,1.25) .. (0.5,1.5);
\draw (0.5,1) ..controls (0.5,1.25) and (0,1.25) .. (0,1.5);
\node at (0.42,0.10) {$\bullet$};
	} \qquad \tikzdiag{
\draw (0,-0.5) ..controls (0,-0.25) and (0.5,-0.25) .. (0.5,0);
\draw (0.5,-0.5) ..controls (0.5,-0.25) and (0,-0.25) .. (0,0);
\draw (0,0) ..controls (0,0.25) and (0.5,0.25) .. (0.5,0.5);
\draw (0,0.5) ..controls (0,0.25) and (0.5,0.25) .. (0.5,0);
\draw (-0.5,-0.5) to (-0.5,0.5);
\draw (0.5,0.5) to (0.5,1);
\draw (-0.5,0.5) ..controls (-0.5,0.75) and (0,0.75) .. (0,1);
\draw (0,0.5) ..controls (0,0.75) and (-0.5,0.75) .. (-0.5,1);
\draw (-0.5,1) to (-0.5,1.5);
\draw (0,1) ..controls (0,1.25) and (0.5,1.25) .. (0.5,1.5);
\draw (0.5,1) ..controls (0.5,1.25) and (0,1.25) .. (0,1.5);
	} \qquad \tikzdiag{
\draw (0,0) ..controls (0,0.25) and (0.5,0.25) .. (0.5,0.5);
\draw (0,0.5) ..controls (0,0.25) and (0.5,0.25) .. (0.5,0);
\draw (-0.5,0) to (-0.5,0.5);
\draw (0.5,0.5) to (0.5,1);
\draw (-0.5,0.5) ..controls (-0.5,0.75) and (0,0.75) .. (0,1);
\draw (0,0.5) ..controls (0,0.75) and (-0.5,0.75) .. (-0.5,1);
\draw (-0.5,1) to (-0.5,2);
\draw (0,1) ..controls (0,1.25) and (0.5,1.25) .. (0.5,1.5);
\draw (0.5,1) ..controls (0.5,1.25) and (0,1.25) .. (0,1.5);
\draw (0,1.5) ..controls (0,1.75) and (0.5,1.75) .. (0.5,2);
\draw (0.5,1.5) ..controls (0.5,1.75) and (0,1.75) .. (0,2);
	} \qquad
\tikzdiag[scale=1,xscale=-1]{
	\draw   (0,1) to (0,1.5);
	\draw  (0.5,0.5) .. controls (0.5,0.75) and (0, 0.75) ..  (0,1);
	\draw  (0,0.5) .. controls (0,0.75) and (0.5, 0.75) ..  (0.5,1);
	\draw  (1,1) .. controls (1,1.25) and (0.5,1.25) .. (0.5,1.5);
	\draw (0.5,1) .. controls (0.5,1.25) and (1,1.25) .. (1,1.5);
	\draw (1,1.5) to (1,2);
	\draw (0.5,1.5) .. controls (0.5,1.75) and (0,1.75) .. (0,2);
	\draw (0,1.5) .. controls (0,1.75) and (0.5,1.75) .. (0.5,2);
	\draw (0,0.5) to (0,0);
	\draw  (0.5,0) .. controls (0.5,0.25) and (1, 0.25) ..  (1,0.5);
	\draw  (1,0) .. controls (1,0.25) and (0.5, 0.25) ..  (0.5,0.5);
	\draw  (0.5,0) .. controls (0.5,0.25) and (1, 0.25) ..  (1,0.5);
	\draw  (0.5,-0.5) .. controls (0.5,-0.25) and (0,-0.25) ..  (0,0);
	\draw  (0,-0.5) .. controls (0,-0.25) and (0.5,-0.25) ..  (0.5,0);
	\draw (1,0) to (1,-0.5);
	\draw (1,0.5) to (1,1);
}
\]
and left-indexed critical branchings given by the overlapping of the Reidemeister $3$ relation with itself (the orientation of the indexation depends on the orientation of the Reidemeister $3$-relation):
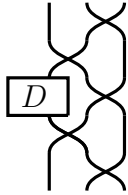
\[ \tikzdiag[scale=1,xscale=-1]{
	\draw   (0,1) to (0,1.5);
	\draw  (0.5,0.5) .. controls (0.5,0.75) and (0, 0.75) ..  (0,1);
	\draw  (0,0.5) .. controls (0,0.75) and (0.5, 0.75) ..  (0.5,1);
	\draw  (1,1) .. controls (1,1.25) and (0.5,1.25) .. (0.5,1.5);
	\draw (0.5,1) .. controls (0.5,1.25) and (1,1.25) .. (1,1.5);
	\draw (1,1.5) to (1,2);
	\draw (0.5,1.5) .. controls (0.5,1.75) and (0,1.75) .. (0,2);
	\draw (0,1.5) .. controls (0,1.75) and (0.5,1.75) .. (0.5,2);
	\draw (0,0.5) to (0,0);
	\draw  (0.5,0) .. controls (0.5,0.25) and (1, 0.25) ..  (1,0.5);
	\draw  (1,0) .. controls (1,0.25) and (0.5, 0.25) ..  (0.5,0.5);
	\draw  (0.5,0) .. controls (0.5,0.25) and (1, 0.25) ..  (1,0.5);
	\draw  (0.5,-0.5) .. controls (0.5,-0.25) and (0,-0.25) ..  (0,0);
	\draw  (0,-0.5) .. controls (0,-0.25) and (0.5,-0.25) ..  (0.5,0);
	\draw (1,0) to (1,-0.5);
	\draw (0.75,0.5) to (1.55,0.5) to (1.55,1) to (0.75,1) to (0.75,0.5);
	\node at (1.2,0.75) {$D$};
} 
 \]
for any monomial $D$. Following \cite{DUP21}, it suffices to check the confluence of these indexed critical branchings for 
\[ D = \tikzdiag{
\draw (0,0) ..controls (0,0.5) and (1,0.5) .. (1,1);
\draw (1,0) ..controls (1,0.5) and (0,0.5) .. (0,1);
\node at (0.9,0.75) {$\bullet$};
\node at (1.15,0.75) {$p$};
}, \qquad \qquad \tikzdiag{
\draw (0,0) to (0,1);
\node at (0,0.5) {$\bullet$};
\node at (0.2,0.5) {$p$};
} \quad \text{for any $p \in \mathbb{N}$}. \]
One proves following the proof of convergence for the KLR algebras of \cite{DUP21}, that all these critical branchings are confluent modulo braid-like isotopies. As a consequence, $P$ is a convergent presentation of $\nh_n$ and the monomials in normal form with respect to $P$ yield a linear basis of $\nh_n$, recovering the usual basis for the nilHecke algebra (see for example~\cite[Section~2.3]{KL1}).

\end{exe}

\input{sections/basisthmrewriting.tex}

\section{Additional computations}\label{sec:computations}

This appendix contains some extra computations that are helpful for some proofs in the main text and the other appendices. 

\begin{lem}\label{lem:dotslideoversevercrossings}
We have
\[
\tikzdiagl{
	\draw (0,0) .. controls (0,.5) and (.5,.5) .. (.5,1);
	\draw (1,0) .. controls (1,.5) and (1.5,.5) .. (1.5,1);
	\node at(.5,.1) {\small $\dots$};
	\node at(1,.9) {\small $\dots$};
	\tikzbrace{0}{1}{0}{\small $k$};
	\draw (1.5,0) .. controls (1.5,.5) and (0,.5) .. (0,1) 
		node[pos=.8, tikzdot]{} node[pos=.8, xshift=-1.5ex, yshift=-.75ex]{\small $u$};
}
\ = \
\tikzdiagl{
	\draw (0,0) .. controls (0,.5) and (.5,.5) .. (.5,1);
	\draw (1,0) .. controls (1,.5) and (1.5,.5) .. (1.5,1);
	\node at(.5,.1) {\small $\dots$};
	\node at(1,.9) {\small $\dots$};
	\tikzbrace{0}{1}{0}{\small $k$};
	\draw (1.5,0) .. controls (1.5,.5) and (0,.5) .. (0,1) 
		node[pos=.2, tikzdot]{} node[pos=.2, xshift=1.5ex, yshift=.75ex]{\small $u$};
}
\ +
\sum_{\ell=0}^{k-1} \sssum{s+t=\\u-1} \ 
\tikzdiagl{
	\draw (0,0) .. controls (0,.5) and (.5,.5) .. (.5,1);
	\draw (1,0) .. controls (1,.5) and (1.5,.5) .. (1.5,1);
	\node at(.5,.1) {\small $\dots$};
	\node at(1,.9) {\small $\dots$};
	\tikzbrace{0}{1}{0}{\small $\ell$};
	\draw (1.5,0) .. controls (1.5,.5) and (0,.5) .. (0,1) 
		node[pos=.8, tikzdot]{} node[pos=.8, xshift=-1.5ex, yshift=-.75ex]{\small $s$};
	\draw (2,0) .. controls (2,.5) and (2.5,.5) .. (2.5,1);
	\draw (3,0) .. controls (3,.5) and (3.5,.5) .. (3.5,1);
	\node at(2.5,.1) {\small $\dots$};
	\node at(3,.9) {\small $\dots$};
	\tikzbrace{2}{3}{0}{\small $k-\ell-1$};
	\draw (3.5,0) .. controls (3.5,.5) and (2,.5) .. (2,1) 
		node[pos=.2, tikzdot]{} node[pos=.2, xshift=1.5ex, yshift=.75ex]{\small $t$};
}
\]
for all $u,k \geq 0$. 
\end{lem}

\begin{proof}
It follows from applying the relations in \cref{eq:nhdotslide} recursively. 
\end{proof}

\begin{lem}\label{lem:dotslidenails}
We have
\begin{align*}
\tikzdiagl{
	%
	\draw (.5,0) .. controls (.5,.25) .. (0,.5) .. controls (2,1.5) .. (2,2.5);
	\draw (1.5,0) .. controls (1.5,.75) .. (0,1.5) .. controls (1,2) .. (1,2.5);
	\node at(1,.1) {\small $\dots$};
	\node at(1.5,2.4) {\small $\dots$};
	\node at(.25,1.1) {\small $\vdots$};
	\tikzbrace{.5}{1.5}{0}{\small $k$};
	\draw (2,0) .. controls (2,1.5) .. (0,2) node[pos=.2,tikzdot]{} .. controls (.5,2.25) .. (.5,2.5);
	\draw [pstdhl] (0,0) node[below]{\small $\mu$} -- (0,2.5) node[pos=.2,nail]{}  node[pos=.6,nail]{} node[pos=.8,nail]{}  ;
}
\ = \
\tikzdiagl{
	%
	\draw (.5,0) .. controls (.5,.25) .. (0,.5) .. controls (2,1.5) .. (2,2.5);
	\draw (1.5,0) .. controls (1.5,.75) .. (0,1.5) .. controls (1,2) .. (1,2.5);
	\node at(1,.1) {\small $\dots$};
	\node at(1.5,2.4) {\small $\dots$};
	\node at(.25,1.1) {\small $\vdots$};
	\tikzbrace{.5}{1.5}{0}{\small $k$};
	\draw (2,0) .. controls (2,1.5) .. (0,2).. controls (.5,2.25) .. (.5,2.5)  node[pos=.5,tikzdot]{} ;
	\draw [pstdhl] (0,0) node[below]{\small $\mu$} -- (0,2.5) node[pos=.2,nail]{}  node[pos=.6,nail]{} node[pos=.8,nail]{}  ;
}
\end{align*}
\end{lem}

\begin{proof}
The statement follows from \cref{eq:nhdotslide}, \cref{eq:nhR2andR3} and \cref{eq:nailsrel}.
\end{proof}

\begin{lem}\label{lem:Ndotsonestep}
We have
\begin{align*}
&\tikzdiagl{
	\draw (.5,0) .. controls (.5,.75) and (2.5,.75) .. (2.5,2.5) node[pos=.6, tikzdot]{} node[pos=.6,xshift=1.5ex,yshift=-.75ex]{\small $N$};
	\draw (1,0) .. controls (1,.25) .. (0,.5) .. controls (2,1.5) .. (2,2.5);
	\draw (2,0) .. controls (2,.75) .. (0,1.5) .. controls (1,2) .. (1,2.5);
	\node at(1.5,.1) {\small $\dots$};
	\node at(1.5,2.4) {\small $\dots$};
	\node at(.25,1.1) {\small $\vdots$};
	\tikzbrace{1}{2}{0}{\small $k$};
	\draw (2.5,0) .. controls (2.5,1.5) .. (0,2) .. controls (.5,2.25) .. (.5,2.5);
	\draw [pstdhl] (0,0) node[below]{\small $\mu$} -- (0,2.5) node[pos=.2,nail]{}  node[pos=.6,nail]{} node[pos=.8,nail]{}  ;
}
+ \sum_{\ell=0}^{k-1} (-1)^{\ell} \sssum{u+v\\=N-1}
\tikzdiagl{
	\draw (.5,0) .. controls (.5,.25) .. (0,.5) .. controls (3.5,1.5) .. (3.5,2) node[pos=1, tikzdot]{} node[pos=1, xshift=-1.5ex, yshift=.75ex]{\small $u$}
		.. controls (3.5,2.5) and (2.5,2.5) .. (2.5,4.5);
	\draw (1,0) .. controls (1,.5) .. (0,1) .. controls (4,3) .. (4,4.5);
	\draw (2,0) .. controls (2,1) .. (0,2) .. controls (3,3.5) .. (3,4.5); 
	\node at(1.5,.1) {\small $\dots$};
	\node at(3.5,4.4) {\small $\dots$};
	\node at(.25,1.6) {\small $\vdots$};
	\tikzbrace{1}{2}{0}{\small $\ell$};
	%
	\draw (2.5,0) .. controls (2.5,1) and (4,1) .. (4,2) node[pos=1, tikzdot]{}  node[pos=1, xshift=1.5ex, yshift=.75ex]{\small $v$}
		 .. controls (4,3) and (4.5,3) .. (4.5,4.5);
	\draw (3,0) .. controls (3,1) .. (0,2.5) .. controls (2,3.5) .. (2,4.5);
	\draw (4,0) .. controls (4,1.5) .. (0,3.5) .. controls (1,4) .. (1,4.5); 
	\node at(3.5,.1) {\small $\dots$};
	\node at(1.5,4.4) {\small $\dots$};
	\node at(.25,3.1) {\small $\vdots$};
	\tikzbrace{3}{4}{0}{\small $k-\ell-1$};
	%
	\draw (4.5,0) -- (4.5,2) .. controls (4.5,3.5) .. (0,4) .. controls (.5,4.25) .. (.5,4.5);
	\draw [pstdhl] (0,0) node[below]{\small $\mu$} -- (0,4.5) 
		node[pos=.11,nail]{}  node[pos=.22,nail]{} node[pos=.44,nail]{} node[pos=.55,nail]{} node[pos=.77,nail]{} node[pos=.88,nail]{};
}
\\
\ =& \ 
\tikzdiagl{
	\draw (.5,0) .. controls (.5,.75) and (2.5,.75) .. (2.5,2.5) node[pos=.85, tikzdot]{} node[pos=.85,xshift=1.5ex,yshift=-.75ex]{\small $N$};
	\draw (1,0) .. controls (1,.25) .. (0,.5) .. controls (2,1.5) .. (2,2.5);
	\draw (2,0) .. controls (2,.75) .. (0,1.5) .. controls (1,2) .. (1,2.5);
	\node at(1.5,.1) {\small $\dots$};
	\node at(1.5,2.4) {\small $\dots$};
	\node at(.25,1.1) {\small $\vdots$};
	\tikzbrace{1}{2}{0}{\small $k$};
	\draw (2.5,0) .. controls (2.5,1.5) .. (0,2) .. controls (.5,2.25) .. (.5,2.5);
	\draw [pstdhl] (0,0) node[below]{\small $\mu$} -- (0,2.5) node[pos=.2,nail]{}  node[pos=.6,nail]{} node[pos=.8,nail]{}  ;
}
+ \sum_{\ell=0}^{k} (-1)^{\ell} \sssum{u+v\\=N-1}
\tikzdiagl{
	\draw (.5,0) .. controls (.5,.25) .. (0,.5) .. controls (3.5,1.5) .. (3.5,2) node[pos=1, tikzdot]{} node[pos=1, xshift=-1.5ex, yshift=.25ex]{\small $u$}
		.. controls (3.5,2.5) and (2,2.5) .. (2,4);
	\draw (1,0) .. controls (1,.5) .. (0,1) .. controls (3.5,3) .. (3.5,4);
	\draw (2,0) .. controls (2,1) .. (0,2) .. controls (2.5,3.5) .. (2.5,4); 
	\node at(1.5,.1) {\small $\dots$};
	\node at(3,3.9) {\small $\dots$};
	\node at(.25,1.6) {\small $\vdots$};
	\tikzbrace{1}{2}{0}{\small $\ell$};
	%
	\draw (2.5,0) .. controls (2.5,1) and (4,1) .. (4,2) node[pos=1, tikzdot]{}  node[pos=1, xshift=1.5ex, yshift=.25ex]{\small $v$}
		 .. controls (4,3) and (4,3) .. (4,4);
	\draw (3,0) .. controls (3,1) .. (0,2.5) .. controls (1.5,3.5) .. (1.5,4);
	\draw (4,0) .. controls (4,1.5) .. (0,3.5) .. controls (.5,3.75) .. (.5,4); 
	\node at(3.5,.1) {\small $\dots$};
	\node at(1,3.9) {\small $\dots$};
	\node at(.25,3.1) {\small $\vdots$};
	\tikzbrace{3}{4}{0}{\small $k-\ell$};
	\draw [pstdhl] (0,0) node[below]{\small $\mu$} -- (0,4) 
		node[pos=.125,nail]{}  node[pos=.25,nail]{} node[pos=.5,nail]{} node[pos=.625,nail]{} node[pos=.875,nail]{};
}
\end{align*}
\end{lem}

\begin{proof}
First, we rewrite the RHS of the equation in the statement as
\begin{align} \label{eq:RHSeqlem_Ndotsonestep}
\tikzdiagl[xscale=.9]{
	\draw (.5,0) .. controls (.5,.75) and (2.5,.75) .. (2.5,2.5) node[pos=.85, tikzdot]{} node[pos=.85,xshift=1.5ex,yshift=-.75ex]{\small $N$};
	\draw (1,0) .. controls (1,.25) .. (0,.5) .. controls (2,1.5) .. (2,2.5);
	\draw (2,0) .. controls (2,.75) .. (0,1.5) .. controls (1,2) .. (1,2.5);
	\node at(1.5,.1) {\small $\dots$};
	\node at(1.5,2.4) {\small $\dots$};
	\node at(.25,1.1) {\small $\vdots$};
	\tikzbrace{1}{2}{0}{\small $k$};
	\draw (2.5,0) .. controls (2.5,1.5) .. (0,2) .. controls (.5,2.25) .. (.5,2.5);
	\draw [pstdhl] (0,0) node[below]{\small $\mu$} -- (0,2.5) node[pos=.2,nail]{}  node[pos=.6,nail]{} node[pos=.8,nail]{}  ;
}
+
(-1)^k
\sssum{u+v\\=N-1}
\tikzdiagl[xscale=.9]{
	\draw (.5,0) .. controls (.5,.25) .. (0,.5) .. controls (2,1) .. (2,1.5) node[pos=1, tikzdot]{} node[pos=1, xshift=-1.5ex, yshift=-.25ex]{\small $u$}
		.. controls (2,2) and (.5,2) .. (.5,2.5);
	\draw (1,0) .. controls (1,.5) .. (0,1) .. controls (2,1.75) .. (2,2.5);
	\draw (2,0) .. controls (2,1) .. (0,2) .. controls (1,2.25) .. (1,2.5);
	\node at(1.5,.1) {\small $\dots$};
	\node at(1.5,2.4) {\small $\dots$};
	\node at(.25,1.6) {\small $\vdots$};
	\tikzbrace{1}{2}{0}{\small $k$};
	\draw (2.5,0) -- (2.5,2.5)  node[pos=.6, tikzdot]{} node[pos=.6, xshift=1.5ex, yshift=-.25ex]{\small $v$};
	\draw [pstdhl] (0,0) node[below]{\small $\mu$} -- (0,2.5) node[pos=.2,nail]{}  node[pos=.4,nail]{} node[pos=.8,nail]{}  ;
}
+ \sum_{\ell=0}^{k-1} (-1)^{\ell} \sssum{u+v\\=N-1}
\tikzdiagl[xscale=.9]{
	\draw (.5,0) .. controls (.5,.25) .. (0,.5) .. controls (4,2) .. (4,2.5) node[pos=1, tikzdot]{} node[pos=1, xshift=-1.5ex, yshift=0ex]{\small $u$}
		.. controls (4,3.5) and (2.5,3.5) .. (2.5,4.5);
	\draw (1,0) .. controls (1,.5) .. (0,1) .. controls (4,3) .. (4,4.5);
	\draw (2,0) .. controls (2,1) .. (0,2) .. controls (3,3.5) .. (3,4.5); 
	\node at(1.5,.1) {\small $\dots$};
	\node at(3.5,4.4) {\small $\dots$};
	\node at(.25,1.6) {\small $\vdots$};
	\tikzbrace{1}{2}{0}{\small $\ell$};
	%
	\draw (2.5,0) .. controls (2.5,1) and (4.5,1) .. (4.5,2.5) node[pos=1, tikzdot]{}  node[pos=1, xshift=1.5ex, yshift=0ex]{\small $v$}
		 .. controls (4.5,3) and (4.5,3) .. (4.5,4.5);
	\draw (3,0) .. controls (3,1) .. (0,2.5) .. controls (2,3.5) .. (2,4.5);
	\draw (4,0) .. controls (4,1.5) .. (0,3.5) .. controls (1,4) .. (1,4.5); 
	\node at(3.5,.1) {\small $\dots$};
	\node at(1.5,4.4) {\small $\dots$};
	\node at(.25,3.1) {\small $\vdots$};
	\tikzbrace{3}{4}{0}{\small $k-\ell-1$};
	\draw (4.5,0) .. controls (4.5,1.75) .. (0,4) .. controls (.5,4.25) .. (.5,4.5);
	%
	\draw [pstdhl] (0,0) node[below]{\small $\mu$} -- (0,4.5) 
		node[pos=.11,nail]{}  node[pos=.22,nail]{} node[pos=.44,nail]{} node[pos=.55,nail]{} node[pos=.77,nail]{} node[pos=.88,nail]{};
}
\end{align}
Then we compute
\begin{align}\label{eq:lem_Ndotsonestep_eq2}
\tikzdiagl{
	\draw (.5,0) .. controls (.5,.75) and (2.5,.75) .. (2.5,2.5) node[pos=.85, tikzdot]{} node[pos=.85,xshift=1.5ex,yshift=-.75ex]{\small $N$};
	\draw (1,0) .. controls (1,.25) .. (0,.5) .. controls (2,1.5) .. (2,2.5);
	\draw (2,0) .. controls (2,.75) .. (0,1.5) .. controls (1,2) .. (1,2.5);
	\node at(1.5,.1) {\small $\dots$};
	\node at(1.5,2.4) {\small $\dots$};
	\node at(.25,1.1) {\small $\vdots$};
	\tikzbrace{1}{2}{0}{\small $k$};
	\draw (2.5,0) .. controls (2.5,1.5) .. (0,2) .. controls (.5,2.25) .. (.5,2.5);
	\draw [pstdhl] (0,0) node[below]{\small $\mu$} -- (0,2.5) node[pos=.2,nail]{}  node[pos=.6,nail]{} node[pos=.8,nail]{}  ;
}
\ &= \ 
\tikzdiagl{
	\draw (.5,0) .. controls (.5,.75) and (2.5,.75) .. (2.5,2.5) node[pos=.6, tikzdot]{} node[pos=.6,xshift=1.5ex,yshift=-.75ex]{\small $N$};
	\draw (1,0) .. controls (1,.25) .. (0,.5) .. controls (2,1.5) .. (2,2.5);
	\draw (2,0) .. controls (2,.75) .. (0,1.5) .. controls (1,2) .. (1,2.5);
	\node at(1.5,.1) {\small $\dots$};
	\node at(1.5,2.4) {\small $\dots$};
	\node at(.25,1.1) {\small $\vdots$};
	\tikzbrace{1}{2}{0}{\small $k$};
	\draw (2.5,0) .. controls (2.5,1.5) .. (0,2) .. controls (.5,2.25) .. (.5,2.5);
	\draw [pstdhl] (0,0) node[below]{\small $\mu$} -- (0,2.5) node[pos=.2,nail]{}  node[pos=.6,nail]{} node[pos=.8,nail]{}  ;
}
\ - \
(-1)^k
 \sssum{u+v=\\N-1}
\tikzdiagl{
	\draw (.5,0) .. controls (.5,.25) .. (0,.5) .. controls (2,1) .. (2,1.5) 
		.. controls (2,2) and (.5,2) .. (.5,2.5) node[pos=.9, tikzdot]{} node[pos=.9, xshift=-1.5ex, yshift=-.25ex]{\small $u$};
	\draw (1,0) .. controls (1,.5) .. (0,1) .. controls (2,1.75) .. (2,2.5);
	\draw (2,0) .. controls (2,1) .. (0,2) .. controls (1,2.25) .. (1,2.5);
	\node at(1.5,.1) {\small $\dots$};
	\node at(1.5,2.4) {\small $\dots$};
	\node at(.25,1.6) {\small $\vdots$};
	\tikzbrace{1}{2}{0}{\small $k$};
	\draw (2.5,0) -- (2.5,2.5)  node[pos=.6, tikzdot]{} node[pos=.6, xshift=1.5ex, yshift=-.25ex]{\small $v$};
	\draw [pstdhl] (0,0) node[below]{\small $\mu$} -- (0,2.5) node[pos=.2,nail]{}  node[pos=.4,nail]{} node[pos=.8,nail]{}  ;
}
\end{align}
using first \cref{eq:nhdotslide}, then \cref{lem:dotslidenails} and finally \cref{eq:nailsrel}.

We also compute
\begin{equation}\label{eq:lem_Ndotsonestep_eq3}
\tikzdiagl[xscale=.9]{
	\draw (.5,0) .. controls (.5,.25) .. (0,.5) .. controls (4,2) .. (4,2.5) node[pos=1, tikzdot]{} node[pos=1, xshift=-1.5ex, yshift=0ex]{\small $u$}
		.. controls (4,3.5) and (2.5,3.5) .. (2.5,4.5);
	\draw (1,0) .. controls (1,.5) .. (0,1) .. controls (4,3) .. (4,4.5);
	\draw (2,0) .. controls (2,1) .. (0,2) .. controls (3,3.5) .. (3,4.5); 
	\node at(1.5,.1) {\small $\dots$};
	\node at(3.5,4.4) {\small $\dots$};
	\node at(.25,1.6) {\small $\vdots$};
	\tikzbrace{1}{2}{0}{\small $\ell$};
	%
	\draw (2.5,0) .. controls (2.5,1) and (4.5,1) .. (4.5,2.5) node[pos=1, tikzdot]{}  node[pos=1, xshift=1.5ex, yshift=0ex]{\small $v$}
		 .. controls (4.5,3) and (4.5,3) .. (4.5,4.5);
	\draw (3,0) .. controls (3,1) .. (0,2.5) .. controls (2,3.5) .. (2,4.5);
	\draw (4,0) .. controls (4,1.5) .. (0,3.5) .. controls (1,4) .. (1,4.5); 
	\node at(3.5,.1) {\small $\dots$};
	\node at(1.5,4.4) {\small $\dots$};
	\node at(.25,3.1) {\small $\vdots$};
	\tikzbrace{3}{4}{0}{\small $k-\ell-1$};
	\draw (4.5,0) .. controls (4.5,1.75) .. (0,4) .. controls (.5,4.25) .. (.5,4.5);
	%
	\draw [pstdhl] (0,0) node[below]{\small $\mu$} -- (0,4.5) 
		node[pos=.11,nail]{}  node[pos=.22,nail]{} node[pos=.44,nail]{} node[pos=.55,nail]{} node[pos=.77,nail]{} node[pos=.88,nail]{};
}
\ = \ 
\tikzdiagl[xscale=.9]{
	\draw (.5,0) .. controls (.5,.25) .. (0,.5) .. controls (4,2) .. (4,2.5) node[pos=1, tikzdot]{} node[pos=1, xshift=-1.5ex, yshift=0ex]{\small $u$}
		.. controls (4,3.5) and (2.5,3.5) .. (2.5,4.5);
	\draw (1,0) .. controls (1,.5) .. (0,1) .. controls (4,3) .. (4,4.5);
	\draw (2,0) .. controls (2,1) .. (0,2) .. controls (3,3.5) .. (3,4.5); 
	\node at(1.5,.1) {\small $\dots$};
	\node at(3.5,4.4) {\small $\dots$};
	\node at(.25,1.6) {\small $\vdots$};
	\tikzbrace{1}{2}{0}{\small $\ell$};
	%
	\draw (2.5,0) .. controls (2.5,1) and (4.5,1) .. (4.5,2.5) node[pos=.65, tikzdot]{}  node[pos=.65, xshift=1.5ex, yshift=-.75ex]{\small $v$}
		 .. controls (4.5,3) and (4.5,3) .. (4.5,4.5);
	\draw (3,0) .. controls (3,1) .. (0,2.5) .. controls (2,3.5) .. (2,4.5);
	\draw (4,0) .. controls (4,1.5) .. (0,3.5) .. controls (1,4) .. (1,4.5); 
	\node at(3.5,.1) {\small $\dots$};
	\node at(1.5,4.4) {\small $\dots$};
	\node at(.25,3.1) {\small $\vdots$};
	\tikzbrace{3}{4}{0}{\small $k-\ell-1$};
	\draw (4.5,0) .. controls (4.5,1.75) .. (0,4) .. controls (.5,4.25) .. (.5,4.5);
	%
	\draw [pstdhl] (0,0) node[below]{\small $\mu$} -- (0,4.5) 
		node[pos=.11,nail]{}  node[pos=.22,nail]{} node[pos=.44,nail]{} node[pos=.55,nail]{} node[pos=.77,nail]{} node[pos=.88,nail]{};
}
\ - \sssum{t+s=\\v-1} \ 
\tikzdiagl[xscale=.9]{
	\draw (.5,0) .. controls (.5,.25) .. (0,.5) .. controls (4,2) .. (4,2.5) node[pos=1, tikzdot]{} node[pos=1, xshift=-1.5ex, yshift=0ex]{\small $u$}
		.. controls (4,3.5) and (2.5,3.5) .. (2.5,4.5);
	\draw (1,0) .. controls (1,.5) .. (0,1) .. controls (4,3) .. (4,4.5);
	\draw (2,0) .. controls (2,1) .. (0,2) .. controls (3,3.5) .. (3,4.5); 
	\node at(1.5,.1) {\small $\dots$};
	\node at(3.5,4.4) {\small $\dots$};
	\node at(.25,1.6) {\small $\vdots$};
	\tikzbrace{1}{2}{0}{\small $\ell$};
	%
	\draw (2.5, 0) .. controls (2.5,.75) and (4,.75) .. (4,1.5) node[pos=1, tikzdot]{} node[pos=1, xshift=-1.5ex, yshift=.5ex]{\small $s$}
		.. controls (4,2) .. (0,4) .. controls (.5,4.25) .. (.5,4.5);
	\draw (3,0) .. controls (3,1) .. (0,2.5) .. controls (2,3.5) .. (2,4.5);
	\draw (4,0) .. controls (4,1.5) .. (0,3.5) .. controls (1,4) .. (1,4.5); 
	\node at(3.5,.1) {\small $\dots$};
	\node at(1.5,4.4) {\small $\dots$};
	\node at(.25,3.1) {\small $\vdots$};
	\tikzbrace{3}{4}{0}{\small $k-\ell-1$};
	%
	\draw (4.5,0) -- (4.5,4.5) node[pos=.33, tikzdot]{} node[pos=.33, xshift=1.5ex, yshift=.75ex]{\small $t$};
	\draw [pstdhl] (0,0) node[below]{\small $\mu$} -- (0,4.5) 
		node[pos=.11,nail]{}  node[pos=.22,nail]{} node[pos=.44,nail]{} node[pos=.55,nail]{} node[pos=.77,nail]{} node[pos=.88,nail]{};
}
\end{equation}
and for similar reasons as in \cref{eq:lem_Ndotsonestep_eq2} we have
\[
\tikzdiagl{
	\draw (.5,0) .. controls (.5,.25) .. (0,.5) .. controls (4,2) .. (4,2.5) node[pos=1, tikzdot]{} node[pos=1, xshift=-1.5ex, yshift=0ex]{\small $u$}
		.. controls (4,3.5) and (2.5,3.5) .. (2.5,4.5);
	\draw (1,0) .. controls (1,.5) .. (0,1) .. controls (4,3) .. (4,4.5);
	\draw (2,0) .. controls (2,1) .. (0,2) .. controls (3,3.5) .. (3,4.5); 
	\node at(1.5,.1) {\small $\dots$};
	\node at(3.5,4.4) {\small $\dots$};
	\node at(.25,1.6) {\small $\vdots$};
	\tikzbrace{1}{2}{0}{\small $\ell$};
	%
	\draw (2.5, 0) .. controls (2.5,.75) and (4,.75) .. (4,1.5) node[pos=1, tikzdot]{} node[pos=1, xshift=-1.5ex, yshift=.5ex]{\small $s$}
		.. controls (4,2) .. (0,4) .. controls (.5,4.25) .. (.5,4.5);
	\draw (3,0) .. controls (3,1) .. (0,2.5) .. controls (2,3.5) .. (2,4.5);
	\draw (4,0) .. controls (4,1.5) .. (0,3.5) .. controls (1,4) .. (1,4.5); 
	\node at(3.5,.1) {\small $\dots$};
	\node at(1.5,4.4) {\small $\dots$};
	\node at(.25,3.1) {\small $\vdots$};
	\tikzbrace{3}{4}{0}{\small $k-\ell-1$};
	%
	\draw (4.5,0) -- (4.5,4.5) node[pos=.33, tikzdot]{} node[pos=.33, xshift=1.5ex, yshift=.75ex]{\small $t$};
	\draw [pstdhl] (0,0) node[below]{\small $\mu$} -- (0,4.5) 
		node[pos=.11,nail]{}  node[pos=.22,nail]{} node[pos=.44,nail]{} node[pos=.55,nail]{} node[pos=.77,nail]{} node[pos=.88,nail]{};
}
\ = \ 
(-1)^{k-\ell-1}
\tikzdiagl{
	\draw (.5,0) .. controls (.5,.25) .. (0,.5) .. controls (4,2) .. (4,2.5) node[pos=1, tikzdot]{} node[pos=1, xshift=-1.5ex, yshift=0ex]{\small $u$}
		.. controls (4,3.5) and (2.5,3.5) .. (2.5,4.5);
	\draw (1,0) .. controls (1,.5) .. (0,1) .. controls (4,3) .. (4,4.5);
	\draw (2,0) .. controls (2,1) .. (0,2) .. controls (3,3.5) .. (3,4.5); 
	\node at(1.5,.1) {\small $\dots$};
	\node at(3.5,4.4) {\small $\dots$};
	\node at(.25,1.6) {\small $\vdots$};
	\tikzbrace{1}{2}{0}{\small $\ell$};
	%
	\draw (2.5,0) .. controls (2.5,1.25) .. (0,2.5)	
		.. controls (2,3.25) .. (2,3.5)
		.. controls  (2,4) and (.5,4) .. (.5,4.5) node[pos=.9, tikzdot]{} node[pos=.9,xshift=-1.5ex,yshift=0ex]{\small $s$};
	\draw (3,0) .. controls (3,1.5) .. (0,3) .. controls (2,3.75) .. (2,4.5);
	\draw (4,0) .. controls (4,2) .. (0,4) .. controls (1,4.25) .. (1,4.5); 
	\node at(3.5,.1) {\small $\dots$};
	\node at(1.5,4.4) {\small $\dots$};
	\node at(.25,3.6) {\small $\vdots$};
	\tikzbrace{3}{4}{0}{\small $k-\ell-1$};
	%
	\draw (4.5,0) -- (4.5,4.5) node[pos=.33, tikzdot]{} node[pos=.33, xshift=1.5ex, yshift=.75ex]{\small $t$};
	\draw [pstdhl] (0,0) node[below]{\small $\mu$} -- (0,4.5) 
		node[pos=.11,nail]{}  node[pos=.22,nail]{} node[pos=.44,nail]{} node[pos=.55,nail]{} node[pos=.66,nail]{} node[pos=.88,nail]{};
}
\]
Therefore, by \cref{lem:dotslideoversevercrossings}, the rightmost term of \cref{eq:lem_Ndotsonestep_eq2} together with the the rightmost term of \cref{eq:lem_Ndotsonestep_eq3} gives
\[
\ - \
(-1)^k
 \sssum{u+v=\\N-1}
\tikzdiagl{
	\draw (.5,0) .. controls (.5,.25) .. (0,.5) .. controls (2,1) .. (2,1.5) 
		.. controls (2,2) and (.5,2) .. (.5,2.5) node[pos=0, tikzdot]{} node[pos=0, xshift=-1.5ex, yshift=-.25ex]{\small $u$};
	\draw (1,0) .. controls (1,.5) .. (0,1) .. controls (2,1.75) .. (2,2.5);
	\draw (2,0) .. controls (2,1) .. (0,2) .. controls (1,2.25) .. (1,2.5);
	\node at(1.5,.1) {\small $\dots$};
	\node at(1.5,2.4) {\small $\dots$};
	\node at(.25,1.6) {\small $\vdots$};
	\tikzbrace{1}{2}{0}{\small $k$};
	\draw (2.5,0) -- (2.5,2.5)  node[pos=.6, tikzdot]{} node[pos=.6, xshift=1.5ex, yshift=-.25ex]{\small $v$};
	\draw [pstdhl] (0,0) node[below]{\small $\mu$} -- (0,2.5) node[pos=.2,nail]{}  node[pos=.4,nail]{} node[pos=.8,nail]{}  ;
}
\]
These elements cancel with the middle terms of \cref{eq:RHSeqlem_Ndotsonestep}, so that what remains is
\[
\tikzdiagl{
	\draw (.5,0) .. controls (.5,.75) and (2.5,.75) .. (2.5,2.5) node[pos=.6, tikzdot]{} node[pos=.6,xshift=1.5ex,yshift=-.75ex]{\small $N$};
	\draw (1,0) .. controls (1,.25) .. (0,.5) .. controls (2,1.5) .. (2,2.5);
	\draw (2,0) .. controls (2,.75) .. (0,1.5) .. controls (1,2) .. (1,2.5);
	\node at(1.5,.1) {\small $\dots$};
	\node at(1.5,2.4) {\small $\dots$};
	\node at(.25,1.1) {\small $\vdots$};
	\tikzbrace{1}{2}{0}{\small $k$};
	\draw (2.5,0) .. controls (2.5,1.5) .. (0,2) .. controls (.5,2.25) .. (.5,2.5);
	\draw [pstdhl] (0,0) node[below]{\small $\mu$} -- (0,2.5) node[pos=.2,nail]{}  node[pos=.6,nail]{} node[pos=.8,nail]{}  ;
}
+ \sum_{\ell=0}^{k-1} (-1)^{\ell} \sssum{u+v\\=N-1}
\tikzdiagl{
	\draw (.5,0) .. controls (.5,.25) .. (0,.5) .. controls (4,2) .. (4,2.5) node[pos=1, tikzdot]{} node[pos=1, xshift=-1.5ex, yshift=0ex]{\small $u$}
		.. controls (4,3.5) and (2.5,3.5) .. (2.5,4.5);
	\draw (1,0) .. controls (1,.5) .. (0,1) .. controls (4,3) .. (4,4.5);
	\draw (2,0) .. controls (2,1) .. (0,2) .. controls (3,3.5) .. (3,4.5); 
	\node at(1.5,.1) {\small $\dots$};
	\node at(3.5,4.4) {\small $\dots$};
	\node at(.25,1.6) {\small $\vdots$};
	\tikzbrace{1}{2}{0}{\small $\ell$};
	%
	\draw (2.5,0) .. controls (2.5,1) and (4.5,1) .. (4.5,2.5) node[pos=.65, tikzdot]{}  node[pos=.65, xshift=1.5ex, yshift=-.75ex]{\small $v$}
		 .. controls (4.5,3) and (4.5,3) .. (4.5,4.5);
	\draw (3,0) .. controls (3,1) .. (0,2.5) .. controls (2,3.5) .. (2,4.5);
	\draw (4,0) .. controls (4,1.5) .. (0,3.5) .. controls (1,4) .. (1,4.5); 
	\node at(3.5,.1) {\small $\dots$};
	\node at(1.5,4.4) {\small $\dots$};
	\node at(.25,3.1) {\small $\vdots$};
	\tikzbrace{3}{4}{0}{\small $k-\ell-1$};
	\draw (4.5,0) .. controls (4.5,1.75) .. (0,4) .. controls (.5,4.25) .. (.5,4.5);
	%
	\draw [pstdhl] (0,0) node[below]{\small $\mu$} -- (0,4.5) 
		node[pos=.11,nail]{}  node[pos=.22,nail]{} node[pos=.44,nail]{} node[pos=.55,nail]{} node[pos=.77,nail]{} node[pos=.88,nail]{};
}
\]

We compute
\[
\tikzdiagl{
	\draw (.5,0) .. controls (.5,.25) .. (0,.5) .. controls (4,2) .. (4,2.5) node[pos=1, tikzdot]{} node[pos=1, xshift=-1.5ex, yshift=0ex]{\small $u$}
		.. controls (4,3.5) and (2.5,3.5) .. (2.5,4.5);
	\draw (1,0) .. controls (1,.5) .. (0,1) .. controls (4,3) .. (4,4.5);
	\draw (2,0) .. controls (2,1) .. (0,2) .. controls (3,3.5) .. (3,4.5); 
	\node at(1.5,.1) {\small $\dots$};
	\node at(3.5,4.4) {\small $\dots$};
	\node at(.25,1.6) {\small $\vdots$};
	\tikzbrace{1}{2}{0}{\small $\ell$};
	%
	\draw (2.5,0) .. controls (2.5,1) and (4.5,1) .. (4.5,2.5) node[pos=.65, tikzdot]{}  node[pos=.65, xshift=1.5ex, yshift=-.75ex]{\small $v$}
		 .. controls (4.5,3) and (4.5,3) .. (4.5,4.5);
	\draw (3,0) .. controls (3,1) .. (0,2.5) .. controls (2,3.5) .. (2,4.5);
	\draw (4,0) .. controls (4,1.5) .. (0,3.5) .. controls (1,4) .. (1,4.5); 
	\node at(3.5,.1) {\small $\dots$};
	\node at(1.5,4.4) {\small $\dots$};
	\node at(.25,3.1) {\small $\vdots$};
	\tikzbrace{3}{4}{0}{\small $k-\ell-1$};
	\draw (4.5,0) .. controls (4.5,1.75) .. (0,4) .. controls (.5,4.25) .. (.5,4.5);
	%
	\draw [pstdhl] (0,0) node[below]{\small $\mu$} -- (0,4.5) 
		node[pos=.11,nail]{}  node[pos=.22,nail]{} node[pos=.44,nail]{} node[pos=.55,nail]{} node[pos=.77,nail]{} node[pos=.88,nail]{};
}
\ = \ 
\tikzdiagl{
	\draw (.5,0) .. controls (.5,.25) .. (0,.5) .. controls (4,2) .. (4,2.5) node[pos=.5, tikzdot]{} node[pos=.5, xshift=-1.5ex, yshift=.75ex]{\small $u$}
		.. controls (4,3.5) and (2.5,3.5) .. (2.5,4.5);
	\draw (1,0) .. controls (1,.5) .. (0,1) .. controls (4,3) .. (4,4.5);
	\draw (2,0) .. controls (2,1) .. (0,2) .. controls (3,3.5) .. (3,4.5); 
	\node at(1.5,.1) {\small $\dots$};
	\node at(3.5,4.4) {\small $\dots$};
	\node at(.25,1.6) {\small $\vdots$};
	\tikzbrace{1}{2}{0}{\small $\ell$};
	%
	\draw (2.5,0) .. controls (2.5,1) and (4.5,1) .. (4.5,2.5) node[pos=.65, tikzdot]{}  node[pos=.65, xshift=1.5ex, yshift=-.75ex]{\small $v$}
		 .. controls (4.5,3) and (4.5,3) .. (4.5,4.5);
	\draw (3,0) .. controls (3,1) .. (0,2.5) .. controls (2,3.5) .. (2,4.5);
	\draw (4,0) .. controls (4,1.5) .. (0,3.5) .. controls (1,4) .. (1,4.5); 
	\node at(3.5,.1) {\small $\dots$};
	\node at(1.5,4.4) {\small $\dots$};
	\node at(.25,3.1) {\small $\vdots$};
	\tikzbrace{3}{4}{0}{\small $k-\ell-1$};
	\draw (4.5,0) .. controls (4.5,1.75) .. (0,4) .. controls (.5,4.25) .. (.5,4.5);
	%
	\draw [pstdhl] (0,0) node[below]{\small $\mu$} -- (0,4.5) 
		node[pos=.11,nail]{}  node[pos=.22,nail]{} node[pos=.44,nail]{} node[pos=.55,nail]{} node[pos=.77,nail]{} node[pos=.88,nail]{};
}
\ = \ 
\tikzdiagl{
	\draw (.5,0) .. controls (.5,.25) .. (0,.5) .. controls (3.5,1.5) .. (3.5,2) node[pos=1, tikzdot]{} node[pos=1, xshift=-1.5ex, yshift=.75ex]{\small $u$}
		.. controls (3.5,2.5) and (2.5,2.5) .. (2.5,4.5);
	\draw (1,0) .. controls (1,.5) .. (0,1) .. controls (4,3) .. (4,4.5);
	\draw (2,0) .. controls (2,1) .. (0,2) .. controls (3,3.5) .. (3,4.5); 
	\node at(1.5,.1) {\small $\dots$};
	\node at(3.5,4.4) {\small $\dots$};
	\node at(.25,1.6) {\small $\vdots$};
	\tikzbrace{1}{2}{0}{\small $\ell$};
	%
	\draw (2.5,0) .. controls (2.5,1) and (4,1) .. (4,2)  node[pos=1, tikzdot]{}  node[pos=1, xshift=1.5ex, yshift=.75ex]{\small $v$}
		 .. controls (4,3) and (4.5,3) .. (4.5,4.5);
	\draw (3,0) .. controls (3,1) .. (0,2.5) .. controls (2,3.5) .. (2,4.5);
	\draw (4,0) .. controls (4,1.5) .. (0,3.5) .. controls (1,4) .. (1,4.5); 
	\node at(3.5,.1) {\small $\dots$};
	\node at(1.5,4.4) {\small $\dots$};
	\node at(.25,3.1) {\small $\vdots$};
	\tikzbrace{3}{4}{0}{\small $k-\ell-1$};
	%
	\draw (4.5,0) -- (4.5,2) .. controls (4.5,3.5) .. (0,4) .. controls (.5,4.25) .. (.5,4.5);
	\draw [pstdhl] (0,0) node[below]{\small $\mu$} -- (0,4.5) 
		node[pos=.11,nail]{}  node[pos=.22,nail]{} node[pos=.44,nail]{} node[pos=.55,nail]{} node[pos=.77,nail]{} node[pos=.88,nail]{};
}
\]
using \cref{lem:dotslidenails} again. 
Putting all of the above together yields the equation in the statement. 
\end{proof}

\begin{prop}\label{prop:Ndotsslideovernails}
We have
\[
\tikzdiagl{
	\draw (.5,0) .. controls (.5,.75) and (2,.75) .. (2,2) node[pos=.1, tikzdot]{} node[pos=.1,xshift=-1.5ex,yshift=-.75ex]{\small $N$};
	\draw (1,0) .. controls (1,.25) .. (0,.5) .. controls (1.5,1.25) .. (1.5,2);
	\draw (2,0) .. controls (2,.75) .. (0,1.5) .. controls (.5,1.75) .. (.5,2);
	\node at(1.5,.1) {\small $\dots$};
	\node at(1,1.9) {\small $\dots$};
	\node at(.25,1.1) {\small $\vdots$};
	\tikzbrace{1}{2}{0}{\small $k$};
	%
	%
	\draw [pstdhl] (0,0) node[below]{\small $\mu$} -- (0,2) node[pos=.25,nail]{}  node[pos=.75,nail]{}  ;
}
\ = \ 
\tikzdiagl{
	\draw (.5,0) .. controls (.5,.75) and (2,.75) .. (2,2) node[pos=.85, tikzdot]{} node[pos=.85,xshift=1.5ex,yshift=-.75ex]{\small $N$};
	\draw (1,0) .. controls (1,.25) .. (0,.5) .. controls (1.5,1.25) .. (1.5,2);
	\draw (2,0) .. controls (2,.75) .. (0,1.5) .. controls (.5,1.75) .. (.5,2);
	\node at(1.5,.1) {\small $\dots$};
	\node at(1,1.9) {\small $\dots$};
	\node at(.25,1.1) {\small $\vdots$};
	\tikzbrace{1}{2}{0}{\small $k$};
	%
	%
	\draw [pstdhl] (0,0) node[below]{\small $\mu$} -- (0,2) node[pos=.25,nail]{}  node[pos=.75,nail]{}  ;
}
+ \sum_{\ell=0}^{k-1} (-1)^{\ell} \sssum{u+v\\=N-1}
\tikzdiagl{
	\draw (.5,0) .. controls (.5,.25) .. (0,.5) .. controls (3.5,1.5) .. (3.5,2) node[pos=1, tikzdot]{} node[pos=1, xshift=-1.5ex, yshift=.25ex]{\small $u$}
		.. controls (3.5,2.5) and (2,2.5) .. (2,4);
	\draw (1,0) .. controls (1,.5) .. (0,1) .. controls (3.5,3) .. (3.5,4);
	\draw (2,0) .. controls (2,1) .. (0,2) .. controls (2.5,3.5) .. (2.5,4); 
	\node at(1.5,.1) {\small $\dots$};
	\node at(3,3.9) {\small $\dots$};
	\node at(.25,1.6) {\small $\vdots$};
	\tikzbrace{1}{2}{0}{\small $\ell$};
	%
	\draw (2.5,0) .. controls (2.5,1) and (4,1) .. (4,2) node[pos=1, tikzdot]{}  node[pos=1, xshift=1.5ex, yshift=.25ex]{\small $v$}
		 .. controls (4,3) and (4,3) .. (4,4);
	\draw (3,0) .. controls (3,1) .. (0,2.5) .. controls (1.5,3.5) .. (1.5,4);
	\draw (4,0) .. controls (4,1.5) .. (0,3.5) .. controls (.5,3.75) .. (.5,4); 
	\node at(3.5,.1) {\small $\dots$};
	\node at(1,3.9) {\small $\dots$};
	\node at(.25,3.1) {\small $\vdots$};
	\tikzbrace{3}{4}{0}{\small $k-\ell-1$};
	\draw [pstdhl] (0,0) node[below]{\small $\mu$} -- (0,4) 
		node[pos=.125,nail]{}  node[pos=.25,nail]{} node[pos=.5,nail]{} node[pos=.625,nail]{} node[pos=.875,nail]{};
}
\]
\end{prop}

\begin{proof}
We apply recursively the lemma. 
\end{proof}

\subsection{Detailed computations for rewriting}

In order the make the following proofs less notational heavy, we introduce the following shorthand. Fix $p \geq 0$. When in a diagram we draw $m$ stars on the black strands, it means we consider the sum over all diagrams where we replace each star by $p_i$ dots for $\sum_i p_i = p - m  +1$, where we assume the sum is  empty whenever $p - m  +1 < 0$. For example
\[
\tikzdiag{
	\draw (0,0) -- (0,1) node[midway,tikzstar]{};
	\draw (.5,0) -- (.5,1) node[midway,tikzstar]{};
	\draw (1,0) -- (1,1) node[midway,tikzstar]{};
}
\ = \sssum{p_1+p_2+p_3 \\ = p-2}
\tikzdiag{
	\draw (0,0) -- (0,1) node[midway,tikzdot]{} node[midway,xshift=1.25ex,yshift=1ex]{\small $p_1$};
	\draw (.5,0) -- (.5,1) node[midway,tikzdot]{} node[midway,xshift=1.25ex,yshift=1ex]{\small $p_2$};
	\draw (1,0) -- (1,1) node[midway,tikzdot]{} node[midway,xshift=1.25ex,yshift=1ex]{\small $p_3$};
}
\]
This allows us to write local relations as
\begin{align}
\tikzdiag{
	\draw (0,0) .. controls (0,.5) and (1,.5) .. (1,1);
	\draw (1,0) .. controls (1,.5) and (0,.5) .. (0,1) node[near start, tikzstar]{}; 
}
\ &= \ 
\tikzdiag{
	\draw (0,0) .. controls (0,.5) and (1,.5) .. (1,1);
	\draw (1,0) .. controls (1,.5) and (0,.5) .. (0,1) node[near end, tikzstar]{}; 
}
\ - \ 
\tikzdiag{
	\draw (0,0) -- (0,1) node[midway, tikzstar]{};
	\draw (1,0) -- (1,1) node[midway, tikzstar]{};
}
\\
\tikzdiag{
	\draw (0,0) .. controls (0,.5) and (1,.5) .. (1,1) node[near start, tikzstar]{};
	\draw (1,0) .. controls (1,.5) and (0,.5) .. (0,1) node[near start, tikzstar]{}; 
}
\ &= \ 
\tikzdiag{
	\draw (0,0) .. controls (0,.5) and (1,.5) .. (1,1) node[near end, tikzstar]{};
	\draw (1,0) .. controls (1,.5) and (0,.5) .. (0,1) node[near end, tikzstar]{}; 
}
\label{eq:doublestarcommutes}
\end{align}

\begin{lem}\label{lem:doublecrossingnailzero}
We have
\[
\tikzdiagl
{
	\draw (1,-1) 	.. controls (1,-.75) .. (0,-.5) 
			.. controls (1,0) .. (1,.5)
			.. controls (1,.75) and (.5,.75) .. (.5,1);
	\draw (1,1) 	.. controls (1,.75) .. (0,.5) 
			.. controls (1,0) .. (1,-.5)
			.. controls (1,-.75) and (.5,-.75) .. (.5,-1);
	\draw[pstdhl] (0,-1)node[below]{\small $\mu_1$} -- (0,0) node[midway, nail]{}
			 -- (0,1) node[midway, nail]{};
}
\ = 0,
\]
in $\muT_b$.
\end{lem}

\begin{proof}
The statement immediately follows from \cref{eq:nailsrel} and \cref{eq:nhR2andR3}. 
\end{proof}

\begin{prop}\label{prop:doublenailsamestrandiszero}
We have
\[
\tikzdiagl
{
	\draw (2,-1) 	.. controls (2,-.75) .. (0,-.5) 
			.. controls (2,-.25) .. (2,0) node[pos=1, tikzdot]{}  node[pos=1,xshift=1.25ex,yshift=1ex]{\small $p$}
			.. controls (2,.25) .. (0,.5)
			.. controls (2,.75) .. (2,1);
	\draw (.5,-1) -- (.5,1);
	\draw[pstdhl,dashed] (.5,-1) -- (.5,1);
	\draw (1.5,-1) -- (1.5,1);
	\draw[pstdhl,dashed] (1.5,-1) -- (1.5,1);
	\node at (1,-.9){\small $\dots$}; \node at (1,0){\small $\dots$}; \node at (1,.9){\small $\dots$};
	\draw[pstdhl] (0,-1) node[below]{\small $\mu_1$} -- (0,0) node[midway, nail]{}
			 -- (0,1) node[midway, nail]{};
	\tikzbrace{.5}{1.5}{-1}{\small $\ell$};
}
\ =  0,
\]
in $\muT_b$.
\end{prop}

\begin{proof}
We prove the statement by induction on the number of strands $\ell$. The base case $\ell = 0$ is given by
\[
\tikzdiagl
{
	\draw (.5,-1) 	.. controls (.5,-.75) .. (0,-.5) 
			.. controls (.5,-.25) .. (.5,0) node[pos=1, tikzdot]{}  node[pos=1,xshift=1.25ex,yshift=1ex]{\small $p$}
			.. controls (.5,.25) .. (0,.5)
			.. controls (.5,.75) .. (.5,1);
	\draw[pstdhl] (0,-1)node[below]{\small $\mu_1$} -- (0,0) node[midway, nail]{}
			 -- (0,1) node[midway, nail]{};
}
\ = \ 
\tikzdiagl
{
	\draw (.5,-1) 	.. controls (.5,-.75) .. (0,-.5)  
			.. controls (.5,-.25) .. (.5,0)
			.. controls (.5,.25) .. (0,.5)
			.. controls (.5,.75) .. (.5,1)  node[pos=.5, tikzdot]{}  node[pos=.5,xshift=1.25ex,yshift=-1ex]{\small $p$};
	\draw[pstdhl] (0,-1) node[below]{\small $\mu_1$} -- (0,0) node[midway, nail]{}
			 -- (0,1) node[midway, nail]{};
}
\ =  0.
\]
Suppose the statement is true for $\ell - 1$. We have two cases to consider:
\begin{align*}
\tikzdiagl
{
	\draw (2.5,-1) 	.. controls (2.5,-.75) .. (0,-.5) 
				.. controls (2.5,-.25) .. (2.5,0) node[pos=1, tikzdot]{}  node[pos=1,xshift=1.25ex,yshift=1ex]{\small $p$}
				.. controls (2.5,.25) .. (0,.5)
				.. controls (2.5,.75) .. (2.5,1);
	\draw[pstdhl] (2,-1) node[below]{\small $\mu_i$} --  (2,1);
	\draw (.5,-1) -- (.5,1);
	\draw[pstdhl,dashed] (.5,-1) -- (.5,1);
	\draw (1.5,-1) -- (1.5,1);
	\draw[pstdhl,dashed] (1.5,-1) -- (1.5,1);
	\node at (1,-.9){\small $\dots$}; \node at (1,0){\small $\dots$}; \node at (1,.9){\small $\dots$};
	\draw[pstdhl] (0,-1) node[below]{\small $\mu_1$} -- (0,0) node[midway, nail]{}
			 -- (0,1) node[midway, nail]{};
}
&&
\text{and}
&&
\tikzdiagl
{
	\draw (2.5,-1) 	.. controls (2.5,-.75) .. (0,-.5) 
				.. controls (2.5,-.25) .. (2.5,0) node[pos=1, tikzdot]{}  node[pos=1,xshift=1.25ex,yshift=1ex]{\small $p$}
				.. controls (2.5,.25) .. (0,.5)
				.. controls (2.5,.75) .. (2.5,1);
	\draw (2,-1) --  (2,1);
	\draw (.5,-1) -- (.5,1);
	\draw[pstdhl,dashed] (.5,-1) -- (.5,1);
	\draw (1.5,-1) -- (1.5,1);
	\draw[pstdhl,dashed] (1.5,-1) -- (1.5,1);
	\node at (1,-.9){\small $\dots$}; \node at (1,0){\small $\dots$}; \node at (1,.9){\small $\dots$};
	\draw[pstdhl] (0,-1) node[below]{\small $\mu_1$} -- (0,0) node[midway, nail]{}
			 -- (0,1) node[midway, nail]{};
}
\end{align*}
For the first one, we directly have
\[
\tikzdiagl
{
	\draw (2.5,-1) 	.. controls (2.5,-.75) .. (0,-.5) 
				.. controls (2.5,-.25) .. (2.5,0) node[pos=1, tikzdot]{}  node[pos=1,xshift=1.25ex,yshift=1ex]{\small $p$}
				.. controls (2.5,.25) .. (0,.5)
				.. controls (2.5,.75) .. (2.5,1);
	\draw[pstdhl] (2,-1) node[below]{\small $\mu_i$} --  (2,1);
	\draw (.5,-1) -- (.5,1);
	\draw[pstdhl,dashed] (.5,-1) -- (.5,1);
	\draw (1.5,-1) -- (1.5,1);
	\draw[pstdhl,dashed] (1.5,-1) -- (1.5,1);
	\node at (1,-.9){\small $\dots$}; \node at (1,0){\small $\dots$}; \node at (1,.9){\small $\dots$};
	\draw[pstdhl] (0,-1) node[below]{\small $\mu_1$} -- (0,0) node[midway, nail]{}
			 -- (0,1) node[midway, nail]{};
}
\ = \ 
\tikzdiagl
{
	\draw (2.5,-1) 	.. controls (2.5,-.75) .. (0,-.5) 
				.. controls (2,-.25) .. (2,0) node[pos=1, tikzdot]{}  node[pos=1,xshift=1.5ex,yshift=.25ex]{\small $p'$}
				.. controls (2,.25) .. (0,.5)
				.. controls (2.5,.75) .. (2.5,1);
	\draw[pstdhl] (2,-1) node[below]{\small $\mu_i$} .. controls (2,-.5) and (2.5,-.5) .. (2.5,0) 
								     .. controls (2.5,.5) and (2,.5) ..  (2,1);
	\draw (.5,-1) -- (.5,1);
	\draw[pstdhl,dashed] (.5,-1) -- (.5,1);
	\draw (1.5,-1) -- (1.5,1);
	\draw[pstdhl,dashed] (1.5,-1) -- (1.5,1);
	\node at (1,-.9){\small $\dots$}; \node at (1,0){\small $\dots$}; \node at (1,.9){\small $\dots$};
	\draw[pstdhl] (0,-1) node[below]{\small $\mu_1$} -- (0,0) node[midway, nail]{}
			 -- (0,1) node[midway, nail]{};
}
\]
where $p' := p+\mu_i$, 
and we conclude by using the induction hypothesis.
For the second case, we compute
\[
\tikzdiagl
{
	\draw (2.5,-1) 	.. controls (2.5,-.75) .. (0,-.5) 
				.. controls (2.5,-.25) .. (2.5,0) node[pos=1, tikzstar]{} 
				.. controls (2.5,.25) .. (0,.5)
				.. controls (2.5,.75) .. (2.5,1);
	\draw (2,-1) --  (2,1);
	\draw (.5,-1) -- (.5,1);
	\draw[pstdhl,dashed] (.5,-1) -- (.5,1);
	\draw (1.5,-1) -- (1.5,1);
	\draw[pstdhl,dashed] (1.5,-1) -- (1.5,1);
	\node at (1,-.9){\small $\dots$}; \node at (1,0){\small $\dots$}; \node at (1,.9){\small $\dots$};
	\draw[pstdhl] (0,-1) node[below]{\small $\mu_1$} -- (0,0) node[midway, nail]{}
			 -- (0,1) node[midway, nail]{};
}
\ = - \ 
\tikzdiagl
{
	\draw (2.5,-1) 	.. controls (2.5,-.75) .. (0,-.5) 
				.. controls (2.5,0) .. (2.5,.5) node[pos=1, tikzstar]{}
				.. controls (2.5,.75) and (2,.75) .. (2,1);
	\draw (2.5,1) 	.. controls (2.5,.75) .. (0,.5) 
				.. controls (2.5,0) .. (2.5,-.5) node[pos=1, tikzstar]{}
				.. controls (2.5,-.75) and (2,-.75) .. (2,-1);
	\draw (.5,-1) -- (.5,1);
	\draw[pstdhl,dashed] (.5,-1) -- (.5,1);
	\draw (1.5,-1) -- (1.5,1);
	\draw[pstdhl,dashed] (1.5,-1) -- (1.5,1);
	\node at (1,-.9){\small $\dots$}; \node at (1,0){\small $\dots$}; \node at (1,.9){\small $\dots$};
	\draw[pstdhl] (0,-1) node[below]{\small $\mu_1$} -- (0,0) node[midway, nail]{}
			 -- (0,1) node[midway, nail]{};
}
\ - \ 
\tikzdiagl
{
	\draw (2.5,-1) 	.. controls (2.5,-.75) .. (0,-.5) 
				.. controls (2,-.25) .. (2,0) node[pos=1, tikzstar]{} 
				.. controls (2,.25) .. (0,.5)
				.. controls (2.5,.75) .. (2.5,1);
	\draw		 (2,-1) node[below]{\small $\mu_i$} .. controls (2,-.5) and (2.5,-.5) .. (2.5,0) node[near end, tikzstar]{}
								     .. controls (2.5,.5) and (2,.5) ..  (2,1) node[near start, tikzstar]{};
	\draw (.5,-1) -- (.5,1);
	\draw[pstdhl,dashed] (.5,-1) -- (.5,1);
	\draw (1.5,-1) -- (1.5,1);
	\draw[pstdhl,dashed] (1.5,-1) -- (1.5,1);
	\node at (1,-.9){\small $\dots$}; \node at (1,0){\small $\dots$}; \node at (1,.9){\small $\dots$};
	\draw[pstdhl] (0,-1) node[below]{\small $\mu_1$} -- (0,0) node[midway, nail]{}
			 -- (0,1) node[midway, nail]{};
}
\]
The rightmost term is zero by induction hypothesis. Then we compute
\[
\tikzdiagl
{
	\draw (2.5,-1) 	.. controls (2.5,-.75) .. (0,-.5) 
				.. controls (2.5,0) .. (2.5,.5) node[pos=1, tikzstar]{}
				.. controls (2.5,.75) and (2,.75) .. (2,1);
	\draw (2.5,1) 	.. controls (2.5,.75) .. (0,.5) 
				.. controls (2.5,0) .. (2.5,-.5) node[pos=1, tikzstar]{}
				.. controls (2.5,-.75) and (2,-.75) .. (2,-1);
	\draw (.5,-1) -- (.5,1);
	\draw[pstdhl,dashed] (.5,-1) -- (.5,1);
	\draw (1.5,-1) -- (1.5,1);
	\draw[pstdhl,dashed] (1.5,-1) -- (1.5,1);
	\node at (1,-.9){\small $\dots$}; \node at (1,0){\small $\dots$}; \node at (1,.9){\small $\dots$};
	\draw[pstdhl] (0,-1) node[below]{\small $\mu_1$} -- (0,0) node[midway, nail]{}
			 -- (0,1) node[midway, nail]{};
}
\ = \ 
\tikzdiagl
{
	\draw (2.5,-1) 	.. controls (2.5,-.75) .. (0,-.5) 
				.. controls (.5-.125,.125) .. (.5,.25)
				.. controls (.5+.125,.375) and (2.5,.375) .. (2.5,.5) node[pos=1, tikzstar]{}
				.. controls (2.5,.75) and (2,.75) .. (2,1);
	\draw (2.5,1) 	.. controls (2.5,.75) .. (0,.5) 
				.. controls (.5-.125,-.125) .. (.5,-.25)
				.. controls (.5+.125,-.375) and (2.5,-.375) .. (2.5,-.5) node[pos=1, tikzstar]{}
				.. controls (2.5,-.75) and (2,-.75) .. (2,-1);
	\draw (.5,-1) .. controls (.5,-.5) and (1,-.5) .. (1,0) .. controls (1,.5) and (.5,.5) .. (.5,1);
	\draw[pstdhl,dashed] (.5,-1) .. controls (.5,-.5) and (1,-.5) .. (1,0) .. controls (1,.5) and (.5,.5) .. (.5,1);
	\draw (1.5,-1) .. controls (1.5,-.5) and (2,-.5) .. (2,0) .. controls (2,.5) and (1.5,.5) .. (1.5,1);
	\draw[pstdhl,dashed] (1.5,-1) .. controls (1.5,-.5) and (2,-.5) .. (2,0) .. controls (2,.5) and (1.5,.5) .. (1.5,1);
	\node at (1,-.9){\small $\dots$}; \node at (1.5,0){\small $\dots$}; \node at (1,.9){\small $\dots$};
	\draw[pstdhl] (0,-1) node[below]{\small $\mu_1$} -- (0,0) node[midway, nail]{}
			 -- (0,1) node[midway, nail]{};
}
\ + \sum_{i} \sssum{u+v= \\ \mu_i-1} \  
\tikzdiagl
{
	\draw (5.5,-1) 	.. controls (5.5,-.75) .. (0,-.5)
			.. controls (2,-.25) .. (2,0) node[pos=1, tikzdot]{}  node[pos=1,xshift=1ex,yshift=1ex]{\small $u$}
			.. controls (2,.25) .. (0,.5)
			.. controls (5.5,.75) .. (5.5,1);
	\draw (.5,-1) -- (.5,1);
	\draw[pstdhl,dashed] (.5,-1) -- (.5,1);
	\draw (1.5,-1) -- (1.5,1);
	\draw[pstdhl,dashed] (1.5,-1) -- (1.5,1);
	\node at (1,-.9){\small $\dots$}; \node at (1,0){\small $\dots$}; \node at (1,.9){\small $\dots$};
	\draw[pstdhl](2.5,-1) node[below]{\small $\mu_i$} -- (2.5,1);
	\draw (5,-1) 	.. controls (5,-.75) and (5.5,-.75) .. (5.5,-.5) node[pos=1, tikzstar]{}
			.. controls (5.5,-.25) and (3,-.25) .. (3,0) node[pos=1,tikzdot]{} node[pos=1, xshift=-1ex,yshift=1ex]{\small $v$}
			.. controls (3,.25) and (5.5,.25) .. (5.5,.5) node[pos=1, tikzstar]{}
			.. controls (5.5,.75) and (5,.75) .. (5,1);
	\draw (3.5,-1) -- (3.5,1);
	\draw[pstdhl,dashed] (3.5,-1) -- (3.5,1);
	\draw (4.5,-1) -- (4.5,1);
	\draw[pstdhl,dashed] (4.5,-1) -- (4.5,1);
	\node at (4,-.9){\small $\dots$}; \node at (4,0){\small $\dots$}; \node at (4,.9){\small $\dots$};
	\draw[pstdhl] (0,-1) node[below]{\small $\mu_1$} -- (0,0) node[midway, nail]{}
			 -- (0,1) node[midway, nail]{};
}
\]
where the elements in the sum are all zero by induction hypothesis. Then we compute
\begin{align*}
&\tikzdiagl
{
	\draw (2.5,-1) 	.. controls (2.5,-.75) .. (0,-.5) 
				.. controls (.5-.125,.125) .. (.5,.25)
				.. controls (.5+.125,.375) and (2.5,.375) .. (2.5,.5) node[pos=1, tikzstar]{}
				.. controls (2.5,.75) and (2,.75) .. (2,1);
	\draw (2.5,1) 	.. controls (2.5,.75) .. (0,.5) 
				.. controls (.5-.125,-.125) .. (.5,-.25)
				.. controls (.5+.125,-.375) and (2.5,-.375) .. (2.5,-.5) node[pos=1, tikzstar]{}
				.. controls (2.5,-.75) and (2,-.75) .. (2,-1);
	\draw (.5,-1) .. controls (.5,-.5) and (1,-.5) .. (1,0) .. controls (1,.5) and (.5,.5) .. (.5,1);
	\draw[pstdhl,dashed] (.5,-1) .. controls (.5,-.5) and (1,-.5) .. (1,0) .. controls (1,.5) and (.5,.5) .. (.5,1);
	\draw (1.5,-1) .. controls (1.5,-.5) and (2,-.5) .. (2,0) .. controls (2,.5) and (1.5,.5) .. (1.5,1);
	\draw[pstdhl,dashed] (1.5,-1) .. controls (1.5,-.5) and (2,-.5) .. (2,0) .. controls (2,.5) and (1.5,.5) .. (1.5,1);
	\node at (1,-.9){\small $\dots$}; \node at (1.5,0){\small $\dots$}; \node at (1,.9){\small $\dots$};
	\draw[pstdhl] (0,-1) node[below]{\small $\mu_1$} -- (0,0) node[midway, nail]{}
			 -- (0,1) node[midway, nail]{};
} \\
&\ = \ 
\tikzdiagl
{
	\draw (2,-1) 	.. controls (2,-.75) .. (0,-.5)   node[pos=.2, tikzstar]{}
				.. controls (.5-.125,.125) .. (.5,.25)
				.. controls (.5+.125,.375) and (2.5,.375) .. (2.5,.5) 
				-- (2.5,1) node[pos=.5, tikzstar]{};
	\draw (2,1) 	.. controls (2,.75) .. (0,.5)   node[pos=.2, tikzstar]{}
				.. controls (.5-.125,-.125) .. (.5,-.25)
				.. controls (.5+.125,-.375) and (2.5,-.375) .. (2.5,-.5)
				-- (2.5,-1)  node[pos=.5, tikzstar]{};
	\draw (.5,-1) .. controls (.5,-.5) and (1,-.5) .. (1,0) .. controls (1,.5) and (.5,.5) .. (.5,1);
	\draw[pstdhl,dashed] (.5,-1) .. controls (.5,-.5) and (1,-.5) .. (1,0) .. controls (1,.5) and (.5,.5) .. (.5,1);
	\draw (1.5,-1) .. controls (1.5,-.5) and (2,-.5) .. (2,0) .. controls (2,.5) and (1.5,.5) .. (1.5,1);
	\draw[pstdhl,dashed] (1.5,-1) .. controls (1.5,-.5) and (2,-.5) .. (2,0) .. controls (2,.5) and (1.5,.5) .. (1.5,1);
	\node at (1,-.9){\small $\dots$}; \node at (1.5,0){\small $\dots$}; \node at (1,.9){\small $\dots$};
	\draw[pstdhl] (0,-1) node[below]{\small $\mu_1$} -- (0,0) node[midway, nail]{}
			 -- (0,1) node[midway, nail]{};
}
\ - \ 
\tikzdiagl
{
	\draw (2,-1) 	.. controls (2,-.75) .. (0,-.5) node[pos=.2, tikzstar]{}
				.. controls (.5-.125,.125) .. (.5,.25)
				.. controls (.5+.125,.375) and (2.5,.375) .. (2.5,.5) 
				.. controls (2.5,.75) and (2,.75) .. (2,1)  node[pos=.8, tikzstar]{};
	\draw (2.5,1) 	.. controls (2.5,.75) .. (0,.5) 
				.. controls (.5-.125,-.125) .. (.5,-.25)
				.. controls (.5+.125,-.375) and (2.5,-.375) .. (2.5,-.5)
				-- (2.5,-1)  node[pos=.5, tikzstar]{};
	\draw (.5,-1) .. controls (.5,-.5) and (1,-.5) .. (1,0) .. controls (1,.5) and (.5,.5) .. (.5,1);
	\draw[pstdhl,dashed] (.5,-1) .. controls (.5,-.5) and (1,-.5) .. (1,0) .. controls (1,.5) and (.5,.5) .. (.5,1);
	\draw (1.5,-1) .. controls (1.5,-.5) and (2,-.5) .. (2,0) .. controls (2,.5) and (1.5,.5) .. (1.5,1);
	\draw[pstdhl,dashed] (1.5,-1) .. controls (1.5,-.5) and (2,-.5) .. (2,0) .. controls (2,.5) and (1.5,.5) .. (1.5,1);
	\node at (1,-.9){\small $\dots$}; \node at (1.5,0){\small $\dots$}; \node at (1,.9){\small $\dots$};
	\draw[pstdhl] (0,-1) node[below]{\small $\mu_1$} -- (0,0) node[midway, nail]{}
			 -- (0,1) node[midway, nail]{};
}
\ - \ 
\tikzdiagl
{
	\draw (2.5,-1) 	.. controls (2.5,-.75) .. (0,-.5) 
				.. controls (.5-.125,.125) .. (.5,.25)
				.. controls (.5+.125,.375) and (2.5,.375) .. (2.5,.5) 
				-- (2.5,1) node[pos=.5, tikzstar]{};
	\draw (2,1) 	.. controls (2,.75) .. (0,.5)   node[pos=.2, tikzstar]{}
				.. controls (.5-.125,-.125) .. (.5,-.25)
				.. controls (.5+.125,-.375) and (2.5,-.375) .. (2.5,-.5) 
				.. controls (2.5,-.75) and (2,-.75) .. (2,-1) node[pos=.8, tikzstar]{};
	\draw (.5,-1) .. controls (.5,-.5) and (1,-.5) .. (1,0) .. controls (1,.5) and (.5,.5) .. (.5,1);
	\draw[pstdhl,dashed] (.5,-1) .. controls (.5,-.5) and (1,-.5) .. (1,0) .. controls (1,.5) and (.5,.5) .. (.5,1);
	\draw (1.5,-1) .. controls (1.5,-.5) and (2,-.5) .. (2,0) .. controls (2,.5) and (1.5,.5) .. (1.5,1);
	\draw[pstdhl,dashed] (1.5,-1) .. controls (1.5,-.5) and (2,-.5) .. (2,0) .. controls (2,.5) and (1.5,.5) .. (1.5,1);
	\node at (1,-.9){\small $\dots$}; \node at (1.5,0){\small $\dots$}; \node at (1,.9){\small $\dots$};
	\draw[pstdhl] (0,-1) node[below]{\small $\mu_1$} -- (0,0) node[midway, nail]{}
			 -- (0,1) node[midway, nail]{};
}
\ + \ 
\tikzdiagl
{
	\draw (2.5,-1) 	.. controls (2.5,-.75) .. (0,-.5) 
				.. controls (.5-.125,.125) .. (.5,.25)
				.. controls (.5+.125,.375) and (2.5,.375) .. (2.5,.5) 
				.. controls (2.5,.75) and (2,.75) .. (2,1) node[pos=.8, tikzstar]{};
	\draw (2.5,1) 	.. controls (2.5,.75) .. (0,.5) 
				.. controls (.5-.125,-.125) .. (.5,-.25)
				.. controls (.5+.125,-.375) and (2.5,-.375) .. (2.5,-.5)
				.. controls (2.5,-.75) and (2,-.75) .. (2,-1)  node[pos=.8, tikzstar]{};
	\draw (.5,-1) .. controls (.5,-.5) and (1,-.5) .. (1,0) .. controls (1,.5) and (.5,.5) .. (.5,1);
	\draw[pstdhl,dashed] (.5,-1) .. controls (.5,-.5) and (1,-.5) .. (1,0) .. controls (1,.5) and (.5,.5) .. (.5,1);
	\draw (1.5,-1) .. controls (1.5,-.5) and (2,-.5) .. (2,0) .. controls (2,.5) and (1.5,.5) .. (1.5,1);
	\draw[pstdhl,dashed] (1.5,-1) .. controls (1.5,-.5) and (2,-.5) .. (2,0) .. controls (2,.5) and (1.5,.5) .. (1.5,1);
	\node at (1,-.9){\small $\dots$}; \node at (1.5,0){\small $\dots$}; \node at (1,.9){\small $\dots$};
	\draw[pstdhl] (0,-1) node[below]{\small $\mu_1$} -- (0,0) node[midway, nail]{}
			 -- (0,1) node[midway, nail]{};
}
\\
& \ = \ 
\tikzdiagl
{
	\draw (2,-1) 	.. controls (2,-.75) .. (0,-.25)   node[pos=.2, tikzstar]{}
				.. controls (2.5,.5) .. (2.5,.75) 
				-- (2.5,1) node[pos=.5, tikzstar]{};
	\draw (2,1) 	.. controls (2,.75) .. (0,.25)   node[pos=.2, tikzstar]{}
				.. controls (2.5,-.5) .. (2.5,-.75) 
				-- (2.5,-1) node[pos=.5, tikzstar]{};
	\draw (.5,-1) .. controls (.5,-.5) and (1.5,-.5) .. (1.5,0) .. controls (1.5,.5) and (.5,.5) .. (.5,1);
	\draw[pstdhl,dashed] (.5,-1) .. controls (.5,-.5) and (1.5,-.5) .. (1.5,0) .. controls (1.5,.5) and (.5,.5) .. (.5,1);
	\draw (1.5,-1) .. controls (1.5,-.5) and (2.5,-.5) .. (2.5,0) .. controls (2.5,.5) and (1.5,.5) .. (1.5,1);
	\draw[pstdhl,dashed] (1.5,-1) .. controls (1.5,-.5) and (2.5,-.5) .. (2.5,0) .. controls (2.5,.5) and (1.5,.5) .. (1.5,1);
	\node at (1,-.9){\small $\dots$}; \node at (2,0){\small $\dots$}; \node at (1,.9){\small $\dots$};
	\draw[pstdhl] (0,-1) node[below]{\small $\mu_1$} -- (0,0) node[pos=.75, nail]{}
			 -- (0,1) node[pos=.25, nail]{};
}
\ - \ 
\tikzdiagl
{
	\draw (2,-1) 		.. controls (2,-.75) .. (0,-.25) node[pos=.2, tikzstar]{}
		 		.. controls (1,0) .. (1,.25)
				.. controls (1,.375) and (.4,.375) .. (.4,.6)
				.. controls (.4,.75) and (2,.75) .. (2,1) node[pos=.85, tikzstar]{};
	\draw (2.5,1) 	.. controls (2.5,.75) .. (0,.25) 
				.. controls (2.5,-.5) .. (2.5,-.75) 
				-- (2.5,-1) node[pos=.5, tikzstar]{};
	\draw (.5,-1) .. controls (.5,-.5) and (1.5,-.5) .. (1.5,0) .. controls (1.5,.5) and (.5,.5) .. (.5,1);
	\draw[pstdhl,dashed] (.5,-1) .. controls (.5,-.5) and (1.5,-.5) .. (1.5,0) .. controls (1.5,.5) and (.5,.5) .. (.5,1);
	\draw (1.5,-1) .. controls (1.5,-.5) and (2.5,-.5) .. (2.5,0) .. controls (2.5,.5) and (1.5,.5) .. (1.5,1);
	\draw[pstdhl,dashed] (1.5,-1) .. controls (1.5,-.5) and (2.5,-.5) .. (2.5,0) .. controls (2.5,.5) and (1.5,.5) .. (1.5,1);
	\node at (1,-.9){\small $\dots$}; \node at (2,0){\small $\dots$}; \node at (1,.9){\small $\dots$};
	\draw[pstdhl] (0,-1) node[below]{\small $\mu_1$} -- (0,0) node[pos=.75, nail]{}
			 -- (0,1) node[pos=.25, nail]{};
}
\ - \ 
\tikzdiagl
{
	\draw (2.5,-1) 	.. controls (2.5,-.75) .. (0,-.25) 
				.. controls (2.5,.5) .. (2.5,.75) 
				-- (2.5,1) node[pos=.5, tikzstar]{};
	\draw (2,1) 		.. controls (2,.75) .. (0,.25)  node[pos=.2, tikzstar]{}
		 		.. controls (1,0) .. (1,-.25)
				.. controls (1,-.375) and (.4,-.375) .. (.4,-.6)
				.. controls (.4,-.75) and (2,-.75) .. (2,-1) node[pos=.85, tikzstar]{};
	\draw (.5,-1) .. controls (.5,-.5) and (1.5,-.5) .. (1.5,0) .. controls (1.5,.5) and (.5,.5) .. (.5,1);
	\draw[pstdhl,dashed] (.5,-1) .. controls (.5,-.5) and (1.5,-.5) .. (1.5,0) .. controls (1.5,.5) and (.5,.5) .. (.5,1);
	\draw (1.5,-1) .. controls (1.5,-.5) and (2.5,-.5) .. (2.5,0) .. controls (2.5,.5) and (1.5,.5) .. (1.5,1);
	\draw[pstdhl,dashed] (1.5,-1) .. controls (1.5,-.5) and (2.5,-.5) .. (2.5,0) .. controls (2.5,.5) and (1.5,.5) .. (1.5,1);
	\node at (1,-.9){\small $\dots$}; \node at (2,0){\small $\dots$}; \node at (1,.9){\small $\dots$};
	\draw[pstdhl] (0,-1) node[below]{\small $\mu_1$} -- (0,0) node[pos=.75, nail]{}
			 -- (0,1) node[pos=.25, nail]{};
}
\ + \ 
\tikzdiagl
{
	\draw (2.5,-1) 	.. controls (2.5,-.75) .. (0,-.25) 
		 		.. controls (1,0) .. (1,.25)
				.. controls (1,.375) and (.4,.375) .. (.4,.6)
				.. controls (.4,.75) and (2,.75) .. (2,1) node[pos=.85, tikzstar]{};
	\draw (2.5,1) 	.. controls (2.5,.75) .. (0,.25) 
		 		.. controls (1,0) .. (1,-.25)
				.. controls (1,-.375) and (.4,-.375) .. (.4,-.6)
				.. controls (.4,-.75) and (2,-.75) .. (2,-1) node[pos=.85, tikzstar]{};
	\draw (.5,-1) .. controls (.5,-.5) and (1.5,-.5) .. (1.5,0) .. controls (1.5,.5) and (.5,.5) .. (.5,1);
	\draw[pstdhl,dashed] (.5,-1) .. controls (.5,-.5) and (1.5,-.5) .. (1.5,0) .. controls (1.5,.5) and (.5,.5) .. (.5,1);
	\draw (1.5,-1) .. controls (1.5,-.5) and (2.5,-.5) .. (2.5,0) .. controls (2.5,.5) and (1.5,.5) .. (1.5,1);
	\draw[pstdhl,dashed] (1.5,-1) .. controls (1.5,-.5) and (2.5,-.5) .. (2.5,0) .. controls (2.5,.5) and (1.5,.5) .. (1.5,1);
	\node at (1,-.9){\small $\dots$}; \node at (2,0){\small $\dots$}; \node at (1,.9){\small $\dots$};
	\draw[pstdhl] (0,-1) node[below]{\small $\mu_1$} -- (0,0) node[pos=.75, nail]{}
			 -- (0,1) node[pos=.25, nail]{};
}
\end{align*}
where the last term is zero by \cref{lem:doublecrossingnailzero}. 
For the remaining terms, we can gather them by number of dots distributed on the two stars on the left, so that we only need to compare the following terms:
\begin{align*}
&\tikzdiagl
{
	\draw (2,-1) 	.. controls (2,-.75) .. (0,-.25)  
				.. controls (2.5,.5) .. (2.5,.75) 
				-- (2.5,1) node[pos=.5, tikzstar]{};
	\draw (2,1) 	.. controls (2,.75) .. (0,.25)  
				.. controls (2.5,-.5) .. (2.5,-.75) 
				-- (2.5,-1) node[pos=.5, tikzstar]{};
	\draw (.5,-1) .. controls (.5,-.5) and (1.5,-.5) .. (1.5,0) .. controls (1.5,.5) and (.5,.5) .. (.5,1);
	\draw[pstdhl,dashed] (.5,-1) .. controls (.5,-.5) and (1.5,-.5) .. (1.5,0) .. controls (1.5,.5) and (.5,.5) .. (.5,1);
	\draw (1.5,-1) .. controls (1.5,-.5) and (2.5,-.5) .. (2.5,0) .. controls (2.5,.5) and (1.5,.5) .. (1.5,1);
	\draw[pstdhl,dashed] (1.5,-1) .. controls (1.5,-.5) and (2.5,-.5) .. (2.5,0) .. controls (2.5,.5) and (1.5,.5) .. (1.5,1);
	\node at (1,-.9){\small $\dots$}; \node at (2,0){\small $\dots$}; \node at (1,.9){\small $\dots$};
	\draw[pstdhl] (0,-1) node[below]{\small $\mu_1$} -- (0,0) node[pos=.75, nail]{}
			 -- (0,1) node[pos=.25, nail]{};
}
\ - \ 
\tikzdiagl
{
	\draw (2,-1) 		.. controls (2,-.75) .. (0,-.25)
		 		.. controls (1,0) .. (1,.25)
				.. controls (1,.375) and (.4,.375) .. (.4,.6)
				.. controls (.4,.75) and (2,.75) .. (2,1);
	\draw (2.5,1) 	.. controls (2.5,.75) .. (0,.25) 
				.. controls (2.5,-.5) .. (2.5,-.75) 
				-- (2.5,-1) node[pos=.5, tikzstar]{};
	\draw (.5,-1) .. controls (.5,-.5) and (1.5,-.5) .. (1.5,0) .. controls (1.5,.5) and (.5,.5) .. (.5,1);
	\draw[pstdhl,dashed] (.5,-1) .. controls (.5,-.5) and (1.5,-.5) .. (1.5,0) .. controls (1.5,.5) and (.5,.5) .. (.5,1);
	\draw (1.5,-1) .. controls (1.5,-.5) and (2.5,-.5) .. (2.5,0) .. controls (2.5,.5) and (1.5,.5) .. (1.5,1);
	\draw[pstdhl,dashed] (1.5,-1) .. controls (1.5,-.5) and (2.5,-.5) .. (2.5,0) .. controls (2.5,.5) and (1.5,.5) .. (1.5,1);
	\node at (1,-.9){\small $\dots$}; \node at (2,0){\small $\dots$}; \node at (1,.9){\small $\dots$};
	\draw[pstdhl] (0,-1) node[below]{\small $\mu_1$} -- (0,0) node[pos=.75, nail]{}
			 -- (0,1) node[pos=.25, nail]{};
}
\ - \ 
\tikzdiagl
{
	\draw (2.5,-1) 	.. controls (2.5,-.75) .. (0,-.25) 
				.. controls (2.5,.5) .. (2.5,.75) 
				-- (2.5,1) node[pos=.5, tikzstar]{};
	\draw (2,1) 		.. controls (2,.75) .. (0,.25) 
		 		.. controls (1,0) .. (1,-.25)
				.. controls (1,-.375) and (.4,-.375) .. (.4,-.6)
				.. controls (.4,-.75) and (2,-.75) .. (2,-1);
	\draw (.5,-1) .. controls (.5,-.5) and (1.5,-.5) .. (1.5,0) .. controls (1.5,.5) and (.5,.5) .. (.5,1);
	\draw[pstdhl,dashed] (.5,-1) .. controls (.5,-.5) and (1.5,-.5) .. (1.5,0) .. controls (1.5,.5) and (.5,.5) .. (.5,1);
	\draw (1.5,-1) .. controls (1.5,-.5) and (2.5,-.5) .. (2.5,0) .. controls (2.5,.5) and (1.5,.5) .. (1.5,1);
	\draw[pstdhl,dashed] (1.5,-1) .. controls (1.5,-.5) and (2.5,-.5) .. (2.5,0) .. controls (2.5,.5) and (1.5,.5) .. (1.5,1);
	\node at (1,-.9){\small $\dots$}; \node at (2,0){\small $\dots$}; \node at (1,.9){\small $\dots$};
	\draw[pstdhl] (0,-1) node[below]{\small $\mu_1$} -- (0,0) node[pos=.75, nail]{}
			 -- (0,1) node[pos=.25, nail]{};
}
\\
\ = \ &
\tikzdiagl
{
	\draw (2,-1) 	.. controls (2,-.75) .. (0,-.25)  
				.. controls (2.5,.5) .. (2.5,.75) 
				-- (2.5,1) node[pos=.5, tikzstar]{};
	\draw (2,1) 	.. controls (2,.75) .. (0,.25)  
				.. controls (2.5,-.5) .. (2.5,-.75) 
				-- (2.5,-1) node[pos=.5, tikzstar]{};
	\draw (.5,-1) .. controls (.5,-.5) and (1.5,-.5) .. (1.5,0) .. controls (1.5,.5) and (.5,.5) .. (.5,1);
	\draw[pstdhl,dashed] (.5,-1) .. controls (.5,-.5) and (1.5,-.5) .. (1.5,0) .. controls (1.5,.5) and (.5,.5) .. (.5,1);
	\draw (1.5,-1) .. controls (1.5,-.5) and (2.5,-.5) .. (2.5,0) .. controls (2.5,.5) and (1.5,.5) .. (1.5,1);
	\draw[pstdhl,dashed] (1.5,-1) .. controls (1.5,-.5) and (2.5,-.5) .. (2.5,0) .. controls (2.5,.5) and (1.5,.5) .. (1.5,1);
	\node at (1,-.9){\small $\dots$}; \node at (2,0){\small $\dots$}; \node at (1,.9){\small $\dots$};
	\draw[pstdhl] (0,-1) node[below]{\small $\mu_1$} -- (0,0) node[pos=.75, nail]{}
			 -- (0,1) node[pos=.25, nail]{};
}
\ + \ 
\tikzdiagl
{
	\draw (2.5,-1) 	.. controls (2.5,-.75) .. (0,-.25)  node[pos=.2, tikzstar]{}
				.. controls (2.5,.5) .. (2.5,.75) 
				-- (2.5,1) ;
	\draw (2,1) 		.. controls (2,.75) .. (0,.25) 
		 		.. controls (1,0) .. (1,-.25)
				.. controls (1,-.375) and (.4,-.375) .. (.4,-.6)
				.. controls (.4,-.75) and (2,-.75) .. (2,-1);
	\draw (.5,-1) .. controls (.5,-.5) and (1.5,-.5) .. (1.5,0) .. controls (1.5,.5) and (.5,.5) .. (.5,1);
	\draw[pstdhl,dashed] (.5,-1) .. controls (.5,-.5) and (1.5,-.5) .. (1.5,0) .. controls (1.5,.5) and (.5,.5) .. (.5,1);
	\draw (1.5,-1) .. controls (1.5,-.5) and (2.5,-.5) .. (2.5,0) .. controls (2.5,.5) and (1.5,.5) .. (1.5,1);
	\draw[pstdhl,dashed] (1.5,-1) .. controls (1.5,-.5) and (2.5,-.5) .. (2.5,0) .. controls (2.5,.5) and (1.5,.5) .. (1.5,1);
	\node at (1,-.9){\small $\dots$}; \node at (2,0){\small $\dots$}; \node at (1,.9){\small $\dots$};
	\draw[pstdhl] (0,-1) node[below]{\small $\mu_1$} -- (0,0) node[pos=.75, nail]{}
			 -- (0,1) node[pos=.25, nail]{};
}
\ - \ 
\tikzdiagl
{
	\draw (2.5,-1) 	.. controls (2.5,-.75) .. (0,-.25) 
				.. controls (2.5,.5) .. (2.5,.75) 
				-- (2.5,1) node[pos=.5, tikzstar]{};
	\draw (2,1) 		.. controls (2,.75) .. (0,.25) 
		 		.. controls (1,0) .. (1,-.25)
				.. controls (1,-.375) and (.4,-.375) .. (.4,-.6)
				.. controls (.4,-.75) and (2,-.75) .. (2,-1);
	\draw (.5,-1) .. controls (.5,-.5) and (1.5,-.5) .. (1.5,0) .. controls (1.5,.5) and (.5,.5) .. (.5,1);
	\draw[pstdhl,dashed] (.5,-1) .. controls (.5,-.5) and (1.5,-.5) .. (1.5,0) .. controls (1.5,.5) and (.5,.5) .. (.5,1);
	\draw (1.5,-1) .. controls (1.5,-.5) and (2.5,-.5) .. (2.5,0) .. controls (2.5,.5) and (1.5,.5) .. (1.5,1);
	\draw[pstdhl,dashed] (1.5,-1) .. controls (1.5,-.5) and (2.5,-.5) .. (2.5,0) .. controls (2.5,.5) and (1.5,.5) .. (1.5,1);
	\node at (1,-.9){\small $\dots$}; \node at (2,0){\small $\dots$}; \node at (1,.9){\small $\dots$};
	\draw[pstdhl] (0,-1) node[below]{\small $\mu_1$} -- (0,0) node[pos=.75, nail]{}
			 -- (0,1) node[pos=.25, nail]{};
}
\end{align*}
We observe that 
\begin{align*}
\tikzdiagl{
	\draw (2,0) .. controls (2,.5) and (.5,.5) .. (.5,1) node[pos=.1,tikzstar]{};
	\draw (.5,0) .. controls (.5,.5) and (1,.5) .. (1,1);
	\draw[dashed, pstdhl] (.5,0) .. controls (.5,.5) and (1,.5) .. (1,1);
	\draw (1.5,0) .. controls (1.5,.5) and (2,.5) .. (2,1);
	\draw[dashed, pstdhl] (1.5,0) .. controls (1.5,.5) and (2,.5) .. (2,1);
	\node at(1,.1){\small $\dots$}; \node at(1.5,.9){\small $\dots$};
	\draw [pstdhl] (0,0) node[below]{\small $\mu_1$} -- (0,1); 
}
\ =
 \ 
\tikzdiagl{
	\draw (2,0) .. controls (2,.5) and (.5,.5) .. (.5,1) node[pos=.9,tikzstar]{};
	\draw (.5,0) .. controls (.5,.5) and (1,.5) .. (1,1);
	\draw[dashed, pstdhl] (.5,0) .. controls (.5,.5) and (1,.5) .. (1,1);
	\draw (1.5,0) .. controls (1.5,.5) and (2,.5) .. (2,1);
	\draw[dashed, pstdhl] (1.5,0) .. controls (1.5,.5) and (2,.5) .. (2,1);
	\node at(1,.1){\small $\dots$}; \node at(1.5,.9){\small $\dots$};
	\draw [pstdhl] (0,0) node[below]{\small $\mu_1$} -- (0,1); 
}
\ -
 \sum \ 
\tikzdiagl{
	\draw (2,0) .. controls (2,.5) and (.5,.5) .. (.5,1) node[pos=.1,tikzstar]{};
	\draw (.5,0) .. controls (.5,.5) and (1,.5) .. (1,1);
	\draw[dashed, pstdhl] (.5,0) .. controls (.5,.5) and (1,.5) .. (1,1);
	\draw (1.5,0) .. controls (1.5,.5) and (2,.5) .. (2,1);
	\draw[dashed, pstdhl] (1.5,0) .. controls (1.5,.5) and (2,.5) .. (2,1);
	\node at(1,.1){\small $\dots$}; \node at(1.5,.9){\small $\dots$};
	\begin{scope}[xshift=2cm]
		\draw (2,0) .. controls (2,.5) and (.5,.5) .. (.5,1) node[pos=.9,tikzstar]{};
		\draw (.5,0) .. controls (.5,.5) and (1,.5) .. (1,1);
		\draw[dashed, pstdhl] (.5,0) .. controls (.5,.5) and (1,.5) .. (1,1);
		\draw (1.5,0) .. controls (1.5,.5) and (2,.5) .. (2,1);
		\draw[dashed, pstdhl] (1.5,0) .. controls (1.5,.5) and (2,.5) .. (2,1);
		\node at(1,.1){\small $\dots$}; \node at(1.5,.9){\small $\dots$};
	\end{scope}
	\draw [pstdhl] (0,0) node[below]{\small $\mu_1$} -- (0,1); 
}
\end{align*}
where the sum is over all black strands. 
Applying this relation recursively yields
\begin{align}
\tikzdiagl{
	\draw (2,0) .. controls (2,.5) and (.5,.5) .. (.5,1) node[pos=.1,tikzstar]{};
	\draw (.5,0) .. controls (.5,.5) and (1,.5) .. (1,1);
	\draw[dashed, pstdhl] (.5,0) .. controls (.5,.5) and (1,.5) .. (1,1);
	\draw (1.5,0) .. controls (1.5,.5) and (2,.5) .. (2,1);
	\draw[dashed, pstdhl] (1.5,0) .. controls (1.5,.5) and (2,.5) .. (2,1);
	\node at(1,.1){\small $\dots$}; \node at(1.5,.9){\small $\dots$};
	\draw [pstdhl] (0,0) node[below]{\small $\mu_1$} -- (0,1); 
}
\ =
 \sum (-1)^c \ 
\tikzdiagl{
	\draw (2,0) .. controls (2,.5) and (.5,.5) .. (.5,1) node[pos=.9,tikzstar]{};
	\draw (.5,0) .. controls (.5,.5) and (1,.5) .. (1,1);
	\draw[dashed, pstdhl] (.5,0) .. controls (.5,.5) and (1,.5) .. (1,1);
	\draw (1.5,0) .. controls (1.5,.5) and (2,.5) .. (2,1);
	\draw[dashed, pstdhl] (1.5,0) .. controls (1.5,.5) and (2,.5) .. (2,1);
	\node at(1,.1){\small $\dots$}; \node at(1.5,.9){\small $\dots$};
	\begin{scope}[xshift=2cm]
		\draw (2,0) .. controls (2,.5) and (.5,.5) .. (.5,1) node[pos=.9,tikzstar]{};
		\draw (.5,0) .. controls (.5,.5) and (1,.5) .. (1,1);
		\draw[dashed, pstdhl] (.5,0) .. controls (.5,.5) and (1,.5) .. (1,1);
		\draw (1.5,0) .. controls (1.5,.5) and (2,.5) .. (2,1);
		\draw[dashed, pstdhl] (1.5,0) .. controls (1.5,.5) and (2,.5) .. (2,1);
		\node at(1,.1){\small $\dots$}; \node at(1.5,.9){\small $\dots$};
	\end{scope}
	\node at (4.75,.5) {$\dots$};
	\begin{scope}[xshift=5cm]
		\draw (2,0) .. controls (2,.5) and (.5,.5) .. (.5,1) node[pos=.9,tikzstar]{};
		\draw (.5,0) .. controls (.5,.5) and (1,.5) .. (1,1);
		\draw[dashed, pstdhl] (.5,0) .. controls (.5,.5) and (1,.5) .. (1,1);
		\draw (1.5,0) .. controls (1.5,.5) and (2,.5) .. (2,1);
		\draw[dashed, pstdhl] (1.5,0) .. controls (1.5,.5) and (2,.5) .. (2,1);
		\node at(1,.1){\small $\dots$}; \node at(1.5,.9){\small $\dots$};
	\end{scope}
	\draw [pstdhl] (0,0) node[below]{\small $\mu_1$} -- (0,1); 
}
\label{eq:starslidingoverallcrossings}
\end{align}
where the sum is over all ways to resolve black/black crossings in the diagram, and $c$ is the number of resolved crossings. 
By applying \cref{eq:starslidingoverallcrossings} and its symmetric to
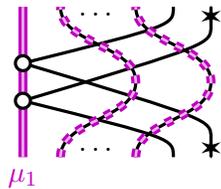
\[
\tikzdiagl
{
	\draw (2,-1) 	.. controls (2,-.75) .. (0,-.25)  
				.. controls (2.5,.5) .. (2.5,.75) 
				-- (2.5,1) node[pos=.5, tikzstar]{};
	\draw (2,1) 	.. controls (2,.75) .. (0,.25)  
				.. controls (2.5,-.5) .. (2.5,-.75) 
				-- (2.5,-1) node[pos=.5, tikzstar]{};
	\draw (.5,-1) .. controls (.5,-.5) and (1.5,-.5) .. (1.5,0) .. controls (1.5,.5) and (.5,.5) .. (.5,1);
	\draw[pstdhl,dashed] (.5,-1) .. controls (.5,-.5) and (1.5,-.5) .. (1.5,0) .. controls (1.5,.5) and (.5,.5) .. (.5,1);
	\draw (1.5,-1) .. controls (1.5,-.5) and (2.5,-.5) .. (2.5,0) .. controls (2.5,.5) and (1.5,.5) .. (1.5,1);
	\draw[pstdhl,dashed] (1.5,-1) .. controls (1.5,-.5) and (2.5,-.5) .. (2.5,0) .. controls (2.5,.5) and (1.5,.5) .. (1.5,1);
	\node at (1,-.9){\small $\dots$}; \node at (1,.9){\small $\dots$};
	\draw[pstdhl] (0,-1) node[below]{\small $\mu_1$} -- (0,0) node[pos=.75, nail]{}
			 -- (0,1) node[pos=.25, nail]{};
}
\]
we obtain a collection of diagrams that typically look like
\begin{equation} \label{eq:doublenaildiagcol}
\tikzdiagl{
	\draw (6,-1.5) .. controls (6,-1) and (.5,-1.25) .. (.5,-.5)
		.. controls (.5,-.375) .. (0,-.25).. controls (1,0) .. 
		(1,.5) node[pos=1,tikzstar]{} .. controls (1+.25,.75) and (4.5,.75) .. (4.5,1) .. controls (4.5,1.25) and (4,1.25) .. (4,1.5);  
	%
	%
	\draw (6.5,-1.5) .. controls (6.5,-1) and (5,-1) .. (5,-.5) node[pos=1,tikzstar]{}
		--
		(5,.5) node[pos=1,tikzstar]{} .. controls (5,1) and (6.5,1) .. (6.5,1.5);
	\draw (2, -1.5) .. controls (2,-1.25) and (2.5,-1.25) .. (2.5,-1) .. controls (2.5,-.75) and (1+.25,-.75) .. (1,-.5) node[pos=1,tikzstar]{}
		.. controls (1,0) .. (0,.25) .. controls (.5,.375) ..
		 (.5,.5) .. controls (.5,1.25) and (6,1) .. (6,1.5);
	%
	%
	%
	\draw (4, -1.5) .. controls (4,-1.25) and (4.5,-1.25) .. (4.5,-1) .. controls (4.5,-.75) and (3,-.75) .. (3,-.5) node[pos=1,tikzstar]{}
		--
		(3,.5)  .. controls (3,1) and (2,1) .. (2,1.5);
	\begin{scope}[xshift=0cm]
		\draw (.5,-1.5) .. controls (.5,-1) and (1.5,-1) .. (1.5,-.5)
			--
			(1.5,.5) .. controls (1.5,1) and (.5,1) .. (.5,1.5);
		\draw[dashed, pstdhl] (.5,-1.5) .. controls (.5,-1) and (1.5,-1) .. (1.5,-.5)
			--
			(1.5,.5) .. controls (1.5,1) and (.5,1) .. (.5,1.5);
		\draw (1.5,-1.5) .. controls (1.5,-1) and (2.5,-1) .. (2.5,-.5)
			--
			(2.5,.5) .. controls (2.5,1) and (1.5,1) .. (1.5,1.5);
		\draw[dashed, pstdhl] (1.5,-1.5) .. controls (1.5,-1) and (2.5,-1) .. (2.5,-.5)
			--
			(2.5,.5) .. controls (2.5,1) and (1.5,1) .. (1.5,1.5);
		\node at(1,-1.4){\small $\dots$};  \node at(2,0){\small $\dots$}; \node at(1,1.4){\small $\dots$};
	\end{scope}
	\begin{scope}[xshift=2cm]
		\draw (.5,-1.5) .. controls (.5,-1) and (1.5,-1) .. (1.5,-.5)
			--
			(1.5,.5) .. controls (1.5,1) and (.5,1) .. (.5,1.5);
		\draw[dashed, pstdhl] (.5,-1.5) .. controls (.5,-1) and (1.5,-1) .. (1.5,-.5)
			--
			(1.5,.5) .. controls (1.5,1) and (.5,1) .. (.5,1.5);
		\draw (1.5,-1.5) .. controls (1.5,-1) and (2.5,-1) .. (2.5,-.5)
			--
			(2.5,.5) .. controls (2.5,1) and (1.5,1) .. (1.5,1.5);
		\draw[dashed, pstdhl] (1.5,-1.5) .. controls (1.5,-1) and (2.5,-1) .. (2.5,-.5)
			--
			(2.5,.5) .. controls (2.5,1) and (1.5,1) .. (1.5,1.5);
		\node at(1,-1.4){\small $\dots$};  \node at(2,0){\small $\dots$}; \node at(1,1.4){\small $\dots$};
	\end{scope}
	\begin{scope}[xshift=4cm]
		\draw (.5,-1.5) .. controls (.5,-1) and (1.5,-1) .. (1.5,-.5)
			--
			(1.5,.5) .. controls (1.5,1) and (.5,1) .. (.5,1.5);
		\draw[dashed, pstdhl] (.5,-1.5) .. controls (.5,-1) and (1.5,-1) .. (1.5,-.5)
			--
			(1.5,.5) .. controls (1.5,1) and (.5,1) .. (.5,1.5);
		\draw (1.5,-1.5) .. controls (1.5,-1) and (2.5,-1) .. (2.5,-.5)
			--
			(2.5,.5) .. controls (2.5,1) and (1.5,1) .. (1.5,1.5);
		\draw[dashed, pstdhl] (1.5,-1.5) .. controls (1.5,-1) and (2.5,-1) .. (2.5,-.5)
			--
			(2.5,.5) .. controls (2.5,1) and (1.5,1) .. (1.5,1.5);
		\node at(1,-1.4){\small $\dots$};  \node at(2,0){\small $\dots$}; \node at(1,1.4){\small $\dots$};
	\end{scope}
	\draw[pstdhl] (0,-1.5) node[below]{\small $\mu_1$} -- (0,-.5) 
			 -- (0,.5) node[pos=.25, nail]{} node[pos=.75,nail]{}
			 -- (0,1.5);
}
\end{equation}
together with the following diagrams
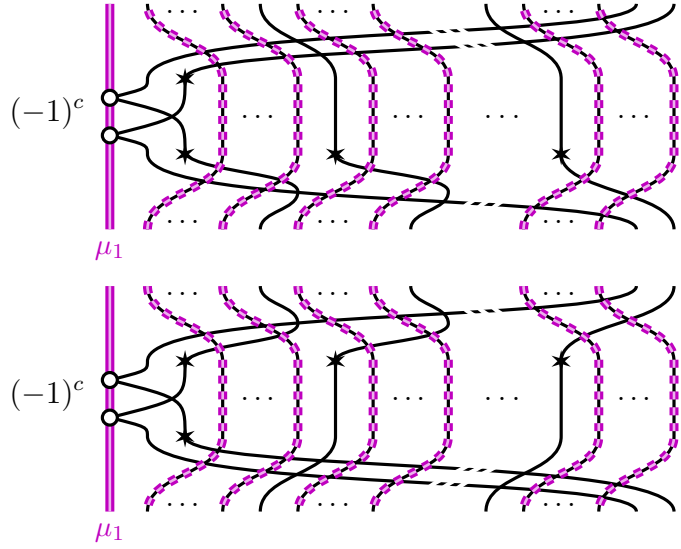
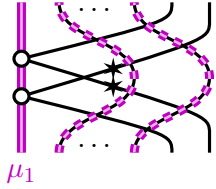
\begin{align*}
(-1)^{c}
\tikzdiagl{
	\draw (7,-1.5) .. controls (7,-1) and (.5,-1.25) .. (.5,-.5)
		.. controls (.5,-.375) .. (0,-.25).. controls (1,0) .. 
		(1,.5) node[pos=1,tikzstar]{}  .. controls (1,1) and (7.5,.75) .. (7.5,1.5);
	%
	%
	%
	\draw (2, -1.5) .. controls (2,-1.25) and (2.5,-1.25) .. (2.5,-1) .. controls (2.5,-.75) and (1+.25,-.75) .. (1,-.5) node[pos=1,tikzstar]{}
		.. controls (1,0) .. (0,.25) .. controls (.5,.375) ..
		 (.5,.5) .. controls (.5,1.25) and (7,1) .. (7,1.5);
	%
	%
	%
	\draw (4, -1.5) .. controls (4,-1.25) and (4.5,-1.25) .. (4.5,-1) .. controls (4.5,-.75) and (3,-.75) .. (3,-.5) node[pos=1,tikzstar]{}
		--
		(3,.5)  .. controls (3,1) and (2,1) .. (2,1.5);
	\node at (5.25,0){$\dots$};
	\draw[color=white,xshift=-.2cm,line width=1.6pt] (4.75,-1.5) .. controls (4.75,-1) and (5.75,-1) .. (5.75,-.5);
	\draw[color=white,line width=1.6pt] (4.75,-1.5) .. controls (4.75,-1) and (5.75,-1) .. (5.75,-.5);
	\draw[color=white,xshift=.2cm,line width=1.6pt] (4.75,-1.5) .. controls (4.75,-1) and (5.75,-1) .. (5.75,-.5);
	\draw[color=white,xshift=-.2cm,line width=1.6pt] (4.25,1.5) .. controls (4.25,1) and (5.25,1) .. (5.25,.5);
	\draw[color=white,line width=1.6pt] (4.25,1.5) .. controls (4.25,1) and (5.25,1) .. (5.25,.5);
	\draw[color=white,xshift=.2cm,line width=1.6pt] (4.25,1.5) .. controls (4.25,1) and (5.25,1) .. (5.25,.5);
	\draw (7.5,-1.5) .. controls (7.5,-1) and (6,-1) .. (6,-.5) node[pos=1,tikzstar]{}
		--
		(6,.5) .. controls (6,1) and (5,1) .. (5,1.5);
	\begin{scope}[xshift=0cm]
		\draw (.5,-1.5) .. controls (.5,-1) and (1.5,-1) .. (1.5,-.5)
			--
			(1.5,.5) .. controls (1.5,1) and (.5,1) .. (.5,1.5);
		\draw[dashed, pstdhl] (.5,-1.5) .. controls (.5,-1) and (1.5,-1) .. (1.5,-.5)
			--
			(1.5,.5) .. controls (1.5,1) and (.5,1) .. (.5,1.5);
		\draw (1.5,-1.5) .. controls (1.5,-1) and (2.5,-1) .. (2.5,-.5)
			--
			(2.5,.5) .. controls (2.5,1) and (1.5,1) .. (1.5,1.5);
		\draw[dashed, pstdhl] (1.5,-1.5) .. controls (1.5,-1) and (2.5,-1) .. (2.5,-.5)
			--
			(2.5,.5) .. controls (2.5,1) and (1.5,1) .. (1.5,1.5);
		\node at(1,-1.4){\small $\dots$};  \node at(2,0){\small $\dots$}; \node at(1,1.4){\small $\dots$};
	\end{scope}
	\begin{scope}[xshift=2cm]
		\draw (.5,-1.5) .. controls (.5,-1) and (1.5,-1) .. (1.5,-.5)
			--
			(1.5,.5) .. controls (1.5,1) and (.5,1) .. (.5,1.5);
		\draw[dashed, pstdhl] (.5,-1.5) .. controls (.5,-1) and (1.5,-1) .. (1.5,-.5)
			--
			(1.5,.5) .. controls (1.5,1) and (.5,1) .. (.5,1.5);
		\draw (1.5,-1.5) .. controls (1.5,-1) and (2.5,-1) .. (2.5,-.5)
			--
			(2.5,.5) .. controls (2.5,1) and (1.5,1) .. (1.5,1.5);
		\draw[dashed, pstdhl] (1.5,-1.5) .. controls (1.5,-1) and (2.5,-1) .. (2.5,-.5)
			--
			(2.5,.5) .. controls (2.5,1) and (1.5,1) .. (1.5,1.5);
		\node at(1,-1.4){\small $\dots$};  \node at(2,0){\small $\dots$}; \node at(1,1.4){\small $\dots$};
	\end{scope}
	\begin{scope}[xshift=5cm]
		\draw (.5,-1.5) .. controls (.5,-1) and (1.5,-1) .. (1.5,-.5)
			--
			(1.5,.5) .. controls (1.5,1) and (.5,1) .. (.5,1.5);
		\draw[dashed, pstdhl] (.5,-1.5) .. controls (.5,-1) and (1.5,-1) .. (1.5,-.5)
			--
			(1.5,.5) .. controls (1.5,1) and (.5,1) .. (.5,1.5);
		\draw (1.5,-1.5) .. controls (1.5,-1) and (2.5,-1) .. (2.5,-.5)
			--
			(2.5,.5) .. controls (2.5,1) and (1.5,1) .. (1.5,1.5);
		\draw[dashed, pstdhl] (1.5,-1.5) .. controls (1.5,-1) and (2.5,-1) .. (2.5,-.5)
			--
			(2.5,.5) .. controls (2.5,1) and (1.5,1) .. (1.5,1.5);
		\node at(1,-1.4){\small $\dots$};  \node at(2,0){\small $\dots$}; \node at(1,1.4){\small $\dots$};
	\end{scope}
	\draw[pstdhl] (0,-1.5) node[below]{\small $\mu_1$} -- (0,-.5) 
			 -- (0,.5) node[pos=.25, nail]{} node[pos=.75,nail]{}
			 -- (0,1.5);
}
\\
(-1)^{c}
\tikzdiagl[yscale=-1]{
	\draw (7,-1.5) .. controls (7,-1) and (.5,-1.25) .. (.5,-.5)
		.. controls (.5,-.375) .. (0,-.25).. controls (1,0) .. 
		(1,.5) node[pos=1,tikzstar]{}  .. controls (1,1) and (7.5,.75) .. (7.5,1.5);
	%
	%
	%
	\draw (2, -1.5) .. controls (2,-1.25) and (2.5,-1.25) .. (2.5,-1) .. controls (2.5,-.75) and (1+.25,-.75) .. (1,-.5) node[pos=1,tikzstar]{}
		.. controls (1,0) .. (0,.25) .. controls (.5,.375) ..
		 (.5,.5) .. controls (.5,1.25) and (7,1) .. (7,1.5);
	%
	%
	%
	\draw (4, -1.5) .. controls (4,-1.25) and (4.5,-1.25) .. (4.5,-1) .. controls (4.5,-.75) and (3,-.75) .. (3,-.5) node[pos=1,tikzstar]{}
		--
		(3,.5)  .. controls (3,1) and (2,1) .. (2,1.5);
	\node at (5.25,0){$\dots$};
	\draw[color=white,xshift=-.2cm,line width=1.6pt] (4.75,-1.5) .. controls (4.75,-1) and (5.75,-1) .. (5.75,-.5);
	\draw[color=white,line width=1.6pt] (4.75,-1.5) .. controls (4.75,-1) and (5.75,-1) .. (5.75,-.5);
	\draw[color=white,xshift=.2cm,line width=1.6pt] (4.75,-1.5) .. controls (4.75,-1) and (5.75,-1) .. (5.75,-.5);
	\draw[color=white,xshift=-.2cm,line width=1.6pt] (4.25,1.5) .. controls (4.25,1) and (5.25,1) .. (5.25,.5);
	\draw[color=white,line width=1.6pt] (4.25,1.5) .. controls (4.25,1) and (5.25,1) .. (5.25,.5);
	\draw[color=white,xshift=.2cm,line width=1.6pt] (4.25,1.5) .. controls (4.25,1) and (5.25,1) .. (5.25,.5);
	\draw (7.5,-1.5) .. controls (7.5,-1) and (6,-1) .. (6,-.5) node[pos=1,tikzstar]{}
		--
		(6,.5) .. controls (6,1) and (5,1) .. (5,1.5);
	\begin{scope}[xshift=0cm]
		\draw (.5,-1.5) .. controls (.5,-1) and (1.5,-1) .. (1.5,-.5)
			--
			(1.5,.5) .. controls (1.5,1) and (.5,1) .. (.5,1.5);
		\draw[dashed, pstdhl] (.5,-1.5) .. controls (.5,-1) and (1.5,-1) .. (1.5,-.5)
			--
			(1.5,.5) .. controls (1.5,1) and (.5,1) .. (.5,1.5);
		\draw (1.5,-1.5) .. controls (1.5,-1) and (2.5,-1) .. (2.5,-.5)
			--
			(2.5,.5) .. controls (2.5,1) and (1.5,1) .. (1.5,1.5);
		\draw[dashed, pstdhl] (1.5,-1.5) .. controls (1.5,-1) and (2.5,-1) .. (2.5,-.5)
			--
			(2.5,.5) .. controls (2.5,1) and (1.5,1) .. (1.5,1.5);
		\node at(1,-1.4){\small $\dots$};  \node at(2,0){\small $\dots$}; \node at(1,1.4){\small $\dots$};
	\end{scope}
	\begin{scope}[xshift=2cm]
		\draw (.5,-1.5) .. controls (.5,-1) and (1.5,-1) .. (1.5,-.5)
			--
			(1.5,.5) .. controls (1.5,1) and (.5,1) .. (.5,1.5);
		\draw[dashed, pstdhl] (.5,-1.5) .. controls (.5,-1) and (1.5,-1) .. (1.5,-.5)
			--
			(1.5,.5) .. controls (1.5,1) and (.5,1) .. (.5,1.5);
		\draw (1.5,-1.5) .. controls (1.5,-1) and (2.5,-1) .. (2.5,-.5)
			--
			(2.5,.5) .. controls (2.5,1) and (1.5,1) .. (1.5,1.5);
		\draw[dashed, pstdhl] (1.5,-1.5) .. controls (1.5,-1) and (2.5,-1) .. (2.5,-.5)
			--
			(2.5,.5) .. controls (2.5,1) and (1.5,1) .. (1.5,1.5);
		\node at(1,-1.4){\small $\dots$};  \node at(2,0){\small $\dots$}; \node at(1,1.4){\small $\dots$};
	\end{scope}
	\begin{scope}[xshift=5cm]
		\draw (.5,-1.5) .. controls (.5,-1) and (1.5,-1) .. (1.5,-.5)
			--
			(1.5,.5) .. controls (1.5,1) and (.5,1) .. (.5,1.5);
		\draw[dashed, pstdhl] (.5,-1.5) .. controls (.5,-1) and (1.5,-1) .. (1.5,-.5)
			--
			(1.5,.5) .. controls (1.5,1) and (.5,1) .. (.5,1.5);
		\draw (1.5,-1.5) .. controls (1.5,-1) and (2.5,-1) .. (2.5,-.5)
			--
			(2.5,.5) .. controls (2.5,1) and (1.5,1) .. (1.5,1.5);
		\draw[dashed, pstdhl] (1.5,-1.5) .. controls (1.5,-1) and (2.5,-1) .. (2.5,-.5)
			--
			(2.5,.5) .. controls (2.5,1) and (1.5,1) .. (1.5,1.5);
		\node at(1,-1.4){\small $\dots$};  \node at(2,0){\small $\dots$}; \node at(1,1.4){\small $\dots$};
	\end{scope}
	\draw[pstdhl] (0,-1.5) -- (0,-.5) 
			 -- (0,.5) node[pos=.25, nail]{} node[pos=.75,nail]{}
			 -- (0,1.5) node[below]{\small $\mu_1$};
}
\end{align*}
and
\[
\tikzdiagl
{
	\draw (2,-1) 	.. controls (2,-.75) .. (0,-.25)  
				.. controls (2.5,.5) .. (2.5,.75) node[pos=.2, tikzstar]{}
				-- (2.5,1) ;
	\draw (2,1) 	.. controls (2,.75) .. (0,.25)  
				.. controls (2.5,-.5) .. (2.5,-.75) node[pos=.2, tikzstar]{}
				-- (2.5,-1) ;
	\draw (.5,-1) .. controls (.5,-.5) and (1.5,-.5) .. (1.5,0) .. controls (1.5,.5) and (.5,.5) .. (.5,1);
	\draw[pstdhl,dashed] (.5,-1) .. controls (.5,-.5) and (1.5,-.5) .. (1.5,0) .. controls (1.5,.5) and (.5,.5) .. (.5,1);
	\draw (1.5,-1) .. controls (1.5,-.5) and (2.5,-.5) .. (2.5,0) .. controls (2.5,.5) and (1.5,.5) .. (1.5,1);
	\draw[pstdhl,dashed] (1.5,-1) .. controls (1.5,-.5) and (2.5,-.5) .. (2.5,0) .. controls (2.5,.5) and (1.5,.5) .. (1.5,1);
	\node at (1,-.9){\small $\dots$}; \node at (1,.9){\small $\dots$};
	\draw[pstdhl] (0,-1) node[below]{\small $\mu_1$} -- (0,0) node[pos=.75, nail]{}
			 -- (0,1) node[pos=.25, nail]{};
}
\]
All the diagrams of the same shape as in \cref{eq:doublenaildiagcol} are zero since
\[
\tikzdiagl
{
	\draw (.5,-1) .. controls (.5,-.75) .. (0,-.5) 
			.. controls (1,0) .. (1,1) node[pos=.75,tikzstar]{};
	\draw (.5,1) .. controls (.5,.75) .. (0,.5) 
			.. controls (1,0) .. (1,-1)  node[pos=.75,tikzstar]{};
	\draw[pstdhl] (0,-1)node[below]{\small $\mu_1$} -- (0,0) node[midway, nail]{}
			 -- (0,1) node[midway, nail]{};
}
\ = - \ 
\tikzdiagl
{
	\draw (.5,-1) .. controls (.5,-.75) .. (0,-.5) 
			.. controls (1,0) .. (1,1) node[pos=.75,tikzstar]{};
	\draw (.5,1) .. controls (.5,.75) .. (0,.5)  node[pos=.25,tikzstar]{}
			.. controls (1,0) .. (1,-1) ;
	\draw[pstdhl] (0,-1)node[below]{\small $\mu_1$} -- (0,0) node[midway, nail]{}
			 -- (0,1) node[midway, nail]{};
}
\]
and because of \cref{eq:doublestarcommutes} and \cref{lem:doublecrossingnailzero}. 
Applying \cref{eq:starslidingoverallcrossings} to
\[
\tikzdiagl
{
	\draw (2.5,-1) 	.. controls (2.5,-.75) .. (0,-.25)  node[pos=.2, tikzstar]{}
				.. controls (2.5,.5) .. (2.5,.75) 
				-- (2.5,1) ;
	\draw (2,1) 		.. controls (2,.75) .. (0,.25) 
		 		.. controls (1,0) .. (1,-.25)
				.. controls (1,-.375) and (.4,-.375) .. (.4,-.6)
				.. controls (.4,-.75) and (2,-.75) .. (2,-1);
	\draw (.5,-1) .. controls (.5,-.5) and (1.5,-.5) .. (1.5,0) .. controls (1.5,.5) and (.5,.5) .. (.5,1);
	\draw[pstdhl,dashed] (.5,-1) .. controls (.5,-.5) and (1.5,-.5) .. (1.5,0) .. controls (1.5,.5) and (.5,.5) .. (.5,1);
	\draw (1.5,-1) .. controls (1.5,-.5) and (2.5,-.5) .. (2.5,0) .. controls (2.5,.5) and (1.5,.5) .. (1.5,1);
	\draw[pstdhl,dashed] (1.5,-1) .. controls (1.5,-.5) and (2.5,-.5) .. (2.5,0) .. controls (2.5,.5) and (1.5,.5) .. (1.5,1);
	\node at (1,-.9){\small $\dots$}; \node at (2,0){\small $\dots$}; \node at (1,.9){\small $\dots$};
	\draw[pstdhl] (0,-1) node[below]{\small $\mu_1$} -- (0,0) node[pos=.75, nail]{}
			 -- (0,1) node[pos=.25, nail]{};
}
\]
yields a collection of elements
\begin{align*}
&
(-1)^{c}
\tikzdiagl{
	 \draw (7,-1.5) .. controls (7,-1.25) and (.5,-1.25) .. (.5,-1)
	 	.. controls (.5,-.75) and (1,-.75) .. (1,-.5)
		.. controls (1,0) .. (0,.25) .. controls (.5,.375) ..
		 (.5,.5) .. controls (.5,1.25) and (7,1) .. (7,1.5);
	%
	%
	%
	%
	\draw (2, -1.5) .. controls (2,-1.25) and (2.5,-1.25) .. (2.5,-1) .. controls (2.5,-.75) and (.5+.25,-.75) .. (.5,-.5) node[pos=.65,tikzstar]{}
		.. controls (.5,-.375) .. (0,-.25).. controls (1,0) .. 
		(1,.5)   .. controls (1,1) and (7.5,.75) .. (7.5,1.5);
	%
	%
	%
	\draw (4, -1.5) .. controls (4,-1.25) and (4.5,-1.25) .. (4.5,-1) .. controls (4.5,-.75) and (3,-.75) .. (3,-.5) node[pos=1,tikzstar]{}
		--
		(3,.5)  .. controls (3,1) and (2,1) .. (2,1.5);
	\node at (5.25,0){$\dots$};
	\draw[color=white,xshift=-.2cm,line width=1.6pt] (4.75,-1.5) .. controls (4.75,-1) and (5.75,-1) .. (5.75,-.5);
	\draw[color=white,line width=1.6pt] (4.75,-1.5) .. controls (4.75,-1) and (5.75,-1) .. (5.75,-.5);
	\draw[color=white,xshift=.2cm,line width=1.6pt] (4.75,-1.5) .. controls (4.75,-1) and (5.75,-1) .. (5.75,-.5);
	\draw[color=white,xshift=-.2cm,line width=1.6pt] (4.25,1.5) .. controls (4.25,1) and (5.25,1) .. (5.25,.5);
	\draw[color=white,line width=1.6pt] (4.25,1.5) .. controls (4.25,1) and (5.25,1) .. (5.25,.5);
	\draw[color=white,xshift=.2cm,line width=1.6pt] (4.25,1.5) .. controls (4.25,1) and (5.25,1) .. (5.25,.5);
	\draw (7.5,-1.5) .. controls (7.5,-1) and (6,-1) .. (6,-.5) node[pos=1,tikzstar]{}
		--
		(6,.5) .. controls (6,1) and (5,1) .. (5,1.5);
	\begin{scope}[xshift=0cm]
		\draw (.5,-1.5) .. controls (.5,-1) and (1.5,-1) .. (1.5,-.5)
			--
			(1.5,.5) .. controls (1.5,1) and (.5,1) .. (.5,1.5);
		\draw[dashed, pstdhl] (.5,-1.5) .. controls (.5,-1) and (1.5,-1) .. (1.5,-.5)
			--
			(1.5,.5) .. controls (1.5,1) and (.5,1) .. (.5,1.5);
		\draw (1.5,-1.5) .. controls (1.5,-1) and (2.5,-1) .. (2.5,-.5)
			--
			(2.5,.5) .. controls (2.5,1) and (1.5,1) .. (1.5,1.5);
		\draw[dashed, pstdhl] (1.5,-1.5) .. controls (1.5,-1) and (2.5,-1) .. (2.5,-.5)
			--
			(2.5,.5) .. controls (2.5,1) and (1.5,1) .. (1.5,1.5);
		\node at(1,-1.4){\small $\dots$};  \node at(2,0){\small $\dots$}; \node at(1,1.4){\small $\dots$};
	\end{scope}
	\begin{scope}[xshift=2cm]
		\draw (.5,-1.5) .. controls (.5,-1) and (1.5,-1) .. (1.5,-.5)
			--
			(1.5,.5) .. controls (1.5,1) and (.5,1) .. (.5,1.5);
		\draw[dashed, pstdhl] (.5,-1.5) .. controls (.5,-1) and (1.5,-1) .. (1.5,-.5)
			--
			(1.5,.5) .. controls (1.5,1) and (.5,1) .. (.5,1.5);
		\draw (1.5,-1.5) .. controls (1.5,-1) and (2.5,-1) .. (2.5,-.5)
			--
			(2.5,.5) .. controls (2.5,1) and (1.5,1) .. (1.5,1.5);
		\draw[dashed, pstdhl] (1.5,-1.5) .. controls (1.5,-1) and (2.5,-1) .. (2.5,-.5)
			--
			(2.5,.5) .. controls (2.5,1) and (1.5,1) .. (1.5,1.5);
		\node at(1,-1.4){\small $\dots$};  \node at(2,0){\small $\dots$}; \node at(1,1.4){\small $\dots$};
	\end{scope}
	\begin{scope}[xshift=5cm]
		\draw (.5,-1.5) .. controls (.5,-1) and (1.5,-1) .. (1.5,-.5)
			--
			(1.5,.5) .. controls (1.5,1) and (.5,1) .. (.5,1.5);
		\draw[dashed, pstdhl] (.5,-1.5) .. controls (.5,-1) and (1.5,-1) .. (1.5,-.5)
			--
			(1.5,.5) .. controls (1.5,1) and (.5,1) .. (.5,1.5);
		\draw (1.5,-1.5) .. controls (1.5,-1) and (2.5,-1) .. (2.5,-.5)
			--
			(2.5,.5) .. controls (2.5,1) and (1.5,1) .. (1.5,1.5);
		\draw[dashed, pstdhl] (1.5,-1.5) .. controls (1.5,-1) and (2.5,-1) .. (2.5,-.5)
			--
			(2.5,.5) .. controls (2.5,1) and (1.5,1) .. (1.5,1.5);
		\node at(1,-1.4){\small $\dots$};  \node at(2,0){\small $\dots$}; \node at(1,1.4){\small $\dots$};
	\end{scope}
	\draw[pstdhl] (0,-1.5) node[below]{\small $\mu_1$} -- (0,-.5) 
			 -- (0,.5) node[pos=.25, nail]{} node[pos=.75,nail]{}
			 -- (0,1.5);
}
\\
& \ = - \ 
(-1)^{c}
\tikzdiagl{
	\draw (7,-1.5) .. controls (7,-1) and (.5,-1.25) .. (.5,-.5)
		.. controls (.5,-.375) .. (0,-.25).. controls (1,0) .. 
		(1,.5) node[pos=1,tikzstar]{}  .. controls (1,1) and (7.5,.75) .. (7.5,1.5);
	%
	%
	%
	\draw (2, -1.5) .. controls (2,-1.25) and (2.5,-1.25) .. (2.5,-1) .. controls (2.5,-.75) and (1+.25,-.75) .. (1,-.5) node[pos=1,tikzstar]{}
		.. controls (1,0) .. (0,.25) .. controls (.5,.375) ..
		 (.5,.5) .. controls (.5,1.25) and (7,1) .. (7,1.5);
	%
	%
	%
	\draw (4, -1.5) .. controls (4,-1.25) and (4.5,-1.25) .. (4.5,-1) .. controls (4.5,-.75) and (3,-.75) .. (3,-.5) node[pos=1,tikzstar]{}
		--
		(3,.5)  .. controls (3,1) and (2,1) .. (2,1.5);
	\node at (5.25,0){$\dots$};
	\draw[color=white,xshift=-.2cm,line width=1.6pt] (4.75,-1.5) .. controls (4.75,-1) and (5.75,-1) .. (5.75,-.5);
	\draw[color=white,line width=1.6pt] (4.75,-1.5) .. controls (4.75,-1) and (5.75,-1) .. (5.75,-.5);
	\draw[color=white,xshift=.2cm,line width=1.6pt] (4.75,-1.5) .. controls (4.75,-1) and (5.75,-1) .. (5.75,-.5);
	\draw[color=white,xshift=-.2cm,line width=1.6pt] (4.25,1.5) .. controls (4.25,1) and (5.25,1) .. (5.25,.5);
	\draw[color=white,line width=1.6pt] (4.25,1.5) .. controls (4.25,1) and (5.25,1) .. (5.25,.5);
	\draw[color=white,xshift=.2cm,line width=1.6pt] (4.25,1.5) .. controls (4.25,1) and (5.25,1) .. (5.25,.5);
	\draw (7.5,-1.5) .. controls (7.5,-1) and (6,-1) .. (6,-.5) node[pos=1,tikzstar]{}
		--
		(6,.5) .. controls (6,1) and (5,1) .. (5,1.5);
	\begin{scope}[xshift=0cm]
		\draw (.5,-1.5) .. controls (.5,-1) and (1.5,-1) .. (1.5,-.5)
			--
			(1.5,.5) .. controls (1.5,1) and (.5,1) .. (.5,1.5);
		\draw[dashed, pstdhl] (.5,-1.5) .. controls (.5,-1) and (1.5,-1) .. (1.5,-.5)
			--
			(1.5,.5) .. controls (1.5,1) and (.5,1) .. (.5,1.5);
		\draw (1.5,-1.5) .. controls (1.5,-1) and (2.5,-1) .. (2.5,-.5)
			--
			(2.5,.5) .. controls (2.5,1) and (1.5,1) .. (1.5,1.5);
		\draw[dashed, pstdhl] (1.5,-1.5) .. controls (1.5,-1) and (2.5,-1) .. (2.5,-.5)
			--
			(2.5,.5) .. controls (2.5,1) and (1.5,1) .. (1.5,1.5);
		\node at(1,-1.4){\small $\dots$};  \node at(2,0){\small $\dots$}; \node at(1,1.4){\small $\dots$};
	\end{scope}
	\begin{scope}[xshift=2cm]
		\draw (.5,-1.5) .. controls (.5,-1) and (1.5,-1) .. (1.5,-.5)
			--
			(1.5,.5) .. controls (1.5,1) and (.5,1) .. (.5,1.5);
		\draw[dashed, pstdhl] (.5,-1.5) .. controls (.5,-1) and (1.5,-1) .. (1.5,-.5)
			--
			(1.5,.5) .. controls (1.5,1) and (.5,1) .. (.5,1.5);
		\draw (1.5,-1.5) .. controls (1.5,-1) and (2.5,-1) .. (2.5,-.5)
			--
			(2.5,.5) .. controls (2.5,1) and (1.5,1) .. (1.5,1.5);
		\draw[dashed, pstdhl] (1.5,-1.5) .. controls (1.5,-1) and (2.5,-1) .. (2.5,-.5)
			--
			(2.5,.5) .. controls (2.5,1) and (1.5,1) .. (1.5,1.5);
		\node at(1,-1.4){\small $\dots$};  \node at(2,0){\small $\dots$}; \node at(1,1.4){\small $\dots$};
	\end{scope}
	\begin{scope}[xshift=5cm]
		\draw (.5,-1.5) .. controls (.5,-1) and (1.5,-1) .. (1.5,-.5)
			--
			(1.5,.5) .. controls (1.5,1) and (.5,1) .. (.5,1.5);
		\draw[dashed, pstdhl] (.5,-1.5) .. controls (.5,-1) and (1.5,-1) .. (1.5,-.5)
			--
			(1.5,.5) .. controls (1.5,1) and (.5,1) .. (.5,1.5);
		\draw (1.5,-1.5) .. controls (1.5,-1) and (2.5,-1) .. (2.5,-.5)
			--
			(2.5,.5) .. controls (2.5,1) and (1.5,1) .. (1.5,1.5);
		\draw[dashed, pstdhl] (1.5,-1.5) .. controls (1.5,-1) and (2.5,-1) .. (2.5,-.5)
			--
			(2.5,.5) .. controls (2.5,1) and (1.5,1) .. (1.5,1.5);
		\node at(1,-1.4){\small $\dots$};  \node at(2,0){\small $\dots$}; \node at(1,1.4){\small $\dots$};
	\end{scope}
	\draw[pstdhl] (0,-1.5) node[below]{\small $\mu_1$} -- (0,-.5) 
			 -- (0,.5) node[pos=.25, nail]{} node[pos=.75,nail]{}
			 -- (0,1.5);
}
\end{align*}
and the element
\[
\tikzdiagl
{
	\draw (2.5,-1) 	.. controls (2.5,-.75) .. (0,-.25)  node[pos=.875, tikzstar]{}
				.. controls (2.5,.5) .. (2.5,.75) 
				-- (2.5,1) ;
	\draw (2,1) 		.. controls (2,.75) .. (0,.25) 
		 		.. controls (.7,0) .. (.7,-.25)
				.. controls (.7,-.375) and (.4,-.375) .. (.4,-.6)
				.. controls (.4,-.75) and (2,-.75) .. (2,-1);
	\draw (.5,-1) .. controls (.5,-.5) and (1.5,-.5) .. (1.5,0) .. controls (1.5,.5) and (.5,.5) .. (.5,1);
	\draw[pstdhl,dashed] (.5,-1) .. controls (.5,-.5) and (1.5,-.5) .. (1.5,0) .. controls (1.5,.5) and (.5,.5) .. (.5,1);
	\draw (1.5,-1) .. controls (1.5,-.5) and (2.5,-.5) .. (2.5,0) .. controls (2.5,.5) and (1.5,.5) .. (1.5,1);
	\draw[pstdhl,dashed] (1.5,-1) .. controls (1.5,-.5) and (2.5,-.5) .. (2.5,0) .. controls (2.5,.5) and (1.5,.5) .. (1.5,1);
	\node at (1,-.9){\small $\dots$}; \node at (2,0){\small $\dots$}; \node at (1,.9){\small $\dots$};
	\draw[pstdhl] (0,-1) node[below]{\small $\mu_1$} -- (0,0) node[pos=.75, nail]{}
			 -- (0,1) node[pos=.25, nail]{};
}
\  = \ 
\tikzdiagl
{
	\draw (2.5,-1) 	.. controls (2.5,-.75) .. (0,-.25) 
				.. controls (2.5,.5) .. (2.5,.75)  node[pos=.125, tikzstar]{}
				-- (2.5,1) ;
	\draw (2,1) 		.. controls (2,.75) .. (0,.25) 
		 		.. controls (.7,0) .. (.7,-.25)
				.. controls (.7,-.375) and (.4,-.375) .. (.4,-.6)
				.. controls (.4,-.75) and (2,-.75) .. (2,-1);
	\draw (.5,-1) .. controls (.5,-.5) and (1.5,-.5) .. (1.5,0) .. controls (1.5,.5) and (.5,.5) .. (.5,1);
	\draw[pstdhl,dashed] (.5,-1) .. controls (.5,-.5) and (1.5,-.5) .. (1.5,0) .. controls (1.5,.5) and (.5,.5) .. (.5,1);
	\draw (1.5,-1) .. controls (1.5,-.5) and (2.5,-.5) .. (2.5,0) .. controls (2.5,.5) and (1.5,.5) .. (1.5,1);
	\draw[pstdhl,dashed] (1.5,-1) .. controls (1.5,-.5) and (2.5,-.5) .. (2.5,0) .. controls (2.5,.5) and (1.5,.5) .. (1.5,1);
	\node at (1,-.9){\small $\dots$}; \node at (2,0){\small $\dots$}; \node at (1,.9){\small $\dots$};
	\draw[pstdhl] (0,-1) node[below]{\small $\mu_1$} -- (0,0) node[pos=.75, nail]{}
			 -- (0,1) node[pos=.25, nail]{};
}
\ - \ 
\tikzdiagl
{
	\draw (2,-1) 	.. controls (2,-.75) .. (0,-.25)  
				.. controls (2.5,.5) .. (2.5,.75) node[pos=.2, tikzstar]{}
				-- (2.5,1) ;
	\draw (2,1) 	.. controls (2,.75) .. (0,.25)  
				.. controls (2.5,-.5) .. (2.5,-.75) node[pos=.2, tikzstar]{}
				-- (2.5,-1) ;
	\draw (.5,-1) .. controls (.5,-.5) and (1.5,-.5) .. (1.5,0) .. controls (1.5,.5) and (.5,.5) .. (.5,1);
	\draw[pstdhl,dashed] (.5,-1) .. controls (.5,-.5) and (1.5,-.5) .. (1.5,0) .. controls (1.5,.5) and (.5,.5) .. (.5,1);
	\draw (1.5,-1) .. controls (1.5,-.5) and (2.5,-.5) .. (2.5,0) .. controls (2.5,.5) and (1.5,.5) .. (1.5,1);
	\draw[pstdhl,dashed] (1.5,-1) .. controls (1.5,-.5) and (2.5,-.5) .. (2.5,0) .. controls (2.5,.5) and (1.5,.5) .. (1.5,1);
	\node at (1,-.9){\small $\dots$}; \node at (1,.9){\small $\dots$};
	\draw[pstdhl] (0,-1) node[below]{\small $\mu_1$} -- (0,0) node[pos=.75, nail]{}
			 -- (0,1) node[pos=.25, nail]{};
}
\]
Applying the symmetric of \cref{eq:starslidingoverallcrossings} to
\[
\ - \ 
\tikzdiagl
{
	\draw (2.5,-1) 	.. controls (2.5,-.75) .. (0,-.25) 
				.. controls (2.5,.5) .. (2.5,.75) 
				-- (2.5,1) node[pos=.5, tikzstar]{};
	\draw (2,1) 		.. controls (2,.75) .. (0,.25) 
		 		.. controls (1,0) .. (1,-.25)
				.. controls (1,-.375) and (.4,-.375) .. (.4,-.6)
				.. controls (.4,-.75) and (2,-.75) .. (2,-1);
	\draw (.5,-1) .. controls (.5,-.5) and (1.5,-.5) .. (1.5,0) .. controls (1.5,.5) and (.5,.5) .. (.5,1);
	\draw[pstdhl,dashed] (.5,-1) .. controls (.5,-.5) and (1.5,-.5) .. (1.5,0) .. controls (1.5,.5) and (.5,.5) .. (.5,1);
	\draw (1.5,-1) .. controls (1.5,-.5) and (2.5,-.5) .. (2.5,0) .. controls (2.5,.5) and (1.5,.5) .. (1.5,1);
	\draw[pstdhl,dashed] (1.5,-1) .. controls (1.5,-.5) and (2.5,-.5) .. (2.5,0) .. controls (2.5,.5) and (1.5,.5) .. (1.5,1);
	\node at (1,-.9){\small $\dots$}; \node at (2,0){\small $\dots$}; \node at (1,.9){\small $\dots$};
	\draw[pstdhl] (0,-1) node[below]{\small $\mu_1$} -- (0,0) node[pos=.75, nail]{}
			 -- (0,1) node[pos=.25, nail]{};
}
\]
yields a collection of elements
\begin{align*}
&-
(-1)^{c}
\tikzdiagl[yscale=-1]{
	\draw (7,-1.5) .. controls (7,-1) and (.5,-1.25) .. (.5,-.5)
		.. controls (.5,-.375) .. (0,-.25).. controls (1,0) .. 
		(1,.5) .. controls (1,.75) and (.5,.75) .. (.5,1) .. controls (.5,1.25) and (7,1.25) .. (7,1.5); 
	%
	%
	%
	\draw (2, -1.5) .. controls (2,-1.25) and (2.5,-1.25) .. (2.5,-1) .. controls (2.5,-.75) and (1+.25,-.75) .. (1,-.5) node[pos=1,tikzstar]{}
		.. controls (1,0) .. (0,.25) .. controls (.5,.375) ..
		 (.5,.5) .. controls (.5,1.25) and (7.5,.75) .. (7.5,1.5); 
	%
	%
	%
	\draw (4, -1.5) .. controls (4,-1.25) and (4.5,-1.25) .. (4.5,-1) .. controls (4.5,-.75) and (3,-.75) .. (3,-.5) node[pos=1,tikzstar]{}
		--
		(3,.5)  .. controls (3,1) and (2,1) .. (2,1.5);
	\node at (5.25,0){$\dots$};
	\draw[color=white,xshift=-.2cm,line width=1.6pt] (4.75,-1.5) .. controls (4.75,-1) and (5.75,-1) .. (5.75,-.5);
	\draw[color=white,line width=1.6pt] (4.75,-1.5) .. controls (4.75,-1) and (5.75,-1) .. (5.75,-.5);
	\draw[color=white,xshift=.2cm,line width=1.6pt] (4.75,-1.5) .. controls (4.75,-1) and (5.75,-1) .. (5.75,-.5);
	\draw[color=white,xshift=-.2cm,line width=1.6pt] (4.25,1.5) .. controls (4.25,1) and (5.25,1) .. (5.25,.5);
	\draw[color=white,line width=1.6pt] (4.25,1.5) .. controls (4.25,1) and (5.25,1) .. (5.25,.5);
	\draw[color=white,xshift=.2cm,line width=1.6pt] (4.25,1.5) .. controls (4.25,1) and (5.25,1) .. (5.25,.5);
	\draw (7.5,-1.5) .. controls (7.5,-1) and (6,-1) .. (6,-.5) node[pos=1,tikzstar]{}
		--
		(6,.5) .. controls (6,1) and (5,1) .. (5,1.5);
	\begin{scope}[xshift=0cm]
		\draw (.5,-1.5) .. controls (.5,-1) and (1.5,-1) .. (1.5,-.5)
			--
			(1.5,.5) .. controls (1.5,1) and (.5,1) .. (.5,1.5);
		\draw[dashed, pstdhl] (.5,-1.5) .. controls (.5,-1) and (1.5,-1) .. (1.5,-.5)
			--
			(1.5,.5) .. controls (1.5,1) and (.5,1) .. (.5,1.5);
		\draw (1.5,-1.5) .. controls (1.5,-1) and (2.5,-1) .. (2.5,-.5)
			--
			(2.5,.5) .. controls (2.5,1) and (1.5,1) .. (1.5,1.5);
		\draw[dashed, pstdhl] (1.5,-1.5) .. controls (1.5,-1) and (2.5,-1) .. (2.5,-.5)
			--
			(2.5,.5) .. controls (2.5,1) and (1.5,1) .. (1.5,1.5);
		\node at(1,-1.4){\small $\dots$};  \node at(2,0){\small $\dots$}; \node at(1,1.4){\small $\dots$};
	\end{scope}
	\begin{scope}[xshift=2cm]
		\draw (.5,-1.5) .. controls (.5,-1) and (1.5,-1) .. (1.5,-.5)
			--
			(1.5,.5) .. controls (1.5,1) and (.5,1) .. (.5,1.5);
		\draw[dashed, pstdhl] (.5,-1.5) .. controls (.5,-1) and (1.5,-1) .. (1.5,-.5)
			--
			(1.5,.5) .. controls (1.5,1) and (.5,1) .. (.5,1.5);
		\draw (1.5,-1.5) .. controls (1.5,-1) and (2.5,-1) .. (2.5,-.5)
			--
			(2.5,.5) .. controls (2.5,1) and (1.5,1) .. (1.5,1.5);
		\draw[dashed, pstdhl] (1.5,-1.5) .. controls (1.5,-1) and (2.5,-1) .. (2.5,-.5)
			--
			(2.5,.5) .. controls (2.5,1) and (1.5,1) .. (1.5,1.5);
		\node at(1,-1.4){\small $\dots$};  \node at(2,0){\small $\dots$}; \node at(1,1.4){\small $\dots$};
	\end{scope}
	\begin{scope}[xshift=5cm]
		\draw (.5,-1.5) .. controls (.5,-1) and (1.5,-1) .. (1.5,-.5)
			--
			(1.5,.5) .. controls (1.5,1) and (.5,1) .. (.5,1.5);
		\draw[dashed, pstdhl] (.5,-1.5) .. controls (.5,-1) and (1.5,-1) .. (1.5,-.5)
			--
			(1.5,.5) .. controls (1.5,1) and (.5,1) .. (.5,1.5);
		\draw (1.5,-1.5) .. controls (1.5,-1) and (2.5,-1) .. (2.5,-.5)
			--
			(2.5,.5) .. controls (2.5,1) and (1.5,1) .. (1.5,1.5);
		\draw[dashed, pstdhl] (1.5,-1.5) .. controls (1.5,-1) and (2.5,-1) .. (2.5,-.5)
			--
			(2.5,.5) .. controls (2.5,1) and (1.5,1) .. (1.5,1.5);
		\node at(1,-1.4){\small $\dots$};  \node at(2,0){\small $\dots$}; \node at(1,1.4){\small $\dots$};
	\end{scope}
	\draw[pstdhl] (0,-1.5) -- (0,-.5) 
			 -- (0,.5) node[pos=.25, nail]{} node[pos=.75,nail]{}
			 -- (0,1.5) node[below]{\small $\mu_1$};
}
\\
& \ = \ 
-
(-1)^{c}
\tikzdiagl[yscale=-1]{
	\draw (7,-1.5) .. controls (7,-1) and (.5,-1.25) .. (.5,-.5)
		.. controls (.5,-.375) .. (0,-.25).. controls (1,0) .. 
		(1,.5) .. controls (1,.75) and (.5,.75) .. (.5,1) .. controls (.5,1.25) and (7,1.25) .. (7,1.5); 
	%
	%
	%
	\draw (2, -1.5) .. controls (2,-1.25) and (2.5,-1.25) .. (2.5,-1)  .. controls (2.5,-.75) and (1+.25,-.75) .. (1,-.5)
		.. controls (1,0) .. (0,.25)  .. controls (.5,.375) ..
		 (.5,.5) node[pos=1,tikzstar]{}  .. controls (.5,1.25) and (7.5,.75) .. (7.5,1.5); 
	%
	%
	%
	\draw (4, -1.5) .. controls (4,-1.25) and (4.5,-1.25) .. (4.5,-1) .. controls (4.5,-.75) and (3,-.75) .. (3,-.5) node[pos=1,tikzstar]{}
		--
		(3,.5)  .. controls (3,1) and (2,1) .. (2,1.5);
	\node at (5.25,0){$\dots$};
	\draw[color=white,xshift=-.2cm,line width=1.6pt] (4.75,-1.5) .. controls (4.75,-1) and (5.75,-1) .. (5.75,-.5);
	\draw[color=white,line width=1.6pt] (4.75,-1.5) .. controls (4.75,-1) and (5.75,-1) .. (5.75,-.5);
	\draw[color=white,xshift=.2cm,line width=1.6pt] (4.75,-1.5) .. controls (4.75,-1) and (5.75,-1) .. (5.75,-.5);
	\draw[color=white,xshift=-.2cm,line width=1.6pt] (4.25,1.5) .. controls (4.25,1) and (5.25,1) .. (5.25,.5);
	\draw[color=white,line width=1.6pt] (4.25,1.5) .. controls (4.25,1) and (5.25,1) .. (5.25,.5);
	\draw[color=white,xshift=.2cm,line width=1.6pt] (4.25,1.5) .. controls (4.25,1) and (5.25,1) .. (5.25,.5);
	\draw (7.5,-1.5) .. controls (7.5,-1) and (6,-1) .. (6,-.5) node[pos=1,tikzstar]{}
		--
		(6,.5) .. controls (6,1) and (5,1) .. (5,1.5);
	\begin{scope}[xshift=0cm]
		\draw (.5,-1.5) .. controls (.5,-1) and (1.5,-1) .. (1.5,-.5)
			--
			(1.5,.5) .. controls (1.5,1) and (.5,1) .. (.5,1.5);
		\draw[dashed, pstdhl] (.5,-1.5) .. controls (.5,-1) and (1.5,-1) .. (1.5,-.5)
			--
			(1.5,.5) .. controls (1.5,1) and (.5,1) .. (.5,1.5);
		\draw (1.5,-1.5) .. controls (1.5,-1) and (2.5,-1) .. (2.5,-.5)
			--
			(2.5,.5) .. controls (2.5,1) and (1.5,1) .. (1.5,1.5);
		\draw[dashed, pstdhl] (1.5,-1.5) .. controls (1.5,-1) and (2.5,-1) .. (2.5,-.5)
			--
			(2.5,.5) .. controls (2.5,1) and (1.5,1) .. (1.5,1.5);
		\node at(1,-1.4){\small $\dots$};  \node at(2,0){\small $\dots$}; \node at(1,1.4){\small $\dots$};
	\end{scope}
	\begin{scope}[xshift=2cm]
		\draw (.5,-1.5) .. controls (.5,-1) and (1.5,-1) .. (1.5,-.5)
			--
			(1.5,.5) .. controls (1.5,1) and (.5,1) .. (.5,1.5);
		\draw[dashed, pstdhl] (.5,-1.5) .. controls (.5,-1) and (1.5,-1) .. (1.5,-.5)
			--
			(1.5,.5) .. controls (1.5,1) and (.5,1) .. (.5,1.5);
		\draw (1.5,-1.5) .. controls (1.5,-1) and (2.5,-1) .. (2.5,-.5)
			--
			(2.5,.5) .. controls (2.5,1) and (1.5,1) .. (1.5,1.5);
		\draw[dashed, pstdhl] (1.5,-1.5) .. controls (1.5,-1) and (2.5,-1) .. (2.5,-.5)
			--
			(2.5,.5) .. controls (2.5,1) and (1.5,1) .. (1.5,1.5);
		\node at(1,-1.4){\small $\dots$};  \node at(2,0){\small $\dots$}; \node at(1,1.4){\small $\dots$};
	\end{scope}
	\begin{scope}[xshift=5cm]
		\draw (.5,-1.5) .. controls (.5,-1) and (1.5,-1) .. (1.5,-.5)
			--
			(1.5,.5) .. controls (1.5,1) and (.5,1) .. (.5,1.5);
		\draw[dashed, pstdhl] (.5,-1.5) .. controls (.5,-1) and (1.5,-1) .. (1.5,-.5)
			--
			(1.5,.5) .. controls (1.5,1) and (.5,1) .. (.5,1.5);
		\draw (1.5,-1.5) .. controls (1.5,-1) and (2.5,-1) .. (2.5,-.5)
			--
			(2.5,.5) .. controls (2.5,1) and (1.5,1) .. (1.5,1.5);
		\draw[dashed, pstdhl] (1.5,-1.5) .. controls (1.5,-1) and (2.5,-1) .. (2.5,-.5)
			--
			(2.5,.5) .. controls (2.5,1) and (1.5,1) .. (1.5,1.5);
		\node at(1,-1.4){\small $\dots$};  \node at(2,0){\small $\dots$}; \node at(1,1.4){\small $\dots$};
	\end{scope}
	\draw[pstdhl] (0,-1.5) -- (0,-.5) 
			 -- (0,.5) node[pos=.25, nail]{} node[pos=.75,nail]{}
			 -- (0,1.5) node[below]{\small $\mu_1$};
}
\\
&\ = \ 
- (-1)^{c}
\tikzdiagl[yscale=-1]{
	\draw (7,-1.5) .. controls (7,-1) and (.5,-1.25) .. (.5,-.5)
		.. controls (.5,-.375) .. (0,-.25).. controls (1,0) .. 
		(1,.5) node[pos=1,tikzstar]{}  .. controls (1,1) and (7.5,.75) .. (7.5,1.5);
	%
	%
	%
	\draw (2, -1.5) .. controls (2,-1.25) and (2.5,-1.25) .. (2.5,-1) .. controls (2.5,-.75) and (1+.25,-.75) .. (1,-.5) node[pos=1,tikzstar]{}
		.. controls (1,0) .. (0,.25) .. controls (.5,.375) ..
		 (.5,.5) .. controls (.5,1.25) and (7,1) .. (7,1.5);
	%
	%
	%
	\draw (4, -1.5) .. controls (4,-1.25) and (4.5,-1.25) .. (4.5,-1) .. controls (4.5,-.75) and (3,-.75) .. (3,-.5) node[pos=1,tikzstar]{}
		--
		(3,.5)  .. controls (3,1) and (2,1) .. (2,1.5);
	\node at (5.25,0){$\dots$};
	\draw[color=white,xshift=-.2cm,line width=1.6pt] (4.75,-1.5) .. controls (4.75,-1) and (5.75,-1) .. (5.75,-.5);
	\draw[color=white,line width=1.6pt] (4.75,-1.5) .. controls (4.75,-1) and (5.75,-1) .. (5.75,-.5);
	\draw[color=white,xshift=.2cm,line width=1.6pt] (4.75,-1.5) .. controls (4.75,-1) and (5.75,-1) .. (5.75,-.5);
	\draw[color=white,xshift=-.2cm,line width=1.6pt] (4.25,1.5) .. controls (4.25,1) and (5.25,1) .. (5.25,.5);
	\draw[color=white,line width=1.6pt] (4.25,1.5) .. controls (4.25,1) and (5.25,1) .. (5.25,.5);
	\draw[color=white,xshift=.2cm,line width=1.6pt] (4.25,1.5) .. controls (4.25,1) and (5.25,1) .. (5.25,.5);
	\draw (7.5,-1.5) .. controls (7.5,-1) and (6,-1) .. (6,-.5) node[pos=1,tikzstar]{}
		--
		(6,.5) .. controls (6,1) and (5,1) .. (5,1.5);
	\begin{scope}[xshift=0cm]
		\draw (.5,-1.5) .. controls (.5,-1) and (1.5,-1) .. (1.5,-.5)
			--
			(1.5,.5) .. controls (1.5,1) and (.5,1) .. (.5,1.5);
		\draw[dashed, pstdhl] (.5,-1.5) .. controls (.5,-1) and (1.5,-1) .. (1.5,-.5)
			--
			(1.5,.5) .. controls (1.5,1) and (.5,1) .. (.5,1.5);
		\draw (1.5,-1.5) .. controls (1.5,-1) and (2.5,-1) .. (2.5,-.5)
			--
			(2.5,.5) .. controls (2.5,1) and (1.5,1) .. (1.5,1.5);
		\draw[dashed, pstdhl] (1.5,-1.5) .. controls (1.5,-1) and (2.5,-1) .. (2.5,-.5)
			--
			(2.5,.5) .. controls (2.5,1) and (1.5,1) .. (1.5,1.5);
		\node at(1,-1.4){\small $\dots$};  \node at(2,0){\small $\dots$}; \node at(1,1.4){\small $\dots$};
	\end{scope}
	\begin{scope}[xshift=2cm]
		\draw (.5,-1.5) .. controls (.5,-1) and (1.5,-1) .. (1.5,-.5)
			--
			(1.5,.5) .. controls (1.5,1) and (.5,1) .. (.5,1.5);
		\draw[dashed, pstdhl] (.5,-1.5) .. controls (.5,-1) and (1.5,-1) .. (1.5,-.5)
			--
			(1.5,.5) .. controls (1.5,1) and (.5,1) .. (.5,1.5);
		\draw (1.5,-1.5) .. controls (1.5,-1) and (2.5,-1) .. (2.5,-.5)
			--
			(2.5,.5) .. controls (2.5,1) and (1.5,1) .. (1.5,1.5);
		\draw[dashed, pstdhl] (1.5,-1.5) .. controls (1.5,-1) and (2.5,-1) .. (2.5,-.5)
			--
			(2.5,.5) .. controls (2.5,1) and (1.5,1) .. (1.5,1.5);
		\node at(1,-1.4){\small $\dots$};  \node at(2,0){\small $\dots$}; \node at(1,1.4){\small $\dots$};
	\end{scope}
	\begin{scope}[xshift=5cm]
		\draw (.5,-1.5) .. controls (.5,-1) and (1.5,-1) .. (1.5,-.5)
			--
			(1.5,.5) .. controls (1.5,1) and (.5,1) .. (.5,1.5);
		\draw[dashed, pstdhl] (.5,-1.5) .. controls (.5,-1) and (1.5,-1) .. (1.5,-.5)
			--
			(1.5,.5) .. controls (1.5,1) and (.5,1) .. (.5,1.5);
		\draw (1.5,-1.5) .. controls (1.5,-1) and (2.5,-1) .. (2.5,-.5)
			--
			(2.5,.5) .. controls (2.5,1) and (1.5,1) .. (1.5,1.5);
		\draw[dashed, pstdhl] (1.5,-1.5) .. controls (1.5,-1) and (2.5,-1) .. (2.5,-.5)
			--
			(2.5,.5) .. controls (2.5,1) and (1.5,1) .. (1.5,1.5);
		\node at(1,-1.4){\small $\dots$};  \node at(2,0){\small $\dots$}; \node at(1,1.4){\small $\dots$};
	\end{scope}
	\draw[pstdhl] (0,-1.5) -- (0,-.5) 
			 -- (0,.5) node[pos=.25, nail]{} node[pos=.75,nail]{}
			 -- (0,1.5) node[below]{\small $\mu_1$};
}
\end{align*}
by \cref{lem:doublecrossingnailzero}, 
and the element
\[
\ - \ \tikzdiagl
{
	\draw (2.5,-1) 	.. controls (2.5,-.75) .. (0,-.25) 
				.. controls (2.5,.5) .. (2.5,.75)  node[pos=.125, tikzstar]{}
				-- (2.5,1) ;
	\draw (2,1) 		.. controls (2,.75) .. (0,.25) 
		 		.. controls (.7,0) .. (.7,-.25)
				.. controls (.7,-.375) and (.4,-.375) .. (.4,-.6)
				.. controls (.4,-.75) and (2,-.75) .. (2,-1);
	\draw (.5,-1) .. controls (.5,-.5) and (1.5,-.5) .. (1.5,0) .. controls (1.5,.5) and (.5,.5) .. (.5,1);
	\draw[pstdhl,dashed] (.5,-1) .. controls (.5,-.5) and (1.5,-.5) .. (1.5,0) .. controls (1.5,.5) and (.5,.5) .. (.5,1);
	\draw (1.5,-1) .. controls (1.5,-.5) and (2.5,-.5) .. (2.5,0) .. controls (2.5,.5) and (1.5,.5) .. (1.5,1);
	\draw[pstdhl,dashed] (1.5,-1) .. controls (1.5,-.5) and (2.5,-.5) .. (2.5,0) .. controls (2.5,.5) and (1.5,.5) .. (1.5,1);
	\node at (1,-.9){\small $\dots$}; \node at (2,0){\small $\dots$}; \node at (1,.9){\small $\dots$};
	\draw[pstdhl] (0,-1) node[below]{\small $\mu_1$} -- (0,0) node[pos=.75, nail]{}
			 -- (0,1) node[pos=.25, nail]{};
}
\]
Comparing all the remaining terms, we observe that they cancel with each other, concluding the proof.
\end{proof}



\bibliographystyle{bibliography/habbrv}


\end{document}